\documentclass[12pt]{article}

\usepackage{amsmath,amssymb,amsthm}
\usepackage{tikz}
\usepackage{tikz-cd}
\usepackage[mathcal]{eucal}
\usepackage[pagebackref,colorlinks=true,allcolors=black,bookmarksopen,bookmarksdepth=3]{hyperref}
\usepackage[margin=1in]{geometry}

\usepackage{enumitem}
\setenumerate{label=(\roman*),topsep=1pt,itemsep=1pt,partopsep=0pt,parsep=0pt}

\newtheorem{theorem}{Theorem}[section]
\newtheorem{lemma}[theorem]{Lemma}

\newtheorem{proposition}[theorem]{Proposition}

\theoremstyle{definition}
\newtheorem{definition}[theorem]{Definition}
\newtheorem{convention}[theorem]{Convention}

\theoremstyle{remark}
\newtheorem{remark}[theorem]{Remark}
\newtheorem{example}[theorem]{Example}

\newtheoremstyle{dotless}{}{}{}{}{\bfseries}{}{ }{\thmname{#1}\thmnumber{ #2}\thmnote{#3}}
\theoremstyle{dotless}
\newtheorem{setting}{Setup}

\numberwithin{equation}{section}

\newcommand{\st}{\mathrm{st}}
\newcommand{\sign}{\operatorname{sign}}
\newcommand{\g}{\mathbf g}
\newcommand{\N}{\mathcal N}
\newcommand{\C}{\mathcal C}
\newcommand{\D}{\mathcal D}
\newcommand{\M}{\mathcal M}
\newcommand{\1}{\mathbf 1}
\newcommand{\CC}{\mathbb C}
\newcommand{\PPP}{\mathbb P}
\newcommand{\ZZ}{\mathbb Z}
\newcommand{\J}{\mathcal J}
\newcommand{\F}{\mathcal F}
\newcommand{\G}{\mathcal G}
\newcommand{\HH}{\mathcal H}
\newcommand{\X}{\mathbf X}
\newcommand{\cC}{\check C}
\newcommand{\cH}{\check H}
\newcommand{\RR}{\mathbb R}
\newcommand{\SSS}{\mathcal S}
\newcommand{\TTT}{\mathcal T}
\newcommand{\ttt}{\mathfrak t}
\newcommand{\sL}{\mathcal L}
\newcommand{\QQ}{\mathbb Q}
\newcommand{\s}{\mathfrak s}
\newcommand{\op}{\mathrm{op}}
\newcommand{\vir}{\mathrm{vir}}
\newcommand{\CZ}{\operatorname{CZ}}
\newcommand{\ind}{\operatorname{ind}}
\newcommand{\im}{\operatorname{im}}
\newcommand{\colim}{\operatornamewithlimits{colim}}
\newcommand{\End}{\operatorname{End}}
\newcommand{\Hom}{\operatorname{Hom}}
\newcommand{\Diff}{\operatorname{Diff}}
\newcommand{\Cont}{\operatorname{Cont}}
\newcommand{\cyl}{\mathrm{cyl}}
\newcommand{\contr}{\mathrm{contr}}
\newcommand{\glue}{\operatorname{glue}}
\newcommand{\brk}{\operatorname{break}}
\newcommand{\calib}{\operatorname{calib}}
\newcommand{\Sym}{\operatorname{Sym}}
\newcommand{\coker}{\operatorname{coker}}
\newcommand{\id}{\operatorname{id}}
\newcommand{\vdim}{\operatorname{vdim}}
\newcommand{\eext}{\mathrm{ext}}
\newcommand{\eint}{\mathrm{int}}
\newcommand{\ttop}{\mathrm{max}}
\newcommand{\bad}{\mathrm{bad}}
\newcommand{\good}{\mathrm{good}}
\newcommand{\PT}{\mathsf{PT}}
\newcommand{\Mbar}{\overline{\mathcal M}}
\newcommand{\Cbar}{\overline{\mathcal C}}
\newcommand{\I}{\mathrm I}
\newcommand{\II}{\mathrm{II}}
\newcommand{\III}{\mathrm{III}}
\newcommand{\IV}{\mathrm{IV}}
\newcommand{\cof}{\mathrm{cof}}
\newcommand{\nodes}{\mathrm{nodes}}
\newcommand{\Ch}{\mathsf{Ch}}
\newcommand{\Sp}{\operatorname{Sp}}
\newcommand{\gl}{\operatorname{\mathfrak{gl}}}
\newcommand{\Aut}{\operatorname{Aut}}
\newcommand{\aut}{\operatorname{\mathfrak{aut}}}
\newcommand{\Maps}{\operatorname{Maps}}
\newcommand{\codim}{\operatorname{codim}}
\newcommand{\rel}{\operatorname{rel}}
\newcommand{\reg}{\mathrm{reg}}
\newcommand{\PP}{\mathcal P}
\newcommand{\oo}{\mathfrak o}
\newcommand{\hocolim}{\operatornamewithlimits{hocolim}}
\newcommand{\Contact}{\mathfrak{Contact}}
\newcommand{\Exact}{\mathfrak{Exact}}
\newcommand{\Ring}{\mathfrak{Ring}}

\begin{document}

\title{Contact homology and virtual fundamental cycles}

\author{John Pardon\footnote{This research was partially conducted during the period the author served as a Clay Research Fellow.  The author was also partially supported by a National Science Foundation Graduate Research Fellowship under grant number DGE--1147470.}}

\date{28 October 2015\\(revised 4 February 2019)}

\maketitle

\vspace{-1.9\baselineskip}

\begin{abstract}
We give a construction of contact homology in the sense of Eliashberg--Givental--Hofer.
Specifically, we construct coherent virtual fundamental cycles on the relevant compactified moduli spaces of pseudo-holomorphic curves.
\end{abstract}


The aim of this work is to provide a rigorous construction of \emph{contact homology}, an invariant of contact manifolds and symplectic cobordisms due to Eliashberg--Givental--Hofer \cite{eliashbergicm,sftintro}.
The contact homology of a contact manifold $(Y,\xi)$ is defined by counting pseudo-holomorphic curves in the sense of Gromov \cite{gromov} in its symplectization $\RR\times Y$.
The main problem we solve in this paper is simply to give a rigorous definition of these curve counts, the essential difficulty being that the moduli spaces of such curves are usually not cut out transversally.
It is therefore necessary to construct the virtual fundamental cycles of these moduli spaces (which play the same enumerative role that the ordinary fundamental cycles do for transversally cut out moduli spaces).
For this construction, we use the framework developed in \cite{pardonimplicitatlas}.
Our methods are quite general, and apply equally well to many other moduli spaces of interest.

We use a compactification of the relevant moduli spaces which is smaller than the compactification considered in \cite{sftintro,sftcompactness}.  Roughly speaking, for curves in symplectizations $\RR\times Y$, we do not keep track of the relative vertical positions of different components (in particular, no trivial cylinders appear).
Our compactification is more convenient for proving the master equations of contact homology: the codimension one boundary strata in our compactification correspond bijectively with the desired terms in the ``master equations'', whereas the compactification from \cite{sftintro,sftcompactness} contains additional codimension one boundary strata.  If we were to use the compactification from \cite{sftintro,sftcompactness}, we would need to additionally argue that the contribution of each such extra codimension one boundary stratum vanishes.

\begin{remark}[Historical discussion]
The theory of pseudo-holomorphic curves in closed symplectic manifolds was founded by Gromov \cite{gromov}.  Hofer's breakthrough work on the three-dimensional Weinstein conjecture \cite{hoferweinstein} introduced pseudo-holomorphic curves in symplectizations and their relationship with Reeb dynamics.  The analytic theory of such curves was then further developed by Hofer--Wysocki--Zehnder \cite{hwzsymplectizationsI,hwzsymplectizationsIcorr,hwzsymplectizationsII,hwzsymplectizationsIII,hwzsmallarea}.  On the algebraic side, Eliashberg--Givental--Hofer \cite{sftintro} introduced the theories of contact homology and symplectic field theory, based on counts of pseudo-holomorphic curves in symplectizations and symplectic cobordisms (assuming such counts can be defined).  The key compactness results for such curves were established by Bourgeois--Eliashberg--Hofer--Wyzocki--Zehnder \cite{sftcompactness}.  Gluing techniques applicable to such pseudo-holomorphic curves have been developed by many authors, notably Taubes, Donaldson, Floer, Fukaya--Oh--Ohta--Ono, and Hofer--Wysocki--Zehnder.

A number of important special cases of contact homology and closely related invariants have been constructed rigorously using generic and/or automatic transversality techniques.  Legendrian contact homology in $\RR^{2n+1}$ was constructed by Ekholm--Etnyre--Sullivan \cite{ekholmetnyresullivan}.  Eliashberg--Kim--Polterovich \cite{eliashbergkimpolterovich} defined cylindrical contact homology for fiberwise star-shaped open subsets of certain pre-quantization spaces.  Embedded contact homology was introduced and constructed by Hutchings and Hutchings--Taubes \cite{hutchingsindex,hutchingsindexrevisited,hutchingstaubesI,hutchingstaubesII}.  Cylindrical contact homology was constructed by Hutchings--Nelson \cite{hutchingsnelson} for dynamically convex contact three-manifolds and by Bao--Honda \cite{baohonda} for hypertight contact three-manifolds.
Soon after the present paper was released, work of Bao--Honda \cite{baohondaII} appeared, as well as work of Ishikawa \cite{ishikawa}.
\end{remark}

\begin{remark}[Virtual moduli cycle techniques]
The technique of patching together local finite-dimensional reductions to construct virtual fundamental cycles has been applied to moduli spaces of pseudo-holomorphic curves by many authors, including Fukaya--Ono \cite{fukayaono} (Kuranishi structures), Li--Tian \cite{litianII}, Liu--Tian \cite{liutian}, Ruan \cite{ruan}, Fukaya--Oh--Ohta--Ono \cite{FOOOI,FOOOII,foootechnicaldetails,foooshrinking,fooonewI,foooexpdecay,fukayaliegroupoid,fooonewII,foooconstructionI,foooconstructionII}, McDuff--Wehrheim \cite{mcduffwehrheimI,mcduffwehrheimII,mcduffwehrheimIII,mcduffIV} (Kuranishi atlases), Joyce \cite{joycedefinition,joyceI,joyceII} (Kuranishi spaces and d-orbifolds), and \cite{pardonimplicitatlas} (implicit atlases).  Features of the framework from \cite{pardonimplicitatlas} include that (1) it requires only topological (as opposed to smooth) gluing theorems as input and (2) it applies to moduli spaces with rather general corner structures which we term ``cell-like'' (more general than manifold-with-corners).

More recently, the theory of polyfolds developed by Hofer--Wysocki--Zehnder \cite{polyfoldI,polyfoldII,polyfoldIII,polyfoldint,polyfoldscnew,polyfoldGW,polyfoldftI,polyfoldbook,polyfoldnotesI,polyfoldnotesII,polyfoldlectures} promises to provide a robust new infinite-dimensional context in which all reasonable moduli spaces of pseudo-holomorphic curves may be perturbed ``abstractly'' to obtain transversality.

Any one of the above theories, once sufficiently developed, could be used to prove the main results of this paper.  Although these theories vary in their approach to the myriad of technical issues involved, they are expected to give rise to completely equivalent virtual fundamental cycles.
\end{remark}

\textbf{Acknowledgements:} The author thanks Yasha Eliashberg for introducing the author to this problem and for many useful conversations.  The author thanks the referee for their careful reading and extensive comments which have improved this paper significantly.  The author is also grateful for useful discussions with Mohammed Abouzaid, Fr\'ed\'eric Bourgeois, Roger Casals, Vincent Colin, Tobias Ekholm, Kenji Fukaya, Eleny Ionel, Patrick Massot, Rafe Mazzeo, Peter Ozsv\'ath, Mohan Swaminathan, and Chris Wendl.  A preliminary version of this paper appeared as part of the author's 2015 Ph.D.\ thesis at Stanford University supervised by Ya.\ Eliashberg.

\section{Statement of results}

We now state our main results while simultaneously reviewing the definition of contact homology as given in \cite{eliashbergicm,sftintro,sftappl}.
The definition of contact homology involves counting pseudo-holomorphic curves in various different (but closely related) settings, which we distinguish with roman numerals (I), (II), (III), (IV).
\emph{We strongly recommend that the reader restrict their attention to (I) on a first reading.}
The generalization from (I) to (II), (III), (IV) is mostly straightforward, and may be saved for a second reading.

\subsection{Conventions}\label{conventionsintrosec}

Everything is in the smooth category unless stated otherwise.

To keep track of signs, everything is $\ZZ/2$-graded, and $\otimes$ always denotes the \emph{super tensor product}, meaning that the isomorphism $A\otimes B\xrightarrow\sim B\otimes A$ is given by $a\otimes b\mapsto(-1)^{\left|a\right|\left|b\right|}b\otimes a$ and that $(f\otimes g)(a\otimes b):=(-1)^{\left|g\right|\left|a\right|}f(a)\otimes g(b)$.
Complexes are $(\ZZ,\ZZ/2)$-bigraded, and differentials are always odd.
For specificity, let us declare that $\Hom(A,B)\otimes A\to B$ be given by $f\otimes a\mapsto f(a)$, though it won't really matter for our arguments.

An \emph{orientation line} is a $\ZZ/2$-graded free $\ZZ$-module of rank one.
Note that for an odd orientation line $\oo$, there is no nonzero symmetric pairing $\oo\otimes\oo\to\ZZ$, so we must be careful to distinguish between $\oo$ and its dual $\oo^\vee$.

\subsection{(I) The differential}\label{introI}

Let $Y^{2n-1}$ be a closed manifold, and let $\lambda$ be a \emph{contact form} on $Y$, namely a $1$-form such that $\lambda\wedge(d\lambda)^{n-1}$ is non-vanishing, denoting by $\xi:=\ker\lambda$ the induced co-oriented contact structure.
Contact homology can be regarded as an attempt to define the $S^1$-equivariant Morse--Floer homology of the loop space $C^\infty(S^1,Y)$ with respect to the action functional $a(\gamma):=\int_{S^1}\gamma^\ast\lambda$.

Denote by $R_\lambda$ the \emph{Reeb vector field} associated to $\lambda$, defined by the properties $\lambda(R_\lambda)=1$ and $d\lambda(R_\lambda,\cdot)=0$.  Denote by $\PP=\PP(Y,\lambda)$ the collection of (unparameterized) \emph{Reeb orbits}, namely closed trajectories of $R_\lambda$, not necessarily embedded.  A Reeb orbit is called \emph{non-degenerate} iff the associated linearized Poincar\'e return map $\xi_p\to\xi_p$ has no fixed vector.  A contact form is called non-degenerate iff all its Reeb orbits are non-degenerate, which implies that there are at most finitely many Reeb orbits below any given action threshold.
Suppose now that $\lambda$ is non-degenerate.

There is a natural partition $\PP=\PP_\good\sqcup\PP_\bad$, and for each good Reeb orbit $\gamma\in\PP_\good$, there is an associated orientation line $\oo_\gamma$ with parity $\left|\gamma\right|=\sign(\det(I-A_\gamma))\in\{\pm 1\}=\ZZ/2$, where $A_\gamma$ denotes the linearized Poincar\'e return map of $\gamma$ acting on $\xi$ (see \S\ref{thickorsec}).  We set $\oo_\Gamma:=\bigotimes_{\gamma\in\Gamma}\oo_\gamma$ and $\left|\Gamma\right|:=\sum_{\gamma\in\Gamma}\left|\gamma\right|$ for any finite set of good Reeb orbits $\Gamma\to\PP_\good$.  For a given Reeb orbit $\gamma\in\PP$, let $d_\gamma\in\ZZ_{\geq 1}$ denote its covering multiplicity.

Let
\begin{equation}\label{freesupercommutative}
CC_\bullet(Y,\xi)_\lambda:=\bigoplus_{n\geq 0}\Sym^n_\QQ\Bigl(\bigoplus_{\gamma\in\PP_\good}\oo_\gamma\Bigr)
\end{equation}
denote the free supercommutative (meaning $ab=(-1)^{\left|a\right|\left|b\right|}ba$) unital $\ZZ/2$-graded $\QQ$-algebra generated by $\oo_\gamma$ for $\gamma\in\PP_\good$.
This is the contact homology chain complex, which we now equip with a differential.
To define this differential, fix an almost complex structure $J:\xi\to\xi$ which is \emph{compatible} with $d\lambda$, meaning that $d\lambda(\cdot,J\cdot)$ is a positive definite symmetric pairing on $\xi$.

Let $\hat Y:=\RR\times Y$ (with coordinate $s\in\RR$) denote the \emph{symplectization}\footnote{More intrinsically, the symplectization of a co-oriented contact manifold $(Y,\xi)$ is defined as the total space of the bundle of $1$-forms with kernel $\xi$, namely $\hat Y:=\ker(T^\ast Y\to\xi^\ast)_+$.  The restriction of the tautological Liouville $1$-form on $T^\ast Y$ is a Liouville $1$-form $\hat\lambda$ on $\hat Y$; the associated Liouville vector field on $\hat Y$ generates an $\RR$-action on $\hat Y$ which is simply scaling by $e^s$.  A choice of contact form $\lambda$ for $\xi$ induces an identification of $(\hat Y,\hat\lambda)$ with $(\RR\times Y,e^s\lambda)$.} of $Y$, which is equipped with the Liouville form $\hat\lambda:=e^s\lambda$.  Now $J$ induces an $\RR$-invariant almost complex structure $\hat J$ on $\hat Y$ defined by the property that $\hat J(\partial_s)=R_\lambda$ and $\hat J|_\xi=J$.  Given a Reeb orbit $\gamma^+\in\PP$ and a finite set of Reeb orbits $\Gamma^-\to\PP$, let
\begin{equation}
\pi_2(Y,\gamma^+\sqcup\Gamma^-):=[(S,\partial S),(Y,\gamma^+\sqcup\Gamma^-)]/\!\Diff(S,\partial S),
\end{equation}
where $S$ is any compact connected oriented surface of genus zero with boundary, equipped with a homeomorphism between $\partial S$ and $\gamma^+\sqcup\Gamma^-$ (preserving orientation on $\gamma^+$ and reversing orientation on $\Gamma^-$), and $\Diff(S,\partial S)$ denotes diffeomorphisms of $S$ fixing $\partial S$ pointwise.

Let $\Mbar_\I(\gamma^+,\Gamma^-;\beta)_J$ denote the compactified moduli space of connected $\hat J$-holomorphic curves of genus zero in $\hat Y$ modulo $\RR$-translation, with one positive puncture asymptotic to $\gamma^+$ and negative punctures asymptotic to $\Gamma^-$, in the homotopy class $\beta$, along with asymptotic markers on the domain mapping to fixed basepoints on $\gamma^+$ and $\Gamma^-$ (see \S\S\ref{Tcatdefsec}--\ref{gromovtopology}).  We denote by $\mu(\gamma^+,\Gamma^-;\beta)\in\ZZ$ the \emph{index} of this moduli problem; we have $\mu(\gamma^+,\Gamma^-;\beta)\equiv\left|\gamma^+\right|-\left|\Gamma^-\right|\in\ZZ/2$ (see \S\ref{indexsec}).

It is shown in \cite{sftcompactness} that $\bigsqcup_{(\Gamma^-,\beta)}\Mbar_\I(\gamma^+,\Gamma^-;\beta)_J$ is compact for any fixed $\gamma^+$ (see \S\ref{compactness}).

We say that $\Mbar_\I(\gamma^+,\Gamma^-;\beta)_J$ is \emph{regular} (or \emph{cut out transversally}, or just \emph{transverse}) iff the relevant linearized operator (which takes into account both variations in the map and variations in the almost complex structure of the domain) is everywhere surjective (see \S\ref{linearizedopsec}).
In this case, $\Mbar_\I(\gamma^+,\Gamma^-;\beta)_J$ is a manifold with corners of dimension $\mu(\gamma^+,\Gamma^-;\beta)-1$, whose orientation local system is naturally isomorphic to $\oo_{\gamma^+}\otimes\oo_{\Gamma^-}^\vee\otimes\oo_\RR^\vee$ (see \S\ref{firstgluingstatements}); there are no orbifold points due to our use of asymptotic markers and to the fact that the symplectization is exact so there can be no nodes.
In particular, if $\mu(\gamma^+,\Gamma^-;\beta)=1$ and $\Mbar_\I(\gamma^+,\Gamma^-;\beta)_J$ is cut out transversally, the moduli count
\begin{equation}
\#\Mbar_\I(\gamma^+,\Gamma^-;\beta)_J\in\oo_{\gamma^+}^\vee\otimes\oo_{\Gamma^-}\otimes\oo_\RR
\end{equation}
is well-defined.

We would now like to define a differential
\begin{equation}
d_J:CC_\bullet(Y,\xi)_\lambda\to CC_{\bullet-1}(Y,\xi)_\lambda
\end{equation}
which satisfies the Leibniz rule ($d(1)=0$ and $d(ab)=da\cdot b+(-1)^{\left|a\right|}a\cdot db$) and is defined by the property that it acts on $\oo_{\gamma^+}$ by pairing with
\begin{equation}
\sum_{\begin{smallmatrix}\Gamma^-\to\PP_\good\\\mu(\gamma^+,\Gamma^-;\beta)=1\end{smallmatrix}}\frac 1{\left|\Aut\right|}\cdot\#\Mbar_\I(\gamma^+,\Gamma^-;\beta)_J\cdot\Gamma^-
\end{equation}
and contracting on the left with a chosen orientation of $\RR$ (where $\Aut$ denotes a certain group of automorphisms of the triple $(\gamma^+,\Gamma^-;\beta)$ acting as the identity on $\gamma^+$).
When the moduli spaces $\Mbar_\I(\gamma^+,\Gamma^-;\beta)_J$ of index $\leq 2$ are cut out transversally, this definition of $d_J$ makes sense, and consideration of the fact that $\#\partial\Mbar_\I(\gamma^+,\Gamma^-;\beta)_J=0$ for index two moduli spaces (since they are compact $1$-manifolds) shows that $d_J$ squares to zero.

Our main result, Theorem \ref{main}, allows the above construction of the differential to go through without any transversality assumptions.
More precisely, it provides a non-empty set $\Theta_\I=\Theta_\I(Y,\lambda,J)$, an element of which may be thought of as a specification of ``perturbation data'' for the moduli spaces $\Mbar_\I(\gamma^+,\Gamma^-;\beta)_J$, and for each $\theta\in\Theta_\I$ it defines \emph{rational} ``virtual moduli counts''
\begin{equation}
\#\Mbar_\I(\gamma^+,\Gamma^-;\beta)_{J,\theta}^\vir\in\oo_{\gamma^+}^\vee\otimes\oo_{\Gamma^-}\otimes\oo_\RR\otimes\QQ.
\end{equation}
These virtual moduli counts furthermore satisfy $\#\partial\Mbar_\I(\gamma^+,\Gamma^-;\beta)_{J,\theta}^\vir=0$ and coincide with the ordinary moduli counts when the moduli spaces are transverse.
We thus obtain a differential
\begin{equation}
d_{J,\theta}:CC_\bullet(Y,\xi)_\lambda\to CC_{\bullet-1}(Y,\xi)_\lambda,
\end{equation}
and we denote the resulting homology by $CH_\bullet(Y,\xi)_{\lambda,J,\theta}$.

\subsection{(II) The cobordism map}

Let $\hat X^{2n}$ be a manifold and let $\hat\omega$ be a symplectic form on $\hat X$, namely a closed $2$-form such that $\hat\omega^n$ is non-vanishing.
Let $(Y^\pm,\lambda^\pm)$ be closed manifolds equipped with contact forms, and let
\begin{align}
\label{endmarkingsI}([N,\infty)\times Y^+,d\hat\lambda^+)&\to(\hat X,\hat\omega),\\
\label{endmarkingsII}((-\infty,-N]\times Y^-,d\hat\lambda^-)&\to(\hat X,\hat\omega),
\end{align}
(any large $N<\infty$) be diffeomorphisms onto their (disjoint) images, proper, and such that together they cover a neighborhood of infinity.
A \emph{symplectic cobordism} is a symplectic manifold $(\hat X,\hat\omega)$ equipped with such markings \eqref{endmarkingsI}--\eqref{endmarkingsII}, and it is called \emph{exact} when there is a globally defined primitive $\hat\lambda$ for $\hat\omega$ coinciding with $\hat\lambda^\pm$ in the ends.

Fix an almost complex structure $\hat J:T\hat X\to T\hat X$ which is \emph{tamed} by $\hat\omega$, meaning $\hat\omega(v,\hat Jv)>0$ for nonzero $v\in T\hat X$, such that outside a compact set, $\hat J$ coincides with some $\hat J^\pm$ for $J^\pm:\xi^\pm\to\xi^\pm$.
Also suppose that $\lambda^\pm$ are non-degenerate.

Let $\Mbar_\II(\gamma^+,\Gamma^-;\beta)_{\hat J}$ denote the compactified moduli space of connected $\hat J$-holomorphic curves of genus zero in $\hat X$ from $\gamma^+\in\PP^+=\PP(Y^+,\lambda^+)$ to $\Gamma^-\to\PP^-=\PP(Y^-,\lambda^-)$ in the homotopy class $\beta\in\pi_2(\hat X,\gamma^+\sqcup\Gamma^-)$.

It is shown in \cite{sftcompactness} that $\bigsqcup_{\langle\hat\omega,\beta\rangle<N}\Mbar_\II(\gamma^+,\Gamma^-;\beta)_{\hat J}$ is compact for any fixed $\gamma^+$, $\Gamma^-$, and $N<\infty$.
The condition $\langle\hat\omega,\beta\rangle<N$ may be interpreted with respect to any fixed null-homology of $[\gamma^+]-[\Gamma^-]$.
Note that when the symplectic cobordism is exact, it follows that $\bigsqcup_{(\Gamma^-,\beta)}\Mbar_\II(\gamma^+,\Gamma^-;\beta)_{\hat J}$ is compact for any fixed $\gamma^+$.

Theorem \ref{main} provides virtual moduli counts for the moduli spaces $\Mbar_\II$ in the following form.
There is a set $\Theta_\II=\Theta_\II(\hat X,\hat\omega,\hat J,\eqref{endmarkingsI},\eqref{endmarkingsII})$ together with a surjective forgetful map $\Theta_\II\twoheadrightarrow\Theta_\I^+\times\Theta_\I^-$, where $\Theta_\I^\pm=\Theta_\I(Y^\pm,\lambda^\pm,J^\pm)$, along with virtual moduli counts
\begin{equation}
\#\Mbar_\II(\gamma^+,\Gamma^-,\beta)_{\hat J,\theta}^\vir\in\oo_{\gamma^+}^\vee\otimes\oo_{\Gamma^-}\otimes\QQ
\end{equation}
for $\theta\in\Theta_\II$ and $\mu(\gamma^+,\Gamma^-,\beta)=0$.

Now for exact symplectic cobordisms, we may define a unital $\QQ$-algebra map
\begin{equation}
\Phi(\hat X,\hat\lambda)_{\hat J,\theta}:CC_\bullet(Y^+,\xi^+)_{\lambda^+,J^+,\theta^+}\to CC_\bullet(Y^-,\xi^-)_{\lambda^-,J^-,\theta^-}
\end{equation}
for any $\theta\in\Theta_\II$ mapping to $(\theta^+,\theta^-)\in\Theta_\I^+\times\Theta_\I^-$, by pairing with the virtual moduli counts as before (exactness ensures the relevant sum is finite).
This is a chain map by virtue of the fact that $\#\partial\Mbar_\II=0$.

\subsection{(III) The deformation homotopy}

Let $(\hat X,\hat\omega^t)_{t\in[0,1]}$ be a one-parameter family of symplectic cobordisms, fixed near infinity (meaning \eqref{endmarkingsI}--\eqref{endmarkingsII} are independent of $t$).
Note that the \emph{a priori} more general setup where \eqref{endmarkingsI}--\eqref{endmarkingsII} are allowed to vary with $t$ is easily reduced to the case of being fixed near infinity, simply by conjugating by an appropriate family of diffeomorphisms $\varphi^t:\hat X\to\hat X$.

Fix a family $\hat J^t$ of almost complex structures as above, agreeing with $\hat J^\pm$ outside of a compact subset independent of $t$.

We consider moduli spaces $\Mbar_\III(\{\gamma^+_i,\Gamma^-_i;\beta_i\}_{i\in I})_{\hat J^t}$ parameterizing a choice of $t\in[0,1]$ together with a $\hat J^t$-holomorphic curve from $\gamma^+_i$ to $\Gamma^-_i$ in homotopy class $\beta_i$ for every $i\in I$.
For such moduli spaces, there are two natural notions of transversality: a point is called \emph{regular} iff the associated linearized operator is surjective as before, and a point is called \emph{weakly regular} iff the associated linearized operator becomes surjective once we also take into account variations in $t\in(0,1)$ (but \emph{not} at the endpoints of this interval).

It is shown in \cite{sftcompactness} that $\bigsqcup_{\sup_t\langle\hat\omega^t,\beta\rangle<N}\Mbar_\III(\gamma^+,\Gamma^-;\beta)_{\hat J}$ is compact for any fixed $\gamma^+$, $\Gamma^-$, and $N<\infty$.

Theorem \ref{main} provides virtual moduli counts
\begin{equation}
\#\Mbar_\III(\{\gamma^+_i,\Gamma^-_i;\beta_i\}_{i\in I})_{\hat J^t,\theta}^\vir\in\bigotimes_{i\in I}\oo_{\gamma^+_i}^\vee\otimes\oo_{\Gamma^-_i}\otimes\oo_{[0,1]}^\vee\otimes\QQ
\end{equation}
for $\theta\in\Theta_\III:=\Theta_\III(\hat X,(\hat\omega^t,\hat J^t)_{t\in[0,1]})$ with a surjective map $\Theta_\III\twoheadrightarrow\Theta_\II^{t=0}\times_{\Theta^+\times\Theta^-}\Theta_\II^{t=1}$.

For families of exact symplectic cobordisms, pairing with these virtual moduli counts (and an orientation of $[0,1]$) defines a $\QQ$-linear map
\begin{equation}
K(\hat X,\hat\lambda^t)_{\hat J^t,\theta}:CC_\bullet(Y^+,\xi^+)_{\lambda^+,J^+,\theta^+}\to CC_{\bullet+1}(Y^-,\xi^-)_{\lambda^-,J^-,\theta^-}.
\end{equation}
Note that, whereas in cases (I) and (II), we defined the differential and the cobordism map on algebra generators and then extended using the multiplicative structure, here in case (III) we define the homotopy on each monomial separately.

Consideration of $\#\partial\Mbar_\III=0$ shows that $K(\hat X,\hat\lambda^t)_{\hat J^t,\theta}$ is a chain homotopy between $\Phi(\hat X,\hat\lambda^0)_{\hat J^0,\theta^0}$ and $\Phi(\hat X,\hat\lambda^1)_{\hat J^1,\theta^1}$, implying that the induced maps on homology
\begin{equation}
\begin{tikzcd}[column sep = large]
CH_\bullet(Y^+,\xi^+)_{\lambda^+,J^+,\theta^+}\ar[shift left=0.6ex]{r}{\Phi(\hat X,\hat\lambda^0)_{\hat J^0,\theta^0}}\ar[shift left=-0.6ex]{r}[swap]{\Phi(\hat X,\hat\lambda^1)_{\hat J^1,\theta^1}}&CH_\bullet(Y^-,\xi^-)_{\lambda^-,J^-,\theta^-}
\end{tikzcd}
\end{equation}
coincide.

\subsection{(IV) The composition homotopy}\label{introIV}

Let $(\hat X^{01},\hat\omega^{01})$ be a symplectic cobordism with positive end $(Y^0,\lambda^0)$ and negative end $(Y^1,\lambda^1)$, and let $(\hat X^{12},\hat\omega^{12})$ be a symplectic cobordism with positive end $(Y^1,\lambda^1)$ and negative end $(Y^2,\lambda^2)$.
Given any sufficiently large $t<\infty$, we can form a symplectic cobordism $(\hat X^{02,t},\hat\omega^{02,t})$ by truncating the negative end of $\hat X^{01}$ to $(-t,0]\times Y_1$, truncating the positive end of $\hat X^{12}$ to $[0,t)\times Y_1$, and identifying these truncated ends by translation by $t$.
Under this identification, the forms $\hat\omega^{01}$ and $\hat\omega^{12}$ only agree up to a scaling factor of $e^t$, so their descent $\hat\omega^{02,t}$ to $\hat X^{02,t}$ is only well-defined up to scale; similarly the natural end markings \eqref{endmarkingsI}--\eqref{endmarkingsII} of $\hat X^{02,t}$ only respect symplectic forms up to scale (these ambiguities are never problematic, however, so they will be ignored from now on).
Let $(\hat X^{02,t},\hat\omega^{02,t})_{t\in[0,\infty)}$ denote this one-parameter family of symplectic cobordisms with positive end $(Y^0,\lambda^0)$ and negative end $(Y^2,\lambda^2)$.

Fix almost complex structures $\hat J^{01}$ and $\hat J^{12}$ on $\hat X^{01}$ and $\hat X^{12}$ as before, agreeing with fixed $\hat J^0$, $\hat J^1$, $\hat J^2$ near infinity.
These descend naturally to $\hat J^{02,t}$ on $\hat X^{02,t}$ for sufficiently large $t<\infty$, and we fix an extension $\hat J^{02,t}$ to all $t\in[0,\infty)$, agreeing with $\hat J^0$ and $\hat J^2$ near infinity.

We consider moduli spaces $\Mbar_\IV(\{\gamma^+_i,\Gamma^-_i;\beta_i\}_{i\in I})_{\hat J^{02,t}}$ parameterizing a choice of $t\in[0,\infty]$ together with a $\hat J^{02,t}$-holomorphic curve in $\hat X^{02,t}$ from $\gamma^+_i$ to $\Gamma^-_i$ in homotopy class $\beta_i$ for every $i\in I$.

Theorem \ref{main} provides virtual moduli counts
\begin{equation}
\#\Mbar_\IV(\{\gamma^+_i,\Gamma^-_i;\beta_i\}_{i\in I})_{\hat J^t,\theta}^\vir\in\bigotimes_{i\in I}\oo_{\gamma^+_i}^\vee\otimes\oo_{\Gamma^-_i}\otimes\oo_{[0,\infty]}^\vee\otimes\QQ
\end{equation}
for $\theta\in\Theta_\IV$ with a surjective map $\Theta_\IV\twoheadrightarrow\Theta_\II^{02}\times_{\Theta_\I^0\times\Theta_\I^2}(\Theta_\II^{01}\times_{\Theta_\I^1}\Theta_\II^{12})$.

In the exact setting, pairing with these virtual moduli counts and using $\#\partial\Mbar_\IV=0$ produces a chain homotopy which shows that the following diagram commutes:
\begin{equation}
\begin{tikzcd}[row sep = small]
&CH_\bullet(Y^1,\xi^1)_{\lambda^1,J^1,\theta^1}\ar{rd}{\Phi(\hat X^{12},\hat\lambda^{12})_{\hat J^{12},\theta^{12}}}\\
CH_\bullet(Y^0,\xi^0)_{\lambda^0,J^0,\theta^0}\ar{rr}[swap]{\Phi(\hat X^{02},\hat\lambda^{02})_{\hat J^{02},\theta^{02}}}\ar{ru}{\Phi(\hat X^{01},\hat\lambda^{01})_{\hat J^{01},\theta^{01}}}&&CH_\bullet(Y^2,\xi^2)_{\lambda^2,J^2,\theta^2}
\end{tikzcd}
\end{equation}
for any $(\theta^{0,1,2},\theta^{01,12,02})\in\Theta_\II^{02}\times_{\Theta_\I^0\times\Theta_\I^2}(\Theta_\II^{01}\times_{\Theta_\I^1}\Theta_\II^{12})$.

\subsection{Main result}

We now give a precise statement of our main result, Theorem \ref{main}, which has already been alluded to in \S\S\ref{introI}--\ref{introIV} above.
This result takes as input a datum $\D$ (as in any of Setups \ref{setupI}--\ref{setupIV} below) and produces a set $\Theta(\D)$ together with virtual moduli counts $\#\Mbar{\vphantom{\M}}_\theta^\vir$ for $\theta\in\Theta(\D)$ satisfying certain properties.

\begin{setting}\label{setupI}
consists of a closed manifold $Y$ equipped with a non-degenerate contact form $\lambda$ and a $d\lambda$-compatible almost complex structure $J:\xi\to\xi$.
\end{setting}

\begin{setting}\label{setupII}
consists of a symplectic cobordism $(\hat X,\hat\omega)$ with positive end $(Y^+,\lambda^+)$ and negative end $(Y^-,\lambda^-)$, together with an $\hat\omega$-tame almost complex structure $\hat J:T\hat X\to T\hat X$ agreeing with $\hat J^\pm$ outside a compact set, where $(Y^\pm,\lambda^\pm,J^\pm)$ are as in Setup \ref{setupI}.
\end{setting}

\begin{setting}\label{setupIII}
consists of a one-parameter family of symplectic cobordisms $(\hat X,\hat\omega^t)_{t\in[0,1]}$ (fixed near infinity) with positive/negative ends $(Y^\pm,\lambda^\pm)$, together with $\hat\omega^t$-tame almost complex structures $\hat J^t$ agreeing with $\hat J^\pm$ outside a compact set, where $(Y^\pm,\lambda^\pm,J^\pm)$ are as in Setup \ref{setupI}.
\end{setting}

\begin{setting}\label{setupIV}
consists of a one-parameter family of symplectic cobordisms $(\hat X^{02,t},\hat\omega^{02,t})_{t\in[0,\infty)}$ with ends $(Y^0,\lambda^0)$ and $(Y^2,\lambda^2)$, which for sufficiently large $t$ agrees with the $t$-gluing of a symplectic cobordism $(\hat X^{01},\hat\omega^{01})$ with ends $(Y^0,\lambda^0)$ and $(Y^1,\lambda^1)$ and a symplectic cobordism $(\hat X^{12},\hat\omega^{12})$ with ends $(Y^1,\lambda^1)$ and $(Y^2,\lambda^2)$, together with an appropriately compatible family of almost complex structures $\hat J^{02,t}$, with $(Y^i,\lambda^i,J^i)$ as in Setup \ref{setupI}.
\end{setting}

\begin{theorem}\label{main}
Let $\D$ be a datum as in any of Setups \ref{setupI}--\ref{setupIV}.
There exists a set $\Theta(\D)$ along with \emph{rational} virtual moduli counts $\#\Mbar{\vphantom{\M}}_\theta^\vir$ satisfying the following properties:
\begin{enumerate}
\item$\Theta(\D)$ and $\#\Mbar{\vphantom{\M}}_\theta^\vir$ are functorial in $\D$, meaning that an isomorphism $i:\D\to\D'$ induces an isomorphism $i_\ast:\Theta(\D)\xrightarrow\sim\Theta(\D')$ such that $\id_\ast=\id$, $(i\circ j)_\ast=i_\ast\circ j_\ast$, and $\#\Mbar{\vphantom{\M}}_\theta^\vir=\#\Mbar{\vphantom{\M}}_{i_\ast\theta}^\vir$.
\item There are natural surjective forgetful maps
\begin{align}
\Theta_\I&\twoheadrightarrow\ast,\\
\Theta_\II&\twoheadrightarrow\Theta_\I^+\times\Theta_\I^-,\\
\Theta_\III&\twoheadrightarrow\Theta_\II^0\times_{\Theta_\I^+\times\Theta_\I^-}\Theta_\II^1,\\
\Theta_\IV&\twoheadrightarrow\Theta_\II^{02}\times_{\Theta_\I^0\times\Theta_\I^2}(\Theta_\II^{01}\times_{\Theta_\I^1}\Theta_\II^{12}),
\end{align}
covering the natural forgetful maps of data $\D$.
\item\label{mainindex}$\#\Mbar{\vphantom{\M}}_\theta^\vir$ is nonzero only for moduli spaces of virtual dimension zero.
\item\label{maintransverse}$\#\Mbar{\vphantom{\M}}_\theta^\vir=\#\Mbar$ for (weakly) regular moduli spaces of dimension zero.
In particular, if a moduli space is empty then its virtual count is zero.
\item The virtual moduli counts satisfy the ``master equation''
\begin{equation}\label{master}
\#\partial\Mbar{\vphantom{\M}}^\vir_\theta=0,
\end{equation}
where the left hand side denotes the sum over all codimension one boundary strata of the relevant products of virtual moduli counts and inverse covering multiplicities of intermediate orbits (this sum is finite by compactness).
\end{enumerate}
\end{theorem}

\subsection{The contact homology functor}\label{wrapup}

Contact homology can be viewed as a functor \eqref{contacthomologyfunctor}, whose (mostly formal) construction from Theorem \ref{main} we now describe.

Let $(\Contact,\Exact)_n$ denote the category whose objects are closed co-oriented contact manifolds $(Y^{2n-1},\xi)$ and whose morphisms are deformation classes of exact symplectic cobordisms $(\hat X^{2n},\hat\lambda)$ between contact manifolds.  Let $\Ring^{\ZZ/2}_\QQ$ denote the category whose objects are supercommutative $\ZZ/2$-graded unital $\QQ$-algebras and whose morphisms are graded unital $\QQ$-algebra homomorphisms.  Contact homology is a symmetric monoidal functor
\begin{equation}\label{contacthomologyfunctor}
CH_\bullet:(\Contact,\Exact)_n^\sqcup\to(\Ring^{\ZZ/2}_\QQ)^\otimes.
\end{equation}
The symmetric monoidal structure on $(\Contact,\Exact)_n$ is disjoint union $\sqcup$, and the symmetric monoidal structure on $\Ring^{\ZZ/2}_\QQ$ is the super tensor product $\otimes$ ($A\otimes B$ is endowed with the multiplication $(a\otimes b)(a'\otimes b'):=(-1)^{\left|a'\right|\left|b\right|}aa'\otimes bb'$, and the isomorphism $A\otimes B\xrightarrow\sim B\otimes A$ is given by $a\otimes b\mapsto(-1)^{\left|a\right|\left|b\right|}b\otimes a$).

Concretely, to define \eqref{contacthomologyfunctor} we should specify:
\begin{itemize}
\item For every co-oriented contact manifold $(Y,\xi)$, a supercommutative $\ZZ/2$-graded unital $\QQ$-algebra $CH_\bullet(Y,\xi)$.
\item For every exact symplectic cobordism $(\hat X,\hat\lambda)$ from $(Y^+,\xi^+)$ to $(Y^-,\xi^-)$, a graded unital $\QQ$-algebra map $\Phi(\hat X,\hat\lambda):CH_\bullet(Y^+,\xi^+)\to CH_\bullet(Y^-,\xi^-)$.
\item Isomorphisms $CH_\bullet(Y,\xi)\otimes CH_\bullet(Y',\xi')=CH_\bullet(Y\sqcup Y',\xi\sqcup\xi')$.
\end{itemize}
such that:
\begin{itemize}
\item The morphism $\Phi(\hat X,\hat\lambda)$ depends only on the deformation class of $(\hat X,\hat\lambda)$.
\item The morphism associated to the identity cobordism is the identity map.
\item The morphism $\Phi(\hat X^{02},\hat\lambda^{02})$ associated to a composition of exact symplectic cobordisms $\hat X^{02}=\hat X^{01}\#\hat X^{12}$ coincides with the composition $\Phi(\hat X^{12},\hat\lambda^{12})\circ\Phi(\hat X^{01},\hat\lambda^{01})$.
\item The isomorphisms $CH_\bullet(Y,\xi)\otimes CH_\bullet(Y',\xi')=CH_\bullet(Y\sqcup Y',\xi\sqcup\xi')$ are commutative, associative, and compatible with the cobordism maps.
\end{itemize}
The construction is as follows.

Theorem \ref{main} applied to Setup \ref{setupI} provides a supercommutative $\ZZ/2$-graded unital $\QQ$-algebra
\begin{equation}\label{invariantI}
CH_\bullet(Y,\xi)_{\lambda,J,\theta}
\end{equation}
for any co-oriented contact manifold $(Y,\xi)$ with non-degenerate contact form $\lambda$, admissible almost complex structure $J$, and $\theta\in\Theta_\I(Y,\lambda,J)$.

Theorem \ref{main} applied to Setup \ref{setupII} provides a graded unital $\QQ$-algebra map
\begin{equation}\label{invariantII}
CH_\bullet(Y^+,\xi^+)_{\lambda^+,J^+,\theta^+}\xrightarrow{\Phi(\hat X,\hat\lambda)_{\hat J,\theta}}CH_\bullet(Y^-,\xi^-)_{\lambda^-,J^-,\theta^-}
\end{equation}
for any exact symplectic cobordism $(\hat X,\hat\lambda)$ with $\lambda^\pm$ non-degenerate, admissible almost complex structure $\hat J$ coinciding with $\hat J^\pm$ near infinity, and $\theta\in\Theta_\II(\hat X,\hat\lambda,\hat J)$ mapping to $\theta^\pm\in\Theta_\I^\pm$.

Theorem \ref{main} applied to Setup \ref{setupIII} shows that the following two maps coincide:
\begin{equation}\label{invariantIII}
\begin{tikzcd}[column sep = large]
CH_\bullet(Y^+,\xi^+)_{\lambda^+,J^+,\theta^+}\ar[shift left=0.6ex]{r}{\Phi(\hat X,\hat\lambda^0)_{\hat J^0,\theta^0}}\ar[shift left=-0.6ex]{r}[swap]{\Phi(\hat X,\hat\lambda^1)_{\hat J^1,\theta^1}}&CH_\bullet(Y^-,\xi^-)_{\lambda^-,J^-,\theta^-}.
\end{tikzcd}
\end{equation}
Note that this immediately implies that $\Phi(\hat X,\hat\lambda)_{\hat J,\theta}$ is independent of $\hat J$ and $\theta$, and depends only on the deformation class of $(\hat X,\hat\lambda)$.  Thus we may rewrite \eqref{invariantII} as
\begin{equation}\label{invariantIIa}
CH_\bullet(Y^+,\xi^+)_{\lambda^+,J^+,\theta^+}\xrightarrow{\Phi(\hat X,\hat\lambda)}CH_\bullet(Y^-,\xi^-)_{\lambda^-,J^-,\theta^-}.
\end{equation}

Theorem \ref{main} applied to Setup \ref{setupIV} shows that the following diagram commutes:
\begin{equation}\label{invariantIV}
\begin{tikzcd}[row sep = small]
&CH_\bullet(Y^1,\xi^1)_{\lambda^1,J^1,\theta^1}\ar{rd}{\Phi(\hat X^{12},\hat\lambda^{12})}\\
CH_\bullet(Y^0,\xi^0)_{\lambda^0,J^0,\theta^0}\ar{rr}[swap]{\Phi(\hat X^{02},\hat\lambda^{02})}\ar{ru}{\Phi(\hat X^{01},\hat\lambda^{01})}&&CH_\bullet(Y^2,\xi^2)_{\lambda^2,J^2,\theta^2}.
\end{tikzcd}
\end{equation}

\begin{lemma}\label{isolemma}
Let $(Y,\xi)$ be a co-oriented contact manifold with two non-degenerate contact forms $\lambda^+,\lambda^-$.
Let $(\hat X,\hat\lambda)$ denote the trivial exact symplectic cobordism with positive end $(Y,\lambda^+)$ and negative end $(Y,\lambda^-)$ (namely, $\hat X$ is simply the symplectization $\hat Y$ marked appropriately).
The map
\begin{equation}
CH_\bullet(Y,\xi)_{\lambda^+,J^+,\theta^+}\xrightarrow{\Phi(\hat X,\hat\lambda)}CH_\bullet(Y,\xi)_{\lambda^-,J^-,\theta^-}
\end{equation}
is an isomorphism for any $J^\pm$ and $\theta^\pm$.
\end{lemma}

\begin{proof}
In view of the commutativity of \eqref{invariantIV}, it suffices to treat the case $\lambda^+=\lambda^-=\lambda$ and $J^+=J^-=J$.

Choose the $\RR$-invariant almost complex structure $\hat J=\hat J^\pm$ on $\hat X$, and choose any $\theta\in\Theta_\II$ mapping to $\theta^\pm$.  We will show that the map on chains
\begin{equation}\label{chainisomorphism}
CC_\bullet(Y,\xi)_{\lambda,J,\theta^+}\xrightarrow{\Phi(\hat X,\hat\lambda)_{\hat J,\theta}}CC_\bullet(Y,\xi)_{\lambda,J,\theta^-}
\end{equation}
is an isomorphism, which is clearly sufficient.

We consider the ascending filtration on both sides of \eqref{chainisomorphism} whose $^{\leq(a,k)}$ filtered piece is the $\QQ$-subspace generated by all monomials of Reeb orbits with total action $<a$ or total action $=a$ and degree $\geq k$.  We claim that the map \eqref{chainisomorphism} as well as the differentials on both its domain and codomain all respect this filtration.  Indeed, the integral of $d\lambda$ over any $\hat J$-holomorphic curve is $\geq 0$, with equality iff the curve is a branched cover of a trivial cylinder (note the difference between $d\lambda$ and the symplectic form $d\hat\lambda=d(e^s\lambda)=e^s(d\lambda+ds\wedge\lambda)$).  Every branched cover of a trivial cylinder has at least one negative end, which proves the claim.  Since the filtration is well-ordered, to show that \eqref{chainisomorphism} is an isomorphism, it suffices to show that the induced map on associated gradeds is an isomorphism.

The curves contributing to the action of \eqref{chainisomorphism} on associated gradeds are the branched covers of trivial cylinders with exactly one negative end, and such curves are themselves necessarily trivial cylinders by Riemann--Hurwitz.  Since there is exactly one such trivial cylinder for every Reeb orbit, it suffices to show that trivial cylinders are cut out transversally.  This is a standard fact, whose proof we recall in Lemma \ref{trivialtransverse}.
\end{proof}

Now for a contact manifold $(Y,\xi)$, all groups $CH_\bullet(Y,\xi)_{\lambda,J,\theta}$ are canonically isomorphic via the morphisms $\Phi(\hat X,\hat\lambda)$ associated to the trivial cobordisms (by Lemma \ref{isolemma} and the commutativity of \eqref{invariantIV}).  Thus we get a well-defined object
\begin{equation}
CH_\bullet(Y,\xi),
\end{equation}
independent of $(\lambda,J,\theta)$.  Formally speaking, $CH_\bullet(Y,\xi)$ is the limit (and the colimit) of $\{CH_\bullet(Y,\xi)_{\lambda,J,\theta}\}_{\lambda,J,\theta}$, which is attained at any particular triple $(\lambda,J,\theta)$.  Note that for any contact structure, the set of non-degenerate contact forms is generic (and in particular non-empty).

The commutativity of \eqref{invariantIV} also implies that a deformation class of exact symplectic coboordism $(\hat X,\hat\lambda)$ from $(Y^+,\xi^+)$ to $(Y^-,\xi^-)$ induces a well-defined graded unital $\QQ$-algebra map
\begin{equation}\label{defII}
\Phi(\hat X,\hat\lambda):CH_\bullet(Y^+,\xi^+)\to CH_\bullet(Y^-,\xi^-),
\end{equation}
and that $\Phi(\hat X^{02},\hat\lambda^{02})=\Phi(\hat X^{12},\hat\lambda^{12})\circ\Phi(\hat X^{01},\hat\lambda^{01})$ for $\hat X^{02}=\hat X^{01}\#\hat X^{12}$.

To construct the symmetric monoidal structure on $CH_\bullet$, it suffices to observe that (as we shall prove in Proposition \ref{symmetricmonoidal}) the sets $\Theta_\I$, $\Theta_\II$ are themselves weakly symmetric monoidal in a manner which preserves the virtual moduli counts.  This completes the construction of the contact homology functor \eqref{contacthomologyfunctor} in terms of Theorem \ref{main}.

\subsection{Applications and extensions}\label{applextend}

There is quite some literature devoted to properties and applications of contact homology and symplectic field theory which relies, explicitly or implicitly, on the existence of ``suitable'' virtual curve counting foundations.
An exhaustive investigation of the applicability and extendibility of our results to these myriad of settings is beyond the scope of this paper.
We nevertheless include some brief indications in this direction.

\begin{itemize}
\item(Grading by $H_1(Y)$) Contact homology $CH_\bullet(Y,\xi)$ has a grading by $H_1(Y)$ (the grading of a given monomial in Reeb orbits equals its total homology class).

\item(Refinement of $\ZZ/2$-grading) Contact homology $CH_\bullet(Y,\xi)$ has a relative grading by $\ZZ/2c_1(\xi)\cdot H_2(Y)$, which is absolute over the $0\in H_1(Y)$ graded piece.  The homological grading of $\gamma$ is given by $\left|\gamma\right|=\CZ(\gamma)+n-3$.

\item(Action filtration) If we equip $(Y,\xi)$ with a contact form $\lambda$, then for $a\in\RR$, there is an invariant $CH_\bullet(Y,\lambda)^{<a}$ which is equipped with functorial maps $CH_\bullet(Y,\lambda)^{<a}\to CH_\bullet(Y,\lambda')^{<a'}$ for $\frac\lambda a\geq\frac{\lambda'}{a'}$ (pointwise) such that
\begin{equation}\label{actionfiltrationlimit}
CH_\bullet(Y,\xi)=\varinjlim CH_\bullet(Y,\lambda)^{<a}.
\end{equation}
Namely, for $\lambda$ non-degenerate, $CH_\bullet(Y,\lambda)^{<a}$ is defined as the homology of the subcomplex $CC_\bullet(Y,\xi)^{<a}_{\lambda,J,\theta}\subseteq CC_\bullet(Y,\xi)_{\lambda,J,\theta}$ spanned by those monomials of total action $<a$, and for general $\lambda$ we define $CH_\bullet(Y,\lambda)^{<a}$ as the direct limit of $CH_\bullet(Y,\lambda')^{<a}$ over non-degenerate $\lambda'>\lambda$.  This invariant $CH_\bullet^{<a}$ may be constructed out of Theorem \ref{main} as in \S\ref{wrapup}.

\item(Coefficients in $\QQ[H_2(Y)]$) Contact homology $CH_\bullet(Y,\xi)$ has a natural lift $\overline{CH}_\bullet(Y,\xi)$ to the group ring $\QQ[H_2(Y;\ZZ)]$.  More intrinsically, $\overline{CH}_\bullet$ may be thought of as a local system over the space of $1$-cycles in $Y$, namely $\tau_{\geq 0}C_{\bullet+1}(Y)$.  Contact homology with group ring coefficients $\overline{CH}_\bullet(Y,\xi)$ has a relative $\ZZ$-grading, where $\QQ[H_2(Y)]$ is $\ZZ$-graded by $2c_1(\xi):H_2(Y)\to\ZZ$.  This invariant $\overline{CH}_\bullet$ may be constructed out of Theorem \ref{main} as in \S\ref{wrapup}.

\item(Contact homology of contractible orbits) An invariant $CH_\bullet^\contr(Y,\xi)$ is obtained from the chain complex $CC_\bullet^\contr(Y,\lambda)$ generated as an algebra by contractible Reeb orbits (with a differential which counts curves whose asymptotic orbits are all contractible).  There is also an invariant $CH_\bullet^\alpha(Y,\xi)$ obtained from the chain complex $CC_\bullet^\alpha(Y,\lambda)$ generated as a module over $CC_\bullet^\contr(Y,\lambda)$ by Reeb orbits in a fixed nontrivial homotopy class $\alpha$ (with differential counting curves whose asymptotic orbits are either all contractible or all contractible except for the positive end and one negative end both in class $\alpha$).  These invariants $CH_\bullet^\contr$ and $CH_\bullet^\alpha$ may be constructed out of Theorem \ref{main} as in \S\ref{wrapup}.

\item(Cylindrical contact homology) If $(Y,\xi)$ is \emph{hypertight} (admits a contact form with no contractible Reeb orbits) then there is an invariant $CH_\bullet^\cyl(Y,\xi)$ defined as follows.  If $(Y,\xi)$ admits a \emph{non-degenerate} contact form with no contractible Reeb orbits, then $CH_\bullet^\cyl(Y,\xi)$ is defined as the homology of the complex $CC_\bullet^\cyl(Y,\lambda):=\bigoplus_{\gamma\in\PP_\good}\oo_\gamma$ with the differential which counts pseudo-holomorphic cylinders.  If this is not the case, then one must first define $CH_\bullet^\cyl(Y,\lambda)^{<a}$ for non-degenerate contact forms $\lambda$ with no contractible Reeb orbits of action $<a$, and then let $CH_\bullet^\cyl(Y,\xi):=\varinjlim CH_\bullet^\cyl(Y,\lambda)^{<a}$.  As this definition makes clear, to define $CH_\bullet^\cyl(Y,\xi)$ we actually only need to assume that $(Y,\xi)$ is \emph{asymptotically hypertight}, namely that it admits a sequence of contact forms $\lambda_i$ converging uniformly to zero, each with no contractible Reeb orbits of action $<1$.  This invariant $CH_\bullet^\cyl$ may be constructed out of Theorem \ref{main} as in \S\ref{wrapup}.

Eliashberg--Givental--Hofer \cite{sftintro} assert that $CH_\bullet^\cyl(Y,\xi)$ can be defined assuming only that $(Y,\xi)$ admits a non-degenerate contact form with no contractible Reeb orbits of index $1$, $0$, or $-1$.  We have nothing positive to say about the applicability of our methods to this question.

\item(Even contact forms)
It is trivial to calculate contact homology given a non-degenerate contact form all of whose Reeb orbits are even, since the differential then vanishes for index reasons.  For examples of such situations, we refer the reader to Ustilovsky \cite{ustilovsky} and Abreu--Macarini \cite{abreumacarini}.

\item(Overtwisted contact manifolds)
A given (connected, non-empty) contact manifold is either \emph{tight} or \emph{overtwisted}.  Overtwisted contact structures are classified completely by an $h$-principle due to Eliashberg \cite{eliashbergovertwisted} in dimension three and Borman--Eliashberg--Murphy \cite{bormaneliashbergmurphy} in general.

Contact homology (even with group ring coefficients) vanishes on any overtwisted contact manifold.  In dimension three, this is a result of Eliashberg \cite[p334, Theorem 3.5(2)]{eliashbergicm} (a proof is given in Yau \cite{mlyauovertwisted} and the appendix by Eliashberg).  In all dimensions, this follows from the result of Bourgeois--van Koert \cite[Theorem 1.3]{bourgeoisvankoert} that contact homology vanishes for any contact manifold admitting a negatively stabilized open book, together with the result of Casals--Murphy--Presas \cite[Theorem 1.1]{casalsmurphypresas} that a contact manifold admits a negatively stabilized open book iff it is overtwisted.  These vanishing results are proved by exhibiting a contact form with a non-degenerate Reeb orbit bounding exactly one pseudo-holomorphic plane in the symplectization (which is cut out transversally); in particular, they are valid for the contact homology we construct here.

An \emph{a priori} weaker notion of a contact manifold being \emph{PS-overtwisted} was introduced by Niederkr\"uger \cite{niederkruger} and Massot--Niederkr\"uger--Wendl \cite{massotniederkrugerwendl}.
An argument, due to Bourgeois--Niederkr\"uger, that contact homology vanishes for PS-overtwisted contact manifolds is sketched in \cite{bourgeoissurvey}; the point is to count pseudo-holomorphic disks with boundary on the plastikstufe/bLob, one boundary marked point constrained to map to a fixed curve on the plastikstufe/bLob from its core/binding to its boundary, and an arbitrary number of negative punctures.
The proof of Theorem \ref{main} generalizes easily to provide sufficient virtual moduli counts for this argument.

\item(Filling and cobordism obstructions)
A contact manifold with vanishing contact homology is not symplectically fillable, and more generally if the positive end of a symplectic cobordism has vanishing contact homology then so does the negative end.
Indeed, the existence of a unital ring map $0\to R$ implies $R=0$, and the cobordism in question induces a unital map $CH_\bullet(Y^+,\xi^+)\to CH_\bullet(Y^-,\xi^-)$ (to obtain such a map in the non-exact case, one can count curves which represent zero in $H_2(\hat X,Y^+\sqcup Y^-)$).
The curve counts produced by Theorem \ref{main} are sufficient for this argument.

It is an observation of Niederkr\"uger--Wendl \cite{niederkrugerwendl} and Latschev--Wendl \cite{latschevwendl} that a similar implication holds for ``stable symplectic cobordisms'' and contact homology with twisted coefficients.
(What we call a symplectic cobordism is often called a ``strong symplectic cobordism'', and the notion of a stable symplectic cobordism is more general; see also Massot--Niederkr\"uger--Wendl \cite{massotniederkrugerwendl}.)
The literal statement of Theorem \ref{main} does not cover stable symplectic cobordisms, however its proof applies to them without modification, as the differences between the two settings are simply not relevant to any part of the argument (except for compactness, which is proved in the requisite generality in \cite{sftcompactness}).

\item(Invariants of contactomorphisms) There is a natural homomorphism
\begin{equation}
\label{pifamilycontact}\pi_0\Cont(Y,\xi)\to\Aut_\QQ(CH_\bullet(Y,\xi)),
\end{equation}
namely the tautological action of $\Cont(Y,\xi)$ on $CH_\bullet(Y,\xi)$.  This action admits the following description in terms of cobordism maps which shows that it descends to $\pi_0$.  For any $\varphi\in\Cont(Y,\xi)$, denote by $X_\varphi$ the exact symplectic cobordism from $Y$ to itself obtained from the trivial cobordism by changing the marking on the negative end by $\varphi$.  The action of $\varphi$ on $CH_\bullet(Y,\xi)$ clearly coincides with the cobordism map $\Phi(X_\varphi)$.  On the other hand, $\Phi(X_\varphi)$ only depends on the class of $\varphi$ in $\pi_0$ since cobordism maps are invariant under deformation.  It is also clear that $\varphi\mapsto\Phi(X_\varphi)$ is a group homomorphism since $X_\varphi\# X_\psi=X_{\varphi\psi}$.

As pointed out by P.\ Massot, any contactomorphism which is \emph{symplectically pseudo-isotopic} to the identity (a notion due to Cieliebak--Eliashberg \cite[\S 14.5]{cieliebakeliashberg}) lies in the kernel of \eqref{pifamilycontact}; indeed, $\varphi$ is symplectically pseudo-isotopic to the identity iff $X_\varphi$ is isomorphic to $X_{\id}$ as symplectic cobordisms from $Y$ to itself.  Thus a contactomorphism which acts nontrivially on contact homology cannot be symplectically pseudo-isotopic to the identity.

\item(Invariants of families of contactomorphisms)
The above construction generalizes to give a natural homomorphism
\begin{equation}
\label{Hkfamilycontact}H_k(\Cont(Y,\xi))\to\Hom_\QQ(CH_\bullet(Y,\xi),CH_{\bullet+k}(Y,\xi)),
\end{equation}
as introduced by Bourgeois \cite{bourgeoishomotopy} (more precisely, Bourgeois introduced the pre-compositions of \eqref{pifamilycontact}--\eqref{Hkfamilycontact} with the map $\Omega_\xi\Xi(Y)\to\Cont(Y,\xi)$ coming from Gray's fibration sequence $\Cont(Y,\xi)\to\Diff(Y,\xi)\to\Xi(Y)$, where $\Xi(Y)$ denotes the space of contact structures on $Y$ and $\Omega_\xi$ denotes the space of loops based at $\xi\in\Xi(Y)$).
Namely, a family of $\varphi\in\Cont(Y,\xi)$ gives rise to a family of cobordisms $X_\varphi$, and case (III) of Theorem \ref{main} (and its proof) generalize immediately to such higher-dimensional families of cobordisms, in the following form.
For any family of cobordisms $(\hat X,(\hat\omega^t,\hat J^t)_{t\in\Delta^n})$, there is a set $\Theta_{\III(n)}=\Theta_{\III(n)}(\hat X,(\hat\omega^t,\hat J^t)_{t\in\Delta^n})$ (specializing to $\Theta_{\III(0)}=\Theta_\II$ and $\Theta_{\III(1)}=\Theta_\III$) along with surjective maps
\begin{equation}
\Theta_{\III(n)}(\hat X,(\hat\omega^t,\hat J^t)_{t\in\Delta^n})\twoheadrightarrow\lim_{\Delta^k\subsetneq\Delta^n}\bigl[\Theta_{\III(k)}(\hat X,(\hat\omega^t,\hat J^t)_{t\in\Delta^k})\to(\Theta_\I^+\times\Theta_\I^-)\bigr]
\end{equation}
(where the limit is in the category of sets over $\Theta_\I^+\times\Theta_\I^-$), and there are associated virtual moduli counts satisfying the natural master equation, thus giving rise (in the exact case) to ``higher homotopies'' $K(\hat X,\hat\lambda^t)_{\hat J^t,\theta}$ satisfying
\begin{multline}
d_{J^-,\theta^-}K(\hat X,\hat\lambda^t)_{\hat J^t,\theta}-(-1)^nK(\hat X,\hat\lambda^t)_{\hat J^t,\theta}d_{J^+,\theta^+}\\
=\sum_{\begin{smallmatrix}k=0\\i:\Delta^{[0\ldots\hat k\ldots n]}\hookrightarrow\Delta^{[0\ldots n]}\end{smallmatrix}}^n(-1)^kK(\hat X,\hat\lambda^{i(t)})_{\hat J^{i(t)},i^\ast\theta}.
\end{multline}
This defines \eqref{Hkfamilycontact}.

We expect that it is possible to upgrade \eqref{pifamilycontact}--\eqref{Hkfamilycontact} into the statement that $CC_\bullet(Y,\xi)$ is an $A_\infty$-module over $C_\bullet(\Cont(Y,\xi))$ (i.e.\ that contact chains of $(Y,\xi)$ is a ``derived local system'' over $B\Cont(Y,\xi)$), by proving a more intricate higher-dimensional version of case (IV) of Theorem \ref{main}.
To realize this rigorously, it would be sufficient to specify the correct compactified moduli spaces to consider and to show that their stratifications are locally cell-like as in \S\S\ref{localmodelsection},\ref{celllikeproofsec} (these are essentially combinatorial questions).

\item(Chain homotopy vs DGA homotopy)
Given a one-parameter family of exact symplectic cobordisms, case (III) of Theorem \ref{main} shows that the two cobordism maps
\begin{equation}
\begin{tikzcd}[column sep = large]
CC_\bullet(Y^+,\xi^+)_{\lambda^+,J^+,\theta^+}\ar[shift left=0.6ex]{r}{\Phi(\hat X,\hat\lambda^0)_{\hat J^0,\theta^0}}\ar[shift left=-0.6ex]{r}[swap]{\Phi(\hat X,\hat\lambda^1)_{\hat J^1,\theta^1}}&CC_\bullet(Y^-,\xi^-)_{\lambda^-,J^-,\theta^-}
\end{tikzcd}
\end{equation}
are \emph{chain homotopic}, by constructing virtual moduli counts for moduli spaces of \emph{disconnected curves} (and analogously for case (IV)).
For certain purposes (such as defining linearized contact homology), one may need to know that such cobordism maps are homotopic in a stronger sense, namely one which takes into account the algebra structure.
The construction of such homotopies is sketched in Eliashberg--Givental--Hofer \cite[\S 2.4]{sftintro} and in Ekholm--Oancea \cite[\S\S 5.4--5.5]{ekholmoancea} in more detail (also see Ekholm--Honda--K\'alm\'an \cite[Lemma 3.13]{ekholmhondakalman} for the related, but easier, case of Legendrian contact homology).
We have nothing positive to say about the applicability of our methods to the rigorous construction of such homotopies.
One can generalize case (III) of Theorem \ref{main} (and its proof) a little by counting curves occurring at times $0\leq t_1\leq\cdots\leq t_k\leq 1$ and obtain an $A_\infty$-homotopy, though it is not clear if this is any improvement.

\item(Symplectic field theory)
To construct the symplectic field theory ``homology groups'' as described in Cieliebak--Latschev \cite{cieliebaklatschev} and Latschev--Wendl \cite{latschevwendl} (specializing work of Eliashberg--Givental--Hofer \cite{sftintro}) using the methods of this paper, it would be sufficient to specify the correct compactified moduli spaces to consider and to show that their stratifications are locally cell-like as in \S\S\ref{localmodelsection},\ref{celllikeproofsec} (these are essentially combinatorial questions).

\item(Morse--Bott contact forms)
It should be possible to define contact homology using Morse--Bott contact forms, by counting appropriate pseudo-holomorphic cascades as in Bourgeois \cite{bourgeoisthesis,bourgeoisthesissurvey}.
To realize this rigorously, the main tasks would be to specify the correct compactified moduli spaces to consider, to show that their stratifications are locally cell-like as in \S\S\ref{localmodelsection},\ref{celllikeproofsec} (these are essentially combinatorial questions), and to establish the relevant gluing result as in \S\ref{gluingsec}.

\item(Integer coefficients) It is an interesting open question (promoted by Abouzaid) how to naturally lift contact homology from $\QQ$ to $\ZZ$.

\item(Non-equivariant contact homology) We conjecture that contact homology can be lifted to a framed $E_2$-algebra, namely an algebra over the operad whose space of $n$-ary operations is the space of embeddings $(D^2)^{\sqcup n}\hookrightarrow D^2$.  More precisely, the free symmetric $\QQ$-algebra on $\oo_\gamma$ for $\gamma\in\PP_\good$ in \eqref{freesupercommutative} should be replaced with the free framed $E_2$-algebra on $\oo_\gamma$ for $\gamma\in\PP$ subject to the relations that rotating any $d$-fold multiply covered orbit by $2\pi/d$ acts on $\oo_\gamma$ by the corresponding sign ($\pm 1$ according to whether the Reeb orbit is good or bad).  The differential of a given orbit $\gamma^+$ should lift naturally to an element of this (almost) free framed $E_2$-algebra by taking into account the conformal types of the domains of the pseudo-holomorphic curves in question.  This lift can be viewed as the ``non-equivariant'' version of contact homology (which should itself be viewed as an $S^1$-equivariant theory).
\end{itemize}

\section{Moduli spaces of pseudo-holomorphic curves}\label{modulisec}

In this section, we define the moduli spaces of pseudo-holomorphic curves which we will use to define contact homology.

\subsection{Categories of strata \texorpdfstring{$\SSS$}{S}}\label{Tcatdefsec}

We begin by introducing, for any datum $\D$ from any of Setups \ref{setupI}--\ref{setupIV}, a collection $\SSS=\SSS(\D)$ of labeled trees  which index the strata of the compactified moduli spaces of pseudo-holomorphic curves (which of Setups \ref{setupI}--\ref{setupIV} is being considered is often indicated with a subscript, as in $\SSS_\I$, $\SSS_\II$, etc.).
A labeled tree describes the ``combinatorial type'' of a pseudo-holomorphic curve: the tree is (almost) the dual graph of the domain, and it is labeled with the homotopy class and asymptotics of the map.

Gluing pseudo-holomorphic curves corresponds to contracting a subset of the edges of a tree and updating the labels accordingly.  We regard these collections $\SSS$ as categories, with such edge contractions as morphisms.  Morphisms of labeled trees will induce maps of compactified moduli spaces.

The categories $\SSS$ carry some additional (vaguely monoidal) structure, corresponding to the fact that boundary strata in compactified moduli spaces can be expressed as products of other ``smaller'' moduli spaces.  The relevant operation on trees is that of ``concatenation'', where we take some collection of trees and identify some pairs of input/output edges with matching labels to produce a new tree.

\begin{figure}[ht]
\centering
\includegraphics{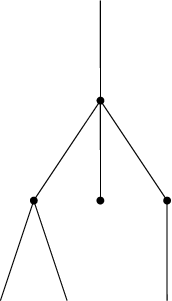}
\caption{A tree.  The edges are all directed downwards.}\label{treefig}
\end{figure}

\begin{figure}[ht]
\centering
\includegraphics{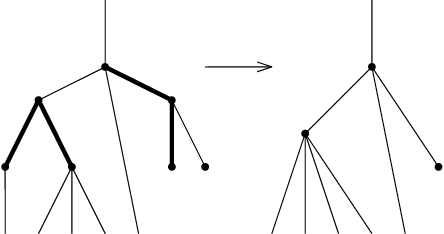}
\caption{A contraction of trees.  The edges which have been contracted are marked bold.}\label{treecontractionfig}
\end{figure}

The category $\SSS_\I$ is defined as follows, with respect to some fixed datum of Setup \ref{setupI}.
An object of $\SSS_\I$ consists first of a finite directed tree $T$ (for us, the word ``tree'' entails being non-empty) in which every vertex has a unique incoming edge (see Figure \ref{treefig}).
Edges with missing source (called input edges) and edges with missing sink (called output edges) are allowed, together these are called external edges, and the remaining edges will be called interior edges.
The tree $T$ is further equipped with decorations as follows:
\begin{enumerate}
\item For each edge $e\in E(T)$, a Reeb orbit $\gamma_e\in\PP$.
\item For each vertex $v\in V(T)$, a homotopy class $\beta_v\in\pi_2(Y,\gamma_{e^+(v)}\sqcup\{\gamma_{e^-}\}_{e^-\in E^-(v)})$, where we denote by $e^+(v)\in E(T)$ the unique incoming edge at $v$, and by $E^-(v)$ the set of outgoing edges at $v$.
\item For each external edge $e\in E^\eext(T)$, a basepoint $b_e\in\left|\gamma_e\right|$, where $\left|\gamma_e\right|$ denotes the underlying simple orbit of $\gamma_e$.
\end{enumerate}
A morphism $\pi:T\to T'$ in $\SSS_\I$ consists first of a contraction of underlying trees.
By this we mean that some collection of interior edges of $T$ is specified, and $T'$ is identified with the result of contracting these edges (see Figure \ref{treecontractionfig}).
This identification must be compatible with the decorations in the following sense:
\begin{enumerate}
\item For each non-contracted edge $e\in E(T)$, we have $\gamma_{\pi(e)}=\gamma_e$.
\item For each vertex $v'\in V(T')$, we have $\beta_{v'}=\#_{\pi(v)=v'}\beta_v$.
\end{enumerate}
Finally, we must specify for each external edge $e\in E^\eext(T)=E^\eext(T')$ a path along $\left|\gamma_e\right|$ between the basepoints $b_e$ and $b_e'$, modulo the relation that identifies two such paths iff their ``difference'' lifts to $\gamma_e$ (i.e.\ iff it has degree divisible by the covering multiplicity of $\gamma_e$).
Such paths will correspond, in the realm of pseudo-holomorphic curves, to ways of moving asymptotic markers.
Conceptually, we should really regard the input/output edges of $T$ as being labeled by objects of the category $\tilde\PP$ whose objects are Reeb orbits $\gamma\in\PP$ together with a basepoint $b\in\left|\gamma\right|$ and whose morphisms are paths as just described (the set of isomorphism classes in this category is $|\tilde\PP|=\PP$, and the automorphism group of an object in the isomorphism class of $\gamma\in\PP$ is (canonically) $\ZZ/d_\gamma$).

We now introduce the notion of a concatenation in $\SSS_\I$.
A concatenation in $\SSS_\I$ consists of a finite non-empty collection of objects $T_i\in\SSS_\I$ along with a matching between some pairs of external edges of the $T_i$'s (with matching $\gamma_e$) such that the resulting gluing is a tree, along with a choice of paths between the basepoints for each pair of matched edges.
Given a concatenation $\{T_i\}_i$ in $\SSS_\I$, there is a resulting object $\#_iT_i\in\SSS_\I$.
A morphism of concatenations $\{T_i\}_i\to\{T_i'\}_i$ shall mean a collection of morphisms $T_i\to T_i'$ covering a bijection of index sets; a morphism of concatenations $\{T_i\}_i\to\{T_i'\}_i$ induces a morphism $\#_iT_i\to\#_iT_i'$.
If $\{T_i\}_i$ is a concatenation in $\SSS_\I$ and $T_i=\#_jT_{ij}$ for some concatenations $\{T_{ij}\}_j$, there is a resulting composite concatenation $\{T_{ij}\}_{ij}$ with natural isomorphisms $\#_{ij}T_{ij}=\#_i\#_jT_{ij}=\#_iT_i$.

\begin{figure}[p]
\centering
\includegraphics{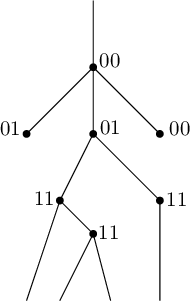}
\caption{An object of $\SSS_\II$, with vertex labels as shown.  Note that the vertex labels determine the edge labels uniquely (and conversely except for vertices with no outgoing edges).}\label{treeIIfig}
\end{figure}

\begin{figure}[p]
\centering
\includegraphics{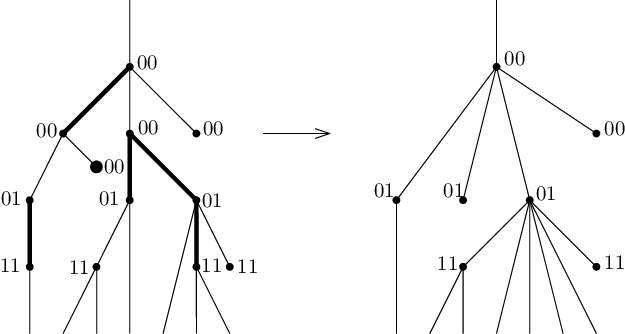}
\caption{A morphism in $\SSS_\II$.  A morphism in $\SSS_\II$ is determined uniquely by the set of contracted edges and the set of vertices whose label changes from $00$ to $01$ (only needed for vertices with no outgoing edges).}\label{treeIIcontractionfig}
\end{figure}

\begin{figure}[p]
\centering
\includegraphics{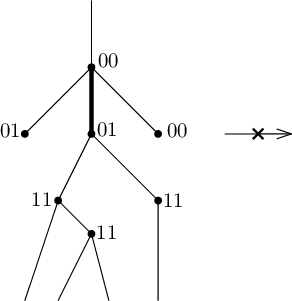}
\caption{This object of $\SSS_\II$ cannot be contracted along (only) the marked edge (there is no way to consistently label the result).}\label{treeIIcontractionnotfig}
\end{figure}

Continuing on to Setup \ref{setupII}, the category $\SSS_\II$ (again for a particular choice of datum of Setup \ref{setupII}) is defined as follows.
An object of $\SSS_\II$ is a decorated tree $T$ as before, but with additional labels which specify the target of the pseudo-holomorphic curve corresponding to each vertex.
For ease of notation, we set $\hat X^{00}=\hat Y^0:=\hat Y^+$, $\hat X^{01}:=\hat X$, and $\hat X^{11}=\hat Y^1:=\hat Y^-$.
The additional labels (which in turn determine the allowable decorations) are as follows:
\begin{enumerate}
\item For each edge $e\in E(T)$, a symbol $\ast(e)\in\{0,1\}$ such that input edges are labeled with $0$ and output edges are labeled with $1$.
\item For each vertex $v\in V(T)$, a pair of symbols $\ast^\pm(v)\in\{0,1\}$ such that $\ast^+(v)\leq\ast^-(v)$ and $\ast(e^\pm(v))=\ast^\pm(v)$.
If $\ast^+(v)=\ast^-(v)$, we call $v$ a symplectization vertex.
\end{enumerate}
See Figure \ref{treeIIfig}.
These labels determine the allowable decorations: $\gamma_e\in\PP(Y^{\ast(e)},\lambda^{\ast(e)})$ and $\beta_v\in\pi_2(\hat X^{\ast(v)},\gamma_{e^+(v)}\sqcup\{\gamma_{e^-}\}_{e^-\in E^-(v)})$.
A morphism $\pi:T\to T'$ in $\SSS_\II$ consists first of a contraction of underlying trees, required to satisfy $\ast^+(\pi(v))\leq\ast^+(v)$, $\ast^-(\pi(v))\geq\ast^-(v)$, and $\ast(\pi(e))=\ast(e)$ for non-contracted edges $e$.
Given these, it makes sense to require compatibility with the decorations and to specify paths between basepoints as before.
See Figures \ref{treeIIcontractionfig} and \ref{treeIIcontractionnotfig}.

A concatenation $\{T_i\}_i$ in $\SSS_\II$ is defined as before, except that each $T_i$ may be an object of either $\SSS_\II$, $\SSS_\I^+$, or $\SSS_\I^-$ (where $\SSS_\I^\pm=\SSS_\I(Y^\pm,\lambda^\pm)$).  Given a concatenation $\{T_i\}_i$ in $\SSS_\II$, there is a resulting object $\#_iT_i\in\SSS_\II$.  If $\{T_i\}_i$ is a concatenation in $\SSS_\II$ and $T_i=\#_jT_{ij}$ for some concatenations $\{T_{ij}\}_j$ (in whichever of $\SSS_\II$, $\SSS_\I^+$, $\SSS_\I^-$ contains $T_i$), there is a resulting \emph{composite concatenation} $\{T_{ij}\}_{ij}$ with natural isomorphisms $\#_{ij}T_{ij}=\#_i\#_jT_{ij}=\#_iT_i$.

Continuing on to Setup \ref{setupIII}, the category $\SSS_\III$ is defined as follows.
An object of $\SSS_\III$ is a forest (meaning connected and non-empty assumptions are dropped) with labels and decorations as in the case of $\SSS_\II$, together with a choice of set $\s\in\{\{0\},\{1\},(0,1)\}$.
A morphism $T\to T'$ in $\SSS_\III$ is defined as for $\SSS_\II$, with the additional requirement that $\s(T)\subseteq\overline{\s(T')}$.

There are three types of concatenations in $\SSS_\III$.  The first type consists of objects $T_i\in\SSS_\I^+\sqcup\SSS_\II^{t=0}\sqcup\SSS_\I^-$ with the usual matching data, producing an object $\#_iT_i\in\SSS_\III$ with $\s(\#_iT_i)=\{0\}$.  The second type consists of objects $T_i\in\SSS_\I^+\sqcup\SSS_\II^{t=1}\sqcup\SSS_\I^-$ with matching data, producing an object $\#_iT_i\in\SSS_\III$ with $\s(\#_iT_i)=\{1\}$.  The third type consists of objects $T_i\in\SSS_\I^+\sqcup\SSS_\III\sqcup\SSS_\I^-$ (exactly one of which lies in $\SSS_\III$) with matching data, producing an object $\#_iT_i\in\SSS_\III$ (where $\s(\#_iT_i)=\s(T_i)$ for the unique $T_i\in\SSS_\III$).
A composition of concatenations is defined as before.

The reason for the restriction that concatenations in $\SSS_\III$ of the third type must contain exactly one object of $\SSS_\III$ is explained as follows.
As mentioned earlier, a concatenation $\{T_i\}_i$ is supposed to induce an identification of $\Mbar(\#_iT_i)$ with $\prod_i\Mbar(T_i)$ (modulo the symmetry coming from multiply covered orbits).
The moduli spaces $\Mbar_\III$ parameterize pseudo-holomorphic curves at varying times $t\in[0,1]$, and thus they come with a forgetful map to $[0,1]$.
Corresponding to a concatenation in $\SSS_\III$ of the third type without the requirement of containing exactly one object of $\SSS_\III$, there is not an identification of $\Mbar(\#_iT_i)$ with $\prod_i\Mbar(T_i)$, but rather with the corresponding \emph{fiber product} over $[0,1]$.
We do not discuss such identifications in this paper because we do not know how to prove (or even formulate precisely) the corresponding expected compatibility of virtual fundamental cycles with such fiber products (this is related to the problem of showing compatibility of the homotopy with the algebra structure on contact homology, see \S\ref{applextend}).
The same remark will apply to concatenations in $\SSS_\IV$.

The category $\SSS_\IV$ is defined as follows.
An object in $\SSS_\IV$ is a forest $T$ with the following labels and decorations.  Each edge shall be labeled with a symbol $\ast(e)\in\{0,1,2\}$, such that input edges are labeled with $0$ and output edges are labeled with $2$.  Each vertex shall be labeled with a pair of symbols $\ast^\pm(v)\in\{0,1,2\}$ such that $\ast^+(v)\leq\ast^-(v)$ and $\ast(e^\pm(v))=\ast^\pm(v)$.  There shall also be decorations $\gamma_e\in\PP(Y^{\ast(e)})$ and $\beta_v\in\pi_2(X^{\ast(v)},\gamma_{e^+(v)}\sqcup\{\gamma_{e^-}\}_{e^-\in E^-(v)})$ with basepoints $b_e\in\left|\gamma_e\right|$ for input/output edges as before.  Finally, we specify a set $\s\in\{\{0\},\{\infty\},(0,\infty)\}$ and we require that if $\s\in\{\{0\},(0,\infty)\}$ then $\ast(v)\in\{00,02,22\}$ for all $v$, and if $\s=\{\infty\}$ then $\ast(v)\in\{00,01,11,12,22\}$ for all $v$.
A morphism $T\to T'$ in $\SSS_\IV$ is defined as for $\SSS_\II$, with the additional requirement that $\s(T)\subseteq\overline{\s(T')}$.

There are three types of concatenations in $\SSS_\IV$.  The first type consists of objects $T_i\in\SSS_\I^0\sqcup\SSS_\II^{02}\sqcup\SSS_\I^2$ with matching data, producing an object $\#_iT_i\in\SSS_\IV$ with $\s(\#_iT_i)=\{0\}$.  The second type consists of objects $T_i\in\SSS_\I^0\sqcup\SSS_\II^{01}\sqcup\SSS_\I^1\sqcup\SSS_\II^{12}\sqcup\SSS_\I^2$ with matching data, producing an object $\#_iT_i\in\SSS_\IV$ with $\s(\#_iT_i)=\{\infty\}$.  The third type consists of objects $T_i\in\SSS_\I^0\sqcup\SSS_\IV\sqcup\SSS_\I^2$ (exactly one of which lies in $\SSS_\IV$) with matching data, producing an object $\#_iT_i\in\SSS_\IV$ (where $\s(\#_iT_i)=\s(T_i)$ for the unique $T_i\in\SSS_\IV$).  A composition of concatenations is defined as before.  

\subsection{Properties of categories \texorpdfstring{$\SSS$}{S}}

The categories $\SSS$ together with their concatenation structure satisfy some basic properties which will be used frequently.
We begin with some basic definitions.

\begin{definition}[Slice categories]
For any category $\TTT$ and an object $T\in\TTT$, denote by $\TTT_{/T}$ the ``over-category'' whose objects are morphisms $T'\to T$.
Similarly define the ``under-category'' $\TTT_{T/}$ (objects are morphisms $T\to T'$) and $\TTT_{T//T'}$ (objects are factorizations $T\to T''\to T'$ of a fixed morphism $T\to T'$).
\end{definition}

\begin{definition}[Poset]\label{posetdef}
A \emph{poset} is a category $\TTT$ for which for every pair of objects $x,y\in\TTT$, there is at most one morphism $x\to y$.
The \emph{cardinality} of a poset is its number of isomorphism classes of objects.
\end{definition}

Both $\SSS_{T/}$ and $\SSS_{T//T'}$ are finite posets for any $T$ or $T\to T'$, respectively.
On the other hand, $\SSS_{/T}$ is almost never finite nor a poset (see Figure \ref{overautfigure}).

\begin{figure}[ht]
\centering
\includegraphics{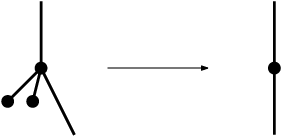}
\caption{A map of trees $T'\to T$ for which $\Aut(T'/T)=\ZZ/2$, corresponding to exchanging the lower two vertices of $T'$.}\label{overautfigure}
\end{figure}

\begin{definition}[Automorphism groups]
Given a map $T'\to T$, we denote by $\Aut(T'/T)$ the subgroup of $\Aut(T')$ compatible with the map $T'\to T$ (i.e.\ the automorphism group of $(T'\to T)\in\SSS_{/T}$).
Given a concatenation $\{T_i\}_i$, we denote by $\Aut(\{T_i\}_i/\#_iT_i)$ the group of automorphisms of $\{T_i\}_i$ acting trivially on $\#_iT_i$ (i.e.\ the product $\prod_e\ZZ/d_{\gamma_e}$ over junction edges, acting diagonally).
\end{definition}

\begin{definition}
An object $T\in\SSS$ is called \emph{maximal} iff the only morphism out of $T$ is the identity map (i.e.\ $\SSS_{T/}$ consists of a single object with trivial endomorphism monoid).
\end{definition}

Maximality can be characterized concretely as follows:
\begin{enumerate}
\item$T\in\SSS_\I$ is maximal iff it has exactly one vertex.
\item$T\in\SSS_\II$ is maximal iff it has exactly one vertex and this vertex has $\ast(v)=01$.
\item$T\in\SSS_\III$ is maximal iff every component of $T$ has exactly one vertex, all these vertices have $\ast(v)=01$, and $\s(T)=(0,1)$.
\item$T\in\SSS_\IV$ is maximal iff every component of $T$ has exactly one vertex, all these vertices have $\ast(v)=02$, and $\s(T)=(0,\infty)$.
\end{enumerate}

\begin{lemma}\label{concatenationmaximal}
Every object $T\in\SSS$ can be expressed as a concatenation of maximal $T_i$, uniquely up to isomorphism.
In particular, $T\in\SSS$ is maximal iff it cannot be expressed nontrivially as a concatenation.
\qed
\end{lemma}

\subsection{Moduli spaces \texorpdfstring{$\Mbar(T)$}{Mbar(T)}}\label{mbardefsec}

We now define the compactified moduli spaces of pseudo-holomorphic curves $\Mbar(T)$ relevant for contact homology.  We first define the strata $\M(T)$ for each $T\in\SSS$, and then we define the compactified moduli spaces $\Mbar(T)$ as unions of strata $\Mbar(T')$ for $T'\to T$.

Equip $\RR\times S^1$ with coordinates $(s,t)$ and with the standard almost complex structure $j(\partial_s)=\partial_t$, i.e.\ $z=e^{s+it}$.

\begin{definition}\label{asymptoticdef}
Fix $(Y,\lambda,J)$ as in Setup \ref{setupI}, and let $u:[0,\infty)\times S^1\to\hat Y$ be a smooth map.  We say that $u$ is \emph{$C^0$-convergent} to (the positive trivial cylinder over) a Reeb orbit $\gamma\in\PP$ iff
\begin{equation}\label{asymptoticend}
u(s,t)=(Ls+b,\tilde\gamma(t))+o(1)\quad\text{as }s\to\infty,
\end{equation}
for some $b\in\RR$ and some $\tilde\gamma:S^1\to Y$ with $\partial_t\tilde\gamma=L\cdot R_\lambda(\tilde\gamma)$ parameterizing $\gamma$, where $o(1)$ is meant with respect to any $\RR$-invariant Riemannian metric.
When the error $o(1)$ decays like $O(e^{-\delta\left|s\right|})$ in all derivatives for some $\delta>0$, we say $u$ is \emph{$C^\infty$-convergent with weight $\delta$}.
Similarly, we define the notion of $u:(-\infty,0]\times S^1\to\hat Y$ being convergent to a negative trivial cylinder by considering \eqref{asymptoticend} in the limit $s\to-\infty$.
\end{definition}

It is straightforward to check that for any Riemann surface $C$ and $p\in C$, it makes sense to say that $u:C\setminus p\to\hat Y$ is $C^0$- or $C^\infty$-convergent of weight $\delta\in(0,1)$ to a trivial cylinder at $p$, i.e.\ that this notion is independent of the choice of coordinates $[0,\infty)\times S^1\to C\setminus p$ near $p$ (though note the importance of $\delta<1$ in this statement).
This notion also generalizes immediately to maps to any $\hat X$ from Setups \ref{setupII}, \ref{setupIII}, \ref{setupIV}, by virtue of the markings \eqref{endmarkingsI}--\eqref{endmarkingsII}.

If $u$ is $\hat J$-holomorphic, then $C^0$-convergence to a trivial cylinder implies $C^\infty$-convergence of every weight $\delta<\delta_\gamma$, where $\delta_\gamma>0$ denotes the smallest nonzero absolute value among eigenvalues of the linearized operator of $\gamma$ as in Definition \ref{smalleigenvalue} (by Hofer--Wysocki--Zehnder \cite[Theorems 1.1, 1.2, and 1.3]{hwzsmallarea} which we restate as Proposition \ref{hwzcylinderestimate}).

If $u$ converges to a trivial cylinder over $\gamma\in\PP$ at $p\in C$, then it induces a well-defined constant speed parameterization of $\gamma$ denoted $u_p:S_pC\to Y$, where $S_pC:=(T_pC\setminus 0)/\RR_{>0}$ is the tangent sphere at $p$, which is a $U(1)$-torsor due to the complex structure on $C$.

\begin{figure}[ht]
\centering
\includegraphics{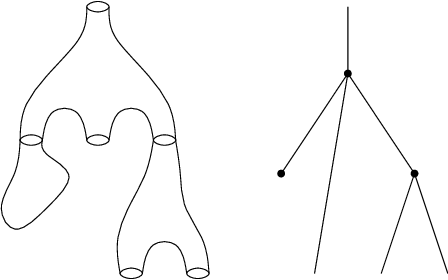}
\caption{A stable pseudo-holomorphic building and the corresponding tree.}\label{MIfig}
\end{figure}

\begin{definition}[Moduli space $\M_\I(T)$]\label{MIdef}
A \emph{pseudo-holomorphic building of type $T\in\SSS_\I$} consists of the following data (see Figure \ref{MIfig}):
\begin{enumerate}
\item\label{MIdefC}For every vertex $v$, a closed, connected, nodal Riemann surface of genus zero $C_v$, along with distinct points $\{p_{v,e}\in C_v\}_e$ indexed by the edges $e$ incident at $v$.
\item\label{MIdefmap}For every vertex $v$, a smooth map $u_v:C_v\setminus\{p_{v,e}\}_e\to\hat Y$.
\item\label{MIdefEA}We require that $u_v$ be $C^0$-convergent to the positively trivial cylinder over $\gamma_{e^+(v)}$ at $p_{v,e^+(v)}$ and to the negative trivial cylinders over $\gamma_{e^-}$ at $p_{v,e^-}$ for $e^-\in E^-(v)$, and be in the homotopy class $\beta_v$.
\item\label{MIdefbase}For every input/output edge $e$, an ``asymptotic marker'' $\tilde b_e\in S_{p_{v,e}}C_v$ which is mapped to the basepoint $b_e\in\left|\gamma_e\right|$ by $(u_v)_{p_{v,e}}$.
\item\label{MIdefM}For every interior edge $v\xrightarrow ev'$, a ``matching isomorphism'' $m_e:S_{p_{v,e}}C_v\to S_{p_{v',e}}C_{v'}$ intertwining $(u_v)_{p_{v,e}}$ and $(u_{v'})_{p_{v',e}}$.
\item\label{MIdefH}We require that $u_v$ be $\hat J$-holomorphic, i.e.\ $(du)^{0,1}_{\hat J}=0$.
\end{enumerate}
An \emph{isomorphism} $(\{C_v\},\{p_{v,e}\},\{u_v\},\{\tilde b_e\},\{m_e\})\to(\{C_v'\},\{p_{v,e}'\},\{u_v'\},\{\tilde b_e'\},\{m_e'\})$ between pseudo-holomorphic buildings of type $T$ consists of isomorphisms $\{i_v:C_v\to C_v'\}$ and real numbers $\{s_v\in\RR\}$ such that $u_v=\tau_{s_v}\circ u_v'\circ i_v$ ($\tau_s:\hat Y\to\hat Y$ denotes translation by $s$), $i_v(p_{v,e})=p_{v,e}'$, $i(\tilde b_e)=\tilde b_e'$, and $i_{v'}\circ m_e=m'_e\circ i_v$.  A pseudo-holomorphic building is called \emph{stable} iff its automorphism group is finite.  We denote by $\M_\I(T)$ the set of isomorphism classes of stable pseudo-holomorphic buildings of type $T$.
Note the tautological action of $\Aut(T)$ on $\M(T)$ by changing the marking.
\end{definition}

\begin{remark}
For topological reasons (exactness of $\hat Y$), the domains $C_v$ of every pseudo-holomorphic building as in Definition \ref{MIdef} are in fact smooth (that is, are without nodes).
We find it convenient to set up all the definitions in the generality of nodal domains (which do occur in non-exact cobordisms).
\end{remark}

\begin{remark}
Concretely, a pseudo-holomorphic building is stable iff it does not contain either (1) a trivial cylinder at a symplectization vertex of $T$, or (2) an unstable irreducible component of $C_v$ over which the map $u_v$ is constant.
\end{remark}

\begin{definition}[Moduli space $\M_\II(T)$]\label{MIIdef}
A \emph{pseudo-holomorphic building of type $T\in\SSS_\II$} consists of the following data:
\begin{enumerate}
\item Domains $C_v$ and punctures $p_{v,e}$ as in Definition \ref{MIdef}\ref{MIdefC}.
\item For every vertex $v$, a smooth map $u_v:C_v\setminus\{p_{v,e}\}_e\to\hat X^{\ast(v)}$.
\item Aymptotic conditions on $u_v$ as Definition \ref{MIdef}\ref{MIdefEA}.
\item Asymptotic markers $\tilde b_e$ as in Definition \ref{MIdef}\ref{MIdefbase}.
\item Matching isomorphisms $m_e$ as in Definition \ref{MIdef}\ref{MIdefM}.
\item $\hat J$-holomorphicity condition on $u_v$ as Definition \ref{MIdef}\ref{MIdefH}.
\end{enumerate}
An \emph{isomorphism} between pseudo-holomorphic buildings of type $T$ is defined as in Definition \ref{MIdef}, except that there is a translation $s_v\in\RR$ only if $v$ is a symplectization vertex.  We denote by $\M_\II(T)$ the set of isomorphism classes of stable pseudo-holomorphic buildings of type $T$.
\end{definition}

\begin{definition}[Moduli space $\M_\III(T)$]\label{MIIIdef}
For $T\in\SSS_\III$, denote by $\M_\III(T)$ the union over $t\in\s(T)$ of the set of isomorphism classes of stable pseudo-holomorphic buildings of type $T$ (as in Definition \ref{MIIdef}) in $(\hat X,\hat J^t)$.
\end{definition}

\begin{definition}[Moduli space $\M_\IV(T)$]\label{MIVdef}
For $T\in\SSS_\IV$, denote by $\M_\IV(T)$ the union over $t\in\s(T)$ of the set of isomorphism classes of stable pseudo-holomorphic buildings of type $T$ (as in Definition \ref{MIIdef}) in $(\hat X^{02,t},\hat J^{02,t})$.
\end{definition}

\begin{definition}[Moduli spaces $\Mbar(T)$]\label{Mbardef}
For $T\in\SSS$, we define
\begin{equation}
\Mbar(T):=\bigsqcup_{T'\to T}\M(T')/\!\Aut(T'/T).
\end{equation}
The union is over the set of isomorphism classes $|\SSS_{/T}|$.
\end{definition}

For a morphism $T'\to T$, there is a natural map
\begin{equation}
\label{Mbarfunct}\Mbar(T')/\!\Aut(T'/T)\to\Mbar(T),
\end{equation}
and for a concatenation $\{T_i\}_i$, there is an induced identification
\begin{equation}
\label{Mbarprod}\prod_i\Mbar(T_i)\Bigm/\!\Aut(\{T_i\}_i/\#_iT_i)\xrightarrow\sim\Mbar(\#_iT_i).
\end{equation}
We wish to warn the reader that \eqref{Mbarfunct} is \emph{not} always an inclusion!  The reason for this is explained in \S\ref{secstratification}, in particular Figure \ref{categorystratificationfigure}.

\subsection{Stratifications of moduli spaces}\label{secstratification}

We now discuss the tautological stratifications of the moduli spaces $\Mbar(T)$ by $\SSS_{/T}$.
This discussion is (slightly) less trivial than one might first expect, since $\SSS_{/T}$ is not a poset (see the discussion following Definition \ref{posetdef}).
A simple nontrivial example of a stratification by a category is given in Figure \ref{categorystratificationfigure}.

\begin{definition}\label{stratificationdef}
A \emph{stratification} of a topological space $X$ by a poset $\TTT$ is a map $X\to\TTT$ for which $X^{\geq\ttt}$ (the inverse image of $\TTT^{\geq\ttt}=\TTT_{\ttt/}$) is open for all $\ttt\in\TTT$.
\end{definition}

\begin{definition}\label{stratificationcategorydef}
A \emph{stratification} of a topological space $X$ by a category $\TTT$ (with the property that $\TTT_{\ttt/}$ is a poset for all $\ttt\in\TTT$) consists of the following data:
\begin{enumerate}
\item For every $x\in X$, we specify an object $\ttt_x\in\TTT$.
\item For every $x\in X$ and every $y$ sufficiently close to $x$, we specify a morphism $f_{yx}:\ttt_x\to\ttt_y$ in $\TTT$ (this data is regarded as a germ near $x$).
\end{enumerate}
subject to the following requirements:
\begin{enumerate}[resume]
\item For every $x\in X$, every $y$ sufficiently close to $x$, and every $z$ sufficiently close to $y$, we have $f_{zy}f_{yx}=f_{zx}$.
\item For every $x\in X$, the induced germ of a map $X\to\TTT_{\ttt_x/}$ near $x$ is a stratification in the sense of Definition \ref{stratificationdef}.
\end{enumerate}
\end{definition}

Given a stratification by a poset $X\to\TTT$, for any element $\ttt\in\TTT$ there is a ``locally closed stratum'' $X^\ttt\subseteq X$ and a ``closed stratum'' $X^{\leq\ttt}\subseteq X$ equipped with a stratification $X^{\leq\ttt}\to\TTT^{\leq\ttt}=\TTT_{/\ttt}$.
Given a stratification by a category $X\to\TTT$, one can similarly define for any $\ttt\in\TTT$ topological spaces $X_\ttt$ and $X_{/\ttt}$, the latter equipped with a stratification by $\TTT_{/\ttt}$.
The points of $X_{/\ttt}$ are pairs $(x\in X,f_x:\ttt_x\to\ttt)$, and a neighborhood of such a pair $(x,f_x)$ consists of those $(y,f_y)$ where $y\in X$ is close to $x$ and $f_yf_{yx}=f_x$; the subspace $X_\ttt\subseteq X_{/\ttt}$ consists of those pairs $(x,f_x)$ for which $f_x$ is an isomorphism.
In general, both natural maps $X_\ttt\to X$ and $X_{/\ttt}\to X$ may fail to be injective.

\begin{figure}[ht]
\centering
\includegraphics{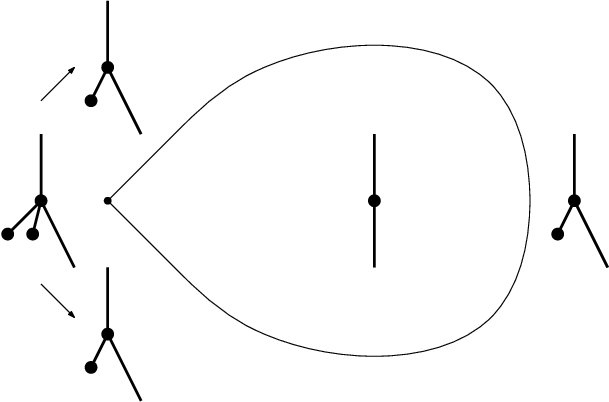}
\caption{Example of a stratification by a category.  The two maps out of the tree with three vertices collapse each of the two interior edges, respectively.  The closed stratum associated to the tree with a single vertex is simply the entire space, homeomorphic to $D^2$.  The closed stratum associated to the tree with two vertices is homeomorphic to the interval $[0,1]$.  The closed stratum associated to the tree with three vertices is two points.  Evidently, the latter two closed strata are ``non-embedded''.}\label{categorystratificationfigure}
\end{figure}

There is a tautological stratification
\begin{equation}
\label{MSstrat}\Mbar(T)\to\SSS_{/T}
\end{equation}
in the sense of Definition \ref{stratificationcategorydef} (this follows immediately from the definition of the Gromov topology given in \S\ref{gromovtopology} below).
The (not necessarily embedded!) closed stratum associated to $T'\to T$ is naturally isomorphic to $\Mbar(T')$.
To understand this stratification, the reader should satisfy themselves that they understand the example in Figure \ref{categorystratificationfigure}.

\subsection{Local models \texorpdfstring{$G$}{G}}\label{localmodelsection}

We introduce spaces $G_{T/}$ for $T\in\SSS$ which parameterize maps $\pi:T\to T'$ together with a collection of maps
\begin{equation}
\{\pi_\ast:\hat X_v\to\hat X_{\pi(v)}\}_{v\in V(T)}
\end{equation}
(often defined only away from a small neighborhood of infinity in $\hat X_v$), where $\hat X_v$ denotes the target of the $v$-component of a pseudo-holomorphic building of type $T$ (despite the notation, $\pi_\ast$ is not determined uniquely by $\pi$).
These spaces have a natural stratification $G_{T/}\to\SSS_{T/}$, and we denote by $G_{T//T'}$ the inverse image of $\SSS_{T//T'}$.
There is a unique basepoint $0\in G_{T/}$ corresponding to the identity map $T\to T$, and we only really care about a neighborhood of this basepoint in $G_{/T}$.

The definition of the Gromov topology on the moduli spaces $\Mbar(T)$ is based on these spaces $G_{T'//T}$, and we will see in Theorem \ref{localgluingnonthickened} that the regular loci of $\Mbar(T)$ are in fact locally modelled on $G_{T'//T}$.
To extract the virtual fundamental cycles of $\Mbar(T)$, a crucial step is to show that the stratifications $G_{T'//T}\to\SSS_{T'//T}$ are \emph{cell-like} (see \S\ref{celllikeproofsec}).

\begin{remark}
The spaces $G_{T/}$ considered here are model real toric singularities, as studied in Joyce \cite{joycegcorners}, Kottke--Melrose \cite{kottkemelrose}, and Gillam--Molcho \cite{gillammlocho} (though we do not appeal to any of their results).
This is most apparent after the change of variables $h=e^{-g}$.
\end{remark}

In Setup \ref{setupI}, all targets $\hat X_v$ are simply $\hat Y$, and we allow the map $\pi_\ast:\hat Y\to\hat Y$ to be any choice of $\RR$-translation.
However, since overall translations of $\hat Y$ for $v\in V(T')$ induce isomorphisms of pseudo-holomorphic buildings, it is only the differences of the translation parameters which really matter.
The result is that for $T\in\SSS_\I$, we define
\begin{equation}
(G_\I)_{T/}:=(0,\infty]^{E^\eint(T)}.
\end{equation}
There is a natural stratification $(G_\I)_{T/}\to(\SSS_\I)_{T/}$, sending $\g=\{\g_e\}_e$ to the map $\pi:T\to T'$ which contracts those edges $e\in E^\eint(T)$ for which $\g_e<\infty$.
We denote by $0\in(G_\I)_{T/}$ the point corresponding to all gluing parameters equal to $\infty$ (i.e.\ corresponding to no gluing at all).

In Setup \ref{setupII}, the targets $\hat X_v$ are either $\hat X$ or one of $\hat Y^\pm$, and the possible maps $\pi_\ast$ are $\hat Y^\pm\to\hat Y^\pm$, $\hat X\to\hat X$, and $\hat Y^\pm\to\hat X$.
Maps of the first type $\hat Y^\pm\to\hat Y^\pm$ are again allowed to be any $\RR$-translation.
Maps of the second type $\hat X\to\hat X$ must be the identity.
Maps of the third type $\hat Y^\pm\to\hat X$ are the pre-composition of the relevant boundary collar \eqref{endmarkingsI}--\eqref{endmarkingsII} with any $\RR$-translation of $\hat Y^\pm$.
For $T\in\SSS_\II$, the space of such collections of maps $\pi_\ast$ is given by
\begin{multline}\label{GIIdefeq}
(G_\II)_{T/}:=\Bigl\{(\{\g_e\}_e,\{\g_v\}_v)\in(0,\infty]^{E^\eint(T)}\times(0,\infty]^{V_{00}(T)}\Bigm|\\
\g_v=\g_e+\g_{v'}\text{ for edges }v\xrightarrow ev'\text{ with }v\in V_{00}(T)\Bigr\}.
\end{multline}
We interpret $\g_{v'}=0$ if $v'\notin V_{00}(T)$.  Here $V_{ij}(T)\subseteq V(T)$ is the subset of $v$ with $(\ast^+(v),\ast^-(v))=(i,j)$.
There is a natural stratification $(G_\II)_{T/}\to(\SSS_\II)_{T/}$, sending $\g=(\{\g_e\}_e,\{\g_v\}_v)$ to the map $\pi:T\to T'$ which contracts those edges $e$ for which $\g_e<\infty$, and changes $\ast(v)={00}$ to $\ast(v)=01$ for those vertices $v$ with $\g_v<\infty$.
We denote by $0\in(G_\II)_{T/}$ the point corresponding to all gluing parameters equal to $\infty$ (i.e.\ corresponding to no gluing at all).

In Setup \ref{setupIII}, the situation is identical to that of Setup \ref{setupII}, except that we also include variation in $t\in[0,1]$.
Namely, for $T\in\SSS_\III$, we set
\begin{multline}\label{GIIIdefeq}
(G_\III)_{T/}:=\Bigl\{(\{\g_e\}_e,\{\g_v\}_v)\in(0,\infty]^{E^\eint(T)}\times(0,\infty]^{V_{00}(T)}\Bigm|\\
\g_v=\g_e+\g_{v'}\text{ for }v\xrightarrow ev'\text{ with }v\in V_{00}(T)\Bigr\}\times\begin{cases}[0,1)&\s(T)=\{0\}\\(0,1)&\s(T)=(0,1)\\(0,1]&\s(T)=\{1\}\end{cases}
\end{multline}
(the first factor is identical to \eqref{GIIdefeq}).
There is a natural stratification $(G_\III)_{T/}\to(\SSS_\III)_{T/}$, sending $\g=(\{\g_e\}_e,\{\g_v\}_v,\g_t)$ to the map $\pi:T\to T'$ which contracts those edges $e$ for which $\g_e<\infty$, changes $\ast(v)=00$ to $\ast(v)=01$ for those vertices $v$ with $\g_v<\infty$, and has $\g_t\in\s(T')$.

In Setup \ref{setupIV}, if $T\in\SSS_\IV$ satisfies $\s(T)=\{0\}$ or $\s(T)=(0,\infty)$, then the situation is identical to that of Setup \ref{setupIII}, so we set
\begin{multline}
(G_\IV)_{T/}:=\Bigl\{(\{\g_e\}_e,\{\g_v\}_v)\in(0,\infty]^{E^\eint(T)}\times(0,\infty]^{V_{00}(T)}\Bigm|\\
\g_v=\g_e+\g_{v'}\text{ for }v\xrightarrow ev'\text{ with }v\in V_{00}(T)\Bigr\}\times\begin{cases}[0,\infty)&\s(T)=\{0\}\\(0,\infty)&\s(T)=(0,\infty)\end{cases}
\end{multline}
(the first factor is identical to \eqref{GIIdefeq}).
The case $\s(T)=\{\infty\}$ (i.e.\ the ``split cobordism'') is more complicated.
The possible targets $\hat X_v$ for a pseudo-holomorphic building of type $T$ with $\s(T)=\{\infty\}$ are $\hat Y^0$, $\hat X^{01}$, $\hat Y^1$, $\hat X^{12}$, $\hat Y^2$.
If $\s(T')=\{\infty\}$ as well, then the allowable maps are the same as in Setup \ref{setupII}.
If $\s(T')=(0,\infty)$, then the targets $\hat X_v$ are $\hat Y^0$, $\hat X^{02,t}$, $\hat Y^2$, for a choice of $t\in(0,\infty)$.
The possible maps $\pi_\ast$ are $\RR$-translations $\hat Y^0\to\hat Y^0$ and $\hat Y^2\to\hat Y^2$, the natural maps $\hat X^{01}\to\hat X^{02,t}$ and $\hat X^{12}\to\hat X^{02,t}$ defined away from a neighborhood of infinity in the negative and positive ends, respectively, and the natural (up to $\RR$-translation) maps $\hat Y^0\to\hat X^{02,t}$, $\hat Y^1\to\hat X^{02,t}$, $\hat Y^2\to\hat X^{02,t}$.
We conclude that for $T\in\SSS_\IV$ with $\s(T)=\{\infty\}$, we should set
\begin{multline}
(G_\IV)_{T/}:=\Bigl\{(\{\g_e\}_e,\{\g_v\}_v,\g_t)\in(0,\infty]^{E^\eint(T)}\times(0,\infty]^{V_{00}(T)\sqcup V_{11}(T)}\times(0,\infty]\Bigm|\\
\left.\begin{matrix}
\hphantom{\g_t}\g_v=\g_e+\g_{v'}\text{ for }v\xrightarrow ev'\text{, }{\ast(e)}=0\text{, }{\ast(v)}=00\hfill\\
\hphantom{\g_t}\g_v=\g_e+\g_{v'}\text{ for }v\xrightarrow ev'\text{, }{\ast(e)}=1\text{, }{\ast(v)}=11\hfill\\
\hphantom{\g_v}\g_t=\g_e+\g_{v'}\text{ for }v\xrightarrow ev'\text{, }{\ast(e)}=1\text{, }{\ast(v)}=01\hfill
\end{matrix}\right\}.
\end{multline}
We interpret $\g_{v'}=0$ if it is undefined.
There is a natural stratification $(G_\IV)_{T/}\to(\SSS_\IV)_{T/}$, which in the case $\s(T)=\{\infty\}$ sends $\g=(\{\g_e\},\{\g_v\},\g_t)$ to the map $\pi:T\to T'$ which contracts those edges $e$ for which $\g_e<\infty$, increments $\ast^-(v)$ for those vertices $v$ with $\g_v<\infty$, and has $\g_t\in\s(T')$.

\subsection{Deformation theory of Riemann surfaces}\label{rsdeformations}

Given a closed Riemann surface $C$ (without nodes), infinitesimal deformations of the almost complex structure are given by $C^\infty(C,\End^{0,1}(TC))$.
An infinitesimal diffeomorphism of $C$, that is a vector field $X$, gives rise to an infinitesimal deformation $\sL_Xj$, where $\sL$ denotes the Lie derivative.
Noting that $\End^{0,1}(TC)=TC\otimes_\CC\Omega_C^{0,1}$, we are thus lead to consider the two-term complex
\begin{equation}
C^\infty(C,TC)\xrightarrow{X\mapsto\sL_Xj}C^\infty(C,TC\otimes\Omega^{0,1}_C),
\end{equation}
which is simply the Dolbeaut complex calculating $H^\bullet(C,TC)$.
Thus $H^0(C,TC)$ classifies infinitesimal automorphisms of $C$ (though this is clear \emph{a priori}) and $H^1(C,TC)$ classifies infinitesimal deformations of $C$.

Now suppose $C$ has (a finite number of distinct) marked points $p_i$.
Infinitesimal deformations of the almost complex structure and the marked points are given by $C^\infty(C,\End^{0,1}(TC))\oplus\bigoplus_iT_{p_i}C$, and the infinitesimal deformation associated to a vector field $X$ is $\sL_Xj\oplus\bigoplus_i(-X(p_i))$.
This gives rise to the two-term complex
\begin{equation}\label{allpointsdefs}
C^\infty(C,TC)\xrightarrow{X\mapsto\sL_Xj\oplus\bigoplus_i(-X(p_i))}C^\infty(C,TC\otimes\Omega^{0,1}_C)\oplus\bigoplus_iT_{p_i}C,
\end{equation}
which calculates $H^\bullet(C,TC(-P))$, where $P=\sum_ip_i$ denotes the divisor of marked points.
It is somewhat more convenient for us to consider just those variations in the complex structure which are supported away from the marked points.
That is, we consider the subcomplex
\begin{equation}\label{cptdef}
\left\{X\in C^\infty(C,TC)\,\middle|\,\begin{matrix}X\text{ holomorphic near }p_i\hfill\\X(p_i)=0\hfill\end{matrix}\right\}\xrightarrow{X\mapsto\sL_Xj}C^\infty_c(C\setminus\{p_i\},TC\otimes\Omega_C^{0,1}).
\end{equation}
The inclusion of this subcomplex into \eqref{allpointsdefs} is a quasi-isomorphism (this is obvious except for surjectivity of the map on cokernels, which follows from elliptic regularity).

Now suppose $C$ has both marked points $p_i$ and ``doubly marked points'' $q_i$, meaning $q_i$ are marked points equipped with markings $\CC\xrightarrow\sim T_{q_i}C$ of their tangent fibers.
The analogue of \eqref{cptdef} is now
\begin{equation}
\left\{X\in C^\infty(C,TC)\,\middle|\,\begin{matrix}X\text{ holomorphic near }p_i,q_i\hfill\\X(p_i)=0\hfill\\X(q_i)=dX(q_i)=0\hfill\end{matrix}\right\}\xrightarrow{X\mapsto\sL_Xj}C^\infty_c(C\setminus\{p_i\}\cup\{q_i\},TC\otimes\Omega_C^{0,1}),
\end{equation}
which calculates $H^\bullet(C,TC(-P-2Q))$, where $P=\sum_ip_i$ and $Q=\sum_iq_i$.

Finally, when $C$ has nodes, we simply treat these as pairs of marked points on its normalization $\tilde C$ and consider the deformation theory of $\tilde C$ instead.

\begin{definition}[Normalization]\label{normalizationdef}
For any nodal Riemann surface $C$, we denote by $\tilde C$ the \emph{normalization} of $C$, i.e.\ the unique smooth Riemann surface equipped with a map $\tilde C\to C$ which identifies points in pairs to form the nodes of $C$.
\end{definition}

\subsection{Moduli of Riemann surfaces}\label{rsmoduli}

We denote by $\Mbar_{0,n}$ the Deligne--Mumford moduli space \cite{delignemumford} of stable genus zero compact nodal Riemann surfaces with $n$ marked points.
We regard $\Mbar_{0,n}$ as a complex manifold (which is compact).
It comes equipped with a universal curve $\Cbar_{0,n}\to\Mbar_{0,n}$ (which is secretly just the map forgetting the last marked point $\Mbar_{0,n+1}\to\Mbar_{0,n}$).
The pair $\Cbar_{0,n}\to\Mbar_{0,n}$ can be defined as the solution to a moduli problem, and it is thus unique up to unique isomorphism.

There is a stratification of $\Mbar_{0,n}$ by locally closed analytic submanifolds $\Mbar{}_{0,n}^{\#\nodes=r}$ indexed by integers $r\geq 0$ consisting of those curves with exactly $r$ nodes.
The open stratum of smooth curves $\Mbar{}_{0,n}^{\#\nodes=0}$ is usually denoted $\M_{0,n}$.
Its complement $\Mbar{}_{0,n}^{\#\nodes\geq 1}$ is a divisor with normal crossings, namely the pair $\Mbar{}_{0,n}^{\#\nodes\geq 1}\subseteq\Mbar_{0,n}$ is locally biholomorphic to $\{z_1\cdots z_k=0\}\subseteq\CC^N$ (any $k\leq N$), and in such local coordinates, the ``number of nodes'' is simply the number of coordinates $z_1,\ldots,z_k$ which vanish.

If $C$ is any point of $\M_{0,n}$ (i.e.\ any stable smooth Riemann surface of genus zero with $n$ marked points), then there is a canonical isomorphism
\begin{equation}
T_C\M_{0,n}=H^1(C,TC(-P)).
\end{equation}
The map $T_C\M_{0,n}\to H^1(C,TC(-P))$ is the Kodaira--Spencer map, which can be defined by (smoothly) trivializing the universal curve in a neighborhood of $C$ and invoking the deformation theory reviewed in \S\ref{rsdeformations}.
This map $T_C\M_{0,n}\to H^1(C,TC(-P))$ is well-known to be an isomorphism (see, for instance, Wendl \cite[Theorem 4.30]{wendljholnotes}).
For $C\in\Mbar_{0,n}$ with $r$ nodes, by considering the normalization $\tilde C$ equipped with the inverse images of the marked points $\tilde P\subseteq\tilde C$ and the inverse images of the nodes $\tilde N\subseteq\tilde C$, we deduce a similar isomorphism
\begin{equation}
T_C\M_{0,n}^{\#\nodes=r}=H^1(\tilde C,T\tilde C(-\tilde P-\tilde N)).
\end{equation}

\subsection{Elliptic \emph{a priori} estimates on pseudo-holomorphic curves}\label{ellipticestimates}

We record here some fundamental \emph{a priori} estimates which guarantee the regularity of pseudo-holomorphic curves.  These estimates play a fundamental role in establishing the basic local properties of moduli spaces of pseudo-holomorphic curves.

The first result implies that for pseudo-holomorphic curves, $C^0$-close implies $C^\infty$-close.

\begin{lemma}[Gromov \cite{gromov}]\label{apriorijhol}
Let $u:D^2(1)\to(B^{2n}(1),J)$ be $J$-holomorphic, where $J$ is tamed by $d\lambda$.  For every $k<\infty$, we have
\begin{equation}
\left\|u\right\|_{C^k(D^2(1-s))}\leq M\cdot(1+s^{-k})
\end{equation}
for some constant $M=M(k,J,\lambda)<\infty$.
\end{lemma}

\begin{proof}
The Gromov--Schwarz Lemma (see Gromov \cite[1.3.A]{gromov} or Muller \cite[Corollary 4.1.4]{muller}) is the case $k=1$.  Standard elliptic bootstrapping allows one to upgrade this to bounds on all higher derivatives (see \cite[Lemma B.11.4]{pardonimplicitatlas}).  Note that it is enough to bound $D^ku$ at $0\in D^2(1)$; the full result is then recovered by restricting to the maximal disk centered at a given point in $D^2(1)$.
\end{proof}

The next two ``long cylinder estimates'' are crucial for proving compactness of moduli spaces (which we do not discuss) and surjectivity of the gluing map (which we discuss in detail).
Note that both of these next results imply corresponding results for semi-infinite cylinders $[0,\infty)\times S^1$, simply by applying the results as stated to $[0,N]\times S^1\subseteq[0,\infty)\times S^1$ and letting $N\to\infty$.

\begin{lemma}[Well-known]\label{apriorinode}
Let $u:[0,N]\times S^1\to(B^{2n}(1),J)$ be $J$-holomorphic, where $J$ is tamed by $d\lambda$.  For every $k<\infty$ and $\varepsilon>0$, we have
\begin{equation}
\left\|u\right\|_{C^k([s,N-s]\times S^1)}\leq M\cdot(1+s^{-k})e^{-(1-\varepsilon)s}
\end{equation}
for some constant $M=M(k,\varepsilon,J,\lambda)<\infty$.
\end{lemma}

\begin{proof}
See \cite{mcduffsalamonJholsymp} or \cite[Proposition B.11.1]{pardonimplicitatlas}.
\end{proof}

\begin{proposition}[{Hofer--Wysocki--Zehnder \cite[Theorems 1.1, 1.2, and 1.3]{hwzsmallarea}}]\label{hwzcylinderestimate}
Let $(Y,\lambda,J)$ be as in Setup \ref{setupI}.
Suppose all periodic orbits of $R_\lambda$ of action $\leq E_0<\infty$ are non-degenerate, let $\sigma>0$ denote the minimum gap between any two elements of $\{0\}\cup a(\PP)\cap[0,E_0]$, and let $\varepsilon>0$.
For any $\hat J$-holomorphic map $u:[0,N]\times S^1\to\hat Y$ satisfying
\begin{align}
\label{hwzcylestlambdaenergy}\sup_{\begin{smallmatrix}\varphi:\RR\to[0,1]\\\varphi'(s)\geq 0\end{smallmatrix}}\int_{[0,N]\times S^1}u^\ast d(\varphi\lambda)&\leq E_0,\\
\label{hwzcylestomegaenergy}\int_{[0,N]\times S^1}u^\ast d\lambda&\leq\sigma-\varepsilon,
\end{align}at least one of the following holds for $H=H(Y,\lambda,J,E_0,\varepsilon)<\infty$:
\begin{itemize}
\item$u([H,N-H]\times S^1)\subseteq B_\varepsilon(u(\frac N2,0))$.
\item There exists $\gamma\in\PP$ with $a(\gamma)=L\leq E_0$ such that for $s\geq H$,
\begin{equation}
\left\|u(s,t)-(Ls+b,\tilde\gamma(t))\right\|_{C^0([s,N-s]\times S^1)}\leq H\cdot e^{-(\delta_\gamma-\varepsilon)\min(s-H,N-H-s)}
\end{equation}
for some $b\in\RR$ and some $\tilde\gamma:S^1\to Y$ with $\partial_t\tilde\gamma=L\cdot R_\lambda(\tilde\gamma)$ parameterizing $\gamma\in\PP$, where $\delta_\gamma>0$ denotes the smallest absolute value of any eigenvalue of the asymptotic operator of $\gamma$ (see Definition \ref{smalleigenvalue}).
\end{itemize}
\end{proposition}

The meaning of the energy bounds \eqref{hwzcylestlambdaenergy}--\eqref{hwzcylestomegaenergy} is discussed further in \S\ref{compactness}.

\begin{remark}
Our statement of Proposition \ref{hwzcylinderestimate} is slightly more precise than that given by Hofer--Wysocki--Zehnder \cite{hwzsmallarea} in that they state it with an unspecified constant $\delta>0$ in place of $\delta_\gamma>0$.
We emphasize, however, that knowing the correct constant is not needed for our work in this paper.
\end{remark}

\subsection{Gromov topology}\label{gromovtopology}

We now define the Gromov topology on the moduli spaces $\Mbar(T)$, and we show that this topology is Hausdorff.

Recall that a \emph{neighborhood} of a point $x$ in a topological space $X$ is a subset $N\subseteq X$ such that there exists an open set $U\subseteq X$ with $x\in U\subseteq N$.
The \emph{neighborhood filter} of $x$ is the collection $\N_x\subseteq 2^X$ of all its neighborhoods.
A \emph{neighborhood base} at $x$ is a cofinal subset $\N_x^\circ\subseteq\N_x$ (meaning, for every $N\in\N_x$ there exists $N'\in\N_x^\circ$ with $N'\subseteq N$).
Given a neighborhood base $\N_x^\circ$, we may recover $\N_x$ as the collection of all subsets of $X$ containing some element of $\N_x^\circ$.

\begin{lemma}\label{nbhdunique}
The system of neighborhood filters $\{\N_x\subseteq 2^X\}_{x\in X}$ at all the points of a topological space $X$ satisfies the following properties:
\begin{enumerate}
\item\label{nbhdI}For every $N\in\N_x$, we have $x\in N$.
\item For every $N\in\N_x$ and every $N'\supseteq N$, we have $N'\in\N_x$.
\item For every $N,N'\in\N_x$, we have $N\cap N'\in\N_x$.
\item\label{nbhdIV}For every $N\in\N_x$, there exists $N'\in\N_x$ such that $N\in\N_y$ for every $y\in N'$.
\end{enumerate}
Conversely, any system of subsets $\{\N_x\subseteq 2^X\}_{x\in X}$ satisfying the above properties is the system of neighborhood filters for a unique topology on $X$.\qed
\end{lemma}

We define the Gromov topology by specifying the neighborhood filter $\N_x$ of any given pseudo-holomorphic building $x$ and appealing to Lemma \ref{nbhdunique}.
More precisely, (1) we define the ``$\varepsilon$-neighborhood of $x$'' for any (sufficiently small) $\varepsilon>0$, given certain choices of auxiliary data (metrics, etc.), (2) we define the neighborhood filter $\N_x$ of $x$ to be that resulting from declaring that the collection of all $\varepsilon$-neighborhoods (with respect to fixed auxiliary data) forms a neighborhood base $\N_x^\circ$ at $x$, and (3) we observe that $\N_x$ is independent of the choice of auxiliary data.

\begin{definition}[Gromov topology]\label{gromovtopologydef}
The $\varepsilon$-neighborhood of a pseudo-holomorphic building $x=(\{C_v\},\{p_{v,e}\},\{u_v\},\{\tilde b_e\},\{m_e\})$ of type $T$ is defined as follows.
Let $\g\in G_{T/}$ be $\varepsilon$-close to the basepoint $0\in G_{T/}$.
Recall from \S\ref{localmodelsection} that $\g$ determines a map $\pi:T\to T'$ together with a collection of (possibly only partially defined) maps
\begin{equation}
\pi_\ast:\hat X_v\to\hat X_{\pi(v)}
\end{equation}
for $v\in V(T)$.
Fix (independently of $\varepsilon>0$) large compact subsets $\hat X_v^c\subseteq\hat X_v$.
Consider the curves
\begin{equation}
C_{v'}:=\bigsqcup_{\pi(v)=v'}C_v\Bigm/\sim
\end{equation}
for $v'\in V(T')$, where $p_{v_1,e}\sim p_{v_2,e}$ for edges $v_1\xrightarrow ev_2$ with $\pi(v_1)=\pi(v_2)=v'$.
These curves $C_{v'}$ are equipped with points $\{p_{v',e'}\}$ for edges $e'\in E(T')$ incident at $v'$, namely the images of the corresponding points $p_{v,e}\in C_v$ for $\pi(e)=e'$.
Let $C_{v'}'$ be any curve obtained from $C_{v'}$ by the following operations:
\begin{enumerate}
\item\label{smallmodificationepsawayfromnodespunctures}An $\varepsilon$-$C^\infty$-small modification of the almost complex structure away from the $\varepsilon$-neighborhood of the points $\{p_{v',e'}\}$ and the nodes of $C_{v'}$.
\item\label{arbnearnodespunctures}An arbitrary modification of the almost complex structure over the $\varepsilon$-neighborhood of the points $\{p_{v',e'}\}$ and the nodes of $C_{v'}$.
\item\label{resolnodespunctures}Resolving the nodes of $C_{v'}$ created by the identifications $p_{v_1,e}\sim p_{v_2,e}$, and possibly resolving the other nodes of $C_{v'}$ (those coming from the nodes of $C_v$ for $\pi(v)=v'$).
(Resolving a node means that its $\varepsilon$-neighborhood $D^2\vee D^2$ gets replaced with a cylinder $[0,1]\times S^1$, equipped here with an arbitrary almost complex structure).
\end{enumerate}
The locus referred to in \ref{smallmodificationepsawayfromnodespunctures} will be called the $\varepsilon$-thick part of $C_{v'}'$, and the locus referred to in \ref{arbnearnodespunctures}--\ref{resolnodespunctures} will be called the $\varepsilon$-thin part of $C_{v'}'$.
The $\varepsilon$-thin part of $C_{v'}'$ is broken down further into the nodal $\varepsilon$-thin part (that coming from nodes of $C_v$) and the Reeb $\varepsilon$-thin part (that coming from the points $\{p_{v,e}\}$ of $C_v$, both those identified to form nodes in $C_{v'}$ and those which become the points $\{p_{v',e'}\}$).
Now let $x'=(\{C_{v'}'\},\{p_{v',e'}\},\{u_{v'}'\},\{\tilde b_{e'}\},\{m_{e'}\})$ be a pseudo-holomorphic building of type $T'$, upon which we impose the following requirements.
Most basically, we require that $u'$ be $\varepsilon$-$C^0$-close to $u$ (measured with respect to fixed metrics on $\hat X_{v'}$ which are $\RR$-invariant in the ends) in the following sense:
\begin{enumerate}[resume]
\item Over the $\varepsilon$-thick part of $C_{v'}'$, we require $u_{v'}'$ to be $\varepsilon$-$C^0$-close to $\pi_\ast\circ\bigsqcup_{\pi(v)=v'}u_v$.
\item Over the nodal $\varepsilon$-thin part of $C_{v'}'$, we require $u_{v'}'$ to map into the $\varepsilon$-ball around the image under $\pi_\ast\circ u_v$ of the corresponding node of $C_v$.
\item\label{reebcloseI}Over the Reeb $\varepsilon$-thin part of $C_{v'}'$ coming from an edge $v_1\xrightarrow ev_2$ with $\pi(v_1)=\pi(v_2)=v'$, we require $u_{v'}'$ to map into the $\varepsilon$-neighborhood of the trivial cylinder over $\gamma_e$ inside the finite cylindrical region in $\hat X_{v'}$ in between $\pi_\ast(\hat X_{v_1}^c)$ and $\pi_\ast(\hat X_{v_2}^c)$ (recall the large compact subsets $\hat X_v^c\subseteq\hat X_v$ fixed above).
\item\label{reebcloseII}Over the Reeb $\varepsilon$-thin part of $C_{v'}'$ near a positive (respectively, negative) puncture $p_{v',e'}$ associated to an edge $v_1\xrightarrow ev_2$ not collapsed by $\pi$ with $\pi(v_2)=v'$ (respectively, $\pi(v_1)=v'$), we require $u_{v'}'$ to map into the $\varepsilon$-neighborhood of the trivial cylinder over $\gamma_e$ inside the half-infinite cylindrical region in $\hat X_{v'}$ above $\pi_\ast(\hat X_{v_2}^c)$ (respectively, below $\pi_\ast(\hat X_{v_1}^c)$).
\end{enumerate}
Note that the \emph{a priori} estimates of Hofer--Wysocki--Zehnder \cite[Theorems 1.1, 1.2, and 1.3]{hwzsmallarea} recalled in Proposition \ref{hwzcylinderestimate} imply that \ref{reebcloseI}--\ref{reebcloseII} imply that over the Reeb $\varepsilon$-thin parts of $C_{v'}'$, the map $u_{v'}'$ actually converges very rapidly to a trivial cylinder (to apply Proposition \ref{hwzcylinderestimate}, we should check that the integral of $d\lambda$ over these $\varepsilon$-thin parts is small, which can be seen by using Lemma \ref{apriorijhol} over the nearby $\varepsilon$-thick part and applying Stokes' theorem).
It follows that our choice of fixed large compact subsets $\hat X_v^c\subseteq\hat X_v$ does not matter.
This convergence to trivial cylinders also allows us to make sense of the following final requirements on the ``discrete data'' of $x'$:
\begin{enumerate}[resume]
\item\label{matchingsameI}The asymptotic markers $\{\tilde b_{e'}\}$ and matching isomorphisms $\{m_{e'}\}$ of $x'$ must agree with those descended from $x$.
\item\label{matchingsameII}Over the glued edges $e\in E(T)$ (meaning those collapsed by $\pi$), the matching isomorphism induced by $u_{v'}'$ must agree with $m_e$.
\end{enumerate}
The $\varepsilon$-neighborhood of $x$ is defined as the set of all possible $x'$ above.
It is straightforward to check that the resulting neighborhood filter $\N_x$ at $x$ is independent of the choices involved.
We leave it as an exercise to check that these neighborhood filters satisfy the conditions in Lemma \ref{nbhdunique} (only \ref{nbhdIV} requires thought).
We thus obtain a well-defined topology on each of the spaces $\Mbar(T)$, which we call the Gromov topology.
\end{definition}

\begin{remark}\label{nearbyinDM}
There are various other ways to define the curve $C_{v'}'$ and its identification with $C_{v'}$ away from their $\varepsilon$-thin parts.
For example, one could add stabilizing marked points to $C_{v'}$ and then take $C_{v'}'$ to be a nearby fiber of the universal curve over Deligne--Mumford space.
It is an exercise to check that this definition of $C_{v'}'$ gives rise to the same neighborhood filter.
\end{remark}

\begin{remark}
A sequence (or net) $x_i$ of pseudo-holomorphic buildings converges in the Gromov topology to a pseudo-holomorphic building $x$ iff for every $\varepsilon>0$, it holds that $x_i$ lies in the $\varepsilon$-neighborhood of $x$ for all sufficiently large $i$.
Indeed, given any neighborhood base $\N_x^\circ$ at a point $x$ of a topological space, a sequence (or net) $x_i$ converges to $x$ iff for every $N\in\N_x^\circ$, we have $x_i\in N$ for all sufficiently large $i$.
\end{remark}

\begin{lemma}
The Gromov topology on $\Mbar(T)$ is Hausdorff.
\end{lemma}

\begin{proof}
Let $x^{(1)}$ and $x^{(2)}$ be two stable pseudo-holomorphic buildings such that for every $\varepsilon>0$, there exists a pseudo-holomorphic building $x^\varepsilon$ lying in the $\varepsilon$-neighborhood of both $x^{(1)}$ and $x^{(2)}$.
Let $C^{(1)}:=\bigsqcup_vC^{(1)}_v/\sim$ (where $p_{v_1,e}\sim p_{v_2,e}$ for edges $v_1\xrightarrow ev_2$ as usual), and similarly define $C^{(2)}$ and $C^\varepsilon$.

Choose finite collections of points $P^{(1)}\subseteq C^{(1)}$ and $P^{(2)}\subseteq C^{(2)}$ (disjoint from the nodes and punctures) which stabilize each.
By Remark \ref{nearbyinDM}, the curve $C^\varepsilon$ may be equipped with marked points $P^{(1),\varepsilon}$ and $P^{(2),\varepsilon}$ (respectively) so that it is $\varepsilon$-close to $(C^{(1)},P^{(1)})$ and $(C^{(2)},P^{(2)})$ (respectively) in Deligne--Mumford space.
We consider the curve $(C^\varepsilon,P^{(1),\varepsilon}\sqcup P^{(2),\varepsilon})$ obtained by adding both $P^{(1),\varepsilon}$ and $P^{(2),\varepsilon}$ as additional marked points (note that, without loss of generality, we may perturb either $P^{(1),\varepsilon}$ or $P^{(2),\varepsilon}$ so that they are disjoint).

Since Deligne--Mumford space is compact, we may assume that $(C^\varepsilon,P^{(1),\varepsilon}\sqcup P^{(2),\varepsilon})$ converges to some stable $(C,P^{(1)}\sqcup P^{(2)})$ as $\varepsilon\to 0$.
Since forgetting marked points and stabilizing is continuous, we obtain canonical identifications
\begin{align}
i_1:(C,P^{(1)})^\st&\xrightarrow\sim(C^{(1)},P^{(1)}),\\
i_2:(C,P^{(2)})^\st&\xrightarrow\sim(C^{(2)},P^{(2)}).
\end{align}
Any irreducible component $C_0\subseteq C$ contracted by $i_1$ but not by $i_2$ would contradict stability of $x^{(2)}$.
Thus every irreducible component $C_0\subseteq C$ contracted by either $i_1$ or $i_2$ is contracted by both.
We conclude that
\begin{equation}
i^{21}:=i_2\circ i_1^{-1}:C^{(1)}\xrightarrow\sim C^{(2)}
\end{equation}
is an isomorphism and that $u^{(2)}\circ i^{21}=u^{(1)}$.

To conclude that $i^{21}$ defines an isomorphism $x^{(1)}\to x^{(2)}$, it remains only to observe that the asymptotic markers and matching isomorphisms also agree as a consequence of \ref{matchingsameI}--\ref{matchingsameII} in Definition \ref{gromovtopologydef}.
\end{proof}

\subsection{Compactness}\label{compactness}

We now recall the compactness results due to Bourgeois--Eliashberg--Hofer--Wysocki--Zehnder \cite{sftcompactness} in the context of the moduli spaces $\Mbar(T)$ which we have defined.

We begin by recalling from \cite{sftcompactness} the notion of \emph{Hofer energy} first introduced in \cite{hoferweinstein}.
Note that the usual notion of energy of a pseudo-holomorphic map (namely the integral of the pullback of the symplectic form) is not well-behaved when the target symplectic manifold is a symplectization $\hat Y$ equipped with the symplectic form $\hat\omega:=d\hat\lambda=d(e^s\lambda)=e^s(d\lambda+ds\wedge\lambda)$, as this energy is not invariant under $\RR$-translation and is almost always infinite.
We therefore remove the $e^s$ factor, leaving $d\lambda+ds\wedge\lambda$, and we further treat each of these terms separately (and differently).
The resulting notions of energy we now discuss are (beyond the case of symplectizations) only well-defined up to an overall multiplicative constant.

The \emph{$\omega$-energy} of a pseudo-holomorphic map $u:C\to\hat Y$ is the integral
\begin{equation}
E_\omega(u):=\int_Cu^\ast\omega\qquad\text{where }\omega:=d\lambda.
\end{equation}
To define the $\omega$-energy of a pseudo-holomorphic map to a symplectic cobordism $u:C\to\hat X$, we integrate the closed $2$-form $\omega$ defined by splicing together $\hat\omega$ with $\omega^\pm=d\lambda^\pm$ in the ends; more explicitly, in the positive end $\omega:=d(\alpha(s)\lambda^+)$ for $\alpha:\RR\to\RR$ satisfying $\alpha(s)=e^s$ for $s\ll 0$, $\alpha'(s)=0$ for $s\gg 0$, and $\alpha'(s)\geq 0$, and similarly in the negative end.
To define the $\omega$-energy of a pseudo-holomorphic map to an almost split symplectic cobordism $\hat X^{02,t}$ with $t\to\infty$ as in case (IV), we integrate the closed $2$-form $\omega^{02,t}$ defined as the descent of $\omega^{01}\sqcup\omega^{12}$ from $\hat X^{01}\sqcup\hat X^{12}$; over the identified regions these $2$-forms agree up to a scaling factor which is fixed as $t\to\infty$ (and hence irrelevant).

The $\omega$-energy of a pseudo-holomorphic building is simply the sum of the $\omega$-energies of each of its components.

The integrand of the $\omega$-energy is pointwise $\geq 0$ for pseudo-holomorphic $u$, with equality iff $du=0$ or $u$ is tangent to the vertical distribution $\RR\partial_s$ at a point of a symplectization region where $\omega$ is (up to a constant) $d\lambda$.
If $u$ is asymptotic to Reeb orbits at positive/negative infinity, then the $\omega$-energy is the pairing between the relative homology class of $u$ and the relative cohomology class of $\omega$ (both relative to the positive/negative asymptotics of $u$).

The \emph{$\lambda$-energy} of a pseudo-holomorphic map $u:C\to\hat Y$ is the supremum
\begin{equation}\label{lambdaenergyforsymplectization}
E_\lambda(u):=\sup_{\begin{smallmatrix}\phi:\RR\to[0,1]\\\phi'(s)\geq 0\\\phi(s)=0\;s\ll 0\\\phi(s)=1\;s\gg 0\end{smallmatrix}}\int_Cu^\ast(d\phi\wedge\lambda).
\end{equation}
The $\lambda$-energy of a pseudo-holomorphic map to a symplectic cobordism $u:C\to\hat X$ is the supremum of $\int_Cu^\ast(d\phi\wedge\lambda^\pm)$ over all $\phi:\hat X\to[0,1]$ which are constant away from the ends \eqref{endmarkingsI}--\eqref{endmarkingsII} and which in each of these ends are as above, namely they are functions of $s$ satisfying $\phi'(s)\geq 0$ and $\phi(s)=0$ for $s\ll 0$ and $\phi(s)=1$ for $s\gg 0$.
To define the $\lambda$-energy of a pseudo-holomorphic map to an almost split symplectic cobordism $\hat X^{02,t}$ for large $t$ as in case (IV), we also allow $\phi$ to be non-constant over the middle finite symplectization region of $\hat X^{02,t}$ where the integrand is $d\phi\wedge\lambda^1$.

To define the $\lambda$-energy of a pseudo-holomorphic building of type $T$, we take the supremum (of the analogous sum of integrals) over all choices of $\{\phi_v:\hat X_v\to[0,1]\}_v$ as follows.
Each $\phi_v$ must be constant away from the symplectization regions, and over each symplectization region $\phi_v$ must be a function of $s$ with $\phi_v'(s)\geq 0$.
Finally, each $\phi_v$ must be constant in a neighborhood of positive/negative infinity; these constants must agree across interior edges of $T$ and must equal $1$ and $0$ for input/output edges of $T$ respectively.

The integrand of the $\lambda$-energy is pointwise $\geq 0$ for pseudo-holomorphic $u$.
An argument involving Stokes' theorem shows that any candidate value $\int_Cu^\ast(d\phi\wedge\lambda)$ of the $\lambda$-energy is equal to (a constant times) any other candidate value, up to an error bounded above by (a constant times) the $\omega$-energy.

The \emph{Hofer energy} of a pseudo-holomorphic map or building is the sum of the $\omega$-energy and the $\lambda$-energy:
\begin{equation}
E(u):=E_\omega(u)+E_\lambda(u).
\end{equation}
Up to a multiplicative constant, the Hofer energy depends only on the asymptotics and homology class of the map/building $u$ (for buildings, by this we mean only $\gamma_{e^+}$, $\{\gamma_{e^-}\}$, and $\#_v\beta_v$, not any of the intermediate Reeb orbits or the individual $\beta_v$); compare \cite[Propositions 5.13 and 6.3]{sftcompactness}.
Indeed, we have already seen above that the $\omega$-energy depends only on the asymptotics and homology class of $u$.
We have also seen above that the various candidate $\lambda$-energies differ by at most the $\omega$-energy, so for the purposes of defining the Hofer energy, we may as well replace the $\lambda$-energy term with one specific such candidate value, such as the action of the positive asymptotic Reeb orbit or the sum of the actions of the negative asymptotic Reeb orbits.
We conclude that for pseudo-holomorphic buildings $u$ of type $T$, we have
\begin{equation}\label{hoferenergyhomology}
E(u)\asymp a(\gamma_{e^+})+\sum_{e^-}a(\gamma_{e^-})+\langle\omega,\#_v\beta_v\rangle
\end{equation}
(equality up to a multiplicative constant).
Finally, note that when we consider families of cobordisms, the last term on the right hand side may be trivially bounded above by $\sup_t\langle\omega^t,\#_v\beta_v\rangle$ in case (III) and by $\sup_t\langle\omega^{02,t},\#_v\beta_v\rangle$ in case (IV) (note that the cohomology class of $\omega^{02,t}$ is independent of $t$ for $t$ sufficiently large, so this supremum is finite).
Define $E(T)$ as the right hand side of \eqref{hoferenergyhomology} with $\langle\omega,\#_v\beta_v\rangle$ replaced with the relevant supremum in cases (III) and (IV), and note that $E(T)=E(T')$ for any morphism $T\to T'$.

\begin{theorem}[{\cite[Theorems 10.1, 10.2, 10.3]{sftcompactness}}]\label{compactnessresult}
For any fixed datum $\D$ as in Setup \ref{setupI}--\ref{setupIV}, each moduli space $\Mbar(T)$ is compact, and for every $E<\infty$ there are at most finitely many (isomorphism classes of) $T\in\SSS$ with $\Mbar(T)$ non-empty and $E(T)\leq E$.
In particular, given $T\in\SSS$ there are at most finitely many (isomorphism classes of) $T'\to T$ with $\Mbar(T')$ non-empty.
\end{theorem}

The arguments of \cite[\S 10]{sftcompactness} show that every net of pseudo-holomorphic buildings of Hofer energy $\leq E$ has a convergent subnet, which proves Theorem \ref{compactnessresult}.
Although \cite[\S 10]{sftcompactness} is presented as a proof of sequential compactness, it applies essentially without change to yield the true compactness result we claim here.
The fact that the moduli spaces considered in \cite{sftcompactness} differ slightly from those considered here (there is a surjective forgetful map from the moduli spaces in \cite{sftcompactness} to ours) also makes no serious difference to the argument.

\subsection{Linearized operators}\label{linearizedopsec}

We now recall the relevant linearized operators associated to the pseudo-holomorphic curves that we consider.

\begin{definition}[Linearized operators of Reeb orbits]\label{reeblinearized}
For any $\gamma\in\PP$ and a choice of parameterization $\tilde\gamma:S^1\to Y$, denote by $D_\gamma^N$ the operator $J\sL_{R_\lambda}$ (where $\sL$ denotes the Lie derivative) acting on sections of $\tilde\gamma^\ast\xi$ over $S^1$, and define $D_\gamma:=D_\gamma^N\oplus i\cdot d$ for $i\cdot d$ acting on functions $S^1\to\CC$.
\end{definition}

The operator $D_\gamma^N$ defines a Fredholm map $H^s(S^1,\tilde\gamma^\ast\xi)\to H^{s-1}(S^1,\tilde\gamma^\ast\xi)$ for any $s\in\RR$ (either by Fourier analysis or by appealing to general elliptic regularity results).
By the analytic Fredholm theorem, the resolvent $(D_\gamma^N-zI)^{-1}:H^{s-1}(S^1,\tilde\gamma^\ast\xi)\to H^s(S^1,\tilde\gamma^\ast\xi)$ varies meromorphically in $z\in\CC$, with poles corresponding to the eigenvalues $\sigma(D_\gamma^N)$ of $D_\gamma^N$.
Since $D_\gamma^N$ is formally self-adjoint, all these eigenvalues are real.
The same discussion applies to $i\cdot d$ and to $D_\gamma$, and
\begin{equation}
\sigma(D_\gamma)=\sigma(D_\gamma^N)\cup\sigma(i\cdot d)=\sigma(D_\gamma^N)\cup\ZZ.
\end{equation}
Clearly $0\notin\sigma(D_\gamma^N)$ iff $\gamma$ is non-degenerate (as is our standing assumption).

\begin{definition}\label{smalleigenvalue}
Denote by $\delta_\gamma>0$ the smallest absolute value of any element of $\sigma(D_\gamma)\setminus\{0\}$.
\end{definition}

\begin{convention}[Choices of metric and connection on domains]\label{metricconnectionchoicedomain}
For the purposes of defining function spaces, stating estimates, etc.\ involving a Riemann surface $C$ equipped with punctures $\{p_e\}_e$, we use a choice of holomorphic cylindrical coordinates
\begin{equation}\label{cylindricalcoordsonC}
[0,\infty)\times S^1\to C\setminus p_e
\end{equation}
near each $p_e$.  We equip $C$ with a Riemannian metric which equals $ds^2+dt^2$ near $p_e$, and we equip $TC$ with a $j$-linear connection for which $\partial_s$ is parallel near $p_e$.  Different choices of this data will result in uniformly commensurable norms and estimates, so the particular choice is not important.
\end{convention}

\begin{convention}[Choices of metric and connection on targets]\label{metricconnectionchoicetarget}
For the purposes of defining function spaces, stating estimates, expressing linearized operators, etc.\ involving symplectizations $\hat Y$, we use any choice of $\RR$-invariant Riemannian metric on $\hat Y$ and any choice of connection on $T\hat Y$ which is $\hat J$-linear (meaning $\hat J\nabla_XY=\nabla_X\hat JY$, i.e.\ $\nabla\hat J=0$) and which is pulled back from a connection on $T\hat Y|_{\{0\}\times Y}=TY\oplus\RR$.
On symplectic cobordisms, use metrics and connections which are of this form in the positive/negative ends.
Different choices of metric and connection will always result in uniformly commensurable norms, so the particular choice is not important.
\end{convention}

\begin{convention}[Regularity classes of functions on nodal curves]\label{nodalregular}
A function on a nodal curve $C$ of a certain regularity class ($C^\infty$, $W^{k,2}$, etc.)\ simply means a function on its normalization $\tilde C$ (see Definition \ref{normalizationdef}) of the given regularity class which agrees across each pair of points identified under $\tilde C\to C$ to form the nodes of $C$.
We will only speak of functions on nodal $C$ in regularity classes which embed into $C^0$, so that the above makes sense.
Recall, in particular, that $W^{k,2}\subseteq C^0$ when $k\geq 2$ \cite[Lemma 5.17]{adamssobolev}.
\end{convention}

We refer to Adams \cite{adamssobolev} for basic properties and definitions of Sobolev spaces, including here only a brief discussion.
Recall that for integers $k\geq 0$, a function is said to be of class $W^{k,2}$ iff all its derivatives (in the sense of distributions) of order $\leq k$ are in $L^2$, and the $W^{k,2}$-norm of such a function is defined as the sum of the $L^2$-norms of these derivatives.
A mollification argument \cite[Theorem 3.16]{adamssobolev} of Meyers--Serrin shows that any compactly supported $W^{k,2}$ function can be approximated in the $W^{k,2}$-topology by a sequence of smooth functions supported in a neighborhood of the original support.

It is well-known that it is frequently useful to introduce ``weights'' into the definition of Sobolev spaces of functions on certain classes of non-compact manifolds.
We now recall the weighted Sobolev spaces relevant to our setting of pseudo-holomorphic curves with cylindrical ends.

\begin{remark}\label{kdeltacomment}
We frame our discussion for general $k$ and $\delta$ so as to be precise as possible regarding the constraints on $k$ and $\delta$ required for each step.
Nevertheless, it may help orient the reader to remark that, for our intended purpose, it is sufficient to fix some sufficiently large value of $k$ ($k\geq 4$ is sufficient) and to restrict our attention to $\delta>0$ satisfying $\delta<1$ and $\delta<\delta_\gamma$ for all Reeb orbits $\gamma$ under consideration.
\end{remark}

\begin{definition}[Weighted Sobolev spaces $W^{k,2,\delta}$]\label{sobolevspaces}
For $k\geq 0$ and $\delta<1$, we define weighted Sobolev spaces
\begin{align}
\label{linearizeddomainrestricted}&W^{k,2,\delta}(C,u^\ast T\hat X),\\
\label{linearizedcodomain}&W^{k,2,\delta}(C,u^\ast T\hat X_{\hat J}\otimes_\CC\Omega^{0,1}_C),
\end{align}
where $u:C\setminus\{p_e\}_e\to\hat X$ is a smooth map which in a neighborhood of each puncture $p_e$ is $C^\infty$-convergent with weight $\delta'>\delta$ to a Reeb orbit $\gamma_e\in\PP^\pm$ in the sense of Definition \ref{asymptoticdef}.
For \eqref{linearizeddomainrestricted}, we allow $C$ to have nodes, though in this case we require $k\geq 2$ in accordance with Convention \ref{nodalregular}.

The $W^{k,2,\delta}$-norm is defined as follows.  Away from $\{p_e\}_e$, we use the usual $W^{k,2}$-norm, and near a given $p_e$, the contribution to the norm squared is given by
\begin{equation}\label{endnorm}
\int_{[0,\infty)\times S^1}\sum_{j=0}^k\left|D^jf\right|^2e^{2\delta s}\;ds\,dt
\end{equation}
for any choice of local coordinates \eqref{cylindricalcoordsonC}.
Equivalently up to commensurability, one can set $\left\|f\right\|_{k,2,\delta}:=\left\|\mu\cdot f\right\|_{k,2}$ for some smooth function $\mu$ which equals $1$ away from the ends and which equals $e^{\delta\left|s\right|}$ in any end (multiplication by any such $\mu$ thus defines an isomorphism $W^{k,2,\delta}\to W^{k,2}$, allowing properties of $W^{k,2}$ to be lifted to $W^{k,2,\delta}$).
Our hypotheses that $\delta<1$ and that $u$ is $C^\infty$-convergent with weight $\delta'>\delta$ imply that the $W^{k,2,\delta}$-norm is independent up to commensurability of the choices involved in its definition.
\end{definition}

\begin{lemma}\label{smoothweighteddense}
Smooth functions of compact support are dense in $W^{k,2,\delta}$.
\end{lemma}

\begin{proof}
By the Meyers--Serrin theorem \cite[Theorem 3.16]{adamssobolev} it is enough to show that functions of compact support are dense in $W^{k,2,\delta}$.
It is enough to examine the model case of an end $[0,\infty)\times S^1$.
Let $\varphi:\RR\to[0,1]$ be any smooth (cutoff) function satisfying $\varphi(s)=1$ for $s\leq 1$ and $\varphi(s)=0$ for $s\geq 2$.
Now the sequence $\varphi(s/N)f(s,t)$ (or even $\varphi(s-N)f(s,t)$) converges to $f(s,t)$ in the $W^{k,2,\delta}$ topology as $N\to\infty$.
\end{proof}

Next, we introduce a space $\tilde W^{k,2,\delta}(C,u^\ast T\hat X)$ which is a modification of the space $W^{k,2,\delta}(C,u^\ast T\hat X)$ taking into account the space of deformations of the almost complex structure of $C$ equipped with the doubly marked points $\{p_e\}$.
Recall from \S\ref{rsdeformations} that when $C$ is smooth, this deformation theory is governed by $H^\bullet(C,TC(-2P))$ where $P:=\sum_ep_e\subseteq C$.
When $C$ has nodes, this deformation theory (not taking into account resolutions of the nodes) is governed by $H^\bullet(\tilde C,T\tilde C(-2P-\tilde N))$, where $\tilde N\subseteq\tilde C$ denotes the inverse images of the nodes $N\subseteq C$ (recall the normalization $\tilde C$ from Definition \ref{normalizationdef}).

\begin{definition}[Weighted Sobolev spaces $\tilde W^{k,2,\delta}$]\label{sobolevspacestilde}
Let $k$, $\delta$, and $u$ be as in Definition \ref{sobolevspaces}.
We define a weighted Sobolev space
\begin{equation}
\tilde W^{k,2,\delta}(C,u^\ast T\hat X)
\end{equation}
as
\begin{equation}\label{Wtildetotal}
W^{k,2,\delta}(C,u^\ast T\hat X)\oplus C^\infty_c(C\setminus\{p_e\}\cup N,\End^{0,1}(TC))
\end{equation}
modulo the subspace
\begin{equation}
\left\{Xu\oplus\sL_Xj\,\middle|\,X\in C^\infty(\tilde C,T\tilde C)\quad\begin{matrix}X\text{ holomorphic near }p_e\text{ and }\tilde N\\X(p_e)=dX(p_e)=0\hfill\\X(\tilde n)=0\text{ for }\tilde n\in\tilde N\hfill\end{matrix}\right\}.
\end{equation}
Since $X$ is holomorphic and vanishes to second order at each $p_e$ (so in local coordinates \eqref{cylindricalcoordsonC} it equals $f\cdot\partial_s$ for some holomorphic function $f$ vanishing at $p_e$), a calculation shows that $Xu\in W^{k,2,\delta}(C,u^\ast T\hat X)$ since $\delta<1$ and $u$ converges to a trivial cylinder with weight $\delta'>\delta$.

There is a unique commensurability class of norm on $\tilde W^{k,2,\delta}$ for which the natural map $W^{k,2,\delta}\to\tilde W^{k,2,\delta}$ is Fredholm, since this map has kernel $H^0(\tilde C,T\tilde C(-2P-\tilde N))$ and cokernel $H^1(\tilde C,T\tilde C(-2P-\tilde N))$, both of which are finite-dimensional.
\end{definition}

To better understand the spaces $\tilde W^{k,2,\delta}$, we discuss a few special cases for smooth $C$.

If $H^1(C,TC(-2P))=0$, namely if all variations of $C$ with its markings at $\{p_e\}$ are trivial up to gauge, then we have
\begin{equation*}
\tilde W^{k,2,\delta}(C,u^\ast T\hat X)=W^{k,2,\delta}(C,u^\ast T\hat X)\Bigm/\ker\Bigl(\aut(C,\{p_e\}_e)\to\bigoplus_e\gl(T_{p_e}C)\Bigr),
\end{equation*}
where $\ker\bigl(\aut(C,\{p_e\}_e)\to\bigoplus_e\gl(T_{p_e}C)\bigr)=H^0(C,TC(-2P))$ denotes the (finite-dimensional!) Lie algebra of the group of automorphisms of $C$ which fix each $p_e$ and act as the identity on each $T_{p_e}C$.
This Lie algebra of vector fields on $C$ maps injectively (by stability) to $W^{k,2,\delta}(C,u^\ast T\hat X)$ by differentiating $u$.
In this case, $\tilde W^{k,2,\delta}$ is equipped with the natural quotient norm from the presentation above.

If $H^0(C,TC(-2P))=0$, namely if $C$ equipped with its markings at $\{p_e\}$ has no nontrivial infinitesimal automorphisms, then we have
\begin{equation}\label{linearizeddomain}
\tilde W^{k,2,\delta}(C,u^\ast T\hat X)=W^{k,2,\delta}(C,u^\ast T\hat X)\oplus V,
\end{equation}
where $V\subseteq C^\infty_c(C\setminus\{p_e\}_e,\End^{0,1}(TC))$ is any subspace projecting isomorphically onto $H^1(C,TC(-2P))$ (which by \S\ref{rsdeformations} is the space of deformations of $C$ equipped with the doubly marked points $\{p_e\}$).
In this case, $\tilde W^{k,2,\delta}$ is equipped with the direct sum norm induced by \eqref{linearizeddomain} for any norm on $V$.
This norm on $\tilde W^{k,2,\delta}$ is well-defined up to commensurability, since every other such subspace $V'$ may be obtained as the image of $v\mapsto v+\sL_{X(v)}j$ for some unique $X:V\to C^\infty(C,TC)$ with $X=0$ and $dX=0$ at $p_e$, and the two resulting spaces \eqref{linearizeddomain} are identified via $(\xi,v)\mapsto(\xi+X(v)u,v+\sL_{X(v)}j)$.

To be even more concrete, suppose $C$ is smooth of genus zero.
If $\#\{p_e\}=1$ (i.e.\ the case of the plane), then it falls into the first case above with a one-dimensional (over $\CC$) space of infinitesimal automorphisms, corresponding to translations of $\CC$ (the automorphisms of $\CC$ acting trivially on the tangent space at infinity).
If $\#\{p_e\}=2$ (i.e.\ the case of the cylinder), then it falls into the second case above with a one-dimensional space of deformations of the almost complex structure.
Namely, all cylinders are biholomorphic to $\CC\setminus\{0\}$, the space of tangent space markings at $0$ and $\infty$ is $(\CC^\times)^2$ which is two-dimensional, and quotienting by $\Aut(\CC\setminus\{0\})=\CC^\times$ reduces this to the one-dimensional $\CC^\times$.
In fact, $\#\{p_e\}\geq 2$ all fall into the second case above.

\begin{definition}[Linearized operators]\label{linearizedbasicdef}
Let $k$, $\delta$, and $u$ be as in Definition \ref{sobolevspaces}, and suppose $k\geq 1$.
For any choice of $\hat J$-linear connection $\nabla$ as in Convention \ref{metricconnectionchoicetarget}, there is a linearized operator
\begin{equation*}
D_u^\nabla:W^{k,2,\delta}(C,u^\ast T\hat X)\oplus C^\infty_c(C\setminus\{p_e\}_e,\End^{0,1}(TC))\to W^{k-1,2,\delta}(\tilde C,u^\ast T\hat X_{\hat J}\otimes_\CC\Omega^{0,1}_{\tilde C})
\end{equation*}
(recall the normalization $\tilde C$ from Definition \ref{normalizationdef})
expressing the first order change in $(du)^{0,1}$, measured with respect to the chosen connection $\nabla$, as $u$ and $j$ vary, namely
\begin{equation}
D_u^\nabla(\xi,j'):=\left.\frac d{d(\xi,j')}\right|_{(\xi,j')=(0,j)}\PT^\nabla_{\exp_u\xi\to u}\Bigl[(d\exp_u\xi)^{0,1}_{j',\hat J}\Bigr]
\end{equation}
where $\PT$ denotes parallel transport.
When $u$ is $\hat J$-holomorphic, the operator $D_u^\nabla$ is independent of the choice of connection and descends to a map
\begin{equation}\label{singledomainlinearized}
D_u:\tilde W^{k,2,\delta}(C,u^\ast T\hat X)\to W^{k-1,2,\delta}(\tilde C,u^\ast T\hat X_{\hat J}\otimes_\CC\Omega^{0,1}_{\tilde C}).
\end{equation}
For general $u$, we may still define such a map \eqref{singledomainlinearized} by fixing a choice of connection $\nabla$ and a choice of subspace of \eqref{Wtildetotal} which projects isomorphically onto $\tilde W^{k,2,\delta}$ and intersects $W^{k,2,\delta}$ in a subspace of finite codimension.
\end{definition}

The operator
\begin{equation}\label{linearizedwithoutj}
D_u^\nabla:W^{k,2,\delta}(C,u^\ast T\hat X)\to W^{k-1,2,\delta}(\tilde C,u^\ast T\hat X_{\hat J}\otimes_\CC\Omega^{0,1}_{\tilde C})
\end{equation}
is a real Cauchy--Riemann operator, in the sense that $D_u^\nabla(fs)=fD_u^\nabla(s)+(\bar\partial f)s$ for real-valued functions $f$.
In particular, it is elliptic.

\begin{proposition}\label{linearizedfredholm}
The linearized operator \eqref{singledomainlinearized} is Fredholm provided $\pm\delta\notin\sigma(D_\gamma)$ for every orbit $\gamma$ which $u$ is positively/negatively asymptotic to.
\end{proposition}

\begin{proof}
It is equivalent to show that \eqref{linearizedwithoutj} is Fredholm under the given hypothesis on $\delta$.
The main point of the proof of this (as compared with the standard Fredholm results for elliptic operators on compact manifolds) is to prove suitable elliptic estimates in the ends of $C$, which follow from the hypothesis $\delta\notin\sigma(D_\gamma)$.
This argument can be viewed as a special case of a general result of Lockhart--McOwen \cite{lockhartmcowen}, and proofs in the specific context of pseudo-holomorphic curves with cylindrical ends can be found in Salamon \cite[\S 2]{salamonnotes} and Wendl \cite[Lecture 4]{wendlsftbook}.
\end{proof}

\begin{definition}[Regularity of pseudo-holomorphic buildings]\label{regularitydef}
Given a pseudo-holomorphic building $\{u_v:C_v\to\hat X_v\}_{v\in V(T)}$, we consider the linearized operator
\begin{equation}\label{totaldomainlinearized}
D_u:\bigoplus_{v\in V(T)}\tilde W^{k,2,\delta}(C,u^\ast T\hat X)\to\bigoplus_{v\in V(T)}W^{k-1,2,\delta}(\tilde C,u^\ast T\hat X_{\hat J}\otimes_\CC\Omega^{0,1}_{\tilde C})
\end{equation}
with $\delta\in(0,1)$ and $\delta<\delta_\gamma$ for every Reeb orbit $\gamma=\gamma_e$ for $e\in E(T)$.
This operator is Fredholm by Proposition \ref{linearizedfredholm}, and its kernel and cokernel are independent of the choice of $k$ and $\delta$ by elliptic regularity.

A point in a moduli space $\Mbar(T)$ is called \emph{regular} iff this linearized operator is surjective.
A point is called \emph{weakly regular} iff the linearized operator is surjective after adding variations in $t\in\s(T)$ to its domain (this is only relevant in cases (III) and (IV)).
\end{definition}

\begin{lemma}\label{trivialtransverse}
The trivial cylinder ${\id}\times{\tilde\gamma}:\RR\times S^1\to\RR\times Y$ over any Reeb orbit $\gamma$ is regular.  In fact, the associated linearized operator is an isomorphism.
\end{lemma}

\begin{proof}
The linearized operator $D$ may be decomposed into the tangential and normal deformation operators $D^T$ and $D^N$.  Precisely, there is a diagram whose rows are short exact sequences (we write $C=\RR\times S^1$):
\begin{equation}
\hspace{-1in}
\begin{tikzcd}[column sep = tiny]
\tilde W^{k,2,\delta}(C,TC)\ar{r}{du\circ}\ar{d}{D^T}&\tilde W^{k,2,\delta}(C,\gamma^\ast TY\oplus\RR\partial_s)\ar{r}\ar{d}{D}&W^{k,2,\delta}(C,\gamma^\ast\xi)\ar{d}{D^N}\\
W^{k-1,2,\delta}(C,TC\otimes_\CC\Omega^{0,1}_C)\ar{r}{du\circ}&W^{k-1,2,\delta}(C,(\gamma^\ast TY\oplus\RR\partial_s)\otimes_\CC\Omega^{0,1}_C)\ar{r}&W^{k-1,2,\delta}(C,\gamma^\ast\xi\otimes_\CC\Omega^{0,1}_C).
\end{tikzcd}
\hspace{-1in}
\end{equation}
Note that the domain of $D^T$ includes variations in complex structure on $C$, while the domain of $D^N$ does not.  It suffices to show that both $D^T$ and $D^N$ are isomorphisms.

To show that $D^N$ is an isomorphism, write $D^N=\partial_s+D_\gamma$, where $D_\gamma=J\sL_{R_\lambda}$ is the linearized operator for $\gamma$ (Definition \ref{reeblinearized}).  It follows from the fact that $\delta$ is not an eigenvalue of $D_\gamma$ that $D^N$ is an isomorphism.

To show that $D^T$ is an isomorphism, let us first observe that the complex
\begin{equation}
W^{k,2,\delta}(C,TC)\xrightarrow{D^T}W^{k-1,2,\delta}(C,TC\otimes_\CC\Omega^{0,1}_C)
\end{equation}
calculates $H^\bullet(\PPP^1,T_{\PPP^1}(-2[0]-2[\infty]))$, that is we have
\begin{align}
\label{kercalculatesH}\ker D^T&=H^0(\PPP^1,T_{\PPP^1}(-2[0]-2[\infty])),\\
\label{cokercalculatesH}\coker D^T&=H^1(\PPP^1,T_{\PPP^1}(-2[0]-2[\infty])),
\end{align}
where we identify $C$ with $\PPP^1\setminus\{0,\infty\}$.

To show \eqref{kercalculatesH}, argue as follows.
By elliptic regularity, $\ker D^T$ consists precisely of those holomorphic sections of $TC$ over $C$ which lie in $W^{k,2,\delta}$.
Thus we must show that a holomorphic section of $T\PPP^1$ over $\PPP^1\setminus\{0,\infty\}$ lies in $W^{k,2,\delta}$ iff it vanishes to order $\geq 2$ at $0,\infty\in\PPP^1$.
In fact, it suffices to show the seemingly much weaker statement that a holomorphic section of $T\PPP^1$ over $\PPP^1\setminus\{0,\infty\}$ lying in $W^{k,2,\delta}$ must be meromorphic at $0,\infty\in\PPP^1$ (this implies the previous assertion by a straightforward calculation).
So, let $f$ be a holomorphic section lying in $W^{k,2,\delta}$, and let us show that it is meromorphic at the punctures.
Since $f$ lies in $W^{k,2,\delta}$, it can be integrated (with respect to a fixed non-vanishing area form on $\PPP^1$) against any smooth function on $\PPP^1$ (more precisely, smooth section of the real dual of $T_{\PPP^1}$) which vanishes to order $N$ at $0$ and $\infty$, for some $N<\infty$ depending on $k$ and $\delta$.
The space of such smooth functions is of finite codimension in the space of all smooth functions on $\PPP^1$, so we may extend $f$ arbitrarily to a distribution $\tilde f$ on $\PPP^1$.
Since $f$ is holomorphic on $\PPP^1\setminus\{0,\infty\}$, the distribution $\bar\partial\tilde f$ is supported at $0,\infty\in\PPP^1$, so by Lemma \ref{distrdelta} it is a finite sum of delta functions and their derivatives at $0,\infty\in\PPP^1$.
It follows that $z^M\bar\partial\tilde f=0$ for sufficiently large $M<\infty$, for any local coordinate $z$ near $0$ or $\infty$.
We have $\bar\partial(z^M\tilde f)=z^M\bar\partial\tilde f=0$, so by elliptic regularity $z^M\tilde f$ is holomorphic at $0$ and $\infty$.
It follows that $f$ is meromorphic at $0,\infty\in\PPP^1$, which finishes the proof of \eqref{kercalculatesH}.
To show \eqref{cokercalculatesH}, apply the same reasoning to the formal adjoint to show that
\begin{equation}
(\coker D^T)^\vee=\ker((D^T)^\vee)=H^0(\PPP^1,\Omega_{\PPP^1}^{\otimes 2}(2[0]+2[\infty]))
\end{equation}
and apply Serre duality.

It follows from \eqref{kercalculatesH}--\eqref{cokercalculatesH} that $D^T$ is an isomorphism.
Indeed, $\ker D^T=H^0(\PPP^1,T_{\PPP^1}(-2[0]-2[\infty]))=0$, and by definition $\tilde W^{k,2,\delta}$ is $W^{k,2,\delta}$ direct sum a space which maps isomorphically to $H^1(\PPP^1,T_{\PPP^1}(-2[0]-2[\infty]))=\coker D^T$.
\end{proof}

The following is an easy exercise (see \cite[Theorem 1.5.3]{hormanderlinear}) and is frequently useful.

\begin{lemma}\label{distrdelta}
A distribution supported at a single point is a linear combination of the delta function at that point and its derivatives.\qed
\end{lemma}

\subsection{Index of moduli spaces}\label{indexsec}

We now define a notion of \emph{index} and \emph{virtual dimension} for objects of $\SSS$.

\begin{definition}[Index $\mu(T)$]\label{indexdefn}
We define $\mu(T)$ to be the Fredholm index of the linearized operator \eqref{totaldomainlinearized} for non-nodal $C_v$.
\end{definition}

Note that \eqref{totaldomainlinearized} makes sense for any smooth building $\{u_v\}_v$ for which $u_v$ approach trivial cylinders sufficiently rapidly, and varies nicely in families of such $\{u_v\}_v$, hence $\mu(T)$ is well-defined.

Standard arguments allow one to express $\mu(T)$ in terms of the Conley--Zehnder indices of the Reeb orbits $\gamma_{e^+}$ and $\{\gamma_{e^-_i}\}_i$ and the homology class of $\#_v\beta_v$ (see \cite[Proposition 1.7.1]{sftintro} or \cite[Proposition 4]{bourgeoismohnke}):
\begin{equation}\label{indexofTbyCZ}
\mu(T)=[\CZ_{\tau^+}(\gamma_{e^+})+n-3]-\sum_i[\CZ_{\tau^-_i}(\gamma_{e^-_i})+n-3]+\langle 2c_1(T\hat X,\tau),\#_v\beta_v\rangle,
\end{equation}
where $\tau=(\tau^+,\{\tau^-_i\}_i)$ denotes any collection of trivializations of $\xi^\pm$ over the orbits $\gamma$, and $c_1(T\hat X,\tau)$ denotes the first Chern class relative to the boundary trivialization of $T\hat X$ obtained by summing $\tau$ with the tautological trivialization of $\RR R_\lambda\oplus\RR\partial_s$ (the expression on the right hand side is independent of the choice of $\tau$).
We may thus define the \emph{homological grading} of a null-homologous Reeb orbit to be
\begin{equation}\label{gammagrading}
\left|\gamma\right|:=\CZ_\tau(\gamma)+n-3+\langle 2c_1(\xi,\tau),\beta\rangle\in\ZZ/2c_1(\xi)\cdot H_2(Y)
\end{equation}
for any trivialization $\tau$ and any null-homology $\beta$ of $\gamma$.
Analogously, there is a relative grading on the set of (monomials of) Reeb orbits in a fixed homology class.

The index satisfies $\mu(T)=\mu(T')$ for any morphism $T\to T'$ (this is evident from the formula in terms of Conley--Zehnder indices), and is additive under concatenations, that is $\mu(\#_iT_i)=\sum_i\mu(T_i)$ (trivial by definition).

\begin{definition}[Virtual dimension $\vdim(T)$]
The \emph{virtual dimension} of $T$ is defined as
\begin{equation}
\vdim(T):=\mu(T)-\#V_s(T)+\dim\s(T),
\end{equation}
recalling that $V_s(T)\subseteq V(T)$ denotes the set of symplectization vertices, i.e.\ those $v$ for which $\ast^+(v)=\ast^-(v)$.
This is the ``expected dimension'' of $\Mbar(T)$.
For a morphism $T'\to T$, let $\codim(T'/T):=\vdim T-\vdim T'$.
\end{definition}

\begin{lemma}
We have $\codim(T'/T)\geq 0$ with equality iff $T'\to T$ is an isomorphism, and $\codim(T'/T)>1$ iff $T\to T'$ factors nontrivially.
\qed
\end{lemma}

\subsection{Orientations of moduli spaces}\label{thickorsec}

We now review the theory of orientations in contact homology.  The general analytic methods used to orient moduli spaces of pseudo-holomorphic curves were introduced by Floer--Hofer \cite{floerhofer} (see also Bourgeois--Mohnke \cite{bourgeoismohnke}).  The resulting algebraic structure relevant for contact homology was worked out by Eliashberg--Givental--Hofer \cite{sftintro} (see also Bourgeois--Mohnke \cite{bourgeoismohnke}).

We find it most convenient to encode orientations via the formalism of \emph{orientation lines}, as introduced into the subject by Seidel.
An orientation line is simply a free $\ZZ$-module of rank one, equipped with a $\ZZ/2$-grading (i.e.\ declared to be either even or odd).
For every Reeb orbit $\gamma$ with basepoint $b$, we will define an orientation line $\oo_{\gamma,b}$.
For every object $T\in\SSS$, we will define an orientation line $\oo^\circ_T$, and we will show that there is a canonical isomorphism
\begin{equation}
\oo^\circ_T=\oo_{\gamma^+,b^+}\otimes\bigotimes_{\gamma^-\in\Gamma^-}\oo_{\gamma^-,b^-}.
\end{equation}
The virtual orientation sheaf of $\Mbar(T)$ is canonically isomorphic to
\begin{equation}
\oo_T:=\oo^\circ_T\otimes(\oo_\RR^\vee)^{\otimes V_s(T)}\otimes\oo_{\s(T)}.
\end{equation}
Recall that we are using the super tensor product $\otimes$ (see \S\ref{conventionsintrosec}).

\begin{definition}
Denote by $\oo_V:=H_{\dim V}(V,V\setminus 0)$ with parity $\dim V$ the orientation line\footnote{One should distinguish the orientation line $\oo_V$ from the determinant line $\det V:=\wedge^{\dim V}V$, as there is no functorial isomorphism $\oo_V\otimes_\ZZ\RR=\det V$ (rather only up to scaling by $\RR_{>0}$).} of the vector space $V$.
For a Fredholm map $A:E\to F$, define $\oo(A)=\oo_A:=\oo_{\ker A}\otimes\oo_{\coker A}^\vee$.
\end{definition}

\begin{definition}[Orientation lines $\oo_{\gamma,b}$ of Reeb orbits]\label{reeborientation}
Let $\gamma\in\PP=\PP(Y,\lambda)$, and fix a constant speed parameterization $\tilde\gamma:S^1\to Y$ of $\gamma$ (equivalently, fix a basepoint $b=\tilde\gamma(0)\in\left|\gamma\right|$).  We consider the bundle $V:=\tilde\gamma^\ast\xi\oplus\CC$ over $[0,\infty)\times S^1\subseteq\CC$.  The bundle $V$ is equipped with a connection, namely the connection on $\tilde\gamma^\ast\xi$ induced by the Lie derivative $\sL_{R_\lambda}$ plus the trivial connection on $\CC$, and this gives rise to a real Cauchy--Riemann operator $\bar\partial$ on $V$.  Now extend the pair $(V,\bar\partial)$ to all of $\CC$, and define
\begin{equation}\label{orientationdefiningoperator}
\oo_{\gamma,b}:=\oo(V,\bar\partial):=\oo(W^{k,2,\delta}(\CC,V)\to W^{k-1,2,\delta}(\CC,V\otimes_\CC\Omega^{0,1}_\CC))
\end{equation}
for small $\delta>0$.
We show in Lemma \ref{oogammawelldefined} immediately below that $\oo_{\gamma,b}$ is well-defined.
\end{definition}

Note that (because of the well-definedness of $\oo_{\gamma,b}$) any path between basepoints $b\to b'$ in the sense of \S\ref{Tcatdefsec} gives rise to an isomorphism $\oo_{\gamma,b}\to\oo_{\gamma,b'}$.

\begin{lemma}\label{oogammawelldefined}
The orientation line $\oo_{\gamma,b}$ defined by \eqref{orientationdefiningoperator} independent of the choice of extension $(V,\bar\partial)$ up to canonical isomorphism.
\end{lemma}

(Refer to Solomon \cite[\S 2, Proposition 2.8]{solomonthesis} for a similar result.)

\begin{proof}
The space of extensions of $(V,\bar\partial)$ is homotopy equivalent to $\Maps(S^2,BU(n))$ (noting that $\pi_1(BU(n))=\pi_0(U(n))=0$ so $V$ is trivial over $S^1$) and we have $\pi_i\Maps(S^2,BU(n))=\pi_{i+2}(BU(n))=\pi_{i+1}(U(n))$; in particular $\pi_0=\ZZ$ and $\pi_1=0$.  By simple connectivity, the line $\oo_{\gamma,b}$ depends at most on the choice of connected component of $\Maps(S^2,BU(n))$ (classified by relative Chern class).

To identify (canonically) the $\oo_{\gamma,b}$ from different relative Chern classes, argue as follows.
Choose an extension $(V,\bar\partial)$ for which $\bar\partial$ is complex linear in a neighborhood of the origin $0\in\CC$ (this is a contractible condition), so $(V,\bar\partial)$ is a holomorphic vector bundle near the origin.
Hence we may define $(V(-[0]),\bar\partial)$ near the origin as the holomorphic vector bundle whose sheaf of holomorphic sections is the subsheaf of holomorphic sections of $V$ which vanish at the origin, and we extend $(V(-[0]),\bar\partial)$ to concide with $(V,\bar\partial)$ away from the origin.
Thus $(V(-[0]),\bar\partial)$ is another extension with relative Chern class one less than $(V,\bar\partial)$, so it suffices to produce a canonical isomorphism $\oo(V,\bar\partial)=\oo(V(-[0]),\bar\partial)$.

To produce this isomorphism, it suffices to construct the expected exact sequence
\begin{equation}\label{chernclasschangeES}
0\to H^0_\delta(V([-0]),\bar\partial)\to H^0_\delta(V,\bar\partial)\to V_0\to H^1_\delta(V([-0]),\bar\partial)\to H^1_\delta(V,\bar\partial)\to 0
\end{equation}
where $H^i_\delta(V,\bar\partial)$ denotes the cohomology of the two term complex appearing in \eqref{orientationdefiningoperator} and $V_0$ denotes the fiber of $V$ over $0\in\CC$ (note that $V_0$ is a complex vector space and hence is canonically oriented).
To produce this exact sequence, it suffices to apply the snake lemma to the diagram obtained from the inclusion of the complex calculating $H^\bullet_\delta(V(-[0]),\bar\partial)$ into that calculating $H^\bullet_\delta(V,\bar\partial)$.
At the level of smooth sections, this is the diagram
\begin{equation}
\begin{tikzcd}[column sep = small]
0\ar{r}&C^\infty(\CC,V(-[0]))\ar{d}\ar{r}&C^\infty(\CC,V)\ar{d}\ar{r}&J^{0,\infty}(V)_0\ar{d}\ar{r}&0\\
0\ar{r}&C^\infty(\CC,V(-[0])\otimes\Omega^{0,1}_\CC)\ar{r}&C^\infty(\CC,V\otimes\Omega^{0,1}_\CC)\ar{r}&J^{0,\infty}(V\otimes\Omega^{0,1}_\CC)_0\ar{r}&0,
\end{tikzcd}
\end{equation}
where $J^{0,\infty}(-)_0$ denotes the infinite anti-holomorphic jet space at $0\in\CC$, i.e.\ keeping track of all anti-holomorphic derivatives $(\frac\partial{\partial\bar z})^k$ at zero; note that the right vertical map is surjective with kernel $V_0$.
Now the usual diagram chase (in combination with some appeals to elliptic regularity) produces the desired exact sequence \eqref{chernclasschangeES}.
In doing this diagram chase, it may be helpful to remark that, to define the connecting homomorphism, one may, for convenience, restrict consideration to lifts of elements of $V_0$ to $W^{k,2,\delta}(\CC,V)$ which are holomorphic near zero; also note that by elliptic regularity, all elements of $H^1_\delta(\CC,V)$ can be represented by smooth sections of $V\otimes\Omega^{0,1}_\CC$ supported inside any given non-empty open subset of $\CC$.
\end{proof}

\begin{definition}[Parity of Reeb orbits]
The \emph{parity} $\left|\gamma\right|\in\ZZ/2$ of $\gamma\in\PP$ is the parity of $\oo_\gamma$.
\end{definition}

By \eqref{indexofTbyCZ} we have $\left|\gamma\right|=\CZ(\gamma)+n-3$, i.e.\ the parity is simply the reduction to $\ZZ/2$ of the grading \eqref{gammagrading} (whenever the latter is defined).
This formula for $\left|\gamma\right|$ in terms of the Conley--Zehnder index can also be expressed as $\left|\gamma\right|=\sign(\det(I-A_\gamma))\in\{\pm 1\}=\ZZ/2$ where $A_\gamma$ denotes the linearized return map of $\gamma$ acting on $\xi_p$ (using the general property of the Conley--Zehnder index $(-1)^{\CZ(\Psi)}=(-1)^{n-1}\sign(\det(I-\Psi(1)))$ for $\Psi:[0,1]\to\Sp_{2n-2}(\RR)$).  It thus follows that
\begin{equation}
\left|\gamma\right|=\#(\sigma(A_\gamma)\cap(0,1))\in\ZZ/2,
\end{equation}
where $\sigma(\cdot)$ denotes the spectrum (recall that the spectrum of any matrix $A\in\Sp_{2n}(\RR)$ lies in $\RR^\times\cup\{z\in\CC:\left|z\right|=1\}$ and that $1\notin\sigma(A_\gamma)$ is equivalent to the non-degeneracy of $\gamma$).  Note that by definition $\gamma$ is non-degenerate iff $1\notin\sigma(A_\gamma)$.  It follows that the index of the $k$-fold multiple cover $\gamma_k$ of $\gamma$ is given by
\begin{equation}\label{multcoverindex}
\left|\gamma_k\right|=\left|\gamma\right|+(k+1)\#(\sigma(A_\gamma)\cap(-1,0))\in\ZZ/2.
\end{equation}

\begin{definition}[Good and bad Reeb orbits]
There is an action of $\ZZ/d_\gamma$ on $\oo_{\gamma,b}$ by functoriality, which just amounts to a homomorphism $\ZZ/d_\gamma\to\{\pm 1\}$ (independent of $b$ since $\ZZ/d_\gamma$ is abelian).  The orbit $\gamma$ is called \emph{good} iff this homomorphism is trivial (and \emph{bad} otherwise).  For good $\gamma$, we thus have an orientation line $\oo_\gamma$ independent of $b$ up to canonical isomorphism.
\end{definition}

The bad Reeb orbits are precisely the even multiple covers $\gamma_{2k}$ of \emph{simple} orbits $\gamma$ with $\#(\sigma(A_\gamma)\cap(-1,0))$ odd by \cite[Lemma 1.8.8, Remark 1.9.2]{sftintro}.  To see this, it suffices to show that a generator of the $\ZZ/k$ action on $\oo_{\gamma_k}$ acts by $(-1)^{\left|\gamma_k\right|-\left|\gamma\right|}$ and use \eqref{multcoverindex}.  This can be proven by pulling back the operator from \eqref{orientationdefiningoperator} under $z\mapsto z^k$ and analyzing the representations of $\ZZ/k$ occurring in the kernel and cokernel (see \cite[Proof of Theorem 3]{bourgeoismohnke}).

\begin{definition}[Orientation lines $\oo^\circ_T$]
For any $T\in\SSS$, we define the orientation line $\oo^\circ_T$ to be the orientation line of the linearized operator \eqref{totaldomainlinearized} at any smooth building of type $T$.
\end{definition}

We show in Lemma \ref{otdefined} immediately below that $\oo^\circ_T$ is well-defined (i.e.\ is independent of the choice of smooth building, up to canonical isomorphism).
Note that it does not matter whether we define $\oo^\circ_T$ in terms of the linearized operator with domain $W^{k,2,\delta}$ or $\tilde W^{k,2,\delta}$, since their ``difference'' is a complex vector space and thus is canonically oriented.
Note also that the parity of $\oo^\circ_T$ is just the index $\mu(T)$ mod $2$.

\begin{lemma}\label{otdefined}
The orientation line of the linearized operator \eqref{totaldomainlinearized} at a smooth building of type $T$ is canonically isomorphic to $\oo_{\gamma_{e^+},b_{e^+}}\otimes\bigotimes_{e^-}\oo_{\gamma_{e^-},b_{e^-}}^\vee$.
\end{lemma}

\begin{proof}
We consider the linearized operator \eqref{totaldomainlinearized} direct sum the operators from \eqref{orientationdefiningoperator} for each output Reeb orbit $\gamma_{e^-}$ of $T$.
The orientation line of this direct sum is thus $\oo_T^\circ\otimes\bigotimes_{e^-}\oo_{\gamma_{e^-},b_{e^-}}$.

Now the Floer--Hofer kernel gluing operation \cite{floerhofer} (see also Bourgeois--Mohnke \cite{bourgeoismohnke}) provides a canonical isomorphism between the orientation line of this operator and the orientation line of the operator obtained by gluing the domains end to end in the specified way (note the use of the basepoints $b_{e^-}$ for $\gamma_{e^-}$ specified by $T$).
Note that, for technical convenience, it is enough to do this gluing just for maps $u_v$ which coincide with trivial cylinders near infinity.

This glued up operator is now exactly of the form from \eqref{orientationdefiningoperator} for $\gamma_{e^+}$, so we deduce an isomorphism $\oo_T^\circ\otimes\bigotimes_{e^-}\oo_{\gamma_{e^-},b_{e^-}}=\oo_{\gamma_{e^+},b_{e^+}}$.
\end{proof}

\begin{remark}
Gluing the linearized operators with domain $W^{k,2,\delta}$ results in an index increase of two, gluing the linearized operators with domain $\tilde W^{k,2,\delta}$ preserves the index, and gluing the linearized operators with domain $W^{k,2,\delta}$ plus decay to constants in the ends results in an index decrease of two.
The Floer--Hofer kernel gluing operation can be done in any of these contexts; Bourgeois--Mohnke choose the last.
\end{remark}

For any morphism $T\to T'$, there is a canonical isomorphism $\oo^\circ_T\to\oo^\circ_{T'}$, again by the Floer--Hofer kernel gluing operation.
Associativity of this operation implies that this isomorphism respects the identifications from Lemma \ref{otdefined} and that it makes $\oo^\circ$ into a functor from $\SSS_\I$ to the category of orientation lines and isomorphisms.
For any concatenation $\{T_i\}$, there is a tautological identification $\oo^\circ_{\#_iT_i}=\bigotimes_i\oo^\circ_{T_i}$; associativity of the Floer--Hofer operation implies that this identification is compatible with Lemma \ref{otdefined}.

\begin{remark}
The above reasoning, in particular the proof of Lemma \ref{otdefined}, relies on the fact that the topology of the curves in question is particularly simple.
To prove the analogous result in the SFT setting requires a more complicated argument; see Bourgeois--Mohnke \cite[Proposition 8]{bourgeoismohnke}.
\end{remark}

For any $T\in\SSS$, define
\begin{equation}
\oo_T:=\oo^\circ_T\otimes(\oo_\RR^\vee)^{\otimes V_s(T)}\otimes\oo_{\s(T)},
\end{equation}
where $\oo_{\s(T)}$ denotes the global sections of the orientation sheaf of $\s$ considered as a real manifold.
We shall see that the virtual orientation sheaf of $\Mbar(T)$ is canonically isomorphic to $\oo_T$.
Note that the parity of $\oo_T$ equals $\vdim(T)$.

For any morphism $T'\to T$, there is a canonical identification
\begin{equation}\label{algebraicorientationgluing}
\oo_{T'}\otimes\oo_{G_{T'//T}}\otimes\oo_{s(T')}^\vee=\oo_T,
\end{equation}
where $G_{T'//T}$ denotes the space of gluing parameters defined in \S\ref{localmodelsection}.
The appearance of $\oo_{s(T')}^\vee$ is explained by the fact that $G_{T'/}$ includes all variations of $t$, rather than just those normal to $\s(T')$.

For any concatenation $\{T_i\}$, there is a tautological identification
\begin{equation}
\oo_{\#_iT_i}=\bigotimes_i\oo_{T_i}.
\end{equation}

\subsection{Local structure of moduli spaces via local models \texorpdfstring{$G$}{G}}\label{firstgluingstatements}

We now state the precise sense in which the spaces $G_{T'//T}$ from \S\ref{localmodelsection} are local topological models for the regular loci in the moduli spaces $\Mbar(T)$.  We also state the compatibility of this local topological structure with the natural maps on orientation lines discussed earlier.  These statements are in essence a gluing theorem, and their proofs are given in \S\ref{gluingsec}.

We denote by $\Mbar(T)^\reg\subseteq\Mbar(T)$ the locus of points which are regular and have trivial isotropy (we do not discuss the entire locus of regular points, possibly with isotropy, just for sake of simplicity, although it would not really present any additional difficulty to do so).

The (Banach space) implicit function theorem implies that $\M(T)^\reg$ is a (smooth) manifold of dimension $\vdim T$ over the locus without nodes.
Recall that for any morphism $T\to T'$, we have $\mu(T)=\mu(T')$.
Denote by $N(x)$ the set of nodes of the domain curve of a given point $x\in\Mbar(T)$.

\begin{theorem}[Local structure of $\Mbar(T)^\reg$]\label{localgluingnonthickened}
Let $x_0\in\Mbar(T)^\reg$ be of type $T'\to T$.  Then $\mu(T)-\#V_s(T')-2\#N(x_0)\geq 0$ and there is a local homeomorphism
\begin{equation}\label{ultimategluingmapnonthickened}
\bigl(G_{T'//T}\times \CC^{N(x_0)}\times\RR^{\mu(T)-\#V_s(T')-2\#N(x_0)},(0,0,0)\bigr)\to\bigl(\Mbar(T),x_0\bigr)
\end{equation}
whose image lands in $\Mbar(T)^\reg$ and which commutes with the maps from both sides to $\SSS_{T'//T}\times\overline{\s(T)}$ (as well as the stratifications by number of nodes).
\end{theorem}

\begin{proof}
See \S\S\ref{localgluingproofI}--\ref{localgluingproofIII}.
\end{proof}

The (Banach space) implicit function theorem moreover gives a canonical identification
\begin{equation}\label{Minteriororientationsnonthickened}
\oo_{\Mbar(T)^\reg}=\oo_T.
\end{equation}
More precisely, this identification is made over the locus in $\M(T)^\reg$ without nodes, and has a unique continuous extension to all of $\Mbar(T)^\reg$ by virtue of the local topological description in Theorem \ref{localgluingnonthickened}.  It is also compatible with morphisms $T'\to T$ in the following precise sense.

\begin{theorem}[Compatibility of the ``analytic'' and ``geometric'' maps on orientations]\label{localorientationsnonthickened}
The following diagram commutes:
\begin{equation}
\begin{tikzcd}
\oo_{\Mbar(T')^\reg}\otimes\oo_{G_{T'//T}}\otimes\oo_{s(T')}^\vee\ar[equals]{r}{\eqref{Minteriororientationsnonthickened}}\ar[equals]{d}{\eqref{ultimategluingmapnonthickened}}&\oo_{T'}\otimes\oo_{G_{T'//T}}\otimes\oo_{s(T')}^\vee\ar[equals]{d}{\eqref{algebraicorientationgluing}}\\
\oo_{\Mbar(T)^\reg}\ar[equals]{r}{\eqref{Minteriororientationsnonthickened}}&\oo_T,
\end{tikzcd}
\end{equation}
where the left vertical map is the ``geometric'' map induced by the local topological structure of $\Mbar(T)^\reg$ given in \eqref{ultimategluingmapnonthickened}, and the right vertical map is the ``analytic'' map \eqref{algebraicorientationgluing} defined earlier via the ``kernel gluing'' operation.
\end{theorem}

\begin{proof}
See \S\ref{localorientationsproof}.
\end{proof}

\section{Implicit atlases}\label{IAsection}

In this section, we define (topological) implicit atlases with cell-like stratification on the moduli spaces $\Mbar(T)$ stratified by $\SSS_{/T}$.
The notion of an implicit atlas with cell-like stratification was introduced in \cite[\S\S 3,6]{pardonimplicitatlas}, and our constructions of implicit atlases follow the general procedure introduced in \cite[\S\S 1--2,9--10]{pardonimplicitatlas}.
We will, however, give a self-contained treatment.

\subsection{Implicit atlases with cell-like stratification}

We review the notion of an implicit atlas with cell-like stratification from \cite[Definitions 3.1.1, 3.2.1, 6.1.2, 6.1.6]{pardonimplicitatlas}.
The notion we present here is, in fact, more general, as we must allow stratifications by categories rather than just posets (recall Definitions \ref{stratificationdef} and \ref{stratificationcategorydef}).
Specifically, we consider here stratifications by categories $\TTT$ satisfying the following properties:
\begin{enumerate}
\item$\TTT$ has a final object.
\item$\TTT$ is finite, meaning $\#\left|\TTT\right|<\infty$ and $\#\Hom(\ttt_1,\ttt_2)<\infty$ for $\ttt_1,\ttt_2\in\TTT$.
\item$\TTT_{\ttt/}$ is a poset for every $t\in\TTT$.
\end{enumerate}
It will be useful to consider such categories $\TTT$ equipped with \emph{dimension and orientation data}, which consists of:
\begin{enumerate}
\item A function $\dim:\TTT\to\ZZ$ such that for every morphism $f:\ttt\to\ttt'$, we have $\dim\ttt\leq\dim\ttt'$ with equality iff $f$ is an isomorphism.
\item For every $\ttt\in\TTT$, an orientation line $\oo_\ttt$ of parity $\dim\ttt$ (more formally, this means a functor from the subcategory of isomorphisms in $\TTT$ to the category of orientation lines).
\item For every $\ttt\to\ttt'$ of codimension one (meaning $\dim\ttt'=\dim\ttt+1$), an isomorphism $\oo_\ttt\otimes\oo_{\RR_{\geq 0}}\xrightarrow\sim\oo_{\ttt'}$ (where $\oo_{\RR_{\geq 0}}=\oo_\RR$ is denoted as $\oo_{\RR_{\geq 0}}$ to indicate that it represents the orientation line of directions transverse to the boundary of a manifold with boundary).
\end{enumerate}

\begin{definition}\label{celllikedef}
Let $\TTT$ be a category as above equipped with dimension and orientation data.
An \emph{oriented cell-like stratification} is a stratification $X\to\TTT$ such that:
\begin{enumerate}
\item For every $\ttt\in\TTT$, the closed stratum $X_{/\ttt}$ is a topological manifold with boundary of dimension $\dim\ttt$, whose interior is precisely the locally closed stratum $X_\ttt\subseteq X_{/\ttt}$.
\item There are specified isomorphisms $\oo_{X_{/\ttt}}=\oo_\ttt$ such that for every $\ttt\to\ttt'$ of codimension one, the isomorphism $\oo_{X_{/\ttt}}\otimes\oo_{\RR_{\geq 0}}\xrightarrow\sim\oo_{X_{/\ttt'}}$ induced by the local topology agrees with the specified isomorphism $\oo_\ttt\otimes\oo_{\RR_{\geq 0}}\xrightarrow\sim\oo_{\ttt'}$.
\end{enumerate}
\end{definition}

\begin{definition}\label{implicitatlasdefinition}
Let $X$ be a compact Hausdorff space with stratification by a category $\TTT$ as above equipped with dimension and orientation data.
An \emph{oriented implicit atlas with cell-like stratification} on $X$ is an index set $A$ along with the following data:
\begin{enumerate}
\item(Covering groups) A finite group $\Gamma_\alpha$ for all $\alpha\in A$ (let $\Gamma_I:=\prod_{\alpha\in I}\Gamma_\alpha$).
\item(Obstruction spaces) A finitely generated $\RR[\Gamma_\alpha]$-module $E_\alpha$ for all $\alpha\in A$ (let $E_I:=\bigoplus_{\alpha\in I}E_\alpha$).
\item(Thickenings) A Hausdorff $\Gamma_I$-space $X_I$ for all finite $I\subseteq A$, equipped with a $\Gamma_I$-invariant stratification $X_I\to\TTT$, and a homeomorphism $X=X_\varnothing$ respecting the stratification.
\item(Kuranishi maps) A $\Gamma_\alpha$-equivariant function $s_\alpha:X_I\to E_\alpha$ for all $\alpha\in I\subseteq A$ (for $I\subseteq J$, let $s_I:X_J\to E_I$ denote $\bigoplus_{\alpha\in I}s_\alpha$).
\item(Footprints) A $\Gamma_I$-invariant open set $U_{IJ}\subseteq X_I$ for all $I\subseteq J\subseteq A$.
\item(Footprint maps) A $\Gamma_J$-equivariant function $\psi_{IJ}:(s_{J\setminus I}|X_J)^{-1}(0)\to U_{IJ}$ for all $I\subseteq J\subseteq A$.
\item(Regular and trivial isotropy locus) A $\Gamma_I$-invariant subset $X_I^\reg\subseteq X_I$ for all $I\subseteq A$.
\end{enumerate}
which must satisfy the following ``compatibility axioms'':
\begin{enumerate}
\item$\psi_{IJ}\psi_{JK}=\psi_{IK}$ and $\psi_{II}=\id$.
\item$s_I\psi_{IJ}=s_I$.
\item$U_{IJ_1}\cap U_{IJ_2}=U_{I,J_1\cup J_2}$ and $U_{II}=X_I$.
\item The restriction of the stratification on $X_J$ to $(s_{J\setminus I}|X_J)^{-1}(0)$ is identified with the pullback of the stratification on $X_I$ via $\psi_{IJ}$.
These identifications must be compatible in triples $I\subseteq J\subseteq K$ (this condition can be nontrivial when $\TTT$ is a category rather than a poset).
\item(Homeomorphism axiom) The footprint map $\psi_{IJ}$ induces a homeomorphism $(s_{J\setminus I}|X_J)^{-1}(0)/\Gamma_{J\setminus I}\xrightarrow\sim U_{IJ}$.
\end{enumerate}
and the following ``transversality axioms'':
\begin{enumerate}[resume]
\item$\psi_{IJ}^{-1}(X_I^\reg)\subseteq X_J^\reg$.
\item$\Gamma_{J\setminus I}$ acts freely on $\psi_{IJ}^{-1}(X_I^\reg)$.
\item(Openness axiom) $X_I^\reg\subseteq X_I$ is open.
\item(Submersion axiom) The stratification of $X_I^\reg$ is cell-like of dimension $\dim\ttt+\dim E_I$ and oriented with respect to $\oo_\ttt\otimes\oo_{E_I}$.  Furthermore, near every point of $\psi_{IJ}^{-1}(X_I^\reg)\subseteq X_J^\reg$, the map $s_{J\setminus I}:X_J\to E_{J\setminus I}$ is locally modelled on the projection $X_I^\reg\times\RR^{\dim E_{J\setminus I}}\to\RR^{\dim E_{J\setminus I}}$ compatibly with oriented stratifications.
\item(Covering axiom) $X_\varnothing=\bigcup_{I\subseteq A}\psi_{\varnothing I}((s_I|X_I^\reg)^{-1}(0))$.
\end{enumerate}
\end{definition}

\begin{definition}\label{implicitatlasdefinitionbdry}
An \emph{oriented implicit atlas with boundary} is an oriented implicit atlas with cell-like stratification by the two-element poset $\{\partial,\circ\}$ with $\partial<\circ$.
\end{definition}

\begin{remark}\label{forgetstratificationrmk}
Given an oriented implicit atlas with cell-like stratification, one may obtain an oriented implicit atlas with boundary simply by collapsing all strata other than the maximal stratum $\circ$ into a single stratum $\partial$.
\end{remark}

The \emph{virtual dimension} $d$ of a space equipped with an implicit atlas shall mean $\dim\ttt^\ttop$, where $\ttt^\ttop\in\TTT$ is the final object, and the \emph{virtual orientation line} $\oo$ of such a space shall mean $\oo_{\ttt^\ttop}$.

\subsection{Stratifications of \texorpdfstring{$G$}{G} are cell-like}\label{celllikeproofsec}

We now show that the natural stratifications $G_{T/}\to\SSS_{T/}$ are cell-like in the sense of Definition \ref{celllikedef}, where we equip $\SSS_{T/}$ with the dimension function
\begin{align}
\SSS_{T/}&\to\ZZ\\
(T\to T')&\mapsto\codim(T/T')+\dim\s(T)
\end{align}
and orientation data
\begin{equation}
(T\to T')\mapsto(\oo_\RR^\vee)^{\otimes V_s(T')}\otimes(\oo_\RR)^{\otimes V_s(T)}\otimes\oo_{\s(T')}.
\end{equation}

\begin{lemma}
The stratification $(G_\I)_{T/}\to(\SSS_\I)_{T/}$ is cell-like.
\end{lemma}

\begin{proof}
By inspection.
\end{proof}

\begin{lemma}\label{GIIisPL}
The stratification $(G_\II)_{T/}\to(\SSS_\II)_{T/}$ is cell-like.
\end{lemma}

\begin{proof}
Observe that for any $f:T\to T'$, we have
\begin{align*}
(G_\II)_{T//T'}&=\prod_{\ast(v')=00}(G_\I^+)_{f^{-1}(v')/}\times\prod_{\ast(v')=01}(G_\II)_{f^{-1}(v')/}\times\prod_{\ast(v')=11}(G_\I^-)_{f^{-1}(v')/},\\
(\SSS_\II)_{T//T'}&=\prod_{\ast(v')=00}(\SSS_\I^+)_{f^{-1}(v')/}\times\prod_{\ast(v')=01}(\SSS_\II)_{f^{-1}(v')/}\times\prod_{\ast(v')=11}(\SSS_\I^-)_{f^{-1}(v')/},
\end{align*}
compatibly with stratifications.  Note that a product of cell-like stratifications is cell-like.  Thus by induction (say, on the number of vertices of $T$), it suffices to show that $(G_\II)_{T/}$ is a topological manifold with boundary, whose interior coincides with the top stratum.

To show that $(G_\II)_{T/}$ is a topological manifold with boundary, we perform a change of variables $h=e^{-\g}\in[0,1)$.  For convenience, we will allow $h\in[0,\infty)$ (this relaxation is certainly permitted for the present purpose).  The relation $\g_v=\g_e+\g_{v'}$ now becomes $h_v=h_eh_{v'}$.  Under this relation $h_v=h_eh_{v'}$, observe that $h_v\in[0,\infty)$ and $h_e^2-h_{v'}^2\in(-\infty,\infty)$ determine $h_e\in[0,\infty)$ and $h_{v'}\in[0,\infty)$ uniquely, since
\begin{equation}
(h_e+ih_{v'})^2=(h_e^2-h_{v'}^2)+2ih_v.
\end{equation}
Thus if we perform another change of variables $q_e=h_e^2-h_{v'}^2$ for $v\xrightarrow ev'$, then we have
\begin{equation}\label{GIIalternativedef}
(G_\II)_{T/}=\left\{\begin{matrix}
h_{v^\ttop}\in[0,\infty)\hfill&\text{if }{\ast(v^\ttop)=00}\hfill\\
q_e\in(-\infty,\infty)\hfill&\text{for }v\xrightarrow ev'\text{ with }{\ast(v')=00}\hfill\\
h_e\in[0,\infty)\hfill&\text{for }\ast(e)=1\hfill
\end{matrix}\right\}.
\end{equation}
This is clearly a topological manifold with boundary of dimension $\#V_s(T)$.  Furthermore, the top stratum is the locus where $h_{v^\ttop}>0$ and $h_e>0$, which is clearly its interior.

Though not logically necessary due to the inductive reasoning above, let us remark that one can also express the stratification of $(G_\II)_{T/}$ by $(\SSS_\II)_{T/}$ concretely in terms of the coordinates \eqref{GIIalternativedef} and thereby verify explicitly that it is cell-like.
\end{proof}

\begin{lemma}\label{GIIIisPL}
The stratification $(G_\III)_{T/}\to(\SSS_\III)_{T/}$ is cell-like.
\end{lemma}

\begin{proof}
Express the underlying forest of $T\in\SSS_\III$ as the disjoint union of $T_i\in\SSS_\II$, so we have
\begin{align}
(G_\III)_{T/}&=\prod_{i\in I}(G_\II)_{T_i/}\times\begin{cases}[0,1)&\s(T)=\{0\}\\(0,1)&\s(T)=(0,1)\\(0,1]&\s(T)=\{1\},\end{cases}\\
(\SSS_\III)_{T/}&=\prod_{i\in I}(\SSS_\II)_{T_i/}\times\begin{cases}\{\{0\}<(0,1)\}&\s(T)=\{0\}\\\{(0,1)\}&\s(T)=(0,1)\\\{(0,1)>\{1\}\}&\s(T)=\{1\}.\end{cases}
\end{align}
Now apply Lemma \ref{GIIisPL} and note that the product of two cell-like stratifications is again cell-like.
\end{proof}

\begin{lemma}\label{GIVisPL}
The stratification $(G_\IV)_{T/}\to(\SSS_\IV)_{T/}$ is cell-like.
\end{lemma}

\begin{proof}
For $\s\ne\{\infty\}$, this is just Lemma \ref{GIIIisPL}.

For $\s=\{\infty\}$, argue as follows.  For any given $T\to T'$, we may express $(G_\IV)_{T//T'}\to(\SSS_\IV)_{T//T'}$ as a product as in the first step of the proof of Lemma \ref{GIIisPL}, by expressing $T'$ as a concatenation of maximal $T_i$.  Thus by induction, it suffices to show that $(G_\IV)_{T/}$ is a topological manifold with boundary whose interior is the stratum corresponding to the maximal contraction of $T$.  To see this, we perform the same change of variables as in Lemma \ref{GIIisPL} to write
\begin{equation}\label{GIValternativedef}
(G_\IV)_{T/}=
\left\{\begin{matrix}
h_{v^\ttop}\in[0,\infty)\hfill&\text{if }{\ast(v^\ttop)}=00\hfill\\
q_e\in(-\infty,\infty)\hfill&\text{for }v\xrightarrow ev'\text{ with }{\ast(v')=00}\hfill\\
\tau\in[0,\infty)\hfill\\
q_e\in(-\infty,\infty)\hfill&\text{for }v\xrightarrow ev'\text{ with }{\ast(v')=11}\hfill\\
h_e\in[0,\infty)\hfill&\text{for }{\ast(e)}=2\hfill
\end{matrix}\right\}.
\end{equation}
This is clearly a topological manifold of dimension $\#V_s(T)+1$, and its maximal stratum is the locus where $\tau>0$, $h_{v^\ttop}>0$, and $h_e>0$, which is clearly its interior.
\end{proof}

\subsection{Sets of thickening datums \texorpdfstring{$A(T)$}{A(T)}}\label{Athicksec}

We define sets of ``thickening datums'' $A(T)$ for the implicit atlases on $\Mbar(T)$.  Roughly speaking, a thickening datum $\alpha\in A(T)$ is a collection of data $(r_\alpha,D_\alpha,E_\alpha,\nu_\alpha)$ from which we will construct a ``thickened'' version of a given moduli space by: (1) adding $r_\alpha$ marked points to the domain (constrainted to lie on the codimension two submanifold $\hat D_\alpha$), (2) adding an extra parameter $e_\alpha$ lying in the vector space $E_\alpha$, and (3) adding an extra term $\nu_\alpha(e_\alpha)$ (which depends on the location of the $r_\alpha$ added points and the positive/negative ends $E^\eext(T)$ in the domain) to the pseudo-holomorphic curve equation.

Recall from \S\ref{rsmoduli} that $\Mbar_{0,n}$ ($n\geq 3$) denotes the Deligne--Mumford moduli space of stable nodal Riemann surfaces of genus zero with $n$ marked points labeled with $\{1,\ldots,n\}$.  We denote by $\Cbar_{0,n}\to\Mbar_{0,n}$ the universal family.  Recall that $\Mbar_{0,n}$ is a compact smooth manifold.  We usually prefer to label the marked points using a set other than $\{1,\ldots,n\}$, so we will also use the notation $\Mbar_{0,n}$ and $\Cbar_{0,n}$ when $n$ is a \emph{finite set} ($\#n\geq 3$) used to label the marked points.

\begin{definition}[Set of thickening datums $A_\I$]\label{AIdef}
A \emph{thickening datum} $\alpha$ for $T\in\SSS_\I$ consists of the following data:
\begin{enumerate}
\item\label{AIdefr}$r_\alpha\geq 0$ an integer such that $r_\alpha+\#E^\eext(T)\geq 3$.
\item\label{AIdefE}$n_\alpha\geq 0$ an integer (let $E_\alpha:=\RR^{n_\alpha}$) and an action on $E_\alpha$ of $S_{r_\alpha}$ (the group of permutations of $\{1,\ldots,r_\alpha\}$).
\item\label{AIdefD}$D_\alpha\subseteq Y$ a compact codimension two submanifold with boundary.  We let $\hat D_\alpha:=\RR\times D_\alpha\subseteq\hat Y$.
\item\label{AIdefl}$\nu_\alpha:E_\alpha\to C^\infty(\hat Y\times\Cbar_{0,E^\eext(T)\cup\{1,\ldots,r_\alpha\}},T\hat Y\otimes_\RR\Omega^{0,1}_{\Cbar_{0,E^\eext(T)\cup\{1,\ldots,r_\alpha\}}/\Mbar_{0,E^\eext(T)\cup\{1,\ldots,r_\alpha\}}})^\RR$ an $S_{r_\alpha}$-equivariant linear map which vanishes in a neighborhood of the nodes and the $E^\eext(T)$-marked points of the fibers of $\Cbar_{0,E^\eext(T)\cup\{1,\ldots,r_\alpha\}}\to\Mbar_{0,E^\eext(T)\cup\{1,\ldots,r_\alpha\}}$.  The superscript $^\RR$ indicates taking the subspace of $\RR$-invariant sections (where $\RR$ acts on $\hat Y$ by translation).
\end{enumerate}
We denote by $A_\I(T)$ the set of such thickening datums.
Note that for any morphism $T\to T'$, there is a tautological identification $A_\I(T)=A_\I(T')$.
\end{definition}

\begin{definition}[Set of thickening datums $A_\II$]\label{AIIdef}
A \emph{thickening datum} $\alpha$ for $T\in\SSS_\II$ consists of the following data:
\begin{enumerate}
\item $r_\alpha$, $E_\alpha$ as in Definition \ref{AIdef}\ref{AIdefr},\ref{AIdefE}.
\item $D^\pm_\alpha\subseteq Y^\pm$, $\nu^\pm_\alpha:E_\alpha\to C^\infty(\hat Y^\pm\times\Cbar_{0,E^\eext(T)\cup\{1,\ldots,r_\alpha\}},T\hat Y^\pm\otimes_\RR\Omega^{0,1}_{\Cbar_{0,E^\eext(T)\cup\{1,\ldots,r_\alpha\}}/\Mbar_{0,E^\eext(T)\cup\{1,\ldots,r_\alpha\}}})^\RR$ as in Definition \ref{AIdef}\ref{AIdefD},\ref{AIdefl}.
\item\label{AIIdefD}$\hat D_\alpha\subseteq\hat X$ a closed codimension two submanifold with boundary.  We require that $\hat D_\alpha$ coincide (via \eqref{endmarkingsI}--\eqref{endmarkingsII}) with $\hat D_\alpha^\pm$ outside a compact subset of $\hat X$.
\item\label{AIIdefl}$\nu_\alpha:E_\alpha\to C^\infty(\hat X\times\Cbar_{0,E^\eext(T)\cup\{1,\ldots,r_\alpha\}},T\hat X\otimes_\RR\Omega^{0,1}_{\Cbar_{0,E^\eext(T)\cup\{1,\ldots,r_\alpha\}}/\Mbar_{0,E^\eext(T)\cup\{1,\ldots,r_\alpha\}}})$ an $S_{r_\alpha}$-equivariant linear map vanishing in a neighborhood of the nodes and $E^\eext(T)$-marked points.  We require that $\nu_\alpha$ coincide (via \eqref{endmarkingsI}--\eqref{endmarkingsII}) with $\nu_\alpha^\pm$ outside a compact subset of $\hat X$.
\end{enumerate}
We denote by $A_\II(T)$ the set of such thickening datums.
Note that for any morphism $T\to T'$, there is a tautological identification $A_\II(T)=A_\II(T')$.
\end{definition}

\begin{definition}[Set of thickening datums $A_\III$]\label{AIIIdef}
A \emph{thickening datum} $\alpha$ for a connected non-empty $T\in\SSS_\III$ is defined identically as a thickening datum for $T\in\SSS_\II$.  This makes sense since in Setup \ref{setupIII}, the identifications \eqref{endmarkingsI}--\eqref{endmarkingsII} are independent of $t$ and the definition of a thickening datum does not make reference to $\lambda^t$ or $J^t$.  We denote by $A_\III(T)$ the set of such thickening datums.
Note that for any morphism $T\to T'$, there is a tautological identification $A_\III(T)=A_\III(T')$.
\end{definition}

\begin{definition}[Set of thickening datums $A_\IV$]\label{AIVdef}
A \emph{thickening datum} $\alpha$ for a connected non-empty $T\in\SSS_\IV$ consists of the following data:
\begin{enumerate}
\item $r_\alpha$, $E_\alpha$ as in Definition \ref{AIdef}\ref{AIdefr},\ref{AIdefE}.
\item $\{D^i_\alpha\subseteq Y^i\}_{i=0,1,2}$, $\{\nu^i_\alpha:E_\alpha\to C^\infty(\hat Y^i\times\Cbar_{0,E^\eext(T)\cup\{1,\ldots,r_\alpha\}},T\hat Y^i\otimes_\RR\Omega^{0,1}_{\Cbar_{0,E^\eext(T)\cup\{1,\ldots,r_\alpha\}}/\Mbar_{0,E^\eext(T)\cup\{1,\ldots,r_\alpha\}}})^\RR\}_{i=0,1,2}$ as in Definition \ref{AIdef}\ref{AIdefD},\ref{AIdefl}.
\item $\{\hat D_\alpha^{i,i+1}\subseteq\hat X^{i,i+1}\}_{i=0,1}$ as in Definition \ref{AIIdef}\ref{AIIdefD}.
\item $\hat D_\alpha^{02,t}\subseteq\hat X^{02,t}$ a smoothly varying family of compact codimension two submanifolds with boundary for $t\in[0,\infty)$ as in Definition \ref{AIIdef}\ref{AIIdefD} which coincides with the descent of $\hat D^{01}_\alpha\sqcup\hat D^{12}_\alpha$ for sufficiently large $t$.  ``Smoothly varying'' means the projection to $[0,\infty)$ is a submersion (and remains a submersion when restricted to the boundary).
\item $\{\nu_\alpha^{i,i+1}:E_\alpha\to C^\infty(\hat X^{i,i+1}\times\Cbar_{0,E^\eext(T)\cup\{1,\ldots,r_\alpha\}},T\hat X^{i,i+1}\otimes_\RR\Omega^{0,1}_{\Cbar_{0,E^\eext(T)\cup\{1,\ldots,r_\alpha\}}/\Mbar_{0,E^\eext(T)\cup\{1,\ldots,r_\alpha\}}})\}_{i=0,1}$ as in Definition \ref{AIIdef}\ref{AIIdefl}.
\item $\nu_\alpha^{02,t}:E_\alpha\to C^\infty(\hat X^{02,t}\times\Cbar_{0,E^\eext(T)\cup\{1,\ldots,r_\alpha\}},T\hat X^{02,t},\otimes_\RR\Omega^{0,1}_{\Cbar_{0,E^\eext(T)\cup\{1,\ldots,r_\alpha\}}/\Mbar_{0,E^\eext(T)\cup\{1,\ldots,r_\alpha\}}})$ as in Definition \ref{AIIdef}\ref{AIIdefl} coinciding with the descent of $\nu^{01}_\alpha\sqcup\nu^{12}_\alpha$ for sufficiently large $t$, and varying smoothly with $t$.
\end{enumerate}
We denote by $A_\IV(T)$ the set of such thickening datums.
Note that for any morphism $T\to T'$, there is a tautological identification $A_\IV(T)=A_\IV(T')$.
\end{definition}

\subsection{Index sets \texorpdfstring{$\bar A(T)$}{Abar(T)}}\label{Abarthicksec}

We define the index sets $\bar A(T)$ of the implicit atlases on $\Mbar(T)$ as unions of copies of the sets of thickening datums $A(T)$.
Precisely, we define
\begin{equation}
\bar A(T):=\bigsqcup_{T'\subseteq T}A(T'),
\end{equation}
where the disjoint union is over all \emph{subtrees} $T'$ of $T$.
What qualifies as a subtree depends on which case (I), (II), (III), (IV) we are in, and we discuss each case individually.
Informally, in cases (I) and (II), a subtree of $T$ is a tree $T'$ which appears in some concatenation yielding $T$, however this interpretation breaks down somewhat in cases (III) and (IV) due to connectedness issues.

\begin{figure}[ht]
\centering
\includegraphics{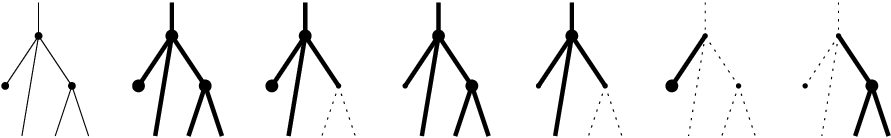}
\caption{A tree and its six physical subtrees.}\label{subtreesfig}
\end{figure}

Let us first explain what we mean by a subtree of an ordinary directed tree $T$ (or forest), as opposed to an object of one of the categories $\SSS$.
Choosing some subset of the vertices $V(T')\subseteq V(T)$, we define $E(T')\subseteq E(T)$ as all edges incident to an element of $V(T')$ (at at least one of its endpoints).
The resulting graph $T'$ will be called a subtree when it is connected and non-empty.
Below, we will use the word ``physical subtree'' for this notion, to distinguish it from the notion of a subtree of an object of some $\SSS$.
An example is illustrated in Figure \ref{subtreesfig}.

\begin{definition}[Index set $\bar A_\I$]
For $T\in\SSS_\I$, we define
\begin{equation}
\label{barAIcoproduct}\bar A_\I(T):=\bigsqcup_{T\supseteq T'\in\SSS_\I}A_\I(T').
\end{equation}
More precisely, there is exactly one subtree $T\supseteq T'\in\SSS_\I$ for every physical subtree of the underlying tree of $T$.
A physical subtree becomes an object of $\SSS_\I$ simply by restricting the vertex and edge decorations from $T$.
Note that there is no canonical way to choose basepoints for the input/output edges of subtrees $T'\in\SSS_\I$, however the disjoint union \eqref{barAIcoproduct} remains well-defined since the subgroup of $\Aut(T')$ given by paths between basepoints acts trivially on $A_\I(T')$.
\end{definition}

\begin{definition}[Index set $\bar A_\II$]
For $T\in\SSS_\II$, we define
\begin{equation}
\label{barAIIcoproduct}\bar A_\II(T):=\bigsqcup_{T\supseteq T'\in\SSS_\I^+}A_\I^+(T')\sqcup\bigsqcup_{T\supseteq T'\in\SSS_\I^-}A_\I^-(T')\sqcup\bigsqcup_{T\supseteq T'\in\SSS_\II}A_\II(T').
\end{equation}
More precisely, physical subtrees of $T$ give rise to terms in this disjoint union as follows:
\begin{enumerate}
\item $T\supseteq T'\in\SSS_\II$ are those with $\ast(e^+)=0$ and $\ast(e^-)=1$ for $e^\pm\in E^\pm(T')$.
\item $T\supseteq T'\in\SSS_\I^+$ are those for which all edges and vertices have $\ast=0$.
\item $T\supseteq T'\in\SSS_\I^-$ are those for which all edges and vertices have $\ast=1$.
\end{enumerate}
Note that those physical subtrees $T'\subseteq T$ for which all edges and vertices have $\ast=0$ and which have no negative external edges can be regarded as objects of $\SSS_\I^+$ or $\SSS_\II$, and thus appear twice in \eqref{barAIIcoproduct}.
\end{definition}

\begin{definition}[Index set $\bar A_\III$]
For $T\in\SSS_\III$, we define
\begin{align}
\bar A_\III(T)&:=\bigsqcup_{T\supseteq T'\in\SSS_\I^+}A_\I^+(T')\sqcup\bigsqcup_{T\supseteq T'\in\SSS_\I^-}A_\I^-(T')\nonumber\\
\label{barAIIIcoproduct}&\qquad\sqcup\bigsqcup_{T\supseteq T'\in\SSS_\II^{t=0}}A_\II^{t=0}(T')\sqcup\bigsqcup_{T\supseteq T'\in\SSS_\II^{t=1}}A_\II^{t=1}(T')\sqcup\bigsqcup_{T\supseteq T'\in\SSS_\III}A_\III(T').
\end{align}
More precisely, physical subtrees of $T$ give rise to terms in this disjoint union as follows:
\begin{enumerate}
\item $T\supseteq T'\in\SSS_\III$ are those with $\ast(e^+)=0$ and $\ast(e^-)=1$ for $e^\pm\in E^\pm(T')$.
\item $T\supseteq T'\in\SSS_\II^{t=0}$ are those with $\ast(e^+)=0$ and $\ast(e^-)=1$ for $e^\pm\in E^\pm(T')$ and $\s(T)=\{0\}$.
\item $T\supseteq T'\in\SSS_\II^{t=1}$ are those with $\ast(e^+)=0$ and $\ast(e^-)=1$ for $e^\pm\in E^\pm(T')$ and $\s(T)=\{1\}$.
\item $T\supseteq T'\in\SSS_\I^+$ are those for which all edges and vertices have $\ast=0$.
\item $T\supseteq T'\in\SSS_\I^-$ are those for which all edges and vertices have $\ast=1$.
\end{enumerate}
It should be emphasized that subtrees $T\supseteq T'\in\SSS_\II^{t=0}$ (respectively $T\supseteq T'\in\SSS_\II^{t=1})$ are only present when $\s(T)=\{0\}$ (respectively $\s(T)=\{1\}$).
It should also be emphasized that we consider only sub\emph{trees} (required to be connected and non-empty), rather than subforests.
\end{definition}

\begin{definition}[Index set $\bar A_\IV$]
For $T\in\SSS_\IV$, we define
\begin{align}
\bar A_\IV(T)&:=
\bigsqcup_{T\supseteq T'\in\SSS_\I^0}A_\I^0(T')\sqcup\bigsqcup_{T\supseteq T'\in\SSS_\I^2}A_\I^2(T')\sqcup\bigsqcup_{T\supseteq T'\in\SSS_\II^{02}}A_\II^{02}(T')\nonumber\\
\label{barAIVcoproduct}&\qquad
\sqcup\bigsqcup_{T\supseteq T'\in\SSS_\II^{01}}A_\II^{01}(T')\sqcup\bigsqcup_{T\supseteq T'\in\SSS_\I^1}A_\I^1(T')\sqcup\bigsqcup_{T\supseteq T'\in\SSS_\II^{12}}A_\II^{12}(T')\\
&\qquad\sqcup\bigsqcup_{T\supseteq T'\in\SSS_\IV}A_\IV(T').\nonumber
\end{align}
More precisely, physical subtrees of $T$ give rise to terms in this disjoint union as follows:
\begin{enumerate}
\item $T\supseteq T'\in\SSS_\IV$ are those with $\ast(e^+)=0$ and $\ast(e^-)=2$ for $e^\pm\in E^\pm(T')$.
\item $T\supseteq T'\in\SSS_\II^{01}$ are those with $\ast(e^+)=0$ and $\ast(e^-)=1$ for $e^\pm\in E^\pm(T')$ and~\mbox{$\s(T)=\{\infty\}$}.
\item $T\supseteq T'\in\SSS_\II^{12}$ are those with $\ast(e^+)=1$ and $\ast(e^-)=2$ for $e^\pm\in E^\pm(T')$ and~\mbox{$\s(T)=\{\infty\}$}.
\item $T\supseteq T'\in\SSS_\II^{02}$ are those with $\ast(e^+)=0$ and $\ast(e^-)=2$ for $e^\pm\in E^\pm(T')$ and $\s(T)=\{0\}$.
\item $T\supseteq T'\in\SSS_\I^0$ are those for which all edges and vertices have $\ast=0$.
\item $T\supseteq T'\in\SSS_\I^1$ are those for which all edges and vertices have $\ast=1$.
\item $T\supseteq T'\in\SSS_\I^2$ are those for which all edges and vertices have $\ast=2$.
\end{enumerate}
\end{definition}

A morphism $T\to T'$ induces a natural inclusion
\begin{equation}\label{Abarfunct}
\bar A(T')\hookrightarrow\bar A(T),
\end{equation}
since given $T\to T'$, any subtree $T''\subseteq T'$ pulls back to a subtree of $T$.

For any concatenation $\{T_i\}_i$, there is a natural inclusion
\begin{equation}\label{Abarprod}
\bigsqcup_i\bar A(T_i)\hookrightarrow\bar A(\#_iT_i).
\end{equation}

\subsection{Thickened moduli spaces}\label{thickenedmodulisec}

We now define the thickened moduli spaces for the implicit atlases on the moduli spaces $\Mbar$.

\begin{figure}[ht]
\centering
\includegraphics{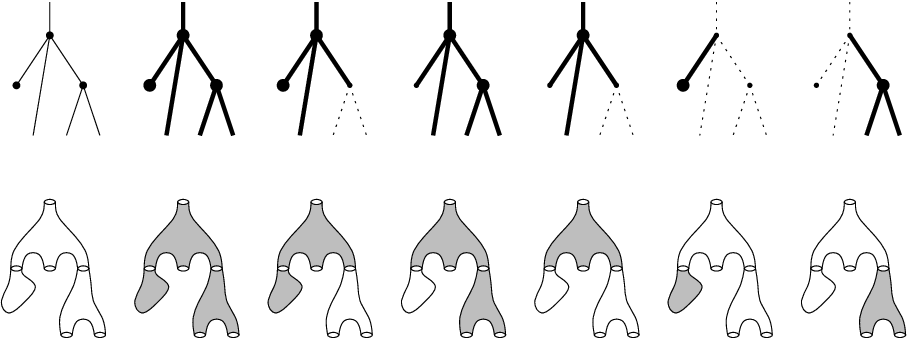}
\caption{A tree with its six subtrees, and a corresponding pseudo-holomorphic building with its corresponding subbuildings.}\label{subtreebuildingfig}
\end{figure}

\begin{definition}[Moduli space $\M_\I(T)_I$]\label{MIthickdef}
Let $T\in\SSS_\I$ and let $I\subseteq\bar A_\I(T)$ be finite.  An \emph{$I$-thickened pseudo-holomorphic building of type $T$} consists of the following data:
\begin{enumerate}
\item\label{MIthickdefCp}Domains $C_v$ and punctures $p_{v,e}$ as in Definition \ref{MIdef}\ref{MIdefC}.  For $\alpha\in I$, let $C_\alpha:=\bigsqcup_{v\in T_\alpha}C_v/\!\!\sim$, where $T\supseteq T_\alpha\in\SSS_\I$ denotes the subtree indexing the term in \eqref{barAIcoproduct} containing $\alpha\in I$, and $\sim$ identifies $p_{v,e}\sim p_{v',e}$ for interior edges $v\xrightarrow ev'$ of $T_\alpha$ (see Figure \ref{subtreebuildingfig}).
\item Maps $u_v$, asymptotic markers $\tilde b_e$, and matching isomorphisms $m_e$ as in Definition \ref{MIdef}\ref{MIdefmap},\ref{MIdefEA},\ref{MIdefbase},\ref{MIdefM}.
\item\label{MIthickdeftr}For all $\alpha\in I$, we require that $(u|C_\alpha)\pitchfork\hat D_\alpha$ with exactly $r_\alpha$ intersections, which together with $\{p_{v,e}\}_{v,e}$ stabilize $C_\alpha$.  By $(u|C_\alpha)\pitchfork\hat D_\alpha$, we mean that $\gamma_e\cap D_\alpha=\varnothing$ for edges $e\in E(T_\alpha)$, $(u|C_\alpha)^{-1}(\partial\hat D_\alpha)=\varnothing$, $(u|C_\alpha)^{-1}(\hat D_\alpha)$ does not contain any node, and $(du)_p:T_pC_\alpha\to T_{u(p)}\hat Y/T_{u(p)}\hat D_\alpha$ is surjective for $p\in(u|C_\alpha)^{-1}(\hat D_\alpha)$.
\item\label{MIthickphi}$\{\phi_\alpha:C_\alpha\to\Cbar_{0,E^\eext(T_\alpha)\cup\{1,\ldots,r_\alpha\}}\}_{\alpha\in I}$, where each $\phi_\alpha$ maps $C_\alpha$ isomorphically onto a fiber of $\Cbar_{0,E^\eext(T_\alpha)\cup\{1,\ldots,r_\alpha\}}$, where $C_\alpha$ is equipped with its given marked points $p_{v,e}$ for $e\in E^\eext(T_\alpha)$ and any marking of $(u|C_\alpha)^{-1}(\hat D_\alpha)$ with $\{1,\ldots,r_\alpha\}$.  Note that under \ref{MIthickdeftr} above, choosing $\phi_\alpha$ is equivalent to choosing a marking of $(u|C_\alpha)^{-1}(\hat D_\alpha)$ with $\{1,\ldots,r_\alpha\}$.
\item\label{MIthickdefe}$\{e_\alpha\in E_\alpha\}_{\alpha\in I}$.
\item We require that $u$ satisfy
\begin{equation}\label{Ithickdelbar}
\Bigl(du+\sum_{\alpha\in I}\nu_\alpha(e_\alpha)(\phi_\alpha(\cdot),u(\cdot))\Bigr)^{0,1}_{\hat J}=0.
\end{equation}
Note that the term in $\sum_{\alpha\in I}$ corresponding to $\alpha$ makes sense only over $C_\alpha$, and we define it to be zero elsewhere.
\end{enumerate}
An \emph{isomorphism} $(\{C_v\},\{p_{v,e}\},\{u_v\},\{\tilde b_e\},\{m_e\},\{\phi_\alpha\},\{e_\alpha\})\to(\{C_v'\},\{p_{v,e}'\},\{u_v'\},\{\tilde b_e'\},\{m_e'\},\{\phi_\alpha'\},\{e_\alpha'\})$ between $I$-thickened pseudo-holomorphic buildings of type $T$ is defined as in Definition \ref{MIdef}, with the additional requirements that $e_\alpha=e_\alpha'$ and $\phi_{\alpha,v}=\phi_{\alpha,v}'\circ i_v$ for $v\in T_\alpha$.  We denote by $\M_\I(T)_I$ the set of isomorphism classes of stable $I$-thickened pseudo-holomorphic buildings of type $T$.
Note the tautological action of $\Aut(T)$ on $\M(T)$ by changing the marking.
\end{definition}

Note that the sum over $\alpha$ in \eqref{Ithickdelbar} is supported away from the punctures $p_{v,e}\in C_v$, and hence $u_v$ is genuinely $\hat J$-holomorphic near $p_{v,e}$.  Note also that \eqref{Ithickdelbar} is equivalent to the assertion that the graph $(\id,u_v):C_v\to C_v\times\hat Y$ is pseudo-holomorphic for the almost complex structure on $C_v\times\hat Y$ given by
\begin{equation}
\left(\begin{matrix}j_{C_v}&0\\(\sum_{\alpha\in I}\nu_\alpha(e_\alpha)(\phi_\alpha(\cdot),\cdot))^{0,1}&\hat J\end{matrix}\right).
\end{equation}
Hence solutions to the $I$-thickened pseudo-holomorphic curve equation enjoy all of the nice local \emph{a priori} estimates from \S\ref{ellipticestimates}.
Specifically, Lemma \ref{apriorijhol} applied to the graph of $u$ implies that $C^0$-estimates imply $C^\infty$-estimates, and Lemma \ref{apriorinode} and Proposition \ref{hwzcylinderestimate} apply to $u$ itself since the perturbation terms $\nu_\alpha$ are supported away from the nodes and the long cylinders of the domain.

\begin{definition}[Moduli space $\M_\II(T)_I$]\label{MIIthickdef}
Let $T\in\SSS_\II$ and let $I\subseteq\bar A_\II(T)$.  An \emph{$I$-thickened pseudo-holomorphic building of type $T$} consists of the following data:
\begin{enumerate}
\item Domains $C_v$ and punctures $p_{v,e}$ as in Definition \ref{MIdef}\ref{MIdefC}.
\item Maps $u_v$, asymptotic markers $\tilde b_e$, and matching isomorphisms $m_e$ as in Definition \ref{MIdef}\ref{MIdefmap},\ref{MIdefEA},\ref{MIdefbase},\ref{MIdefM} (the target of $u_v$ is $\hat X^{\ast(v)}$).
\item For all $\alpha\in I$, we require that $(u|C_\alpha)\pitchfork\hat D_\alpha$ with exactly $r_\alpha$ intersections, which together with $\{p_{v,e}\}_{v,e}$ stabilize $C_\alpha$.
More precisely, we intersect $u_v$ with the relevant $\hat D_\alpha\subseteq\hat X$ or $\hat D_\alpha^\pm\subseteq\hat Y^\pm$ depending on $\ast(v)$.
\item$\{\phi_\alpha:C_\alpha\to\Cbar_{0,E^\eext(T_\alpha)\cup\{1,\ldots,r_\alpha\}}\}_{\alpha\in I}$ as in Definition \ref{MIthickdef}\ref{MIthickphi}.
\item We require that $u$ satisfy \eqref{Ithickdelbar}, noting to use $\nu_\alpha$ or $\nu_\alpha^\pm$ depending on $\ast(v)$.
\end{enumerate}
An \emph{isomorphism} between pseudo-holomorphic buildings of type $T$ is defined as in Definition \ref{MIthickdef}, except that there is a translation $s_v\in\RR$ only if $v$ is a symplectization vertex.  We denote by $\M_\II(T)_I$ the set of isomorphism classes of stable $I$-thickened pseudo-holomorphic buildings of type $T$.
\end{definition}

\begin{definition}[Moduli space $\M_\III(T)_I$]\label{MIIIthickdef}
Let $T\in\SSS_\III$ and let $I\subseteq\bar A_\III(T)$ be finite.  We denote by $\M_\III(T)_I$ the union over $t\in\s(T)$ of the set of isomorphism classes of stable $I$-thickened pseudo-holomorphic buildings of type $T$.
\end{definition}

\begin{definition}[Moduli space $\M_\IV(T)_I$]\label{MIVthickdef}
Let $T\in\SSS_\IV$ and let $I\subseteq\bar A_\IV(T)$ be finite.  We denote by $\M_\IV(T)_I$ the union over $t\in\s(T)$ of the set of isomorphism classes of stable $I$-thickened pseudo-holomorphic buildings of type $T$.
\end{definition}

\begin{definition}[Moduli spaces $\Mbar_I$]\label{Mbarthickdef}
For $T\in\SSS$ and $I\subseteq\bar A(T)$, we define
\begin{equation}
\Mbar(T)_I:=\bigsqcup_{\begin{smallmatrix}T'\to T\\\Mbar(T')\ne\varnothing\end{smallmatrix}}\M(T')_I/\!\Aut(T'/T).
\end{equation}
Each such set $\Mbar(T)_I$ has a natural Gromov topology which is Hausdorff.
\end{definition}

The stratifications \eqref{MSstrat} are clearly defined on the thickened moduli spaces $\Mbar(T)_I$.  The tautological functorial structure \eqref{Mbarfunct} (combined with \eqref{Abarfunct}) and \eqref{Mbarprod} (combined with \eqref{Abarprod}) also exists for the thickened moduli spaces.

\begin{definition}[Regularity of $I$-thickened pseudo-holomorphic buildings]\label{linopthickened}
As in Definition \ref{regularitydef}, given an $I$-thickened pseudo-holomorphic building of type $T$, we consider the linearized operator
\begin{equation}\label{totallinearizedthickened}
E_I\oplus\bigoplus_{v\in V(T)}\tilde W^{k,2,\delta}(C_v,u_v^\ast T\hat X_v)\to\bigoplus_{v\in V(T)}W^{k-1,2,\delta}(\tilde C_v,u_v^\ast(T\hat X_v)_{\hat J_v}\otimes_\CC\Omega^{0,1}_{\tilde C_v})
\end{equation}
(recall that $E_I:=\bigoplus_{\alpha\in I}E_\alpha$).
Note that when $r_\alpha>0$ for some $\alpha\in I$, we must take $k\geq 2$ so that $W^{k,2}\subseteq C^0$ (see \cite[Lemma 5.17]{adamssobolev}) so that the first variation of $\phi_\alpha$ under $W^{k,2}$ deformations of $u$ is defined.

A point in a moduli space $\Mbar(T)_I$ is called \emph{regular} iff this linearized operator is surjective.
We denote by $\Mbar(T)_I^\reg\subseteq\Mbar(T)_I$ the locus of points which are regular and which have trivial automorphism group.
\end{definition}

\subsection{Implicit atlas structure}

We now define the rest of the implicit atlas structure for the implicit atlases $\bar A(T)$ on $\Mbar(T)$.

Recall that the covering groups $\Gamma_\alpha$ and the obstruction spaces $E_\alpha$ are already built into the definition of $\bar A(T)$ from \S\S\ref{Athicksec}--\ref{Abarthicksec}.
The thickened moduli spaces $\Mbar(T)_I$ and their regular loci $\Mbar(T)_I^\reg\subseteq\Mbar(T)_I$ were defined in \S\ref{thickenedmodulisec}.
The Kuranishi maps $s_\alpha:\Mbar(T)_I\to E_\alpha$ are the tautological maps picking out $e_\alpha$ (Definition \ref{MIthickdef}\ref{MIthickdefe}), and the footprint maps $\psi_{IJ}:(s_{J\setminus I}|\Mbar(T)_J)^{-1}(0)\to\Mbar(T)_I$ are the tautological forgetful maps.

The footprint $U_{IJ}\subseteq\Mbar(T)_I$ is defined as the locus of buildings which satisfy the transversality condition Definition \ref{MIthickdef}\ref{MIthickdeftr} for all $\alpha\in J$, and it follows immediately from this definition that $\psi_{IJ}$ induces a bijection $(s_{J\setminus I}|\Mbar(T)_J)^{-1}(0)/\Gamma_{J\setminus I}\to U_{IJ}$.
Inspection of the definition of the Gromov topology shows that this map is a homeomorphism and that $U_{IJ}\subseteq\Mbar(T)_I$ is open.

\begin{theorem}
The above data define oriented implicit atlases $\bar A(T)$ on $\Mbar(T)$.
\end{theorem}

\begin{proof}
Of the axioms which have not already been discussed (and which are nontrivial), the covering axiom follows from Propositions \ref{coveringI}--\ref{coveringIV} below, and the openness and submersion axioms follow from Theorems \ref{localgluing} and \ref{localorientations} below.
\end{proof}

\subsection{Stabilization of pseudo-holomorphic curves}\label{stabilizationsec}

We now verify the covering axiom for the implicit atlases $\bar A(T)$ on $\Mbar(T)$.  Namely, we show that the moduli spaces $\Mbar(T)$ are covered by the regular loci in their thickenings $\Mbar(T)_I^\reg$.  The essential content is to show that at every point in each moduli space $\Mbar(T)$, there exists a codimension two submanifold (as in Definition \ref{AIdef}\ref{AIdefD}) which stabilizes the domain (i.e.\ satisfies Definition \ref{MIthickdef}\ref{MIthickdeftr}).

\begin{lemma}\label{injective}
Let $u:D^2\to(X,J)$ be $J$-holomorphic (for an almost complex manifold $(X,J)$).  Then either $du:T_pC\to T_{u(p)}X$ is injective for some $p\in D^2$ or\footnote{In fact, it is a standard (but nontrivial) fact that the first alternative can be strengthened to state that the zeroes of $du$ form a discrete set (see \cite[Lemma 2.4.1]{mcduffsalamonJholsymp}).} $u$ is constant.
\end{lemma}

\begin{proof}
If $du$ is non-injective, it must be zero by $J$-holomorphicity.
\end{proof}

\begin{lemma}\label{xiinjective}
Let $u:D^2\to(\hat Y,\hat J)$ be $\hat J$-holomorphic (for $(Y,\lambda,J)$ as in Setup \ref{setupI}).  Denote by $\pi_\xi:T\hat Y\to\xi$ the projection under the splitting $T\hat Y=\xi\oplus\RR R_\lambda\oplus\RR\partial_s$.  Then either $\pi_\xi du:T_pC\to\xi_{u(p)}$ is injective for some $p\in D^2$ or\footnote{In fact, it is a standard (but nontrivial) fact that the first alternative can be strengthened to state that the zeroes of $\pi_\xi du$ form a discrete set (see Hofer--Wysocki--Zehnder \cite[Proposition 4.1]{hwzsymplectizationsII}).} $u$ factors through ${\id}\times{\gamma}:\RR\times\RR\to\RR\times Y$ for some Reeb trajectory $\gamma:\RR\to Y$.
\end{lemma}

\begin{proof}
If $\pi_\xi du$ is non-injective, it must be zero.  If $\pi_\xi du$ vanishes identically, then $du$ is everywhere tangent to the $2$-dimensional foliation of $\hat Y$ by $\RR R_\lambda\oplus\RR\partial_s$, and thus $u$ factors through one of its leaves.
\end{proof}

Recall that $U_{IJ}\subseteq\Mbar(T)_I$ denotes the locus of buildings satisfying the transversality condition in Definition \ref{MIthickdef}\ref{MIthickdeftr} for $\alpha\in J\setminus I$.

\begin{proposition}\label{coveringI}
For every $x\in\Mbar_\I(T)$, there exists $\alpha\in A_\I(T)$ such that $x\in U_{\varnothing,\{\alpha\}}$ and $\psi_{\varnothing,\{\alpha\}}^{-1}(x)\subseteq\Mbar_\I(T)_{\{\alpha\}}^\reg$.
\end{proposition}

\begin{proof}
The point $x$ is an isomorphism class of stable pseudo-holomorphic building of type $T'\to T$.

We claim that for all $v\in V(T')$, either $C_v$ is stable (i.e.\ the degree of $v$ is $\geq 3$) or $\pi_\xi du_v$ is injective somewhere on $C_v$.  To see this, suppose that $\pi_\xi du_v\equiv 0$ and apply Lemma \ref{xiinjective} to $u_v:C_v\to\hat Y$.  If the resulting Reeb trajectory $\gamma:\RR\to Y$ is not a closed orbit, then $u_v$ factors through $\RR\times\RR\to\RR\times Y$, and consideration of the positive puncture of $C_v$ leads to a contradiction.  Thus $u_v$ factors through a trivial cylinder $\RR\times S^1\to\RR\times Y$ for some simple Reeb orbit $\gamma:S^1\to Y$.  The map $C_v\to\RR\times S^1$ is holomorphic, and it must have ramification points, as otherwise the building $x$ would be unstable.  It now follows from Riemann--Hurwitz that $C_v$ is stable.  Thus the claim is valid.

Now using the claim, it follows from Sard's theorem that there exists $D_\alpha\subseteq Y$ satisfying Definition \ref{MIthickdef}\ref{MIthickdeftr} for some $r_\alpha\geq 0$.  Now to show the existence of $E_\alpha$ and $\nu_\alpha$ so that $x$ is regular for the thickening datum $\alpha=(D_\alpha,r_\alpha,E_\alpha,\nu_\alpha)$, it suffices to show that
\begin{equation}\label{cinftyc}
\bigoplus_{v\in V(T')}C^\infty_c(\tilde C_v\setminus(\{p_{v,e}\}_e\cup\tilde N_v),u_v^\ast T\hat Y_{\hat J}\otimes_\CC\Omega^{0,1}_{\tilde C_v})
\end{equation}
surjects onto the (finite-dimensional) cokernel of the linearized operator at $x$ (where $\tilde N_v\subseteq\tilde C_v$ denotes the pre-images of the nodes of $C_v$).
Let $\varepsilon$ be a continuous linear functional on
\begin{equation}\label{alltarget}
\bigoplus_{v\in V(T')}W^{k-1,2,\delta}(\tilde C_v,u_v^\ast T\hat Y_{\hat J}\otimes_\CC\Omega^{0,1}_{\tilde C_v})
\end{equation}
which vanishes both on \eqref{cinftyc} and on the image of the linearized operator; it suffices to show that $\varepsilon=0$.
By Lemma \ref{smoothweighteddense}, the space
\begin{equation}\label{cinftycreprise}
\bigoplus_{v\in V(T')}C^\infty_c(\tilde C_v\setminus\{p_{v,e}\}_e,u_v^\ast T\hat Y_{\hat J}\otimes_\CC\Omega^{0,1}_{\tilde C_v})
\end{equation}
is dense in \eqref{alltarget}, so it suffices to show that $\varepsilon$ vanishes as a distribution.
Since $\varepsilon$ annihilates \eqref{cinftyc}, it is supported over the finite set $\tilde N_v$.
When $k=2$, no nonzero distribution of finite support is continuous on $W^{k-1,2}$ (recall Lemma \ref{distrdelta} and note that $W^{1,2}\nsubseteq C^0$ \cite[Example 5.26]{adamssobolev}).
Since we did not use the fact that $\varepsilon$ annihilates the image of the linearized operator, we have in fact shown that \eqref{cinftyc} is dense in \eqref{alltarget}.
\end{proof}

\begin{remark}
There are two alternative ways to conclude the proof of Proposition \ref{coveringI}.
First, using the fact that $\varepsilon$ annihilates the image of $D$, it follows that $D^\ast\varepsilon$ ($D^\ast$ denotes the formal adjoint of $D$) is a linear combination of $\delta$-functions supported over $\tilde N_v$.  But such $\delta$-functions live in $H^{-2}$ (by Sobolev embedding $H^2\subseteq C^0$ \cite[Lemma 5.17]{adamssobolev}) and the formal adjoint is elliptic, so this means $\varepsilon\in H^{-1}$, which contains no nonzero distributions supported at single points as argued above.
Second, we could define the linearized operator for nodal curves using weights near the nodes (as we do in \S\ref{gluingsobolev} to prove gluing) and simply cite Lemma \ref{smoothweighteddense} to conclude that \eqref{cinftyc} is dense in the target weighted Sobolev space for all $k\geq 2$.
\end{remark}

\begin{proposition}\label{coveringII}
For every $x\in\Mbar_\II(T)$, there exists $\alpha\in A_\II(T)$ such that $x\in U_{\varnothing,\{\alpha\}}$ and $\psi_{\varnothing,\{\alpha\}}^{-1}(x)\subseteq\Mbar_\II(T)_{\{\alpha\}}^\reg$.
\end{proposition}

\begin{proof}
The point $x$ is an isomorphism class of stable pseudo-holomorphic building of type $T'\to T$.

As in the proof of Proposition \ref{coveringI}, for $v\in V(T')$ with $\ast(v)=00$ or $\ast(v)=11$, either $C_v$ is stable or $\pi_\xi du_v$ is injective somewhere on $C_v$.  For $\ast(v)=01$, every irreducible component of $C_v$ is either stable or has a point where $du_v$ injective by Lemma \ref{injective}.

Now we can find a (compact) codimension two submanifold $\hat D_\alpha\subseteq\hat X$ such that $u\pitchfork\hat D_\alpha$ with intersections stabilizing every $C_v$ with $\ast(v)=01$.  Now we consider the remaining unstable $C_v$ ($\ast(v)=00$ or $\ast(v)=11$), and we choose codimension two submanifolds $D_\alpha^\pm\subseteq Y^\pm$ stabilizing these.  We then cutoff $\hat D_\alpha^\pm$ near infinity in $\hat X$ and add this to $\hat D_\alpha$.  Thus $u\pitchfork\hat D_\alpha$ with $r_\alpha$ intersections which stabilize $C$.

Now $(E_\alpha,\nu_\alpha)$ are constructed as in Proposition \ref{coveringI}.
\end{proof}

\begin{proposition}\label{coveringIII}
For every $x\in\Mbar_\III(T)$, there exist $\alpha_i\in A_\III(T_i)$ such that $x\in U_{\varnothing,\{\alpha_i\}_i}$ and $\psi_{\varnothing,\{\alpha_i\}_i}^{-1}(x)\subseteq\Mbar_\III(T)_{\{\alpha_i\}_i}^\reg$ (writing $T=\bigsqcup_iT_i$ with $T_i$ connected and non-empty).
\end{proposition}

\begin{proof}
Apply Proposition \ref{coveringII} to each subbuilding of type $T_i$ to get $\alpha_i$.
\end{proof}

\begin{proposition}\label{coveringIV}
For every $x\in\Mbar_\IV(T)$, there exist $\alpha_i\in A_\IV(T_i)$ such that $x\in U_{\varnothing,\{\alpha_i\}_i}$ and $\psi_{\varnothing,\{\alpha_i\}_i}^{-1}(x)\subseteq\Mbar_\IV(T)_{\{\alpha_i\}_i}^\reg$ (writing $T=\bigsqcup_iT_i$ with $T_i$ connected and non-empty).
\end{proposition}

\begin{proof}
Essentially the same as the proof of Propositions \ref{coveringII}--\ref{coveringIII}.
\end{proof}

\subsection{Local structure of thickened moduli spaces via local models \texorpdfstring{$G$}{G}}

We now state the precise sense in which the spaces $G_{T'//T}$ are local topological models for the regular loci in the thickened moduli spaces $\Mbar(T)$.  We also state the compatibility of this local topological structure with the natural maps on orientation lines discussed earlier.  These statements (which contain those from \S\ref{firstgluingstatements} as a special case) are in essence a gluing theorem, and their proofs are given in \S\ref{gluingsec}.  They imply that the openness and submersion axioms hold for the implicit atlases we have defined, and they give canonical isomorphisms $\oo_{\Mbar(T)}=\oo_T$.

The (Banach space) implicit function theorem implies that $\M(T)_I^\reg$ is a (smooth) manifold of dimension $\mu(T)-\#V_s(T)+\dim\s(T)+\dim E_I$ over the locus without nodes.

\begin{theorem}[Local structure of $\Mbar(T)_I^\reg$]\label{localgluing}
Let $I\subseteq J\subseteq\bar A(T)$.  Let $x_0\in\Mbar(T)_J$ be of type $T'\to T$ and satisfy $s_{J\setminus I}(x_0)=0$ and $\psi_{IJ}(x_0)\in\Mbar(T)_I^\reg$.  Then $\mu(T)-\#V_s(T')-2\#N(x_0)+\dim E_I\geq 0$ and there is a local homeomorphism
\begin{equation}\label{ultimategluingmap}
\bigl(G_{T'//T}\times E_{J\setminus I}\times\CC^{N(x_0)}\times\RR^{\mu(T)-\#V_s(T')-2\#N(x_0)+\dim E_I},(0,0,0)\bigr)\to\bigl(\Mbar(T)_J,x_0\bigr)
\end{equation}
whose image lands in $\Mbar(T)_J^\reg$ and which commutes with the maps from both sides to $\SSS_{T'//T}\times\overline{\s(T)}\times E_{J\setminus I}$ (as well as the stratifications by number of nodes).
\end{theorem}

\begin{proof}
See \S\S\ref{localgluingproofI}--\ref{localgluingproofIII}.
\end{proof}

The (Banach space) implicit function theorem moreover gives a canonical identification
\begin{equation}\label{Minteriororientations}
\oo_{\Mbar(T)_I^\reg}=\oo_T\otimes\oo_{E_I}.
\end{equation}
More precisely, this identification is made over the locus in $\M(T)_I^\reg$ without nodes, and has a unique continuous extension to all of $\Mbar(T)_I^\reg$ by virtue of the local topological description in Theorem \ref{localgluing}.  This identification is easily seen to be compatible with $\psi_{IJ}$ and with concatenations.  It is also compatible with morphisms $T'\to T$ in the following precise sense.

\begin{theorem}[Compatibility of the ``analytic'' and ``geometric'' maps on orientations]\label{localorientations}
The following diagram commutes:
\begin{equation}\label{gluingorientations}
\begin{tikzcd}
\oo_{\Mbar(T')^\reg_I}\otimes\oo_{G_{T'//T}}\otimes\oo_{s(T')}^\vee\ar[equals]{r}{\eqref{Minteriororientations}}\ar[equals]{d}{\eqref{ultimategluingmap}}&\oo_{T'}\otimes\oo_{E_I}\otimes\oo_{G_{T'//T}}\otimes\oo_{s(T')}^\vee\ar[equals]{d}{\eqref{algebraicorientationgluing}}\\
\oo_{\Mbar(T)^\reg_I}\ar[equals]{r}{\eqref{Minteriororientations}}&\oo_T\otimes\oo_{E_I},
\end{tikzcd}
\end{equation}
where the left vertical map is the ``geometric'' map induced by the local topological structure of $\Mbar(T)^\reg_I$ coming from \eqref{ultimategluingmap}, and the right vertical map is the ``analytic'' map \eqref{algebraicorientationgluing} defined earlier via the ``kernel gluing'' operation.
\end{theorem}

\begin{proof}
See \S\ref{localorientationsproof}.
\end{proof}

\section{Virtual fundamental cycles}

In this section, we prove Theorem \ref{main} as stated in the introduction.
Specifically, Theorem \ref{main} follows by combining Definition \ref{thetadefs} with Remark \ref{rmkmapofSgivesvmc}, Proposition \ref{vmctransverse}, and Proposition \ref{thetanonempty}.

\subsection{Review of the VFC package}

We begin with a review of the framework introduced in \cite{pardonimplicitatlas} for defining the virtual fundamental cycle of a space equipped with an implicit atlas.
We state here all the results from \cite{pardonimplicitatlas} which we will be appealing to in the rest of this section.
We use $\QQ$ coeffients throughout, and orientation lines are tacitly tensored up to $\QQ$.

\subsubsection{The basic formalism}

Let $X$ be a compact Hausdorff space equipped with an implicit atlas $A$ with boundary in the sense of Definition \ref{implicitatlasdefinitionbdry} (recall also Remark \ref{forgetstratificationrmk}), oriented with respect to $\oo$ and of virtual dimension $d$ (meaning these are the orientation line and virtual dimension of the top stratum).
Assume for the moment that $A$ is finite, and let
\begin{equation}
C_\bullet(E;A):=C_{\dim E_A+\bullet}(E_A,E_A\setminus 0;\oo_{E_A}^\vee)^{\Gamma_A}
\end{equation}
following \cite[Definition 4.2.4]{pardonimplicitatlas} (the superscript indicates taking $\Gamma_A$-invariants); note that there is a canonical isomorphism $H_\bullet(E;A)=\QQ$ (concentrated in degree zero).
Now for any $(X,A)$ as above, \cite[Definition 4.2.6]{pardonimplicitatlas} defines a cochain complex
\begin{equation}
C_\vir^\bullet(X\rel\partial;A)
\end{equation}
together with a canonical pushforward map
\begin{equation}\label{sastvir}
C_\vir^{d+\bullet}(X\rel\partial;A)\to C_{-\bullet}(E;A),
\end{equation}
and \cite[Theorem 4.3.4]{pardonimplicitatlas} provides a canonical isomorphism
\begin{equation}\label{Cviriscech}
H_\vir^\bullet(X\rel\partial;A)=\cH^\bullet(X,j_!\oo),
\end{equation}
where $j:X\setminus\partial X\hookrightarrow X$.
The map \eqref{sastvir} is the (chain level!) virtual fundamental cycle of $X$ (taking cohomology and combining with \eqref{Cviriscech} yields a map $\cH^d(X,j_!\oo)\to\QQ$ which is, by definition, integration over the virtual fundamental class of $X$).

\begin{remark}
As an aside, let us provide some very brief motivation for this formalism of virtual fundamental cycles (the reader may refer to \cite[\S\S 1--2]{pardonimplicitatlas} for details).
Suppose that for some $\alpha\in A$, the $\alpha$-thickened space $X_\alpha$ is everywhere regular (namely $X_\alpha=X_\alpha^\reg$), has no isotropy (meaning $\Gamma_\alpha=1$), and has footprint all of $X$ (that is $U_{\varnothing,\{\alpha\}}=X_\varnothing$).
In this situation, we have $X=s_\alpha^{-1}(0)$ for $s_\alpha:X_\alpha\to E_\alpha$, and we obviously want to define $[X]^\vir:=[X_\alpha]\cap s_\alpha^\ast\tau_{E_\alpha}$, where $\tau_{E_\alpha}$ denotes the Thom class of $E_\alpha$.
Another way of stating the equality $[X]^\vir=[X_\alpha]\cap s_\alpha^\ast\tau_{E_\alpha}$ is to say that evaluation on $[X]^\vir$ is given by the composition
\begin{equation}\label{basicvfc}
\cH^\bullet(X)=H_{\dim X_\alpha-\bullet}(X_\alpha,X_\alpha\setminus X)\xrightarrow{(s_\alpha)_\ast} H_{\dim X_\alpha-\bullet}(E_\alpha,E_\alpha\setminus 0)=\QQ[\dim E_\alpha-\dim X_\alpha]
\end{equation}
(ignoring orientation lines for the moment), where the first equality is a form of Poincar\'e duality.
In general, the complex $C_\vir^\bullet(X\rel\partial;A)$ is built out of relative chain groups of (roughly speaking) the regular loci in the thickened spaces $X_I^\reg$, and \eqref{sastvir}--\eqref{Cviriscech} are a direct generalization of \eqref{basicvfc}.
\end{remark}

For an inclusion of implicit atlases $A\subseteq A'$ on the same space $X$, there are canonical quasi-isomorphisms
\begin{align}
\label{Cvirinflate}C_\vir^\bullet(X\rel\partial;A)&\xrightarrow\sim C_\vir^\bullet(X\rel\partial;A'),\\
\label{CEinflate}C_\bullet(E;A)&\xrightarrow\sim C_\bullet(E;A'),
\end{align}
which compose as expected and are compatible with \eqref{sastvir} (on the chain level), see \cite[Definition 4.2.7]{pardonimplicitatlas}.\footnote{More precisely, the maps defined in \cite{pardonimplicitatlas} carry an extra factor of ${}\otimes C_\bullet(E;A'\setminus A)$ on the left.  Let $[E_\alpha]\in C_0(E;\alpha)$ be the fundamental cycle obtained by pulling back some $[\RR^n]\in C_n(\RR^n,\RR^n\setminus 0;\oo_{\RR^n}^\vee)$ (fixed once and for all) under the specified identification $E_\alpha=\RR^{n_\alpha}$ and averaging over $\Gamma_\alpha$ (alternatively, we could modify the definition of a thickening datum to include the data of a fundamental cycle $[E_\alpha]\in C_0(E;\alpha)$).  Now \eqref{Cvirinflate}--\eqref{CEinflate} are defined by pre-composing the maps in \cite{pardonimplicitatlas} with $\otimes\bigotimes_{\alpha\in A'\setminus A}[E_\alpha]$.}
These maps allow us to remove our assumption that $A$ is finite, namely by defining $C_\bullet(E;A)$ and $C_\vir^\bullet(X\rel\partial;A)$ for general $A$ by taking the direct limit over the collection of finite subsets of $A$.

\begin{lemma}[{\cite[Lemma 5.2.6]{pardonimplicitatlas}}]\label{transversevfc}
If $X=X^\reg$, then the map $\cH^d(X;j_!\oo)\to\QQ$ from \eqref{sastvir} is evaluation on the ordinary fundamental class $[X]\in H_d(X,\partial X;\oo^\vee)$.
\end{lemma}

\subsubsection{Homotopy sheaves}

Recall from \cite[Definition A.2.5]{pardonimplicitatlas} that a presheaf of complexes $\F^\bullet$ on the category of compact subsets of a locally compact Hausdorff space $X$ is called a \emph{homotopy sheaf} iff it satisfies the following three axioms:
\begin{enumerate}
\item[(hSh1)]$\F^\bullet(\varnothing)$ is acyclic.
\item[(hSh2)]$\bigl[\F^\bullet(K_1\cup K_2)\to\F^{\bullet-1}(K_1)\oplus\F^{\bullet-1}(K_2)\to\F^{\bullet-2}(K_1\cap K_2)\bigr]$ is acyclic for all $K_1,K_2\subseteq X$.
\item[(hSh3)]$\varinjlim_{\begin{smallmatrix}K\subseteq U\\U\textrm{ open}\end{smallmatrix}}\F^\bullet(\overline U)\to\F^\bullet(K)$ is a quasi-isomorphism for all $K\subseteq X$.
\end{enumerate}
A homotopy sheaf should be thought of as a complex of presheaves which calculates its own cohomology, in the sense that the natural map $H^\bullet\F^\bullet(K)\to\cH^\bullet(K,\F^\bullet)$ is an isomorphism \cite[Proposition A.4.14]{pardonimplicitatlas}.

Recall from \cite[Definition A.5.1]{pardonimplicitatlas} that a homotopy sheaf $\F^\bullet$ is called \emph{pure} iff it satisfies the following two axioms:{
\setlength\abovedisplayskip{3pt}
\setlength\belowdisplayskip{3pt}
\setlength\abovedisplayshortskip{3pt}
\setlength\belowdisplayshortskip{3pt}
\begin{align}
\label{purestalks}&\text{For all $p\in X$ and $i\ne 0$, we have $H^i\F^\bullet(\{p\})=0$.}\\
\label{pureweakvanishing}&\text{For all $p\in X$, there exists a neighborhood $U\subseteq X$ of $p$ and an integer $i_0$}\\
&\hphantom{\text{For all $p\in X$,}}\;\,\text{such that $H^i\F^\bullet(K)=0$ for all $K\subseteq U$ and $i\leq i_0$.}\nonumber
\end{align}
}From a conceptual standpoint, it is condition \eqref{purestalks} which is most significant, and it should be thought of as saying that $\F^\bullet$ is (stalkwise) quasi-isomorphic to a sheaf $H^0\F^\bullet$ \cite[Lemma A.5.3]{pardonimplicitatlas}.
In particular, for any pure homotopy sheaf $\F^\bullet$, there is a canonical isomorphism
\begin{equation}\label{htpysheaffundiso}
H^\bullet\F^\bullet(X)=\cH^\bullet(X;H^0\F^\bullet)
\end{equation}
by \cite[Proposition A.5.4]{pardonimplicitatlas}.

The complex $C_\vir^\bullet(X\rel\partial;A)$ extends naturally to a presheaf of complexes
\begin{equation}\label{virpresheaf}
K\mapsto C_\vir^\bullet(X\rel\partial;A)_K
\end{equation}
on compact subsets $K\subseteq X$, with complex of global sections $C_\vir^\bullet(X\rel\partial;A)=C_\vir^\bullet(X\rel\partial;A)_X$.\footnote{In \cite{pardonimplicitatlas}, the relatively inferior notation $C_\vir^\bullet(K\rel\partial;A)$ is used in place of $C_\vir^\bullet(X\rel\partial;A)_K$.}
By \cite[Proposition 4.3.3]{pardonimplicitatlas}, this presheaf is a pure homotopy sheaf.
In more detail, axiom (hSh1) is trivial, axiom (hSh2) follows from Mayer--Vietoris for singular chains, axiom (hSh3) follows from compactness of simplices, axiom \eqref{purestalks} is a local homology calculation (essentially $H_{\dim M-\bullet}(M,M\setminus p)=\oo_{M,p}$), and axiom \eqref{pureweakvanishing} follows from the finite-dimensionality of the thickened moduli spaces comprising the implicit atlas $A$.

The sheaf $K\mapsto H_\vir^0(X\rel\partial;A)_K$ associated to \eqref{virpresheaf} is canonically isomorphic to $j_!\oo$ on $X$.
Furthermore, this isomorphism has a natural local description in terms of Poincar\'e duality \cite[Lemma 4.3.2]{pardonimplicitatlas}.
The isomorphism \eqref{Cviriscech} is thus a special case of \eqref{htpysheaffundiso}.

\subsubsection{Stratifications}

Let $X$ be a compact Hausdorff space equipped with an implicit atlas $A$ with oriented cell-like stratification in the sense of Definition \ref{implicitatlasdefinition}.
The implicit atlas on $X$ induces implicit atlases on each closed stratum $X_{/\ttt}$ for $\ttt\in\TTT$, simply by defining $(X_{/\ttt})_I:=(X_I)_{/\ttt}$.
To understand how the virtual fundamental cycles of each of the strata $X_{/\ttt}$ fit together, \cite[Definition 6.2.2]{pardonimplicitatlas} defines
\begin{equation}\label{stratifiedcomplex}
C_\vir^\bullet(X,\TTT;A):=\bigoplus_{\ttt\in\TTT}\left[C_\vir^\bullet(X_{/\ttt}\rel\partial;A)\otimes\oo_\ttt^\vee\otimes\oo_{\ttt^\ttop}\right]^{\Aut(\ttt)},
\end{equation}
(the superscript indicates taking $\Aut(\ttt)$-invariants)
equipped with the differential given by (the internal differential plus) pushing forward along all codimension one maps $\ttt\to\ttt'$ in $\TTT$ (and contracting on the left with a chosen orientation of $\RR$).
There is now a canonical isomorphism
\begin{equation}\label{stratcechviriso}
H_\vir^\bullet(X;A)=\cH^\bullet(X;\oo_{\ttt^\ttop}).
\end{equation}
This is shown in \cite[Propositions 6.2.3 and 4.3.3]{pardonimplicitatlas} for stratifications by posets, and the reasoning there applies without modification to the present setting of stratifications by categories in the sense of Definition \ref{stratificationcategorydef}.

To be more precise, the proof of the isomorphism \eqref{stratcechviriso} proceeds by considering the presheaf of complexes on $X$ given by
\begin{equation}\label{stratifiedcomplexpresheaf}
K\mapsto C_\vir^\bullet(X,\TTT;A)_K:=\bigoplus_{\ttt\in\TTT}\left[C_\vir^\bullet(X_{/\ttt}\rel\partial;A)_{K_{/\ttt}}\otimes\oo_\ttt^\vee\otimes\oo_{\ttt^\ttop}\right]^{\Aut(\ttt)},
\end{equation}
with differential as in \eqref{stratifiedcomplex}.
It is easy to see that $K\mapsto C_\vir^\bullet(X\rel\partial;A)_K$ being a homotopy sheaf implies the same for $K\mapsto C_\vir^\bullet(X,\TTT;A)_K$ \cite[Lemma A.2.11]{pardonimplicitatlas}, and it is straightforward to argue that purity also passes.
The sheaf $K\mapsto H^0_\vir(X,\TTT;A)_K$ now has a filtration whose associated graded is the direct sum
\begin{equation}
K\mapsto\bigoplus_{\ttt\in\TTT}\left[H_\vir^0(X_{/\ttt}\rel\partial;A)_{K_{/\ttt}}\otimes\oo_\ttt^\vee\otimes\oo_{\ttt^\ttop}\right]^{\Aut(\ttt)}
\end{equation}
There is thus a natural stalkwise isomorphism of $K\mapsto H_\vir^0(X,\TTT;A)_K$ with the constant sheaf $\oo_{\ttt^\ttop}$ on $X$, and a local construction \cite[Example 6.2.1]{pardonimplicitatlas} along with \cite[Lemma 4.3.2]{pardonimplicitatlas} shows that it comes from an isomorphism of sheaves.

\subsubsection{Products}

For spaces $X$ and $Y$ with implicit atlases $A$ and $B$, by \cite[Definition 6.3.2]{pardonimplicitatlas} there is a canonical product map
\begin{equation}\label{Cvirprod}
C_\vir^\bullet(X\rel\partial;A)\otimes C_\vir^\bullet(Y\rel\partial;B)\to C_\vir^\bullet(X\times Y\rel\partial;A\sqcup B),
\end{equation}
compatible with \eqref{sastvir} and \eqref{Cvirinflate}--\eqref{CEinflate}.
Under the isomorphisms \eqref{Cviriscech}, this map is simply the usual K\"unneth product map.
This is not mentioned explicitly in \cite{pardonimplicitatlas}, rather it follows from Lemma \ref{kunnethforpure} below and recalling from \cite[Definition 6.3.2]{pardonimplicitatlas} that the map \eqref{Cvirprod} extends naturally to a family of maps
\begin{equation}\label{Cvirprodsheafy}
C_\vir^\bullet(X\rel\partial;A)_K\otimes C_\vir^\bullet(Y\rel\partial;B)_{K'}\to C_\vir^\bullet(X\times Y\rel\partial;A\sqcup B)_{K\times K'}
\end{equation}
compatible with restriction.

\begin{lemma}\label{kunnethforpure}
Let $\F^\bullet$, $\G^\bullet$, $\HH^\bullet$ be pure homotopy sheaves on spaces $X$, $Y$, $X\times Y$, respectively.  Let $\F^\bullet(K)\otimes\G^\bullet(K')\to\HH^\bullet(K\times K')$ be a collection of maps compatible with restriction.  Then the following diagram commutes:
\begin{equation}
\begin{tikzcd}
H^\bullet\F^\bullet(X)\otimes H^\bullet\G^\bullet(Y)\ar{r}\ar{d}{\eqref{htpysheaffundiso}}&H^\bullet\HH^\bullet(X\times Y)\ar{d}{\eqref{htpysheaffundiso}}\\
\cH^\bullet(X;H^0\F^\bullet)\otimes\cH^\bullet(Y;H^0\G^\bullet)\ar{r}&\cH^\bullet(X\times Y;H^0\HH^\bullet),
\end{tikzcd}
\end{equation}
where the bottom horizontal arrow is the usual K\"unneth cup product map.
\end{lemma}

\begin{proof}
We consider the following commutative diagram:
\begin{equation}
\hspace{-1in}
\begin{tikzcd}[column sep = tiny]
H^\bullet\F^\bullet(X)\otimes H^\bullet\G^\bullet(Y)\ar{r}\ar{d}{\sim}&H^\bullet\left[\F^\bullet(X)\otimes\G^\bullet(Y)\right]\ar{r}\ar{d}&H^\bullet\HH^\bullet(X\times Y)\ar{d}{\sim}\\
\cH^\bullet(X;\F^\bullet)\otimes\cH^\bullet(Y;\G^\bullet)\ar{r}\ar[leftarrow]{d}{\sim}&H^\bullet\!\left[\cC^\bullet(X;\F^\bullet)\otimes\cC^\bullet(Y;\G^\bullet)\right]\ar{r}\ar[leftarrow]{d}&\cH^\bullet(X\times Y,\HH^\bullet)\ar[leftarrow]{d}{\sim}\\
\cH^\bullet(X;\tau_{\leq 0}\F^\bullet)\otimes\cH^\bullet(Y;\tau_{\leq 0}\G^\bullet)\ar{r}\ar{d}{\sim}&H^\bullet\!\left[\cC^\bullet(X;\tau_{\leq 0}\F^\bullet)\otimes\cC^\bullet(Y;\tau_{\leq 0}\G^\bullet)\right]\ar{r}\ar{d}&\cH^\bullet(X\times Y,\tau_{\leq 0}\HH^\bullet)\ar{d}{\sim}\\
\cH^\bullet(X;H^0\F^\bullet)\otimes\cH^\bullet(Y;H^0\G^\bullet)\ar{r}&H^\bullet\!\left[\cC^\bullet(X;H^0\F^\bullet)\otimes\cC^\bullet(Y;H^0\G^\bullet)\right]\ar{r}&\cH^\bullet(X\times Y,H^0\HH^\bullet).
\end{tikzcd}
\hspace{-1in}
\end{equation}
Note that there is a natural map $\tau_{\leq 0}\F^\bullet\otimes\tau_{\leq 0}\G^\bullet\to\tau_{\leq 0}(\F^\bullet\otimes\G^\bullet)\to\tau_{\leq 0}\HH^\bullet$, but no such natural map with $\tau_{\geq 0}$ in place of $\tau_{\leq 0}$.
It is therefore important that the vertical columns above are of the form $\F^\bullet\leftarrow\tau_{\leq 0}\F^\bullet\to H^0\F^\bullet$ rather than $\F^\bullet\to\tau_{\geq 0}\F^\bullet\leftarrow H^0\F^\bullet$.

Now the far right and far left vertical maps above are all isomorphisms and define \eqref{htpysheaffundiso} (see the proof of \cite[Proposition A.5.4]{pardonimplicitatlas}).
The outer square is exactly the desired diagram, so we are done.
\end{proof}

\subsection{Sketch of the construction}\label{vfcsketch}

We now sketch the construction of virtual fundamental cycles on the moduli spaces $\Mbar(T)$.  The remainder of this section is then devoted to turning this sketch into an actual proof.

The virtual fundamental cycle of a space with an implicit atlas is represented by the pushforward map \eqref{sastvir}.
Specializing this to the case at hand, the virtual fundamental class $[\Mbar(T)]^\vir$ is represented by a canonical map
\begin{equation}\label{sketchvfcpush}
C_\vir^{\vdim(T)+\bullet}(\Mbar(T)\rel\partial;\bar A(T))\to C_{-\bullet}(E;\bar A(T)).
\end{equation}
The domain of this map is quasi-isomorphic to the complex of \v Cech cochains $\cC^{\vdim(T)+\bullet}(\Mbar(T)\rel\partial;\oo_T)$ by \eqref{Cviriscech}, and the target is quasi-isomorphic to $\QQ$.  Thus, up to quasi-isomorphism, \eqref{sketchvfcpush} can be regarded as a map
\begin{equation}
\cC^{\vdim(T)+\bullet}(\Mbar(T)\rel\partial;\oo_T)\to\QQ,
\end{equation}
which is nothing other than a cycle
\begin{equation}\label{sketchvfc}
[\Mbar(T)]^\vir\in\overline C_{\vdim(T)}(\Mbar(T)\rel\partial;\oo_T^\vee).
\end{equation}
We note that the dual of \v Cech cochains $\cC^\bullet$ is Steenrod chains $\overline C_\bullet$ (see \cite[\S A.9]{pardonimplicitatlas} and the references therein), although this level of precision is not really relevant for the present sketch.

Now the functoriality of \eqref{sketchvfcpush} with respect to morphisms and concatenations in $\SSS$ implies that the virtual fundamental cycles \eqref{sketchvfc} satisfy
\begin{align}
\label{vfcidI}\partial[\Mbar(T)]^\vir&=\sum_{\codim(T'/T)=1}\frac 1{\left|\Aut(T'/T)\right|}[\Mbar(T')]^\vir,\\
\label{vfcidII}[\Mbar(\#_iT_i)]^\vir&=\frac 1{\left|\Aut(\{T_i\}_i/\#T_i)\right|}\prod_i[\Mbar(T_i)]^\vir,
\end{align}
which clearly imply the desired master equations.

To turn this sketch into an actual construction, we must carry out the above arguments on the chain level, with all the necessary chain level functoriality with respect to morphisms and concatenations in $\SSS$.  This is not completely trivial, since many of the chain maps we are given go in the wrong direction (basically, we need to ``invert quasi-isomorphisms'').  The rest of this section is devoted to performing the necessary algebraic manipulations.

\subsection{\texorpdfstring{$\SSS$}{S}-modules}

We begin by introducing the notion of an $\SSS$-module.  This notion formalizes the way in which the moduli spaces $\Mbar(T)$ fit together under morphisms and concatenations in $\SSS$ (namely \eqref{Mbarfunct}--\eqref{Mbarprod}) and will be used below to efficiently encode the identities \eqref{vfcidI}--\eqref{vfcidII} satisfied by the virtual fundamental cycles.  In fact, many (or even most) of the objects introduced and studied earlier in this paper are $\SSS$-modules, as we explain in the examples which follow the definition.  The notion of an $\SSS$-module plays a key organizational role in what follows.

\begin{definition}
An object $T\in\SSS$ will be called \emph{effective} iff $\Mbar(T)\ne\varnothing$.
\end{definition}

Note that (1) for any morphism $T\to T'$, if $T$ is effective, then so is $T'$, and (2) for any concatenation $\{T_i\}_i$, every $T_i$ is effective iff $\#_iT_i$ is effective.
\emph{For the remainder of this section, we will use $\SSS$ to denote the full subcategory spanned by effective objects.}

\begin{definition}\label{SImoddef}
An \emph{$\SSS_\I$-module} $X_\I$ valued in a symmetric monoidal category $\C^\otimes$ consists of the following data:
\begin{enumerate}
\item A functor $X_\I:\SSS_\I\to\C$.
\item For every concatenation of $\{T_i\}_i$ in $\SSS_\I$, a morphism
\begin{equation}\label{TImodprod}
\bigotimes_iX_\I(T_i)\to X_\I(\#_iT_i),
\end{equation}
such that the following diagrams commute:
\begin{equation}\label{TImodidentities}
\begin{tikzcd}
\displaystyle\bigotimes_iX_\I(T_i)\ar{r}{\eqref{TImodprod}}\ar{d}&X_\I(\#_iT_i)\ar{d}\\
\displaystyle\bigotimes_iX_\I(T_i')\ar{r}{\eqref{TImodprod}}&X_\I(\#_iT_i'),
\end{tikzcd}
\begin{tikzcd}[column sep = tiny]
{}&\displaystyle\bigotimes_iX_\I(\#_jT_{ij})\ar{dr}{\eqref{TImodprod}}\\
\displaystyle\bigotimes_{i,j}X_\I(T_{ij})\ar{ur}{\bigotimes_i\eqref{TImodprod}}\ar{rr}{\eqref{TImodprod}}&&X_\I(\#_{ij}T_{ij}),
\end{tikzcd}
\end{equation}
for any morphism of concatenations $\{T_i\}_i\to\{T_i'\}_i$ and any composition of concatenations, respectively.
\end{enumerate}
A morphism of $\SSS_\I$-modules is a natural transformation of functors compatible with \eqref{TImodprod}.

An \emph{$\SSS_\I^\op$-module} is defined similarly, except that $X_\I:\SSS_\I^\op\to\C$ and the vertical arrows in the leftmost diagram in \eqref{TImodidentities} are reversed.
Note that an $\SSS_\I^\op$-module valued in $\C$ is \emph{not} the same thing as an $\SSS_\I$-module valued in $\C^\op$ (this is analogous to the difference between a lax monoidal functor and an oplax monoidal functor).
\end{definition}

\begin{example}\label{SIexI}
The functor $\Mbar_\I$ is an $\SSS_\I$-module (valued in the category of compact Hausdorff spaces, with the product symmetric monoidal structure).
\end{example}

\begin{example}
The functor $\oo^\circ$ is an $\SSS_\I$-module (valued in the category of orientation lines and isomorphisms, with the super tensor product symmetric monoidal structure).
\end{example}

\begin{example}
The functor $\bar A_\I$ is an $\SSS_\I^\op$-module (valued in the category of sets, with the disjoint union symmetric monoidal structure).
\end{example}

\begin{example}\label{SIexIV}
The functor $T\mapsto(\SSS_\I)_{/T}$ is an $\SSS_\I$-module (valued in the category of categories, with the product symmetric monoidal structure).
\end{example}

\begin{definition}\label{SIImoddef}
An \emph{$\SSS_\II$-module} $X_\II$ valued in $\C^\otimes$ consists of the following data:
\begin{enumerate}
\item An $\SSS_\I^+$-module $X_\I^+$ valued in $\C^\otimes$.
\item An $\SSS_\I^-$-module $X_\I^-$ valued in $\C^\otimes$.
\item A functor $X_\II:\SSS_\II\to\C$.
\item For every concatenation $\{T_i\}_i$ in $\SSS_\II$, a morphism
\begin{equation}\label{TIImodprod}
\bigotimes_{T_i\in\SSS_\I^+}X_\I^+(T_i)\otimes
\bigotimes_{T_i\in\SSS_\I^-}X_\I^-(T_i)\otimes
\bigotimes_{T_i\in\SSS_\II}X_\II(T_i)
\to X_\II(\#_iT_i),
\end{equation}
such that the following diagrams commute:
\begin{equation}
\hspace{-1in}
\begin{tikzcd}
\displaystyle\bigotimes_iX_{\I/\II}(T_i)\ar{r}\ar{d}&X_{\I/\II}(\#_iT_i)\ar{d}\\
\displaystyle\bigotimes_iX_{\I/\II}(T_i')\ar{r}&X_{\I/\II}(\#_iT_i'),
\end{tikzcd}
\begin{tikzcd}[column sep = tiny]
{}&\displaystyle\bigotimes_iX_{\I/\II}(\#_jT_{ij})\ar{dr}\\
\displaystyle\bigotimes_{i,j}X_{\I/\II}(T_{ij})\ar{ur}\ar{rr}&&X_{\I/\II}(\#_{ij}T_{ij}),
\end{tikzcd}
\hspace{-1in}
\end{equation}
for any morphism or composition of concatenations, respectively.
\end{enumerate}
A morphism of $\SSS_\II$-modules consists of natural transformations of functors compatible with \eqref{TImodprod}, \eqref{TIImodprod}.
\end{definition}

It would perhaps be more proper to speak of ``$(\SSS_\I^\pm,\SSS_\II)$-modules $(X_\I^\pm,X_\II)$'', though such notation rapidly becomes unwieldy.  It will usually be clear from context what $X_\I^\pm$ are once we have specified $X_\II$, though at times we will specify the pair $(X_\I^\pm,X_\II)$ for sake of clarity.

Examples \ref{SIexI}--\ref{SIexIV} all generalize: $(\Mbar{\vphantom{\M}}_\I^\pm,\Mbar_\II)$ is an $(\SSS_\I^\pm,\SSS_\II)$-module, etc.

\begin{definition}\label{SIIImoddef}
An \emph{$\SSS_\III$-module} $X_\III$ valued in $\C^\otimes$ consists of the following data:
\begin{enumerate}
\item An $\SSS_\I^+$-module $X_\I^+$ valued in $\C^\otimes$.
\item An $\SSS_\I^-$-module $X_\I^-$ valued in $\C^\otimes$.
\item An $(\SSS_\I^\pm,\SSS_\II^{t=0})$-module $(X_\I^\pm,X_\II^{t=0})$ valued in $\C^\otimes$.
\item An $(\SSS_\I^\pm,\SSS_\II^{t=1})$-module $(X_\I^\pm,X_\II^{t=1})$ valued in $\C^\otimes$.
\item A functor $X_\III:S_\III\to\C$.
\item For every concatenation $\{T_i\}_i$ in $\SSS_\III$, a morphism
\begin{equation}\label{TIIImodprod}
\bigotimes_iX_{\I/\II/\III}(T_i)\to X_\III(\#T_i),
\end{equation}
satisfying the natural compatibility conditions, as in Definition \ref{SIImoddef}.
\end{enumerate}
A morphism of $\SSS_\III$-modules consists of natural transformations of functors compatible with \eqref{TImodprod}, \eqref{TIImodprod}, \eqref{TIIImodprod}.
\end{definition}

\begin{definition}\label{SIVmoddef}
An \emph{$\SSS_\IV$-module} $X_\IV$ valued in $\C^\otimes$ consists of the following data:
\begin{enumerate}
\item$\SSS_\I^i$-modules $X_\I^i$ valued in $\C^\otimes$ for $0\leq i\leq 2$.
\item$(\SSS_\I^{i,j},\SSS_\II^{ij})$-modules $(X_\I^{i,j},X_\II^{ij})$ valued in $\C^\otimes$ for $0\leq i<j\leq 2$.
\item A functor $X_\IV:\SSS_\IV\to\C$.
\item For every concatenation $\{T_i\}_i$ in $\SSS_\IV$, a morphism
\begin{equation}\label{TIVmodprod}
\bigotimes_iX_{\I/\II/\IV}(T_i)\to X_\IV(\#_iT_i),
\end{equation}
satisfying the natural compatibility conditions, as in Definition \ref{SIImoddef}.
\end{enumerate}
A morphism of $\SSS_\IV$-modules consists of natural transformations of functors compatible with \eqref{TImodprod}, \eqref{TIImodprod}, \eqref{TIVmodprod}.
\end{definition}

\subsection{Sketch of the construction (revisited)}\label{vfcoutline}

We now revisit the sketch from \S\ref{vfcsketch} using the language of $\SSS$-modules.

Recall that the virtual fundamental cycles of the moduli spaces $\Mbar(T)$ come packaged via the maps \eqref{sketchvfcpush}.  Now the key coherence properties of these cycles \eqref{vfcidI}--\eqref{vfcidII} shall be encoded in the fact that \eqref{sketchvfcpush} is a map of $\SSS$-modules.  Thus our first task is to define the $\SSS$-modules $C_\vir^{\bullet+\vdim}(\Mbar\rel\partial)$ and $C_\bullet(E)$ along with the pushforward map
\begin{equation}\label{revisitedsast}
C_\vir^{\bullet+\vdim}(\Mbar\rel\partial)\to C_{-\bullet}(E).
\end{equation}
We then must relate $C_\vir^{\bullet+\vdim}(\Mbar\rel\partial)$ to $\cC^{\bullet+\vdim}(\Mbar\rel\partial;\oo)$ (with an appropriate $\SSS$-module structure), and we must relate $C_\bullet(E)$ to $\QQ$ (with the trivial $\SSS$-module structure).  If all goes well, this will give virtual fundamental cycles \eqref{sketchvfc} satisfying \eqref{vfcidI}--\eqref{vfcidII}.

In reality, it is somewhat cumbersome to make sense out of $\cC^{\bullet+\vdim}(\Mbar\rel\partial;\oo)$ as an $\SSS$-module, so we will take a shortcut by taking advantage of the fact that we are only interested in integrating the constant function $1$ (and not more general cohomology classes) over the virtual fundamental cycles.  We will define an $\SSS$-module $\QQ[\SSS]$, which can be thought of as ``the subcomplex of $\cC^{\bullet+\vdim}(\Mbar\rel\partial;\oo)$ generated by the Poincar\'e duals of the closed strata'', and we will construct a corresponding map of $\SSS$-modules
\begin{equation}
\QQ[\SSS]\to C_\vir^{\bullet+\vdim}(\Mbar\rel\partial).
\end{equation}
Now combining this map with \eqref{revisitedsast} and our understanding of $C_\bullet(E)\cong\QQ$, we obtain a map of $\SSS$-modules
\begin{equation}\label{sssast}
\QQ[\SSS]\to\QQ.
\end{equation}
Morally speaking, this map is the intersection pairing between the closed strata of $\Mbar(T)$ and its virtual fundamental cycle.  From any map \eqref{sssast}, it is straightforward to read off virtual moduli counts satisfying the master equations.

\subsection{\texorpdfstring{$\SSS$}{S}-modules \texorpdfstring{$\QQ$}{Q} and \texorpdfstring{$\QQ[\SSS]$}{Q[S]}}

We now introduce the two basic $\SSS$-modules $\QQ$ and $\QQ[\SSS]$.  We also make the elementary but crucial observation that a map of $\SSS$-modules $\QQ[\SSS]\to\QQ$ is nothing other than a collection of virtual moduli counts satisfying the relevant master equations.

\begin{definition}[$\SSS$-module $\QQ$]
Denote by $\QQ$ the $\SSS$-module defined by $\QQ(T)=\QQ$ for all $T\in\SSS$, where the pushforward maps are the identity and the concatenation maps are multiplication.
\end{definition}

Before defining $\QQ[\SSS]$, we motivate it as follows.  It is not hard to check that $\cC^\bullet(\Mbar(T)\rel\partial;\oo_T)$ is quasi-isomorphic to the total complex
\begin{multline}\label{bigcechcomplex}
\biggl[\cC^\bullet(\Mbar(T);\oo_T)\longrightarrow
\!\!\!\bigoplus_{\codim(T'/T)=1}\!\!\!\cC^{\bullet-1}(\Mbar(T')/\!\Aut(T'/T);\oo_{T'})\\
\longrightarrow\!\!\!\bigoplus_{\codim(T''/T)=2}\!\!\!\cC^{\bullet-2}(\Mbar(T'')/\!\Aut(T''/T);\oo_{T''})\longrightarrow\cdots\biggr].
\end{multline}
This chain model is convenient because it is manifestly an $\SSS$-module: a morphism $T'\to T$ clearly induces a map of complexes \eqref{bigcechcomplex} of degree $\codim(T'/T)$, while it is not so clear how to canonically define a corresponding map $\cC^\bullet(\Mbar(T')\rel\partial;\oo_{T'})\to\cC^{\bullet+\codim(T'/T)}(\Mbar(T)\rel\partial;\oo_T)$.  The complex $\QQ[\SSS](T)$ which we now define should be thought of as the subcomplex of \eqref{bigcechcomplex} spanned by the images of the constant sections
\begin{equation}
\oo_{T'}\to\cC^0(\Mbar(T');\oo_{T'})\to\cC^0(\Mbar(T')/\!\Aut(T'/T);\oo_{T'})\subseteq\cC^{\codim(T'/T)}(\Mbar(T)\rel\partial;\oo_T)
\end{equation}
for $T'\to T$.

\begin{definition}[{$\SSS$-module $\QQ[\SSS]$}]
For $T\in\SSS$, we define
\begin{equation}
\QQ[\SSS](T):=\QQ[\SSS_{/T}]=\bigoplus_{T'\to T}\oo_{T'}[\vdim(T')]
\end{equation}
(note that this definition makes sense since automorphisms of $T'$ over $T$ act trivially on $\oo_{T'}$; though the action is trivial, it may be conceptually helpful to think of $\oo_{T'}$ instead as its $\Aut(T'/T)$-coinvariants).
We equip $\QQ[\SSS](T)$ with the differential given by the sum over all codimension one maps $T''\to T'$ (maps in $\SSS_{/T}$) of
the boundary map $\oo_{T'}\to\oo_{T''}$ induced by \eqref{algebraicorientationgluing} and pairing on the left with a chosen orientation of $\RR$.

A morphism $T'\to T$ induces a map
\begin{equation}
\QQ[\SSS](T')\to\QQ[\SSS](T)
\end{equation}
via pushforward under the functor $\SSS_{/T'}\to\SSS_{/T}$.

A concatenation $\{T_i\}_i$ induces a map
\begin{equation}
\bigotimes_i\QQ[\SSS](T_i)\to\QQ[\SSS](\#_iT_i)
\end{equation}
by consideration of the isomorphism $\SSS_{/\#_iT_i}=\prod_i\SSS_{/T_i}$, covered by the tautological identifications $\oo_{\#_iT_i'}=\bigotimes_i\oo_{T_i'}$ multiplied by the factor $\left|\Aut(\{T_i'\}_i/\#_iT_i')\right|=\left|\Aut(\{T_i\}_i/\#_iT_i)\right|$ (this factor is explained by it being the ``degree'' of the map $\prod_i\Mbar(T_i')\to\Mbar(\#_iT_i')$).

Thus $\QQ[\SSS]$ is an $\SSS$-module (meaning $\QQ[\SSS_\I]$ is an $\SSS_\I$-module, $(\QQ[\SSS_\I^\pm],\QQ[\SSS_\II])$ is an $(\SSS_\I^\pm,\SSS_\II)$-module, etc.).
\end{definition}

\begin{remark}\label{rmkmapofSgivesvmc}
It is easy to check that there is a natural bijection between morphisms of $\SSS$-modules $\QQ[\SSS]\to\QQ$ and collections of virtual moduli counts $\#\Mbar(T)^\vir\in\left(\oo_T^\vee\right)^{\Aut(T)}$ for $\vdim(T)=0$ satisfying
\begin{align}
0&=\sum_{\codim(T'/T)=1}\frac 1{\left|\Aut(T'/T)\right|}\#\Mbar(T')^\vir\\
\#\Mbar(\#_iT_i)^\vir&=\frac 1{\left|\Aut(\{T_i\}_i/\#T_i)\right|}\prod_i\#\Mbar(T_i)^\vir
\end{align}
(where $\Mbar(T)$ is interpreted as zero if $\vdim(T)\ne 0$).  Note that $\Aut(T)$-invariance forces $\#\Mbar(T)^\vir$ to vanish if any of the input/output edges of $T$ is labeled with a bad Reeb orbit.
\end{remark}

\subsection{\texorpdfstring{$\SSS$}{S}-modules \texorpdfstring{$C_\vir^{\bullet+\vdim}(\Mbar\rel\partial)$}{C\_vir\textasciicircum\{*+vdim\}(Mbar rel d)} and \texorpdfstring{$C_\bullet(E)$}{C\_*(E)}}\label{Cvirhocolimsec}

We now define the two $\SSS$-modules $C_\vir^{\bullet+\vdim}(\Mbar\rel\partial)$ and $C_\bullet(E)$.

Defining $C_\vir^{\bullet+\vdim}(\Mbar\rel\partial)$ as an $\SSS$-module is nontrivial for the following reason.  We would like to associate to $T\in\SSS$ the complex $C_\vir^{\bullet+\vdim(T)}(\Mbar(T)\rel\partial;\bar A(T))$.  However, a map $T'\to T$ does not induce a map on such complexes as desired, rather only a diagram
\begin{equation}\label{CvirnotSmoddiagram}
\begin{tikzcd}
C_\vir^{\bullet+\vdim(T')}(\Mbar(T')\rel\partial;\bar A(T))\ar{d}{\sim}\ar{rd}\\
C_\vir^{\bullet+\vdim(T')}(\Mbar(T')\rel\partial;\bar A(T'))&C_\vir^{\bullet+\vdim(T)}(\Mbar(T)\rel\partial;\bar A(T)).
\end{tikzcd}
\end{equation}
Fortunately, this diagram also suggests a solution.  Namely, we instead associate to $T\in\SSS$ a suitable homotopy colimit of $C_\vir^{\vdim(T')+\bullet}(\Mbar(T')\rel\partial,\bar A(T'))$ over $T'\in\SSS_{/T}$.

\begin{definition}[Homotopy diagram]\label{hodiagram}
Let $\TTT$ be a finite category (meaning $\#\left|\TTT\right|<\infty$ and $\#\Hom(\ttt_1,\ttt_2)<\infty$ for $\ttt_1,\ttt_2\in\TTT$).  A \emph{homotopy diagram} over $\TTT$ shall mean the following:
\begin{enumerate}
\item Let $\Delta^p$ denote the \emph{simplex category} (so that a functor $\Delta^p\to\TTT$ is a chain of morphisms $\ttt_0\to\cdots\to \ttt_p$).  For every functor $\sigma:\Delta^p\to\TTT$, we specify a complex $A^\bullet(\sigma)$, and for every map of simplices $r:\Delta^q\to\Delta^p$, we specify a map $A^\bullet(\sigma)\to A^\bullet(\sigma\circ r)$.  In other words, $A^\bullet$ is a functor from the category whose objects are morphisms $\Delta^p\to\TTT$ and whose morphisms $(\Delta^p\to\TTT)\to(\Delta^q\to\TTT)$ are factorizations $\Delta^q\to\Delta^p\to\TTT$.
\item\label{hodiagramqiso}We require that if $r:\Delta^q\to\Delta^p$ satisfies $r(0)=0$, then the induced map $A^\bullet(\sigma)\to A^\bullet(\sigma\circ r)$ is an isomorphism.
\end{enumerate}
For example, an ordinary diagram $B^\bullet$ over $\TTT$ gives rise to a homotopy diagram by setting $A^\bullet(\sigma):=B^\bullet(\sigma(0))$.
\end{definition}

\begin{definition}[Homotopy colimit]
Let $\TTT$ be a finite category, and let $A^\bullet$ be a homotopy diagram over $\TTT$.  We define
\begin{equation}\label{hocolimdefeq}
\hocolim_\TTT A^\bullet:=\bigoplus_{p\geq 0}\bigoplus_{\ttt_0\to\cdots\to\ttt_p}A^\bullet(\ttt_0\to\cdots\to\ttt_p)_{\Aut(\ttt_0\to\cdots\to\ttt_p)}\otimes\oo_{\Delta^p},
\end{equation}
equipped with the differential which arises upon regarding the right hand side as chains on the nerve of $\TTT$ (the subscript indicates taking coinvariants with respect to the natural action of $\Aut(\ttt_0\to\cdots\to\ttt_p)$).
\end{definition}

Given a functor $f:\TTT\to\TTT'$ and a homotopy diagram $A^\bullet$ over $\TTT'$, there is a natural pullback diagram $f^\ast A^\bullet$ over $\TTT$ and a natural map
\begin{equation}\label{hocolimpush}
\hocolim_\TTT f^\ast A^\bullet\to\hocolim_{\TTT'}A^\bullet.
\end{equation}
Given a finite collection $A^\bullet_i$ of homotopy diagrams over $\TTT_i$, there is an Eilenberg--Zilber quasi-isomorphism
\begin{equation}\label{hocolimEZ}
\bigotimes_i\hocolim_{\TTT_i}A^\bullet_i\xrightarrow\sim\hocolim_{\prod_i\TTT_i}\bigotimes_iA^\bullet_i,
\end{equation}
corresponding to the standard simplicial subdivision of $\Delta^p\times\Delta^q$ into $\binom{p+q}p$ copies of $\Delta^{p+q}$.

\begin{lemma}\label{finalhocolim}
If $\TTT$ has a final object $\ttt^\ttop$, then the natural map
\begin{equation}
A^\bullet(\ttt^\ttop)\xrightarrow\sim\hocolim_\TTT A^\bullet
\end{equation}
is a quasi-isomorphism.
\end{lemma}

\begin{proof}
Filter the homotopy colimit by the number of $\ttt_i$ which are not isomorphic to $\ttt^\ttop$.
\end{proof}

\begin{definition}[$\SSS$-module $C_\vir^{\bullet+\vdim}(\Mbar\rel\partial)$]\label{CvirMIdef}
For $T\in\SSS$, we define
\begin{equation}
C_\vir^\bullet(\Mbar\rel\partial)(T):=\hocolim_{T_0\to\cdots\to T_p\to T}C^{\bullet-\codim(T_0/T)}_\vir(\Mbar(T_0)\rel\partial,\bar A(T_p)).
\end{equation}
This homotopy colimit is over the category $\SSS_{/T}$, and the structure maps of the homotopy diagram are those appearing in \eqref{CvirnotSmoddiagram} (note that Definition \ref{hodiagram}\ref{hodiagramqiso} is satisfied as the leftmost map in \eqref{CvirnotSmoddiagram} is a quasi-isomorphism).  The natural inclusion map
\begin{equation}\label{cvirhocoliminclusion}
C_\vir^\bullet(\Mbar(T)\rel\partial;\bar A(T))\hookrightarrow C_\vir^\bullet(\Mbar\rel\partial)(T)
\end{equation}
is a quasi-isomorphism by Lemma \ref{finalhocolim}.

Now $C_\vir^{\bullet+\vdim}(\Mbar\rel\partial)$ naturally has the structure of an $\SSS$-module, as follows.  A morphism $T\to T'$ induces a natural pushforward map \eqref{hocolimpush} induced by the functor $\SSS_{/T}\to\SSS_{/T'}$, which is covered by a natural isomorphism of homotopy diagrams.  Given a concatenation $\{T_i\}_i$, there is a natural Eilenberg--Zilber map \eqref{hocolimEZ} induced by the isomorphism $\SSS_{/\#_iT_i}=\prod_i\SSS_{/T_i}$, which is covered by a morphism of homotopy diagrams coming from the product maps \eqref{Cvirprod}.
\end{definition}

\begin{definition}[$\SSS$-module $C_\bullet(E)$]\label{CEIdef}
For $T\in\SSS$, we define
\begin{equation}
C_\bullet(E)(T):=\hocolim_{T_0\to\cdots\to T_p\to T}C_\bullet(E;\bar A(T_p)).
\end{equation}
Again by Lemma \ref{finalhocolim}, we have
\begin{equation}\label{HEiso}
H_\bullet(E)(T)=H_\bullet(E;\bar A(T))=\QQ.
\end{equation}
As in Definition \ref{CvirMIdef}, $C_\bullet(E)$ naturally has the structure of a $\SSS$-module.  The isomorphism above is in fact an isomorphism of $\SSS$-modules $H_\bullet(E)=\QQ$.
\end{definition}

There is a canonical map of $\SSS$-modules
\begin{equation}\label{sast}
C_\vir^{\bullet+\vdim}(\Mbar\rel\partial)\to C_{-\bullet}(E)
\end{equation}
induced by \eqref{sastvir}.

\begin{definition}[{$\SSS^\op$-module $\Hom_{\SSS_{/T}}(\QQ[\SSS],C_\vir^{\bullet+\vdim}(\Mbar\rel\partial))$}]\label{HomQSCvirdef}
For $T\in\SSS$, we consider
\begin{equation}\label{HomQSCvir}
\Hom_{\SSS_{/T}}(\QQ[\SSS],C_\vir^{\bullet+\vdim}(\Mbar\rel\partial)).
\end{equation}
Namely, an element of this group is a natural transformation of functors from $\SSS_{/T}$ to the category of \emph{graded $\QQ$-vector spaces}, and we equip it with the usual differential $d\circ f-(-1)^{|f|}f\circ d$.

Now \eqref{HomQSCvir} is an $\SSS^\op$-module as follows.  A map $T\to T'$ clearly gives a restriction map from homomorphisms over $\SSS_{/T'}$ to homomorphisms over $\SSS_{/T}$.  A concatenation $\{T_i\}_i$ induces the necessary map by virtue of the fact that $\SSS_{/\#_iT_i}=\prod_i\SSS_{/T_i}$ and the concatenation maps for $\QQ[\SSS]$ are isomorphisms.
\end{definition}

\begin{definition}[$\SSS^\op$-module $\cH^\bullet(\Mbar(T))$]\label{cHnotrelbSmodule}
We equip the functor $T\mapsto\cH^\bullet(\Mbar(T))$ with the structure of an $\SSS^\op$-module as follows.
A map $T\to T'$ clearly gives a pullback map $\cH^\bullet(\Mbar(T'))\to\cH^\bullet(\Mbar(T))$.
A concatenation $\{T_i\}_i$ induces a map
\begin{equation}\label{HMbarTconcatemaps}
\bigotimes_i\cH^\bullet(\Mbar(T_i))\to\cH^\bullet\Bigl(\prod_i\Mbar(T_i)\Bigr)\to\cH^\bullet\Bigl(\prod_i\Mbar(T_i)\Bigr)^{\Aut(\{T_i\}_i/\#_iT_i)}\xleftarrow\sim\cH^\bullet(\Mbar(\#_iT_i)),
\end{equation}
where the middle map is averaging over the group action, and on the right we use \eqref{Mbarprod} together with the general result that the pullback map $\cH^\bullet(X/G)\xrightarrow\sim\cH^\bullet(X)^G$ is an isomorphism for any compact Hausdorff space $X$ with an action by a finite group $G$ \cite[Lemma A.4.9]{pardonimplicitatlas}.
Commutativity of the first diagram in \eqref{TImodidentities} follows from the fact that a morphism of concatenations $\{T_i\}_i\to\{T_i'\}_i$ induces a map of diagrams \eqref{HMbarTconcatemaps}, noting that $\Aut(\{T_i\}_i/\#_iT_i)=\Aut(\{T_i'\}_i/\#_iT_i')$.
Commutativity of the second diagram in \eqref{TImodidentities} follows from a diagram chase.
\end{definition}

\begin{lemma}\label{sheafylemma}
The cohomology of the $\SSS^\op$-module from Definition \ref{HomQSCvirdef} is canonically isomorphic to the $\SSS^\op$-module from Definition \ref{cHnotrelbSmodule}.
\end{lemma}

\begin{proof}
First, observe that
\begin{equation*}
\Hom_{\SSS_{/T}}(\QQ[\SSS],C_\vir^{\bullet+\vdim}(\Mbar\rel\partial))=\prod_{T'\to T}\left[\oo_{T'}^\vee\otimes C_\vir^\bullet(\Mbar\rel\partial)(T')\right]^{\Aut(T'/T)}.
\end{equation*}
Now
\begin{equation}\label{lastcomplexfiltered}
\prod_{T'\to T}\left[\oo_{T'}^\vee\otimes C_\vir^\bullet(\Mbar(T')\rel\partial;\bar A(T))\right]^{\Aut(T'/T)}
\end{equation}
includes quasi-isomorphically into the right hand side above, as the composition of the final object quasi-isomorphism \eqref{cvirhocoliminclusion} and the atlas enlargment quasi-isomorphism \eqref{Cvirinflate}.
Now \eqref{lastcomplexfiltered} is exactly the complex
\begin{equation}
C_\vir^\bullet(\Mbar(T),\SSS_{/T};\bar A(T))\otimes\oo_T^\vee,
\end{equation}
where the left factor is as defined in \eqref{stratifiedcomplex}.
Appealing to \eqref{stratcechviriso}, we deduce the desired isomorphism
\begin{equation}\label{desirediso}
H^\bullet\Hom_{\SSS_{/T}}(\QQ[\SSS],C_\vir^{\bullet+\vdim}(\Mbar\rel\partial))=\cH^\bullet(\Mbar(T)).
\end{equation}
It remains to compare pullback and concatenation maps.

Compatibility of \eqref{desirediso} with pullback maps is shown as follows.
The pullback map on $\Hom_{\SSS_{/T}}(\QQ[\SSS],C_\vir^{\bullet+\vdim}(\Mbar\rel\partial))$ for $T'\to T$ becomes, under the quasi-isomorphisms above, the natural restriction map
\begin{equation}\label{restrafterqiso}
C_\vir^\bullet(\Mbar(T),\SSS_{/T};\bar A(T))\to C_\vir^\bullet(\Mbar(T'),\SSS_{/T'};\bar A(T))\xrightarrow\sim C_\vir^\bullet(\Mbar(T'),\SSS_{/T'};\bar A(T')).
\end{equation}
Since the isomorphism \eqref{htpysheaffundiso} is natural, the map $\cH^\bullet(\Mbar(T))\to\cH^\bullet(\Mbar(T'))$ induced by \eqref{restrafterqiso} is simply the map induced by the corresponding map of $H^0$ sheaves, which by definition is just restriction, as needed.

To show compatibility of \eqref{desirediso} with concatenation maps, again under the quasi-isomorphisms above, the concatenation maps for $\Hom_{\SSS_{/T}}(\QQ[\SSS],C_\vir^{\bullet+\vdim}(\Mbar\rel\partial))$ become the natural maps
\begin{multline}
\bigotimes_iC_\vir^\bullet(\Mbar(T_i),\SSS_{/T_i};\bar A(T_i))\xrightarrow{\eqref{Cvirprod}}C_\vir^\bullet\Bigl(\prod_i\Mbar(T_i),\prod_i\SSS_{/T_i};\coprod_i\bar A(T_i)\Bigr)\\
\to C_\vir^\bullet(\Mbar(\#_iT_i),\SSS_{/\#_iT_i},\bar A(\#_iT_i)).
\end{multline}
Now appealing to Lemma \ref{kunnethforpure}, it is enough to observe that the induced maps on $H^0$ sheaves are the natural ones.
\end{proof}

\subsection{Cofibrant \texorpdfstring{$\SSS$}{S}-modules}\label{cofibrantsec}

We would like to upgrade the isomorphism of $\SSS$-modules $H_\bullet(E)=\QQ$ to a quasi-isomorphism of $\SSS$-modules $C_\bullet(E)\xrightarrow\sim\QQ$.  Unfortunately, the natural strategy to construct such a quasi-isomorphism, namely by induction on $T$, fails.  As a substitute, we introduce the notion of a \emph{cofibrant} $\SSS$-module, and we construct a cofibrant $\SSS$-module $C_\bullet^\cof(E)$ with a natural quasi-isomorphism $C_\bullet^\cof(E)\xrightarrow\sim C_\bullet(E)$.  Cofibrancy is exactly the condition which ensures that the natural inductive construction of a quasi-isomorphism $C_\bullet^\cof(E)\xrightarrow\sim\QQ$ succeeds.  The resulting diagram
\begin{equation}
C_\bullet(E)\xleftarrow\sim C_\bullet^\cof(E)\xrightarrow\sim\QQ
\end{equation}
turns out to be enough for our purposes.
Cofibrancy of the $\SSS$-module $\QQ[\SSS]$ will also end up being important for the inductive construction of virtual fundamental cycles.

\begin{definition}[Cofibrant $\SSS$-module]\label{cofibrantdefI}
An $\SSS$-module $X$ valued in chain complexes shall be called \emph{cofibrant} iff it satisfies the following two properties:
\begin{enumerate}
\item For all concatenations $\{T_i\}_i$ in $\SSS$, the induced map
\begin{equation}
\biggl[\bigotimes_iX(T_i)\biggr]_{\Aut(\{T_i\}_i/\#T_i)}\xrightarrow\sim X(\#_iT_i)
\end{equation}
is an isomorphism (the subscript indicates taking coinvariants with respect to the natural action of $\Aut(\{T_i\}_i/\#T_i)$).  Note that this follows if \eqref{TImodprod}/\eqref{TIImodprod}/\eqref{TIIImodprod}/\eqref{TIVmodprod} is itself an isomorphism.
\item For maximal $T\in\SSS$, the map
\begin{equation}\label{cofibrancyconditionI}
\colim_{\codim(T'/T)\geq 1}X(T')\rightarrowtail X(T)
\end{equation}
is injective.  The left side denotes the colimit over the full subcategory of $\SSS_{/T}$ spanned by objects $T'\to T$ with $\codim(T'/T)\geq 1$.
\end{enumerate}
Note that for $(\SSS_\I^\pm,\SSS_\II)$-modules, the above conditions are imposed over each of $\SSS_\I^\pm$, $\SSS_\II$, and so on.
\end{definition}

Lemma \ref{concatenationmaximal} will be used frequently below.

\begin{lemma}\label{sometoallinjective}
Injectivity of \eqref{cofibrancyconditionI} for maximal $T$ implies injectivity for all $T$.
\end{lemma}

\begin{proof}
Fix $T\in\SSS$, and write $T=\#_iT_i$ for maximal $T_i$.  Now we have $\SSS_{/T}=\prod_i\SSS_{/T_i}$.  Consider the cubical diagram
\begin{equation}\label{cubediagramI}
\bigotimes_i\left[\displaystyle\colim_{\codim(T'_i/T_i)\geq 1}X(T'_i)\rightarrowtail X(T_i)\right].
\end{equation}
Now \eqref{cofibrancyconditionI} for $T$ is precisely the map to the top vertex of the cube \eqref{cubediagramI} from the colimit over its remaining vertices.  This map is clearly injective given that each map in \eqref{cubediagramI} is injective.
\end{proof}

\begin{lemma}\label{QScofibrant}
The $\SSS$-module $\QQ[\SSS]$ is cofibrant.
\end{lemma}

\begin{proof}
The concatenation maps are isomorphisms by definition since $\SSS_{/\#_iT_i}=\prod_i\SSS_{/T_i}$.  Let us now show that
\begin{equation}\label{QSSiscofibcheck}
\colim_{\codim(T'/T)\geq 1}\QQ[\SSS](T')\rightarrowtail\QQ[\SSS](T)
\end{equation}
is an isomorphism onto the subspace of $\QQ[\SSS](T)$ generated by those $T'\to T$ of codimension $\geq 1$ (certainly this is sufficient).

Both sides of \eqref{QSSiscofibcheck} are graded by $|\SSS_{/T}|$, so it suffices to fix $T'\to T$ of codimension $\geq 1$ and show that \eqref{QSSiscofibcheck} is an isomorphism on $T'$-graded pieces.
The map \eqref{QSSiscofibcheck} is certainly surjective onto the $T'$-graded piece $\oo_{T'}\subseteq\QQ[\SSS](T)$.  Moreover, there is a section from $\oo_{T'}$ back to the left side of \eqref{QSSiscofibcheck} via the inclusion $\oo_{T'}\subseteq\QQ[\SSS](T')$.  Now it is enough to argue that this section is surjective (onto the $T'$-graded piece), but this is clear since a contribution of $\QQ[\SSS](T'')$ to the $T'$-graded piece is the same thing as a factorization $T'\to T''\to T$.
\end{proof}

\begin{definition}[Partial order on $|\SSS|$]\label{partialorderI}
For $T,T'\in\SSS$, let us write $T'\preceq T$ iff there is a morphism $\#_iT_i\to T$ with some $T_i$ isomorphic to $T'$.
\end{definition}

Since we are restricting to effective objects of $\SSS$, compactness of $\Mbar$ implies that the relation $\preceq$ is \emph{well-founded} (i.e.\ there is no infinite strictly decreasing sequence $T_1\succ T_2\succ\cdots$).
The relation $\preceq$ is also a partial order (reflexivity and transitivity are obvious, and antisymmetry follows from well-foundedness).
It follows that induction on $\left|\SSS\right|$ partially ordered by $\preceq$ is justified.

\begin{definition}[Cofibrant $\SSS$-module $C_\bullet^\cof(E)$]\label{CEtildeIdef}
We now define a cofibrant $\SSS$-module $C_\bullet^\cof(E)$ together with a quasi-isomorphism
\begin{equation}\label{qast}
C_\bullet^\cof(E)\xrightarrow\sim C_\bullet(E)
\end{equation}
which is surjective for maximal $T$.  Furthermore, the action of the ``paths between basepoints'' subgroup of $\Aut(T)$ on $C_\bullet^\cof(E)(T)$ will be trivial for all $T$ (as it is for $C_\bullet(E)(T)$).

We construct $C_\bullet^\cof(E)(T)$ by induction on $T$, partially ordered as in Definition \ref{partialorderI}.  For $T$ non-maximal, write $T=\#_iT_i$ with $T_i$ maximal.  The definition of cofibrancy both forces us to take $C_\bullet^\cof(E)(T):=\bigotimes_iC_\bullet^\cof(E)(T_i)$ and assures that the $\SSS$-module structure maps with target $C_\bullet^\cof(E)(T)$ exist and are unique.  For $T$ maximal, consider the diagram
\begin{equation}
\begin{tikzcd}
\displaystyle\colim_{\codim(T'/T)\geq 1}C_\bullet^\cof(E)(T')\ar{d}{\eqref{qast}}\ar[dashed, tail]{r}&C_\bullet^\cof(E)(T)\ar[dashed, two heads]{d}{\eqref{qast}}[swap]{\sim}\\
\displaystyle\colim_{\codim(T'/T)\geq 1}C_\bullet(E)(T')\ar{r}&C_\bullet(E)(T).
\end{tikzcd}
\end{equation}
We define $C_\bullet^\cof(E)(T)$ to be the mapping cylinder of the composition of the two solid maps, which clearly fits into the diagram as desired.  Now the top horizontal map defines the $\SSS$-module structure maps with target $C_\bullet^\cof(E)(T)$, and the commutativity of the diagram ensures that \eqref{qast} is a map of $\SSS$-modules.
\end{definition}

\begin{lemma}\label{HofcolimitCtilde}
Let $\TTT\subseteq\SSS_{/T_a}$ be a full subcategory which is ``downward closed'' in the sense that $T'\to T$ and $T\in\TTT$ implies $T'\in\TTT$.  Then the natural map
\begin{equation}
C_\bullet N_\bullet\TTT\cong\hocolim_{T\in\TTT}C_\bullet^\cof(E)(T)\xrightarrow\sim\colim_{T\in\TTT}C_\bullet^\cof(E)(T)
\end{equation}
is a quasi-isomorphism.  In particular, the homology of the right side is supported in degrees $\geq 0$ and in degree zero is generated by the images of $H_0^\cof(E)(T)$ for $T\in\TTT$.
\end{lemma}

\begin{proof}
We argue by induction on $\#\TTT$ (note that $|\SSS_{/T}|$ is finite by our restriction to effective objects), the case $\#\TTT=0$ being trivial.

Pick a maximal object $T_0\in\TTT$.
If $\TTT=\TTT^{\leq T_0}$, then the desired conclusion follows from Lemma \ref{finalhocolim}; otherwise argue as follows.
Let us abbreviate $X:=C_\bullet^\cof(E)$.
We claim that there is a short exact sequence
\begin{equation}\label{colimspresentation}
0\to\colim_{T\in\TTT^{<T_0}}X(T)\to\colim_{T\in\TTT^{\leq T_0}}X(T)\oplus\colim_{T\in\TTT\setminus\{T_0\}}X(T)\to\colim_{T\in\TTT}X(T)\to 0.
\end{equation}
(Note that the expressions $\TTT^{<T_0}$, $\TTT^{\leq T_0}$, $\TTT\setminus\{T_0\}$ refer to \emph{full subcategories} of $\TTT$, not any sort of over-categories.)
Exactness on the right and in the middle follows from the universal property of $\colim$.
To show exactness on the left, we show the stronger statement that
\begin{equation}\label{wantinjective}
\colim_{T\in\TTT^{<T_0}}X(T)\hookrightarrow\colim_{T\in\TTT^{\leq T_0}}X(T)
\end{equation}
is injective.
To see this, note that cofibrancy of $X$ (and Lemma \ref{sometoallinjective}) and exactness of $\Aut_\TTT(T_0)$-coinvariants imply that the map
\begin{equation}\label{knowinjective}
\Bigl(\colim_{T\in\TTT_{/T_0}}X(T)\Bigr)_{\Aut_\TTT(T_0)}\hookrightarrow X(T_0)_{\Aut_\TTT(T_0)}
\end{equation}
is injective.
The right hand sides of \eqref{knowinjective} and \eqref{wantinjective} coincide, and the left hand side of \eqref{knowinjective} surjects onto the left hand side of \eqref{wantinjective}, so we conclude the desired injectivity of \eqref{wantinjective}.

Now \eqref{colimspresentation} induces a long exact sequence on homology.
There is a similar long exact sequence with $\hocolim$ in place of $\colim$, together with a map from the sequence of $\hocolim$'s to the sequence of $\colim$'s.
Now the five lemma and the induction hypothesis (together with the fact that $\TTT^{\leq T_0}\subsetneqq\TTT$) give the desired result.
\end{proof}

\begin{lemma}\label{qisoCcof}
There exists a quasi-isomorphism of $\SSS$-modules
\begin{equation}
C_\bullet^\cof(E)\xrightarrow\sim\QQ
\end{equation}
inducing the canonical isomorphism $H_\bullet^\cof(E)=H_\bullet(E)=\QQ$ from \eqref{HEiso}.
\end{lemma}

\begin{proof}
We argue by induction on $T\in\SSS$, partially ordered as in Definition \ref{partialorderI}.

For $T$ non-maximal, write $T=\#_iT_i$ for $T_i$ maximal.  Then cofibrancy of $C_\bullet^\cof(E)$ both forces us to take $p_\ast(T):=\bigotimes_ip_\ast(T_i)$ and assures that this choice is compatible with the maps defined thus far.

For maximal $T$, we would like to fill in the diagram
\begin{equation}\label{crucialinjection}
\begin{tikzcd}
\displaystyle\colim_{\codim(T'/T)\geq 1}C_\bullet^\cof(E)(T')\ar[tail]{r}\ar{d}{p_\ast}&C_\bullet^\cof(E)(T)\ar[dashed]{dl}{p_\ast}\\
\QQ
\end{tikzcd}
\end{equation}
with a map $C_\bullet^\cof(E)(T)\to\QQ$ in a particular chain homotopy class.  The horizontal map is injective since $C_\bullet^\cof(E)$ is cofibrant; it follows that it is enough to show that \eqref{crucialinjection} commutes up to chain homotopy.
For chain complexes $C_\bullet$ over $\QQ$, we have an isomorphism $H_0\Hom(C_\bullet,\QQ)\xrightarrow\sim\Hom(H_0C_\bullet,\QQ)$, so it suffices to show that \eqref{crucialinjection} commutes on homology, which follows from Lemma \ref{HofcolimitCtilde}.
Finally, we may ensure $p_\ast$ is $\Aut(T)$-invariant by averaging (this is necessary for $p_\ast$ to be a natural transformation of functors).
\end{proof}

\begin{lemma}\label{QSStocvir}
There exists a map of $\SSS$-modules
\begin{equation}\label{keyvfcinductivemap}
\QQ[\SSS]\to C_\vir^{\bullet+\vdim}(\Mbar\rel\partial)
\end{equation}
with the following property.  Such a map determines for all $T\in\SSS$ a cycle in \eqref{HomQSCvir}, which by Lemma \ref{sheafylemma} determines an element of $\cH^0(\Mbar(T))$.  We require these elements coincide with the class of the constant function $1\in\cH^0(\Mbar(T))$ for all $T\in\SSS$.
\end{lemma}

\begin{proof}
We argue by induction on $T$.

For $T$ non-maximal, write $T=\#_iT_i$ with $T_i$ maximal; now cofibrancy of $\QQ[\SSS]$ (Lemma \ref{QScofibrant}) determines the map completely on $\QQ[\SSS](T)$, and the map has the desired property by Lemma \ref{sheafylemma}.

For maximal $T$, consider the following restriction map:
\begin{equation}\label{restrictionmapI}
\prod_{T'\to T}\left[\oo_{T'}^\vee\otimes C_\vir^\bullet(\Mbar\rel\partial)(T')\right]^{\Aut(T'/T)}\to\prod_{\codim(T'/T)\geq 1}\left[\oo_{T'}^\vee\otimes C_\vir^\bullet(\Mbar\rel\partial)(T')\right]^{\Aut(T'/T)}.
\end{equation}
The part of the map \eqref{keyvfcinductivemap} defined thus far determines a cycle on the right, which we must lift to a cycle on the left in a particular cohomology class.  Since \eqref{restrictionmapI} is surjective, it suffices to show that the desired cohomology class on the left is mapped to the given cohomology class on the right.
Identifying \eqref{restrictionmapI} with the map on global sections of a corresponding map of pure homotopy sheaves as in Lemma \ref{sheafylemma}, we see that on cohomology this map is simply the restriction map $\cH^\bullet(\Mbar(T))\to\cH^\bullet(\partial\Mbar(T))$.
It thus suffices to show that the cycle on the right equals the constant function $1\in\cH^0(\partial\Mbar(T))$, which can be checked on each $\Mbar(T')$ for $\codim(T'/T)\geq 1$, where it holds by the induction hypothesis.
\end{proof}

\subsection{Sets \texorpdfstring{$\Theta$}{Theta}}\label{setthetaIsec}

We now conclude by defining the sets $\Theta$ which index all possible choices of the extra data necessary to fix coherent virtual fundamental cycles on the moduli spaces $\Mbar$.

\begin{definition}[Sets $\Theta(\D)$]\label{thetadefs}
Given a datum $\D$ as in any of Setups \ref{setupI}--\ref{setupIV}, an element $\theta\in\Theta(\D)$ consists of a commuting diagram\footnote{Note that the categories $\SSS$ are essentially small, so the collection of such diagrams forms a set.} of $\SSS$-modules
\begin{equation}\label{keydiagramI}
\begin{tikzcd}
\QQ[\SSS]\ar{r}{\tilde w_\ast}\ar{d}{w_\ast}&C_{-\bullet}^\cof(E)\ar{r}{p_\ast}\ar[two heads]{d}{\eqref{qast}}[swap]{\sim}&\QQ\\
C_\vir^{\bullet+\vdim}(\Mbar\rel\partial)\ar{r}{\eqref{sast}}&C_{-\bullet}(E)
\end{tikzcd}
\end{equation}
satisfying the following properties:
\begin{itemize}
\item We require that $p_\ast$ induce the canonical isomorphism $H_\bullet^\cof(E)=H_\bullet(E)=\QQ$ from Definition \ref{CEIdef}.
\item We require that $w_\ast$ satisfy the conclusion of Lemma \ref{QSStocvir}.
\end{itemize}
\end{definition}

Note that there are natural forgetful maps
\begin{align}
\label{thetaIIforget}\Theta_\II&\to\Theta_\I^+\times\Theta_\I^-,\\
\label{thetaIIIforget}\Theta_\III&\to\Theta_\II^{t=0}\times_{\Theta_\I^+\times\Theta_\I^-}\Theta_\II^{t=1},\\
\label{thetaIVforget}\Theta_\IV&\to\Theta_\II^{02}\times_{\Theta_\I^0\times\Theta_\I^2}(\Theta_\II^{01}\times_{\Theta_\I^1}\Theta_\II^{12}).
\end{align}

An element $\theta\in\Theta$ evidently gives rise to a morphism of $\SSS$-modules $p_\ast\circ\tilde w_\ast:\QQ[\SSS]\to\QQ$.  Such a morphism corresponds to virtual moduli counts $\#\Mbar(T)^\vir_\theta$ satisfying the relevant master equation by Remark \ref{rmkmapofSgivesvmc}.

\begin{proposition}\label{vmctransverse}
If $\Mbar(T)$ is weakly regular and $\vdim(T)=0$, then $\Mbar(T)=\M(T)$ and $\#\Mbar(T)^\vir_\theta=\#\M(T)$.
\end{proposition}

\begin{proof}
There are no nontrivial (effective) $T'\to T$ for dimension reasons, so $\Mbar(T)=\M(T)$.
Evaluating the diagram \eqref{keydiagramI} at $T$ and taking cohomology gives
\begin{equation}
\begin{tikzcd}[column sep = large]
\oo_T\ar{r}{\tilde w_\ast}\ar{d}{w_\ast}&\QQ\ar{r}{\id}\ar{d}{\id}&\QQ\\
\cH^\bullet(\Mbar(T);\oo_T)\ar{r}{[\Mbar(T)]^\vir}&\QQ.
\end{tikzcd}
\end{equation}
Unravelling the isomorphism in Lemma \ref{QSStocvir}, we see that the left vertical map is just the tautological map to $\cH^0$.  The bottom horizontal map is by definition the virtual fundamental class $[\Mbar(T)]^\vir$ from \eqref{sastvir}.
Thus we would like to appeal to Lemma \ref{transversevfc} to conclude that the virtual fundamental class coincides with the ordinary fundamental class.

For cases (I) and (II), our regularity assumption on $\Mbar(T)$ ensures that the hypothesis $\Mbar(T)=\Mbar(T)^\reg$ of Lemma \ref{transversevfc} is satisfied (one should observe here that for moduli spaces of virtual dimension zero, the regular locus automatically has trivial isotropy, due to our use of asymptotic markers and the fact that nodes appear only in codimension two).
For cases (III) and (IV), there is a gap between our assumption that $\Mbar(T)$ is weakly regular and the hypothesis $\Mbar(T)=\Mbar(T)^\reg$ of Lemma \ref{transversevfc} (see Definition \ref{regularitydef}).  It suffices, though, to observe that the implicit atlas on $\Mbar(T)$ remains an implicit atlas if we enlarge $\Mbar(T)_I^\reg$ to include all points which are weakly regular and have trivial isotropy.  The only thing to check is that the manifold/openness/submersion axioms still hold, but this is easy since the thickened moduli spaces have only the top stratum $\M(T)_I$ and hence can be described by the usual Banach manifold Fredholm setup.
\end{proof}

\begin{proposition}\label{thetanonempty}
The set $\Theta_\I$ is non-empty (resp.\ \eqref{thetaIIforget}--\eqref{thetaIVforget} are surjective).
\end{proposition}

\begin{proof}
Concretely, we must construct $p_\ast$, $w_\ast$, $\tilde w_\ast$ as in Definition \ref{thetadefs}.  We will in fact give a construction of such data by \emph{induction} on $T\in\SSS$ with respect to the partial order from Definition \ref{partialorderI}.  The inductive nature of our proof (in addition to being the only reasonable approach) implies surjectivity of \eqref{thetaIIforget}--\eqref{thetaIVforget} (as opposed to mere non-emptiness of $\Theta_\II$, $\Theta_\III$, $\Theta_\IV$).

The existence of $p_\ast$ and $w_\ast$ follow from Lemmas \ref{qisoCcof} and \ref{QSStocvir} respectively.  Note that both are proved by induction on $T$, so their use here is permissible.

To show the existence of $\tilde w_\ast$, we argue by induction on $T$.  Note that $\QQ[\SSS]$ is cofibrant by Lemma \ref{QScofibrant}.  For non-maximal $T$, cofibrancy of $\QQ[\SSS]$ determines $\tilde w_\ast(T)$ uniquely and implies that the diagram still commutes.  For maximal $T$, we are faced with the following lifting problem:
\begin{equation}
\begin{tikzcd}
\displaystyle\colim_{\codim(T'/T)\geq 1}\QQ[\SSS](T')\ar[tail]{d}\ar{r}{\tilde w_\ast}&C_\bullet^\cof(E)(T)\ar[two heads]{d}{\eqref{qast}}[swap]{\sim}\\
\QQ[\SSS](T)\ar{r}{\eqref{sast}\circ w_\ast}\ar[dashed]{ru}{\tilde w_\ast}&C_\bullet(E)(T).
\end{tikzcd}
\end{equation}
The left vertical map is injective by Lemma \ref{QScofibrant}, and the right vertical map is a surjective quasi-isomorphism by Definition \ref{CEtildeIdef}.  All four complexes are bounded below.  Now Lemma \ref{liftingproperty} ensures that a lift exists, and we make it $\Aut(T)$-equivariant by averaging.
\end{proof}

\begin{lemma}\label{liftingproperty}
Consider a diagram of chain complexes bounded below over a ring $R$:
\begin{equation}
\begin{tikzcd}
A_\bullet\ar[tail]{d}\ar{r}&X_\bullet\ar[two heads]{d}{\sim}\\
B_\bullet\ar{r}\ar[dashed]{ru}&Y_\bullet
\end{tikzcd}
\end{equation}
where the right vertical map is a surjective quasi-isomorphism and the left vertical map is an injection whose cokernel is degreewise projective.  Then there exists a lift as illustrated.
\end{lemma}

\begin{proof}
This is the fact that ``cofibrations have the left lifting property with respect to acyclic fibrations'' in the projective model structure on $\Ch_{\bullet\geq 0}(R)$.  It can be proved by a straightforward diagram chase.
\end{proof}

\subsection{Symmetric monoidal structure on \texorpdfstring{$\Theta$}{Theta}}

We now show that the sets $\Theta$ in cases (I) and (II) are naturally symmetric monoidal with respect to disjoint union of data $\D$.
Concretely, this means we are relating the virtual moduli counts associated to a disjoint union of contact manifolds or symplectic cobordisms with those associated to each piece (compare \S\ref{wrapup}).
In fact, to do this we will consider not $\Theta$ as defined in Definition \ref{thetadefs} but rather an (equivalent) enlargement thereof (for which all of the results proven earlier remain valid), explained during the course of the proof.

\begin{proposition}\label{symmetricmonoidal}
There are functorial maps associated to finite disjoint unions (indexed by $i\in I$)
\begin{align}
\label{symmodThetainclusionI}\Theta_\I(Y,\lambda,J)&\leftarrow\prod_{i\in I}\Theta_\I(Y_i,\lambda_i,J_i)&\text{for }(Y,\lambda,J)&=\bigsqcup_{i\in I}\,(Y_i,\lambda_i,J_i)\\
\label{symmodThetainclusionII}\Theta_\II(\hat X,\hat\omega,\hat J)&\leftarrow\prod_{i\in I}\Theta_\II(\hat X_i,\hat\omega_i,\hat J_i)&\text{for }(\hat X,\hat\omega,\hat J)&=\bigsqcup_{i\in I}\,(\hat X_i,\hat\omega_i,\hat J_i)
\end{align}
preserving the virtual moduli counts and turning $\Theta_\I$ and $\Theta_\II$ into weak symmetric monoidal functors.
\end{proposition}

\begin{proof}
Note that (restricting to effective objects), we have
\begin{align}
\SSS_\I(Y,\lambda,J)&=\bigsqcup_{i\in I}\SSS_\I(Y_i,\lambda_i,J_i)&\text{for }(Y,\lambda,J)&=\bigsqcup_{i\in I}\,(Y_i,\lambda_i,J_i),\\
\SSS_\II(\hat X,\hat\omega,\hat J)&=\bigsqcup_{i\in I}\SSS_\II(\hat X_i,\hat\omega_i,\hat J_i)&\text{for }(\hat X,\hat\omega,\hat J)&=\bigsqcup_{i\in I}\,(\hat X_i,\hat\omega_i,\hat J_i),
\end{align}
so an $\SSS_\I(Y,\lambda,J)$-module is the same as a tuple of $\SSS_\I(Y_i,\lambda_i,J_i)$-modules for $i\in I$ (and the same for $\SSS_\II$-modules).  There are also natural inclusions
\begin{align}
\label{symmodAinclusionI}A_\I(Y,\lambda,J)(T)&\hookleftarrow A_\I(Y_i,\lambda_i,J_i)(T)&\text{for }T&\in\SSS_\I(Y_i,\lambda_i,J_i),\\
A_\II(\hat X,\hat\omega,\hat J)(T)&\hookleftarrow A_\II(\hat X_i,\hat\omega_i,\hat J_i)(T)&\text{for }T&\in\SSS_\II(\hat X_i,\hat\omega_i,\hat J_i).
\end{align}
Precisely, these maps preserve $r$, $E$, $D$, and extend $\lambda$ as zero on $Y\setminus Y_i$ (resp.\ $\hat X\setminus\hat X_i$).

Now we modify the definition of $\Theta$ to include a choice of $\Aut(T)$-invariant subatlases $B(T)\subseteq A(T)$ on $\Mbar(T)$ for all maximal $T\in\SSS$ (e.g.\ in case (II), we choose $B_\I^\pm(T)\subseteq A_\I^\pm(T)$ and $B_\II(T)\subseteq A_\II(T)$).
We define $\bar B(T)$ as in \eqref{barAIcoproduct}, and we use $\bar B$ in place of $\bar A$ in the definitions from \S\S\ref{Cvirhocolimsec}--\ref{cofibrantsec}.

With this modification done, the map \eqref{symmodThetainclusionI} is defined by taking the images of the sets $B_\I(T)$ under \eqref{symmodAinclusionI} and using the ``same'' diagrams of $\SSS_\I$-modules.  The map \eqref{symmodThetainclusionII} is defined similarly.  It follows by definition that these maps are functorial and preserve the virtual moduli counts.
\end{proof}

\section{Gluing}\label{gluingsec}

This section is devoted to the proof of Theorems \ref{localgluing} and \ref{localorientations} (of which Theorems \ref{localgluingnonthickened} and \ref{localorientationsnonthickened} are special cases).  Namely, we prove that the regular loci in the thickened moduli spaces $\Mbar(T)_I$ admit the expected local topological description in terms of the spaces $G_{T'//T}$ from \S\ref{localmodelsection}, and we verify that the natural ``geometric'' and ``analytic'' maps between orientation lines agree.  We will give the argument for all cases (I), (II), (III), (IV) simultaneously.

The gluing techniques we use in this section are, to the best of our knowledge, standard in the field (perhaps the most similar published result is Hutchings--Taubes \cite{hutchingstaubesI,hutchingstaubesII}).
We learned these methods (while writing the gluing theorems in \cite[\S\S B--C]{pardonimplicitatlas}) from Abouzaid \cite{abouzaidexotic}, McDuff--Salamon \cite{mcduffsalamonJholsymp,mcduffsalamonJholquant}, and discussions with Mohammed Abouzaid, Tobias Ekholm, Helmut Hofer, and Rafe Mazzeo.

\subsection{Gluing setup}\label{localgluingproofI}

\begin{proof}[Proof of Theorem \ref{localgluing}]
Fix $T\in\SSS$, $I\subseteq J\subseteq\bar A(T)$, and $x_0\in\M(T)_J$ of type $T'\to T$ with $s_{J\setminus I}(x_0)=0$ and $\psi_{IJ}(x_0)\in\M(T)_I^\reg$.  A neighborhood of $x_0\in\Mbar(T)_J$ is stratified by $\SSS_{T'//T}$.

Our goal is to construct a local homeomorphism
\begin{equation}\label{gluinggoalA}
\bigl(G_{T'//T}\times E_{J\setminus I}\times\RR^{\mu(T')-\#V_s(T')+\dim E_I},(0,0,0)\bigr)\to\bigl(\Mbar(T)_J,x_0\bigr)
\end{equation}
which lands in $\Mbar(T)_J^\reg$ and which commutes with the maps from both sides to $\SSS_{T'//T}\times\overline{\s(T)}\times E_{J\setminus I}$.  We denote by $0\in G_{T'//T}$ the basepoint corresponding to all gluing parameters equal to $\infty$ (i.e.\ corresponding to no gluing at all) and $t=t(x_0)$.

The basepoint $x_0$ corresponds to a map $u_0:C_0\to\hat X_0$ and an element $e_0\in E_I\subseteq E_J$, along with ``discrete data'' consisting of asymptotic markers, matching isomorphisms, and markings $\phi_\alpha:(C_0)_\alpha\to\Cbar_{0,E^\eext(T_\alpha)\cup\{1,\ldots,r_\alpha\}}$ for $\alpha\in J$.

\subsubsection{Domain stabilization via submanifolds \texorpdfstring{$\hat D_{v,i}$}{D\_v,i} and points \texorpdfstring{$q_{v,i}$}{q\_v,i}}

We first stabilize the domain $C_0$ by adding marked points $q_{v,i}$ where it intersects certain codimension two submanifolds $\hat D_{v,i}$, arguing much the same as we did in \S\ref{stabilizationsec}.  Such $\hat D_{v,i}$ automatically stabilize the domains of all maps in a neighborhood of $x_0\in\Mbar(T)_J$, and thus we need only consider stable domain curves in the main gluing argument which follows.

For every vertex $v\in V(T')$, we add marked points $q_{v,i}\in(C_0)_v$ and choose local codimension two submanifolds $\hat D_{v,i}\subseteq(\hat X_0)_v$ (required to be $\RR$-invariant if $v$ is a symplectization vertex) with $u_0(q_{v,i})\in\hat D_{v,i}$ intersecting transversally, such that $(C_0)_v$ equipped with the marked points $\{p_{v,e}\}_e$ and $\{q_{v,i}\}_i$ is stable.
These new marked points $q_{v,i}$ bear no logical relation to the marked points labeled by $\{1,\ldots,r_\alpha\}$ for $\alpha\in J$ (in particular, there is no need to require that the $q_{v,i}$ be disjoint from these other marked points).

To show the existence of such points, it suffices to show that each unstable irreducible component of $(C_0)_v$ (equipped with just $\{p_{v,e}\}_e$ as marked points) has a point (and hence a non-empty open set) where $du_0$ (resp.\ $\pi_\xi du_0$ if $v$ is a symplectization vertex) is injective.  If $(C_0)_v\nsubseteq(C_0)_\alpha$ for all $\alpha\in I$, then $u_0|(C_0)_v$ is $\hat J_v$-holomorphic and the existence of such points follows from the arguments given in the proofs of Lemmas \ref{coveringI} and \ref{coveringII}.  For components $(C_0)_v\subseteq(C_0)_\alpha$ for some $\alpha\in J$, such points exist since $u_v|(C_0)_v$ satisfies Definition \ref{MIthickdef}\ref{MIthickdeftr}.

\subsubsection{Glued cobordisms \texorpdfstring{$\hat X_\g$}{X\_g} and points \texorpdfstring{$q_v'$}{q\_v'}}\label{qvprimesec}

Given any gluing parameter $\g\in G_{T'//T}$, we may form the glued cobordism $\hat X_\g$ as follows (all statements involving a choice of $\g\in G_{T'//T}$ carry the (often tacit) assumption that $\g$ lies in a sufficiently small neighborhood of $0$).  Note that $\hat X_0$ is equipped with cylindrical coordinates \eqref{endmarkingsI}--\eqref{endmarkingsII} in each end.  Now we truncate each positive (resp.\ negative) end $[0,\infty)\times Y_e$ (resp.\ $(-\infty,0]\times Y_e$) to $[0,\g_e]\times Y_e$ (resp.\ $[-\g_e,0]\times Y_e$) and identify truncated ends by translation by $\g_e$ (if $\g_e=\infty$ we do nothing).  By construction, $\hat X_\g$ is the target for pseudo-holomorphic buildings of type $T'_\g$ (denoting by $T'\to T'_\g\to T$ the image of $\g$ under the map $G_{T'//T}\to\SSS_{T'//T}$).

For every symplectization vertex $v\in V_s(T')$, fix a section $q_v'$ of the universal curve over a neighborhood of $C_0\in\Mbar_{0,E^\eext(T')\cup\{q_{v,i}\}_{v,i}}$ such that $q_v'(C_0)\in(C_0)_v$.  Now over a neighborhood of $x_0\in\Mbar(T)_J$, the sections $q_v'$ determine points $q_v'$ in the domain of each map.  Over this neighborhood, we can identify the target with $\hat X_\g$ for a unique $\g\in G_{T'//T}$ by requiring that each $q_v'$ be mapped to the corresponding ``zero level'' $0_v\subseteq\hat X_\g$ (i.e.\ the descent of $0_v:=\{0\}\times Y_v\subseteq\hat X_0$).

\subsubsection{Family of almost complex structures \texorpdfstring{$j_y$}{j\_y} on \texorpdfstring{$C_0$}{C\_0}}\label{jysec}

We now proceed to fix a family of almost complex structures on $C_0$ inducing a local biholomorphism onto the relevant moduli space of marked Riemann surfaces.  Let $N\subseteq C_0$ denote the set of nodes (note that $\{p_{v,e}\}_{v,e}$ are not nodes), and let $N_v:=N\cap(C_0)_v$.

We denote by $\Mbar_{0,n+m_{(2)}}\to\Mbar_{0,n+m}$ the total space of the bundle over $\Mbar_{0,n+m}$ with fiber $\prod_{i=1}^m(T_{q_i}C\setminus 0)$ (thought of as the space of markings $\CC\xrightarrow\sim T_{q_i}C$ at the last $m$ marked points).
As in \S\S\ref{rsdeformations}--\ref{rsmoduli}, the tangent space to the nodal stratum $\Mbar{}_{0,n+m_{(2)}}^{\#\nodes=r}$ at a given curve $C$ with marked points $P$ and doubly marked points $Q$ is canonically isomorphic to $H^1(\tilde C,T\tilde C(-\tilde P-2\tilde Q-\tilde N))$.

For every vertex $v\in V(T')$, fix a local biholomorphism
\begin{equation}\label{Jvmapping}
\J_v:=\CC^{2\#\{p_{v,e}\}_e+\#\{q_{v,i}\}_i-\#N_v-3}\to\Mbar_{0,\{q_{v,i}\}_i+(\{p_{v,e}\}_e)_{(2)}}^{\#\nodes=\#N_v}
\end{equation}
sending $0$ to $(C_0)_v$ equipped with its doubly marked points $\{p_{v,e}\}_e$ and marked points $\{q_{v,i}\}_i$.
Further fix a smooth trivialization of the universal curve over these spaces (over a neighborhood of $(C_0)_v$, that is the origin in $\J_v$), which near the nodes and marked points of the fibers is holomorphic (on total spaces) and preserves the tangent space markings.

We thus obtain a description of the universal curve over $\J_v$ as the curve $(C_0)_v$ with its doubly marked points $\{p_{v,e}\}_e$, its marked points $\{q_{v,i}\}_i$, and a modified almost complex structure $j_y$ varying smoothly in $y\in\J_v$ and fixed in a neighborhood of $\{p_{v,e}\}_e$, $\{q_{v,i}\}_i$, and $N_v$.
Since \eqref{Jvmapping} is a local biholomorphism, the derivative
\begin{equation}
\frac{dj_y}{dy}:T_0\J_v=\CC^{2\#\{p_{v,e}\}_e+\#\{q_{v,i}\}_i-\#N_v-3}\to C^\infty_c((C_0)_v\setminus(\{q_{v,i}\}_i\cup\{p_{v,e}\}_e\cup N_v),\End^{0,1}(T(C_0)_v))
\end{equation}
induces an isomorphism between its domain and the deformation space of $(C_0)_v$ equipped with its doubly marked points $\{p_{v,e}\}_e$ and marked points $\{q_{v,i}\}_i$ as recalled above.
We set
\begin{equation}
\J:=\prod_{v\in V(T')}\J_v.
\end{equation}
All statements involving a choice of $y\in\J_v$ or $y\in\J$ carry the (often tacit) assumption that $y$ lies in a sufficiently small neighborhood of $0$.

\subsubsection{Cylindrical coordinates on \texorpdfstring{$C_0$}{C\_0} and \texorpdfstring{$\hat X_0$}{X\_0}}

We now fix positive (resp.\ negative) holomorphic cylindrical coordinates
\begin{align}
\label{cylI}[0,\infty)\times S^1&\to C_0\\
\label{cylII}(-\infty,0]\times S^1&\to C_0
\end{align}
near each positive (resp.\ negative) puncture.  We also fix such cylindrical coordinates on either side of each node $n\in N$ (choosing which side is positive/negative arbitrarily).

We assume that with respect to these fixed cylindrical coordinates, we have
\begin{equation}\label{uzeroexponentialdecay}
u_0(s,t)=(Ls,\tilde\gamma(t))+o(1)\quad\text{as }\left|s\right|\to\infty
\end{equation}
(i.e.\ the constant $b$ in \eqref{asymptoticend} vanishes) with respect to the cylindrical coordinates \eqref{endmarkingsI}--\eqref{endmarkingsII} on $\hat X_0$.  To ensure that we can achieve \eqref{uzeroexponentialdecay}, we allow the possibility that \eqref{cylI}--\eqref{cylII} are only defined for $\left|s\right|$ sufficiently large.
Recall from Proposition \ref{hwzcylinderestimate} due to Hofer--Wysocki--Zehnder \cite[Theorems 1.1, 1.2, and 1.3]{hwzsmallarea} that the convergence \eqref{uzeroexponentialdecay} holds in $C^\infty$ with error $O(e^{-\delta\left|s\right|})$ for every $\delta<\delta_\gamma$.

\subsubsection{Glued curves \texorpdfstring{$C_\g$}{C\_g} and points \texorpdfstring{$q_e''$}{q\_e''}}\label{Cgdefsec}

Given any gluing parameter $\beta\in\CC^{E^\eint(T')}\times\CC^N$ (i.e.\ one for each interior edge of $T'$ and one for each node $n\in N$), we may form the glued curve $C_\beta$ as follows (all statements involving a choice of $\beta\in\CC^{E^\eint(T')}\times\CC^N$ carry the (often tacit) assumption that $\beta$ lies in a sufficiently small neighborhood of $0$).  For each interior edge $v\xrightarrow ev'$ (or node $n\in N$), we truncate the positive (resp.\ negative) end $[0,\infty)\times S^1$ (resp.\ $(-\infty,0]\times S^1$) to $[0,S]\times S^1$ (resp.\ $[-S,0]\times S^1$) and identify them by $s=s'+S$ and $t=t'+\theta$ where $\beta=e^{-S-i\theta}$ (if $\beta=0$ we do nothing).  Note that the points $q_{v,i}$ and the complex structures $j_y$ both descend naturally to $C_\beta$.  This operation gives analytic families of curves over $\J_v\times\CC^{N_v}$ and $\J\times\CC^{E^\eint(T')}\times\CC^N$ inducing local biholomorphisms
\begin{align}
\J_v\times\CC^{N_v}&\to\Mbar_{0,\{q_{v,i}\}_i+(\{p_{v,e}\}_e)_{(2)}},\\
\label{JfullCnodesbiholo}\J\times\CC^N&\to\prod_{v\in V(T')}\Mbar_{0,\{q_{v,i}\}_i+(\{p_{v,e}\}_e)_{(2)}}.
\end{align}
The glued curve $C_\beta$ is actually a bit too general for our purposes: in the current setting the gluing parameters at interior edges must come from the same gluing parameters $\g\in G_{T'//T}$ used to glue the target cobordism.

Namely, to any $\g\in G_{T'//T}$, we associate the gluing parameter $\beta=\beta(\g)\in\CC^{E^\eint(T')}$ given by $S_e:=L_e^{-1}\g_e$ and the unique $\theta_e$ corresponding to the matching isomorphism $S_{p_{v,e}}C_v\to S_{p_{v',e}}C_{v'}$.
For $\g\in G_{T'//T}\times\CC^N$, we thus have a glued curve $C_\g$.
For $e$ with $\g_e<\infty$, let $q_e''\in C_\g$ denote the point in the middle of the corresponding neck, namely $s=\pm\frac 12S$ with respect to the coordinates \eqref{cylI}--\eqref{cylII} (the angular coordinate $t\in S^1$ of $q_e''$ is irrelevant as long as it is fixed).  Now for any fixed $\g\in G_{T'//T}$, this construction gives a map
\begin{equation}\label{gluingdiffeo}
\J\times\CC^N\to\prod_{v\in V(T'_\g)}\Mbar_{0,\{q_{v',i}\}_{v'\mapsto v,i}+\{q_e''\}_{e\mapsto v}+(\{p_{v,e}\}_e)_{(2)}}.
\end{equation}
This map is a local biholomorphism near $0\in\J\times\CC^N$ uniformly in $\g$, in the sense that there is a neighborhood of $0\in\J\times\CC^N$ over which \eqref{gluingdiffeo} is a biholomorphism onto its image for all $\g\in G_{T'//T}$ near zero (we will prove this just below).

The purpose of including the points $q_e''$ as part of the structure of the glued curve is exactly to ensure that \eqref{gluingdiffeo} is a local biholomorphism.
Later, we will impose a codimension two condition at $q_e''$ on our glued maps, so that over a neighborhood of $x_0\in\Mbar(T)_J$, the points $q_e''$ in the domain are determined uniquely, and hence the parameters $(y,\beta)\in\J\times\CC^N$ are determined uniquely.
We denote by $q_v'\in(C_\g,j_y)$ the value of the section $q_v'$ at $(C_\g,j_y)$ for $\g\in G_{T'//T}\times\CC^N$ and $y\in\J$ (note that this $q_v'$ may \emph{not} coincide with the descent of $q_v'\in C_0$, even for $\g=0$).

To conclude, let us give the proof that \eqref{gluingdiffeo} is a local biholomorphism, as claimed above.
In fact, to make the proof work we impose the following additional compatibility condition between the cylindrical coordinates \eqref{cylI}--\eqref{cylII} and the smooth trivialization of the universal curve over $\Mbar_{0,\{q_{v,i}\}_i+(\{p_{v,e}\}_e)_{(2)}}^{\#\nodes=\#N_v}$ near $(C_0)_v$ fixed in \S\ref{jysec}.
The universal curve is pulled back under the forgetful map $\Mbar_{0,\{q_{v,i}\}_i+(\{p_{v,e}\}_e)_{(2)}}\to\Mbar_{0,\{q_{v,i}\}_i+\{p_{v,e}\}_e}$ which is a principal $(\CC\setminus 0)^{\{p_{v,e}\}_e}$-bundle (acting via changing the tangent space markings at $\{p_{v,e}\}_e$).
We require that the induced action of $(\CC\setminus 0)^{\{p_{v,e}\}_e}$ on the total space, when pulled back to $(C_0)_v$ under the chosen trivialization and then expressed in the cylindrical coordinates \eqref{cylI}--\eqref{cylII} near each $p_{v,e}$, must coincide with the tautological action of $\CC\setminus 0$ on the standard end $[0,\infty)\times S^1$ (resp.\ $(-\infty,0]\times S^1$).
To construct such a smooth trivialization compatible with a given choice of cylindrical coordinates \eqref{cylI}--\eqref{cylII}, first fix the trivialization over the image of a section of the forgetful map, then extend to a neighborhood of $\{p_{v,e}\}_e$ as dictated by the desired compatibility property, and then finally extend to the rest.

Given this compatibility property, the local biholomorphism property for \eqref{gluingdiffeo} now follows easily from the standard local biholomorphism property for the usual `local gluing charts' on Deligne--Mumford space.
Namely, for any open subset $U\subseteq\Mbar_{0,n}^{\#\nodes=r}\subseteq\Mbar_{0,n}$ (or $\Mbar_{0,n+m_{(2)}}$) and any choice of local holomorphic (on the total space) cylindrical coordinates
\begin{equation}\label{stdgluingholcoords}
[0,\infty)\times S^1\times\Mbar_{0,n}^{\#\nodes=r}\to\Cbar_{0,n}^{\#\nodes=r}
\end{equation}
(defined over $U$) on each side of each of the $r$ nodes, there is an induced holomorphic map
\begin{equation}\label{stdgluingDM}
U\times\CC^r\to\Mbar_{0,n}
\end{equation}
given by the usual gluing construction (via $e^{-s_1-it_1}e^{-s_2-it_2}=\beta$).
It is well-known that \eqref{stdgluingDM} is a local biholomorphism near $U$ (indeed, it is enough to verify that for every curve $C$ (a point of $U$) and every node $a\in C$, the map $\CC\to\Mbar_{0,n}$ given by gluing $C$ at $a$ with respect to \eqref{stdgluingholcoords} with gluing parameter $\beta\in\CC$ induces an isomorphism between $T_0\CC$ and the tangent space to the local branch of $\Mbar_{0,n}^{\#\nodes=r+1}$ corresponding to $a$, and this can be shown by a standard calculation).
Now to deduce the local biholomorphism property for \eqref{gluingdiffeo}, it suffices to apply the local biholomorphism property for \eqref{stdgluingDM} in a setup which depends only on $T_\g'$ (and hence is uniform in $\g$).

\subsubsection{Preglued maps \texorpdfstring{$u_\g$}{u\_g}}

We now define a ``preglued'' map $u_\g:C_\g\to\hat X_\g$ close to $u_0:C_0\to\hat X_0$.  As we shall see later, this preglued map is very close to solving the relevant (thickened) pseudo-holomorphic curve equation.  Our goal will then be to understand the true solutions near $u_\g$ and to show that this construction gives a local parameterization of the moduli space near $u_0$.

Fix a smooth (cutoff) function $\chi:\RR\to[0,1]$ satisfying
\begin{equation}
\chi(x)=\begin{cases}1&x\leq 0\\0&x\geq 1.\end{cases}
\end{equation}

\begin{definition}[Flattening]
For $\g\in G_{T'//T}\times\CC^N$, we define the ``flattened'' map
\begin{equation}
u_{0|\g}:C_0\to\hat X_0
\end{equation}
as follows.  Away from the ends, $u_{0|\g}$ coincides with $u_0$.  Over a positive end asymptotic to a parameterized Reeb orbit $\tilde\gamma(t):=u_0(\infty,t)$, we define $u_{0|\g}$ by
\begin{equation*}
u_{0|\g}(s,t):=\begin{cases}
u_0(s,t)&s\leq\frac 16S\\
\exp_{(Ls,\tilde\gamma(t))}\left[\chi\bigl(s-\frac 16S\bigr)\cdot\exp_{(Ls,\tilde\gamma(t))}^{-1}u_0(s,t)\right]&\frac 16S\leq s\leq\frac 16S+1\\
(Ls,\tilde\gamma(t))&\frac 16S+1\leq s.
\end{cases}
\end{equation*}
Over a positive end at a node $n\in N$, we define $u_{0|\g}$ by
\begin{equation*}
u_{0|\g}(s,t):=\begin{cases}
u_0(s,t)&s\leq\frac 16S\\
\exp_{u_0(n)}\left[\chi\bigl(s-\frac 16S\bigr)\cdot\exp_{u_0(n)}^{-1}u_0(s,t)\right]&\frac 16S\leq s\leq\frac 16S+1\\
u_0(n)&\frac 16S+1\leq s.
\end{cases}
\end{equation*}
An analogous definition applies over negative ends.  Here $\exp:T\hat X_0\to\hat X_0$ denotes any fixed exponential map (i.e.\ a smooth map defined in a neighborhood of the zero section satisfying $\exp(p,0)=p$ and $d\exp(p,\cdot)=\id_{T_p\hat X_0}$) which is $\RR$-equivariant in any end.
\end{definition}

\begin{definition}[Pregluing]
For $\g\in G_{T'//T}\times\CC^N$, we define the ``preglued'' map
\begin{equation}
u_\g:C_\g\to\hat X_\g
\end{equation}
as the natural ``descent'' of $u_{0|\g}$ from $C_0$ to $C_\g$.
\end{definition}

\subsection{Gluing estimates}

With the above setup understood, our aim is now to describe the ``true solutions'' close to the ``approximate solution'' $u_\g:C_\g\to\hat X_\g$.  This forms the core part of the gluing argument.

\subsubsection{Weighted Sobolev norms}\label{gluingsobolev}

Our first step is simply to define the weighted Sobolev spaces $W^{k,2,\delta}$ relevant for the gluing argument and to specify norms on these spaces, up to commensurability uniform in $\g$ near zero (uniformity in $\g$ is crucial for the key gluing estimates to have any meaning or utility at all).
The weighted Sobolev spaces we define here differ from those of \S\ref{linearizedopsec} in that we put weights not just over the ends asymptotic to Reeb orbits but also near the nodes (this necessitates a comparison of linearized operators in the two contexts, which we do immediately after the definition of Sobolev norms below).
We also put weights over the necks (which has meaning since we work up to commensurability uniform in $\g$).

Fix metrics and connections as in Conventions \ref{metricconnectionchoicedomain}--\ref{metricconnectionchoicetarget} on $C_0$ and $\hat X_0$ with respect to the cylindrical coordinates fixed above and which (for convenience) agree across the parts to be glued (thus descending to $C_\g$ and $\hat X_\g$).
We suppose throughout that $\delta>0$ satisfies $\delta<1$ and $\delta<\delta_\gamma$ for every asymptotic Reeb orbit $\gamma$ of $u_0$.
As in Remark \ref{kdeltacomment}, it is sufficient to fix any $k\geq 4$, though for the moment we only need to assume $k\geq 0$ ($k\geq 2$ if $C_0$ has nodes).

\begin{remark}
Although these choices of metrics, connections, $k$, and $\delta$ affect the constants appearing in the gluing estimates, they do \emph{not} affect either the final gluing map or (the restriction to smooth functions/sections of) any of the intermediate maps we study and/or construct.
\end{remark}

Away from the ends/necks of $C_\g$, we use the usual $W^{k,2}$-norm.  In an end asymptotic to a Reeb orbit, the contribution to the norm squared is
\begin{equation}\label{endnormdecaytozero}
\int_{[0,\infty)\times S^1}\sum_{j=0}^k\left|D^jf\right|^2e^{2\delta s}\;ds\,dt.
\end{equation}
In an end asymptotic to a node $n\in N$, we distinguish two cases: for Sobolev spaces of $(0,1)$-forms on $C_\g$, we use \eqref{endnormdecaytozero}, and for spaces of functions we allow decay to a constant, i.e.\ the contribution to the norm squared is
\begin{equation}
\left|f(n)\right|^2+\int_{[0,\infty)\times S^1}\sum_{j=0}^k\left|D^j[f-f(n)]\right|^2e^{2\delta s}\;ds\,dt
\end{equation}
The contribution of a neck is given as follows (first is for necks over Reeb orbits and necks over nodes for spaces of $(0,1)$-forms; second is for necks over nodes for spaces of functions):
\begin{align}
&\int_{[0,S]\times S^1}\sum_{j=0}^k\left|D^jf\right|^2e^{2\delta\min(s,S-s)}\;ds\,dt,\\
&\int_{[0,S]\times S^1}\sum_{j=0}^k\left|D^j\left[f-\frac 1{2\pi}\int_{S^1}f\left(\textstyle\frac 12S,t\right)dt\right]\right|^2e^{2\delta\min(s,S-s)}\;ds\,dt+\left|\frac 1{2\pi}\int_{S^1}f\left(\textstyle\frac 12S,t\right)dt\right|^2.
\end{align}
We use parallel transport to make sense of the differences $f-f(n)$, etc.  To meausre/differentiate $(0,1)$-forms, we use the usual Riemannian metric on $[0,\infty)\times S^1$ (resp.\ $[0,S]\times S^1$) and usual flat connection on forms.
Different choices of metrics and connections yield norms which are equivalent, uniformly in $\g$, for any fixed $k$ and $\delta$ as above.

\begin{remark}
We emphasize that, although we refer to the weighted Sobolev norms in \S\ref{linearizedopsec} informally as ``without weights at the nodes'', they are \emph{not} the same as the weighted Sobolev norms above specialized to $\delta=0$ at the nodes.
\end{remark}

We now provide the promised comparison of linearized operators defined using the weighted Sobolev norms above (with weights near the nodes) and using the weighted Sobolev norms from \S\ref{linearizedopsec} (without weights near the nodes).
That is, we argue that both linearized operators are Fredholm, with the same kernel and cokernel.
Let us temporarily write $W^{k,2,\delta,\delta}$ for the weighted Sobolev spaces defined above and $W^{k,2,\delta}$ for the weighted Sobolev spaces from \S\ref{linearizedopsec}.
We suppose that $k$ is large enough so that the linearized operators are defined (see Definitions \ref{linearizedbasicdef} and \ref{linopthickened}).
We have already seen in Proposition \ref{linearizedfredholm} that the linearized operator acting on $W^{k,2,\delta}$ is Fredholm for such $\delta$.
The same argument applies to the linearized operator acting on $W^{k,2,\delta,\delta}$, using the fact that $\delta\in(0,1)$ is not an eigenvalue of the asymptotic operators at the nodes (all of which have spectrum $\ZZ$).
There are inclusions $W^{k,2,\delta}\subseteq W^{k,2,\delta,\delta}$ (since $\delta<1$) which are dense by Lemma \ref{smoothweighteddense}.
We get induced maps from the kernel and cokernel of the linearized operator acting on $W^{k,2,\delta}$ to those of the operator acting on $W^{k,2,\delta,\delta}$.
The map on kernels is obviously injective, and the map on cokernels is obviously surjective.  To show surjectivity on kernels and injectivity on cokernels, it suffices to show that if $D\xi=\eta$ with $\xi\in W^{k,2,\delta,\delta}$ and $\eta\in W^{k-1,2,\delta}$, then $\xi\in W^{k,2,\delta}$.  Now as distributions, we have $D\xi=\eta+\varepsilon$, for some distribution $\varepsilon$ supported over the inverse images of the nodes $\tilde N_v\subseteq\tilde C_v$.  By elliptic regularity, it suffices to show that $\varepsilon=0$.  But now for any smooth test function $\varphi$ supported over a small neighborhood of $\tilde N_v$, we have
\begin{equation}
\langle\varepsilon,\varphi\rangle=\langle D\xi-\eta,\varphi\rangle=\langle\xi,D^\ast\varphi\rangle-\langle\eta,\varphi\rangle.
\end{equation}
The first term is bounded by a constant times $\|\varphi\|_{1,1}$ since $D^\ast$ (the formal adjoint) is first order and $\xi\in W^{k,2,\delta,\delta}\subseteq C^0$.  The second term is bounded by a constant times $\|\varphi\|_2$ since $\eta\in W^{k-1,2,\delta}\subseteq L^2$.  Now $\varepsilon$ is supported at a finite number of points, so must be a linear combination of $\delta$-functions and their derivatives (see Lemma \ref{distrdelta}), but these do not satisfy such bounds (recall that $W^{1,1}\nsubseteq C^0$ since we are in two dimensions \cite[Example 5.25]{adamssobolev}).

\subsubsection{Nonlinear Fredholm setup for fixed \texorpdfstring{$\g$}{g}}

We now formulate precisely what we mean by ``solutions close to $u_\g:C_\g\to\hat X_\g$''.  What we mean is ``small zeroes of $\F_\g$'', where $\F_\g$ is the map
\begin{multline}
\F_\g:W^{k,2,\delta}(C_\g,u_\g^\ast T\hat X_\g)_{\begin{smallmatrix}\xi(q_{v,i})\in T\hat D_{v,i}\hfill\\\pi_{\RR\partial_s\oplus\RR R_\lambda}\xi(q_e'')=0\hfill\end{smallmatrix}}\oplus\J\oplus E_J\\
\to W^{k-1,2,\delta}(\tilde C_\g,u_\g^\ast(T\hat X_\g)_{\hat J_{\g_t}}\otimes_\CC\Omega^{0,1}_{\tilde C_\g})\oplus\RR^{V_s(T')}
\end{multline}
defined by
\begin{multline}\label{Fgdef}
\F_\g(\xi,y,e):=\\
\biggl[(\PT_{\exp_{u_\g}\xi\to u_\g}^{\g_t}\otimes I_y)\Bigl(d(\exp_{u_\g}\xi)+\sum_{\alpha\in J}\nu_\alpha((e_0+e)_\alpha)(\phi_\alpha^{\g,\xi},\exp_{u_\g}\xi)\Bigr)^{0,1}_{j_y,\hat J_{\g_t}}\\
\oplus\bigoplus_{v\in V_s(T')}\pi_\RR(\exp_{u_\g}\xi)(q_v'(y))\biggr].
\end{multline}
We explain the notation.  We denote by $\exp:T\hat X_0\to\hat X_0$ a fixed exponential map which is $\RR$-equivariant in ends and over symplectizations and agrees across the parts to be glued, thus descending to $\hat X_\g$.  We fix a $\hat J_0$-linear connection on $T\hat X_0$ which is $\RR$-equivariant in ends and over symplectizations and agrees across the parts to be glued, thus descending to $\hat X_\g$.  We denote by $\PT^{\g_t}$ parallel transport with respect to the $\hat J_{\g_t}$-linear part of this fixed connection.  The map $I_y:\Omega^{0,1}_{C_\g,j_y}\to\Omega^{0,1}_{C_\g,j_0}$ denotes the composition $\Omega^{0,1}_{C_\g,j_y}\to\Omega^1_{C_\g}\otimes_\RR\CC\to\Omega^{0,1}_{C_\g,j_0}$ (which is $\CC$-linear).

The map $\F_\g$ is defined over the ball of some fixed radius $c_{k,\delta}'>0$ uniformly in $\g$ near zero, for any $k\geq 3$.  The constraint $k\geq 3$ ensures that $W^{k,2}\subseteq C^1$ (see \cite[Lemma 5.17]{adamssobolev}), which is needed so that $\phi_\alpha^{\g,\xi}$ is defined for $\|\xi\|_{k,2,\delta}$ small.

\subsubsection{Estimate for \texorpdfstring{$\|\F_\g(0)\|$}{|F\_g(0)|}}

We now show that $\F_\g(0)$ is very small (i.e.\ the preglued map $u_\g:C_\g\to\hat X_\g$ is very close to being a true solution).

\begin{lemma}
We have
\begin{equation}\label{Fgzerosmall}
\|\F_\g(0)\|_{k-1,2,\delta}\to 0\quad\text{as }\g\to 0
\end{equation}
for all $k\geq 1$.
\end{lemma}

\begin{proof}
Away from the necks and ends, the $1$-form part of $\F_\g(0)$ is only nonzero because of using $\hat J_{\g_t}$ in place of $\hat J_0$ and because of using $\phi_\alpha^{\g,0}$ in place of $\phi_\alpha^{0,0}=\phi_\alpha$.  This difference clearly goes to zero as $\g\to 0$.

Over Reeb ends, the $1$-form part of $\F_\g(0)$ is identically zero.  Over Reeb necks, the $1$-form part of $\F_\g(0)$ is supported near $\frac 16S$ and $\frac 56S$, and the desired estimate follows from the exponential convergence of \eqref{uzeroexponentialdecay} and the fact that $\delta<\delta_\gamma$.

The same applies to nodal ends/necks, except that in addition there is a term coming from using $\hat J_{\g_t}$ in place of $\hat J_0$.  This again is bounded as desired since $\delta<1$.

The $\RR^{V_s(T')}$ part of $\F_\g(0)$ satisfies the desired estimate since $q_v'\in(C_\g,j_0)$ approaches the descent of $q_v'\in(C_0,j_0)$ as $\g\to 0$.
\end{proof}

\subsubsection{Regularity of the map \texorpdfstring{$\F_\g$}{F\_g} (quadratic estimate)}

We now prove a ``quadratic estimate'' quantifying the regularity of $\F_\g$ near zero, i.e.\ giving a uniform upper bound on certain of its derivatives near zero.  This estimate is used when we apply the (Banach space) inverse function theorem to understand $\F_\g^{-1}(0)$ near zero.
We begin with a general discussion of the smoothness properties of $\F_\g$ before specializing to the specific estimate we need.

The first term in $\F_\g$ (the usual pseudo-holomorphic curve equation) is smooth and local.  The second term (the ``thickening'' terms $\nu_\alpha$) is non-local; its only non-smoothness comes from the association $\xi\mapsto\phi_\alpha^{\g,\xi}$.  It thus is $C^\ell$ as long as the function which assigns to $\xi$ the set $(\exp_{u_\g}\xi)^{-1}(\hat D_\alpha)$ is $C^\ell$.  By the inverse function theorem, this is the case whenever $W^{k,2}\subseteq C^\ell$, which in turn holds whenever $k\geq\ell+2$ \cite[Lemma 5.17]{adamssobolev}.  The third term is also $C^\ell$ whenever $W^{k,2}\subseteq C^\ell$ (these both come down to the fact that the evaluation map $W^{k,2}(C,X)\times C\to X$ is of class $C^\ell$ whenever $W^{k,2}\subseteq C^\ell$).

The specific sort of estimate we need to apply the inverse function theorem to understand $\F_\g^{-1}(0)$ is a ``quadratic estimate'' quantifying, uniformly in $\g$, the fact that the derivative of $\F_\g$ is Lipschitz.
It is thus sufficient to assume that $k\geq 4$, so that the above discussion shows that $\F_\g$ is of class $C^2$.

\begin{proposition}\label{quadestimate}
For $\left\|\zeta\right\|_{k,2,\delta},\left\|\xi\right\|_{k,2,\delta}\leq c_{k,\delta}'$, we have
\begin{equation}\label{derivislip}
\left\|\F_\g'(0,\xi)-\F_\g'(\zeta,\xi)\right\|_{k-1,2,\delta}\leq c_{k,\delta}\left\|\zeta\right\|_{k,2,\delta}\left\|\xi\right\|_{k,2,\delta}
\end{equation}
for constants $c_{k,\delta}<\infty$ and $c_{k,\delta}'>0$ uniformly in $\g$ near $0$, for all $k\geq 4$.
\end{proposition}

\begin{proof}
We have already observed above that $\F_\g$ is of class $C^2$ for $k\geq 4$, which implies that \eqref{derivislip} holds for some $c_{k,\delta}<\infty$ and $c_{k,\delta}'>0$ possibly depending on $\g$.
Our present task is thus to make the above reasoning precise enough to extract bounds which are uniform in $\g$.
In a word, the point of the proof is simply that the constants $c_{k,\delta}<\infty$ and $c_{k,\delta}'>0$ in question depend only on the $W^{k,\infty}$-norms of $u_\g$ and the various operations comprising $\F_\g$ (in particular, the total area of $C_\g$ is irrelevant).

Choose a covering $\{U_\alpha\}_\alpha$ of $C_\g$ of locally bounded geometry by open sets $U_\alpha$ of bounded diameter (e.g.\ over ends/necks, we can use the covering by sets of the form $[s,s+2]\times S^1$ for integers $s$).
The local geometry of the cover can be bounded uniformly in $\g$, and thus we can write any $\xi\in W^{k,2,\delta}$ as a sum $\xi=\sum_\alpha\xi_\alpha$ where $\xi_\alpha$ is supported inside $U_\alpha$ and
\begin{equation}
\sum_\alpha\left\|\xi_\alpha\right\|_{k,2,\delta}\leq c_{k,\delta}\left\|\xi\right\|_{k,2,\delta}
\end{equation}
for some constant $c_{k,\delta}<\infty$ uniform in $\g$.
Hence it is enough to show the desired estimate \eqref{derivislip} for $\xi$ whose support has bounded diameter.
We know from our earlier reasoning that the estimate \eqref{derivislip} holds for all $\xi$, and it is clear that the resulting constants $c_{k,\delta}<\infty$ and $c_{k,\delta}'>0$ depend only on the derivatives of $u_\g$ over the support of $\xi$ (which we now assume has bounded diameter, and thus bounded geometry).
The derivatives of $u_\g$ are bounded uniformly in $\g$, so we conclude that the constants $c_{k,\delta}<\infty$ and $c_{k,\delta}'>0$ are as well.
\end{proof}

Note that integrating \eqref{derivislip} from $\xi_1$ to $\xi_2$ gives
\begin{equation}\label{quadestgoal}
\bigl\|D_\g(\xi_1-\xi_2)-(\F_\g\xi_1-\F_\g\xi_2)\bigr\|_{k-1,2,\delta}\leq c_{k,\delta}\left\|\xi_1-\xi_2\right\|_{k,2,\delta}\max(\left\|\xi_1\right\|_{k,2,\delta},\left\|\xi_2\right\|_{k,2,\delta})
\end{equation}
for $\left\|\xi_1\right\|_{k,2,\delta},\left\|\xi_2\right\|_{k,2,\delta}\leq c_{k,\delta}'$.

\subsubsection{Bounded right inverses and kernel gluing I: relating \texorpdfstring{$D_0$}{D\_0} and \texorpdfstring{$D_\g$}{D\_g}}

The final step in understanding $\F_\g^{-1}(0)$ is to construct a sufficiently nice bounded right inverse $Q_\g$ for $D_\g:=\F_\g'(0,\cdot)$.  In particular, we will need to show that $\left\|Q_\g\right\|$ is bounded uniformly for $\g$ near $0$, and that $\im Q_\g$ ``varies continuously'' (in a sense which we will make precise) as $\g$ varies.  We will also construct a natural ``kernel gluing'' isomorphism $\ker D_0\to\ker D_\g$.

To study the linearized operator $D_\g$, and in particular to construct $Q_\g$, we consider the following diagram, which allows us to relate $D_\g$ to $D_0$.
\begin{equation}\label{RIdiagram}
\begin{tikzcd}
W^{k,2,\delta}(C_\g,u_\g^\ast T\hat X_\g)_{\begin{smallmatrix}\xi(q_{v,i})\in T\hat D_{v,i}\hfill\\\pi_{\RR\partial_s\oplus\RR R_\lambda}\xi(q_e'')=0\hfill\end{smallmatrix}}\oplus\J\oplus E_J\ar{r}{D_\g}\ar[leftarrow]{d}{\calib}&W^{k-1,2,\delta}(\tilde C_\g,u_\g^\ast(T\hat X_\g)_{\hat J_{\g_t}}\otimes_\CC\Omega^{0,1}_{\tilde C_\g})\oplus\RR^{V_s(T')}\ar[equals]{d}\\
W^{k,2,\delta}(C_\g,u_\g^\ast T\hat X_\g)_{\xi(q_{v,i})\in T\hat D_{v,i}}\oplus\J\oplus E_J\ar{r}{D_\g}\ar[leftarrow]{d}{\glue}&W^{k-1,2,\delta}(\tilde C_\g,u_\g^\ast(T\hat X_\g)_{\hat J_{\g_t}}\otimes_\CC\Omega^{0,1}_{\tilde C_\g})\oplus\RR^{V_s(T')}\ar{d}{\brk}\\
W^{k,2,\delta}(C_0,u_{0|\g}^\ast T\hat X_0)_{\xi(q_{v,i})\in T\hat D_{v,i}}\oplus\J\oplus E_J\ar{r}{D_{0|\g}}\ar[leftrightarrow]{d}{\PT}&W^{k-1,2,\delta}(\tilde C_0,u_{0|\g}^\ast (T\hat X_0)_{\hat J_{\g_t}}\otimes_\CC\Omega^{0,1}_{\tilde C_0})\oplus\RR^{V_s(T')}\ar[leftrightarrow]{d}{{\PT}\circ{\id^{1,0}}}\\
W^{k,2,\delta}(C_0,u_0^\ast T\hat X_0)_{\xi(q_{v,i})\in T\hat D_{v,i}}\oplus\J\oplus E_J\ar{r}{D_0}&W^{k-1,2,\delta}(\tilde C_0,u_0^\ast(T\hat X_0)_{\hat J_0}\otimes_\CC\Omega^{0,1}_{\tilde C_0})\oplus\RR^{V_s(T')}
\end{tikzcd}
\end{equation}
The maps in this diagram are defined as follows.

The horizontal maps are the linearized operators at the maps $u_0$, $u_{0|\g}$, $u_\g$.

The maps $\PT$ are parallel transport with respect to the fixed connection on $T\hat X_0$.

The map $\id^{1,0}:(T\hat X_0)_{\hat J_0}\to(T\hat X_0)_{\hat J_{\g_t}}$ is the $(1,0)$-component of the identity map.

Let us define the $\brk$ map from \eqref{RIdiagram}.  Fix a smooth function $\bar\chi:\RR\to[0,1]$ such that
\begin{equation}
\bar\chi(x)=\begin{cases}1&x\leq-1\\0&x\geq+1\end{cases}\qquad\bar\chi(x)+\bar\chi(-x)=1.
\end{equation}
Now $\brk(\eta)$ is simply $\eta$ except over the ends of $C_0$, where we define it to be
\begin{equation}
\brk(\eta)(s,t):=\begin{cases}
\eta(s,t)&s\leq \frac 12S-1\\
\bar\chi(s-\frac 12S)\cdot\eta(s,t)&\frac 12S-1\leq s\leq\frac 12S+1\\
0&\frac 12S+1\leq s.
\end{cases}
\end{equation}
Thus the ``trace'' of $\brk(\eta)$ from $C_0$ to $C_\g$ (adding along fibers) is precisely $\eta$.

Let us define the $\glue$ map from \eqref{RIdiagram}.  The map $\glue$ acts only on the vector field component (it acts identically on the other components).  Away from the necks, we set $\glue(\xi):=\xi$, and in any particular neck $[0,S]\times S^1\subseteq C_\g$, we define (respectively for necks near Reeb orbits and necks near nodes $n\in N$)
\begin{align*}
\glue(\xi)(s,t):=&\begin{cases}
\xi(s,t)&s\leq\frac 13S-1\\
\chi(s-\frac 23S)\xi(s,t)+\chi(\frac 23S-s')\xi(s',t')&\frac 13S-1\leq s\leq\frac 23S+1\\
\xi(s',t')&\frac 23S+1\leq s,
\end{cases}\\
\glue(\xi)(s,t):=&\begin{cases}
\xi(s,t)&s\leq\frac 13S-1\\
\xi(n)+\chi(s-\frac 23S)[\xi(s,t)-\xi(n)]\\
\phantom{\xi(n)}+\chi(\frac 23S-s')[\xi(s',t')-\xi(n)]&\frac 13S-1\leq s\leq\frac 23S+1\\
\xi(s',t')&\frac 23S+1\leq s,
\end{cases}
\end{align*}
(noting the corresponding ends $(s,t)\in[0,\infty)\times S^1\subseteq C_0$ and $(s',t')\in(-\infty,0]\times S^1\subseteq C_0$, glued via $s=s'+S$ and $t=t'+\theta$).

Let us define the $\calib$ map from \eqref{RIdiagram}.  For every edge $e\in E^{\eint}(T')$ with $\g_e<\infty$, we consider the vector field $C_\g\to u_\g^\ast T\hat X_\g$ given in this neck by
\begin{equation}\label{xcutoff}
\chi(s-{\textstyle\frac 23S})\chi({\textstyle\frac 23S}-s')\cdot\partial_su_\g.
\end{equation}
We denote by $\X$ the $\CC$-span of these vector fields.  Now we have
\begin{equation}\label{calibsplitting}
W^{k,2,\delta}(C_\g,u_\g^\ast T\hat X_\g)_{\xi(q_{v,i})\in T\hat D_{v,i}}=W^{k,2,\delta}(C_\g,u_\g^\ast T\hat X_\g)_{\begin{smallmatrix}\xi(q_{v,i})\in T\hat D_{v,i}\hfill\\\pi_{\RR\partial_s\oplus\RR R_\lambda}\xi(q_e'')=0\hfill\end{smallmatrix}}\oplus\X,
\end{equation}
and the map $\calib$ is simply the associated projection onto the first factor.

This completes the definition of the maps in \eqref{RIdiagram}.

\begin{lemma}
The maps in \eqref{RIdiagram} are all bounded uniformly in $\g$.
\end{lemma}

\begin{proof}
To show uniform boundedness of the linearized operators $D_0$, $D_{0|\g}$, $D_\g$, it is enough to cover the domain curve $C_0$ or $C_\g$ with small open sets of bounded geometry and to prove a uniform bound on $D_0$, $D_{0|\g}$, $D_\g$ restricted to sections supported in each such small open set (compare the proof of Proposition \ref{quadestimate}).
For sake of concreteness, we can take these ``small open sets of bounded geometry'' to be of the form $[n,n+2]\times S^1$ in the ends/necks (note that the total number of these open sets does not matter).
The maps $u_0$, $u_{0|\g}$, $u_\g$ are very well behaved in the ends/necks, uniformly in $\g$, which is enough.
More succinctly, the point is simply that the norms of the linearized operators $D_0$, $D_{0|\g}$, $D_\g$ depend only on the $W^{k,\infty}$-norms of their ``coefficients'' and of $u_0$, $u_{0|\g}$, $u_\g$ (so, in particular, the total area of $C_\g$ is irrelevant).

The maps $\brk$ and $\glue$ are uniformly bounded because the derivatives of $\chi$ and $\bar\chi$ are bounded and the weights agree (up to a constant factor) at the relevant corresponding points of the domains.

To show that the map $\calib$ is uniformly bounded, it is enough to show that the projection onto the second factor in \eqref{calibsplitting} is uniformly bounded.
This projection is uniformly bounded since the norm of \eqref{xcutoff} is uniformly commensurable with its value at $q_e''\in[0,S]\times S^1$ (which holds because $\delta>0$ and $q_e''$ is given the maximum weight).

The maps $\PT$ and $\id^{1,0}$ are bounded uniformly by the same argument used for the linearized operators $D_0$, $D_{0|\g}$, $D_\g$.
\end{proof}

\subsubsection{Bounded right inverses and kernel gluing II: estimates}

The diagram \eqref{RIdiagram} does not commute, but is very close to commuting for $\g$ close to zero, as the following estimates make precise.

\begin{lemma}\label{RIestimates}
We have the following estimates
\begin{align}
\label{flatteningdifference}\|{\PT}\circ{D_0}-{D_{0|\g}}\circ{\PT}\|&\to 0\\
\label{gluebreakdifference}\|({D_\g}\circ{\glue})(\xi)-\eta\|&=o(1)\cdot\left\|\xi\right\|\quad\text{for }\brk(\eta)=D_{0|\g}\xi\\
\label{calibdifference}\|{D_\g}\circ{\calib}-D_\g\|&\to 0
\end{align}
as $\g\to 0$, for any fixed $k\geq 2$.
\end{lemma}

\begin{proof}
To prove \eqref{flatteningdifference}, argue as follows.  The first difference between the two operators is over the $[\frac 16S,\infty)\times S^1$ subset of some ends.  In this region, both are linear differential operators, which we may write in local coordinates $(s,t)$ on $C_0$ and exponential coordinates on the target near the asymptotic orbit or point.  The desired bound then follows from the exponential convergence of \eqref{uzeroexponentialdecay} (near $n\in N$, observe that smoothness of $u_0$ implies decay of all derivatives like $e^{-|s|}$ in cylindrical coordinates).  The second difference between the two operators is $\hat J_0$ vs $\hat J_{\g_t}$, and this is also bounded as desired, since $\hat J_{\g_t}\to\hat J_0$ in $C^\infty$.

To prove \eqref{gluebreakdifference}, argue as follows.
Away from ends/necks, the difference is only nonzero due to using $\phi_\alpha^{\g,\xi}$ in place of $\phi_\alpha^{0,\xi}$, and it is straightforward to see that this is bounded as desired.
In the ends/necks, the difference is only nonzero over the $([\frac 13S-1,\frac 13S]\cup[\frac 23S,\frac 23S+1])\times S^1$ subsets of each neck.  By symmetry, we discuss only the $[\frac 23S,\frac 23S+1]\times S^1$ part, where it equals $D_{0|\g}(\chi(s-\frac 23S)\xi(s,t))$.  Now we note that $D_{0|\g}(\chi(s-\frac 23S)\xi(s,t))$ has $W^{k-1,2,\delta}(\tilde C_0)$-norm bounded by $\|\xi\|_{k,2,\delta}$.  But we are interested in the $W^{k-1,2,\delta}(\tilde C_\g)$-norm, where the weight is smaller by a factor of $e^{-\frac 13\delta S}$, giving the desired estimate since $\delta>0$.

To prove \eqref{calibdifference}, argue as follows.  It suffices to show that $\|D_\g(X(\xi))\|=o(1)\|\xi\|$, where $X(\xi)\in\X$ denotes $\xi-\calib(\xi)$, i.e.\ the second projection in \eqref{calibsplitting}.  Note that the norm of $\xi(q_e'')$ (which determines $X(\xi)$) is bounded by a constant times $e^{-\frac 12\delta S}\|\xi\|_{k,2,\delta}$ (i.e.\ $\|\xi\|_{k,2,\delta}$ divided by the weight in the middle of the neck).  Now $D_\g(X(\xi))$ is only nonzero over the $([\frac 13S-1,\frac 13S]\cup[\frac 23S,\frac 23S+1])\times S^1$ subsets of each neck, where the weight is $e^{\frac 13\delta S}$.  Its norm is thus bounded by $e^{(\frac 13-\frac 12)\delta S}\|\xi\|_{k,2,\delta}$, giving the desired result since $\delta>0$.
\end{proof}

\subsubsection{Bounded right inverses and kernel gluing III: goal}\label{rightinversegoal}

Recall that by assumption, $D_0$ is surjective and the natural projection $\ker D_0\to E_{J\setminus I}$ is surjective; indeed, this is what it means for $\psi_{IJ}(x_0)$ to lie in $\M(T')_I^\reg$.  Let $Q_0$ denote any bounded right inverse for $D_0$, meaning $D_0Q_0=\1$.  Then we have a direct sum decomposition
\begin{equation}
W^{k,2,\delta}(C_0,u_0^\ast T\hat X_0)_{\xi(q_{v,i})\in T\hat D_{v,i}}\oplus\J\oplus E_J=\ker D_0\oplus\im Q_0.
\end{equation}
In fact, choosing a bounded right inverse $Q_0$ is equivalent to choosing a closed complement $\im Q_0$ of $\ker D_0$.  The classical Banach space implicit function theorem (taking as input $Q_0$ and the quadratic estimate \eqref{derivislip}) then implies that the map from $\F_0^{-1}(0)$ to $\ker D_0$ by projection along $\im Q_0$ is a local diffeomorphism near zero.

Our goal is to generalize this setup to $\g$ in a neighborhood of zero (using \eqref{RIdiagram} and Lemma \ref{RIestimates}).  Namely, we will construct a right inverse $Q_\g$ for $D_\g$ (equivalently, we will choose a complement $\im Q_\g$ for $\ker D_\g$), so we that have a direct sum decomposition
\begin{equation}\label{nonzerogdirectsumdecomp}
W^{k,2,\delta}(C_\g,u_\g^\ast T\hat X_\g)_{\begin{smallmatrix}\xi(q_{v,i})\in T\hat D_{v,i}\hfill\\\pi_{\RR\partial_s\oplus\RR R_\lambda}\xi(q_e'')=0\hfill\end{smallmatrix}}\oplus\J\oplus E_J=\ker D_\g\oplus\im Q_\g.
\end{equation}
The same implicit function theorem argument applies as long as $\|Q_\g\|$ is bounded uniformly for $\g$ near zero.  Note also that uniform boundedness of $Q_\g$ implies that both projections in \eqref{nonzerogdirectsumdecomp} are uniformly bounded (since they are given by $\1-Q_\g D_\g$ and $Q_\g D_\g$ respectively).

Now to ensure that the individual parameterizations of $\F_\g^{-1}(0)$ by $K_\g$ near zero fit together continuously as $\g$ varies, we also need to show that the direct sum decomposition \eqref{nonzerogdirectsumdecomp} is ``continuous in $\g$'' in some sense.  Let us now describe more precisely the sense we mean.  For some points $w_i\in C_0\setminus(\{p_{v,e}\}_{v,e}\cup N)$, consider the linear functional
\begin{equation*}
L_0:W^{k,2,\delta}(C_0,u_0^\ast T\hat X_0)_{\xi(q_{v,i})\in T\hat D_{v,i}}\oplus\J\oplus E_J\to\biggl(\bigoplus_iT_{u_0(w_i)}\hat X_0\oplus\J\oplus E_J\biggr)\Bigm/B
\end{equation*}
for some subspace $B$ projecting trivially onto $E_{J\setminus I}$.  Fix $B$ so that $L_0|_{\ker D_0}$ is an isomorphism; this is possible since $\ker D_0\to E_{J\setminus I}$ is surjective.  Since $L_0|_{\ker D_0}$ is an isomorphism, we have a direct sum decomposition
\begin{equation}
W^{k,2,\delta}(C_0,u_0^\ast T\hat X_0)_{\xi(q_{v,i})\in T\hat D_{v,i}}\oplus\J\oplus E_J=\ker D_0\oplus\ker L_0.
\end{equation}
Now denote by
\begin{equation*}
L_\g:W^{k,2,\delta}(C_\g,u_\g^\ast T\hat X_\g)_{\begin{smallmatrix}\xi(q_{v,i})\in T\hat D_{v,i}\hfill\\\pi_{\RR\partial_s\oplus\RR R_\lambda}\xi(q_e'')=0\hfill\end{smallmatrix}}\oplus\J\oplus E_J\to\biggl(\bigoplus_iT_{u_\g(w_i)}\hat X_\g\oplus\J\oplus E_J\biggr)\Bigm/B
\end{equation*}
the ``same'' linear functional, where $w_i\in C_\g$ denote the descents of $w_i\in C_0$, so that there is a natural identification $T_{u_\g(w_i)}\hat X_\g=T_{u_0(w_i)}\hat X_0$.  We will show that $L_\g|_{\ker D_\g}$ is still an isomorphism, and hence there is a direct sum decomposition
\begin{equation}\label{nonzerogdirectsumdecompwithL}
W^{k,2,\delta}(C_\g,u_\g^\ast T\hat X_\g)_{\begin{smallmatrix}\xi(q_{v,i})\in T\hat D_{v,i}\hfill\\\pi_{\RR\partial_s\oplus\RR R_\lambda}\xi(q_e'')=0\hfill\end{smallmatrix}}\oplus\J\oplus E_J=\ker D_\g\oplus\ker L_\g.
\end{equation}
We will construct $Q_\g$ with $\im Q_\g=\ker L_\g$, i.e.\ the direct sum decompositions \eqref{nonzerogdirectsumdecomp} and \eqref{nonzerogdirectsumdecompwithL} coincide.  We will also define natural ``kernel gluing'' isomorphisms $\ker D_0\xrightarrow\sim\ker D_\g$ which agree with $L_\g^{-1}\circ L_0$.

\subsubsection{Bounded right inverses and kernel gluing IV: construction}

We now construct the right inverses $Q_\g$ and the kernel gluing isomorphisms $\ker D_0\xrightarrow\sim\ker D_\g$ satisfying the desired properties discussed above.

We first recall the following general construction, which allows one to upgrade an ``approximate right inverse'' into a (true) right inverse.

\begin{definition}
Let $D:X\to Y$ be a bounded linear map between Banach spaces, and let $T:Y\to X$ be an \emph{approximate right inverse}, meaning that $\left\|\1-DT\right\|<1$.  Then there is a (necessarily unique) associated right inverse $Q:Y\to X$ with the same image $\im Q=\im T$, namely $Q:=T(DT)^{-1}$, where $DT:Y\to Y$ is invertible by the geometric series $\sum_{k=0}^\infty(\1-DT)^k$.  Moreover, we have (trivially) that $\left\|Q\right\|\leq\left\|T\right\|(1-\left\|\1-DT\right\|)^{-1}$.
\end{definition}

To define the right inverse $Q_\g$, first define an approximate right inverse $T_\g$ of $D_\g$ as the following composition of maps in \eqref{RIdiagram}:
\begin{equation}
T_\g:={\calib}\circ{\glue}\circ{\PT}\circ{Q_0}\circ{\PT}\circ{\id^{1,0}}\circ{\brk},
\end{equation}
where $Q_0$ denotes the fixed right inverse of $D_0$ defined by the property that $\im Q_0=\ker L_0$.  A consequence of the estimates \eqref{flatteningdifference}--\eqref{calibdifference} (expressing the fact that \eqref{RIdiagram} almost commutes) is that $\|\1-D_\g T_\g\|\to 0$ as $\g\to 0$ (see \cite[Lemma B.7.6]{pardonimplicitatlas}).  Let $Q_\g$ denote the associated true right inverse, which is uniformly bounded for $\g$ near zero (since all the maps in \eqref{RIdiagram} are uniformly bounded).  Note that ${L_\g}\circ{\calib}\circ{\glue}\circ{\PT}=L_0$ by inspection, so $\im Q_0=\ker L_0$ implies that $\im Q_\g=\im T_\g\subseteq\ker L_\g$.

We define the kernel gluing isomorphism $\ker D_0\xrightarrow\sim\ker D_\g$ as the composition
\begin{equation}\label{kerDgluingduringproof}
{(\1-Q_\g D_\g)}\circ{\calib}\circ{\glue}\circ{\PT}:\ker D_0\to\ker D_\g.
\end{equation}
Note that $L_\g\circ{(1-Q_\g D_\g)}\circ{\calib}\circ{\glue}\circ{\PT}=L_0$ by inspection, and hence \eqref{kerDgluingduringproof} is injective.  Now we have by definition that $\ind D_0=\mu(T')-\#V_s(T')-\#N+\dim E_J$ and $\ind D_\g=\mu(T'_\g)-\#V_s(T')-\#N+\dim E_J$, where $T'\to T'_\g$ denotes the image of $\g$ under the map $G_{T'//T}\to\SSS_{T'//T}$.  These indices coincide as remarked below Definition \ref{indexdefn}, so \eqref{kerDgluingduringproof} is an isomorphism as both $D_0$ and $D_\g$ are surjective.  Since $\eqref{kerDgluingduringproof}$ is an isomorphism, so is $L_\g|_{\ker D_\g}$, and it thus follows that the inclusion $\im Q_\g\subseteq\ker L_\g$ is in fact an equality $\im Q_\g=\ker L_\g$.

\begin{remark}
It is possible to prove that \eqref{kerDgluingduringproof} is surjective directly at the cost of proving a few more estimates (this is similar to \cite[Proposition 9]{floerhofer}).  This thus gives an \emph{a priori} proof that $\mu(T)=\mu(T')$ for $T'\to T$.

Let us sketch the argument.  Let $\ell\in\ker D_\g$.  The \emph{a priori} estimate of exponential convergence to a trivial cylinder due to Hofer--Wysocki--Zehnder \cite[Theorems 1.1, 1.2, and 1.3]{hwzsmallarea} (restated here as Proposition \ref{hwzcylinderestimate}) has an easier linear analogue which says that in any neck, $\ell$ decays rapidly to a constant vector field tangent to the trivial cylinder; moreover, this constant vanishes in Reeb necks since $\pi_{\RR\partial_s\oplus\RR R_\lambda}\ell(q_e'')=0$.  It follows that we can apply an ``ungluing'' operation to produce a $\kappa$ of commensurable norm $\|\kappa\|\asymp\|\ell\|$ with $\|D_0\kappa\|=o(1)\cdot\|\kappa\|$ and $L_0\kappa=L_\g\ell$.  Now we have $\|\ell-({\calib}\circ{\glue}\circ{\PT})(\kappa)\|=o(1)\cdot\|\ell\|$ by explicit calculation, and it follows that the image of $(\1-Q_0D_0)\kappa\in\ker D_0$ under \eqref{kerDgluingduringproof} is within distance $o(1)\cdot\|\ell\|$ of $\ell$.  Since this holds for all $\ell\in\ker D_\g$, it follows that \eqref{kerDgluingduringproof} is surjective.
\end{remark}

\subsection{Gluing map}\label{localgluingproofIII}

We now define the gluing map and show that it is a local homeomorphism.  This is the ``endgame'' of the gluing argument, where we deduce the desired results from the technical work performed above.

\subsubsection{Definition of the gluing map}

We first recall (following our sketch in \S\ref{rightinversegoal}) how our work above implies that $\F_\g^{-1}(0)$ is a manifold near zero and that projection along $\im Q_\g$ provides a diffeomorphism between it and $\ker D_\g$ near zero.

We have fixed a right inverse $Q_\g$ for $D_\g=\F_\g'(0,\cdot)$ with $\|Q_\g\|$ bounded uniformly for $\g$ near zero.  Now it follows from \eqref{derivislip} that $Q_\g$ is an approximate right inverse to $\F_\g'$ over the ball of some radius $c_{k,\delta}'>0$ (uniform in $\g$).  Hence over this ball of radius $c_{k,\delta}'>0$, the operator $\F_\g'$ is surjective, i.e.\ $\F_\g$ is transverse to zero.  By the Banach space implicit function theorem, it thus follows that $\F_\g^{-1}(0)$ is a $C^\ell$-submanifold (for $k\geq\ell+2$) which is transverse to $\im Q_\g$.

Let us now show that map $\ker D_\g\to\F_\g^{-1}(0)$ given by projection along $\im Q_\g$ (is well-defined and) is a diffeomorphism near zero.  The key point is that the map $\1-Q_\g\F_\g$ is a contraction mapping when restricted to any slice $(\xi+\im Q_\g)\cap B(c_{k,\delta}')$ with $\xi$ sufficiently small in terms of $c_{k,\delta}'$.  This follows from \eqref{quadestgoal} and \eqref{Fgzerosmall}, which moreover imply that the contraction constant approaches zero (uniformly in $\g$) as $c_{k,\delta}'\to 0$.  This gives the desired result, and moreover shows that the projection along $\im Q_\g$ to $\F_\g^{-1}(0)$ is given (over the whole ball $B(c_{k,\delta}')$) by the limit of the Newton--Picard iteration $\xi\mapsto\xi-Q_\g\F_\g\xi$.

We can now define the gluing map, by precomposing the above local diffeomorphisms $\ker D_\g\to\F_\g^{-1}(0)$ with the kernel gluing isomorphisms $\ker D_0\xrightarrow\sim\ker D_\g$ and letting $\g$ vary.  In other words, the gluing map
\begin{equation}\label{gluingmap}
\Bigl(G_{T'//T}\times\CC^N\times\ker D_0,(0,0)\Bigr)\to\Bigl(\Mbar(T)_J,x_0\Bigr)
\end{equation}
sends $(\g,\kappa)$ to the map $\exp_{u_\g}\kappa_\g^\infty:C_\g\to\hat X_\g$, where $\kappa_\g^\infty\in\F_\g^{-1}(0)$ is the unique intersection point $\F_\g^{-1}(0)\cap(({(\1-Q_\g D_\g)}\circ{\calib}\circ{\glue}\circ{\PT})(\kappa)+\im Q_\g)$ in $B(c_{k,\delta}')$.  The discrete data for $\exp_{u_\g}\xi:C_\g\to\hat X_\g$ is naturally inherited from that for $u_0:C_0\to\hat X_0$.  Since $\F_\g$ is transverse to zero at $\kappa_\g^\infty$, it follows that the image of the gluing map is contained in $\Mbar(T)_J^\reg$.

The gluing map evidently commutes with the maps from both sides to $\SSS_{T'//T}\times\overline{\s(T)}\times E_{J\setminus I}$ (recall that $\im Q_\g=\ker L_\g$ projects trivially onto $E_{J\setminus I}$ by definition).

Let us also note here that the inequality $\mu(T')-\#V_s(T')-2\#N+\dim E_I\geq 0$ follows from the fact that $D_0$ is surjective, $\ker D_0\to E_{J\setminus I}$ is surjective, and $\ind D_0=\mu(T')-\#V_s(T')-2\#N+\dim E_J$.

\subsubsection{Properties of the gluing map}

We now show that the gluing map is a local homeomorphism.  More precisely, we show that the gluing map is continuous and that it is a local bijection, for which the \emph{a priori} estimates recalled in \S\ref{ellipticestimates} will be crucial.  We then appeal to some point set topology to see that the gluing map is a local homeomorphism (though continuity of the inverse could also be proven directly).

\begin{lemma}\label{gluingcontinuous}
The gluing map is continuous.
\end{lemma}

\begin{proof}
Recall that $\kappa_\g^\infty$ may be described via the Newton--Picard iteration as follows.  Namely, $\kappa_\g^\infty=\lim_{i\to\infty}\kappa_\g^i$, where
\begin{align}
\kappa_\g^{i+1}&=\kappa_\g^i-Q_\g\F_\g\kappa_\g^i,\\
\kappa_\g^0&=({\calib}\circ{\glue}\circ{\PT})(\kappa).
\end{align}
Note that there is no $\1-Q_\g D_\g$ in the definition of $\kappa_\g^0$ (this is ok since $Q_\g D_\g\kappa_\g^0\in\im Q_\g$).

Now suppose $(\g_i,\kappa_i)\to(\g,\kappa)$ (a convergent net), and let us show that the net $\exp_{u_{\g_i}}(\kappa_i)_{\g_i}^\infty:C_{\g_i}\to\hat X_{\g_i}$ approaches $\exp_{u_\g}\kappa_\g^\infty:C_\g\to\hat X_\g$ in the Gromov topology.

First, we claim that $\|(\kappa_i)^\infty_{\g_i}-\kappa^\infty_{\g_i}\|_{k,2,\delta}\to 0$.  By uniform convergence of the Newton--Picard iteration, it suffices to show that $\|(\kappa_i)^n_{\g_i}-\kappa^n_{\g_i}\|_{k,2,\delta}\to 0$ for all $n$.  The case $n=0$ follows from uniform boundedness of ${\calib}\circ{\glue}\circ{\PT}$.  The desired claim then follows by induction on $n$ using \eqref{derivislip}.  Now the claim implies that $\|(\kappa_i)^\infty_{\g_i}-\kappa^\infty_{\g_i}\|_\infty\to 0$, and thus it suffices to show that
\begin{equation}
\exp_{u_{\g_i}}\kappa_{\g_i}^\infty:C_{\g_i}\to\hat X_{\g_i}\text{ approaches }\exp_{u_\g}\kappa_\g^\infty:C_\g\to\hat X_\g.
\end{equation}

Define $(\kappa_\g^\infty)_{\g_i}^0$ by (as the notation suggests) pregluing $\kappa_\g^\infty$ from $C_\g$ to $C_{\g_i}$ as follows.  In any neck of $C_{\g_i}$ corresponding to a pair of ends of $C_\g$, we preglue via ${\calib}\circ{\glue}\circ{\PT}$ as before (this operation is local to the ends/neck).  In any neck of $C_{\g_i}$ corresponding to a neck of $C_\g$, we simply use parallel transport and a nice diffeomorphism between the two necks (say, converging to the identity map in the $C^\infty$ topology as $\g_i\to\g$).  We may assume without loss of generality that there are no pairs of ends of $C_{\g_i}$ corresponding to a neck of $C_\g$.

Now we claim that
\begin{equation}\label{gtogipregluingest}
\|\F_{\g_i}((\kappa_\g^\infty)_{\g_i}^0)\|_{k-1,2,\delta}\to 0.
\end{equation}
Away from the necks/ends, the $1$-form part of $\F_{\g_i}((\kappa_\g^\infty)_{\g_i}^0)$ is nonzero only because of using $\hat J_{(\g_i)_t}$ in place of $\hat J_{\g_t}$ and because of using $\phi_\alpha^{\g_i,\xi}$ in place of $\phi_\alpha^{\g,\xi}$.  We have $\hat J_{(\g_i)_t}\to\hat J_{\g_t}$ and $\phi_\alpha^{\g_i,\xi}\to\phi_\alpha^{\g,\xi}$, so the desired estimate follows since $(\kappa_\g^\infty)_{\g_i}^0=\kappa_\g^\infty$ away from the ends/necks.  Over the Reeb ends of $C_{\g_i}$, the $1$-form part vanishes.  Over the Reeb necks of $C_{\g_i}$ corresponding to necks of $C_\g$, the $1$-form part approaches zero.  Over the Reeb necks of $C_{\g_i}$ corresponding to pairs of ends of $C_\g$, the estimate follows from the exponential decay of $\kappa_\g^\infty$ and $u_0$ and the fact that $\delta<1$ and $\delta<\delta_\gamma$ for all asymptotic Reeb orbits $\gamma$.  The same applies to nodal necks, in addition considering the convergence $\hat J_{(\g_i)_t}\to\hat J_{\g_t}$. The $\RR^{V_s(T')}$ part clearly approaches zero.  This proves \eqref{gtogipregluingest}.

Now we consider the Newton--Picard iteration starting at $(\kappa_\g^\infty)_{\g_i}^0$, with limit $(\kappa_\g^\infty)_{\g_i}^\infty\in\F_{\g_i}^{-1}(0)$.  By uniform contraction of the iteration and \eqref{gtogipregluingest}, we conclude that $\|(\kappa_\g^\infty)_{\g_i}^0-(\kappa_\g^\infty)_{\g_i}^\infty\|_{k,2,\delta}\to 0$.  It thus follows that
\begin{equation}
\exp_{u_{\g_i}}(\kappa_\g^\infty)_{\g_i}^\infty:C_{\g_i}\to\hat X_{\g_i}\text{ approaches }\exp_{u_\g}\kappa_\g^\infty:C_\g\to\hat X_\g.
\end{equation}
Now we claim that $(\kappa_\g^\infty)_{\g_i}^\infty=\kappa_{\g_i}^\infty$, which is clearly sufficient to conclude the proof.

By construction, we have $L_{\g_i}((\kappa_\g^\infty)_{\g_i}^\infty)=L_{\g_i}((\kappa_\g^\infty)_{\g_i}^0)=L_\g(\kappa_\g^\infty)=L_\g(\kappa_\g^0)=L_0(\kappa)$ and similarly $L_{\g_i}(\kappa_{\g_i}^\infty)=L_{\g_i}(\kappa_{\g_i}^0)=L_0(\kappa)$.  Thus $(\kappa_\g^\infty)_{\g_i}^\infty$ and $\kappa_{\g_i}^\infty$ differ by an element of $\ker L_{\g_i}=\im Q_{\g_i}$, which is enough.
\end{proof}

\begin{lemma}\label{gluingbijective}
The gluing map \eqref{gluingmap} is a local bijection.  That is, for every sufficiently small neighborhood $U$ of the basepoint in the domain, there exists an open neighborhood $V$ of the basepoint in the target such that every $v\in V$ has a unique inverse image $u\in U$.
\end{lemma}

\begin{proof}
Let $x\in\Mbar(T)_J$, and denote the corresponding map by $u:C\to\hat X$.  We assume that $x$ is sufficiently close to $x_0$, and we will show that $x$ has a unique inverse image under the gluing map \eqref{gluingmap}.

Concretely, $x$ close to $x_0$ in the Gromov topology means the following.  We may identify $C$ with $(C_\beta,j_w)$, for arbitrarily small gluing parameters $\beta\in\CC^{E^\eint(T')}\times\CC^N$ and an almost complex structure $j_w$ which agrees with $j_0$ except over a compact set (away from the ends/necks) where it is arbitrarily $C^\infty$ close to $j_0$.  Furthermore, the map $u:(C_\beta,j_w)\to\hat X$ is arbitrarily $C^0$-close to $u_0$ away from the Reeb ends/necks (which by Lemma \ref{apriorijhol} implies arbitrarily $C^\infty$ close, at least away from the nodes of $C_0$ which are resolved in $C$; the behavior of $u$ in these neighborhoods is controlled by Lemma \ref{apriorinode}).  We will discuss the situation over the Reeb ends/necks shortly.

Now observe that there are unique points $q_{v,i}\in C$ close to $q_{v,i}\in C_0$ where $u$ intersects $D_{v,i}$ transversally.  These points give rise to unique points $q_v'\in C$ according to the sections chosen in \S\ref{qvprimesec}.  Now regarding $u(q_v')\in\hat X$ as lying on the ``zero section'' determines (uniquely) a gluing parameter $\g\in G_{T'//T}$ and an isomorphism $\hat X=\hat X_\g$.  Thus $x$ corresponds to a map $u:(C_\beta,j_w)\to\hat X_\g$.

Now in any Reeb end $[0,\infty)\times S^1\subseteq C_\beta$ or Reeb neck $[0,S]\times S^1\subseteq C_\beta$, we apply \cite[Theorems 1.1, 1.2, and 1.3]{hwzsmallarea} (restated here as Proposition \ref{hwzcylinderestimate}) to see that $u$ decays exponentially to a trivial cylinder.  For unglued edges, we combine this with the fact that $u$ is $C^\infty$-close to $u_0$ over an arbitrarily compact set (which we can choose to go very deep into the end) to conclude that the tangent space marking of $C_\beta$ at $\{p_{v,e}\}$ induced by the cylindrical coordinates on $\hat X_\g$ is arbitrarily close to the tangent space marking descended from $C_0$.
The same reasoning implies that the gluing parameters $\beta\in\CC^{E^\eint(T')}$ (for the glued edges) are given by $L_e^{-1}g_e+o(1)$ and $\theta_e+o(1)$ (i.e.\ are very close to those coming from $\g$), and we recover the point $q_e''\in C_\beta$ as the inverse image of $\exp_{u_{0|\g}(q_e'')}(\ker\pi_{\RR\partial_s\oplus\RR R_\lambda})$.  From these estimates, we conclude that $x$ corresponds to a map $u:(C_\g,j_y)\to\hat X_\g$ for $\g\in G_{T'//T}\times\CC^N$ which is arbitrarily $C^\infty$-close to $u_\g$ (including over the ends/necks, with respect to the cylindrical coordinates) respecting the tangent markings at $p_{v,e}$ and sending $q_e''$ to $\exp_{u_{0|\g}(q_e'')}(\ker\pi_{\RR\partial_s\oplus\RR R_\lambda})$.  Here $\g$ and $y$ can be assumed arbitrarily small, and we use the fact that \eqref{gluingdiffeo} is a local diffeomorphism uniformly in $\g\in G_{T'//T}$.

The injectivity radius of our fixed exponential map on $\hat X_\g$ is bounded below uniformly, so $x$ corresponds uniquely to some pair $(u=\exp_{u_\g}\xi:(C_\g,j_y)\to\hat X_\g,e)$, where
\begin{equation}
(\xi,y,e-e_0)\in W^{k,2,\delta}(C_\g,u_\g^\ast T\hat X_\g)_{\begin{smallmatrix}\xi(q_{v,i})\in T\hat D_{v,i}\hfill\\\pi_{\RR\partial_s\oplus\RR R_\lambda}\xi(q_e'')=0\hfill\end{smallmatrix}}\oplus\J\oplus E_J
\end{equation}
has arbitrarily small norm due to the \emph{a priori} estimates on $u$ from Lemmas \ref{apriorijhol} and \ref{apriorinode} and Proposition \ref{hwzcylinderestimate}.  Now we observed in the gluing construction that for any fixed sufficiently small $\g$, the gluing map gives a local diffeomorphism between $\ker D_0$ and $\F_\g^{-1}(0)$, over a ball of size uniformly bounded below.  Thus our map is uniquely in the image of the gluing map.
\end{proof}

Since the target of the gluing map is Hausdorff and the domain locally compact Hausdorff, it follows from continuity (Lemma \ref{gluingcontinuous}) and bijectivity in a small neighborhood (Lemma \ref{gluingbijective}) that the gluing map \eqref{gluingmap} is in fact a local homeomorphism, thus completing the proof of Theorem \ref{localgluing}.
\end{proof}

\subsection{Orientations}\label{localorientationsproof}

We now prove the compatibility of the geometric and analytic maps on orientation lines, namely that \eqref{gluingorientations} commutes.  We rely heavily on the gluing construction in \S\S\ref{localgluingproofI}--\ref{localgluingproofIII}.

\begin{proof}[Proof of Theorem \ref{localorientations}]
The gluing map \eqref{gluingmap} (in the case $I=J$) allows us to describe the left vertical ``geometric'' map in \eqref{gluingorientations} (i.e.\ the map induced by the topological structure of $\Mbar(T)_I^\reg$) near the basepoint $x_0$ as follows.  Note that we may assume (for convenience) that $N=\varnothing$, since this locus is dense.  Recall that there is a canonical identification
\begin{equation}
\oo_{\ker D_0}=\oo^\circ_{T'}\otimes(\oo_\RR^\vee)^{\otimes V_s(T')}\otimes\oo_{E_I}.
\end{equation}
Now consider sufficiently small $(\g,\kappa)$, where $\g\in G_{T'//T}$ lies in the top stratum $(T'\to T)$.  Now $\F_\g^{-1}(0)$ is a submanifold with tangent space $\ker\F_\g'(\kappa_\g^\infty,\cdot)$, and there is a canonical identification
\begin{equation}
\oo_{\ker\F_\g'(\kappa_\g^\infty,\cdot)}=\oo^\circ_T\otimes(\oo_\RR^\vee)^{\otimes V_s(T')}\otimes\oo_{E_I}.
\end{equation}
The gluing map is differentiable with respect to $\kappa$ since $\F_\g^{-1}(0)$ is a submanifold transverse to $\im Q_\g$.  Its derivative is clearly given by the composition of $\ker D_0\to\ker D_\g$ and the map $\ker D_\g\to\ker\F_\g'(\kappa_\g^\infty,\cdot)$ given by projecting off $\im Q_\g$.  This map $\ker D_0\to\ker\F_\g'(\kappa_\g^\infty,\cdot)$ thus gives the ``geometric'' map $\oo^\circ_{T'}\to\oo^\circ_T$ when combined with the isomorphisms above.

Now the right inverse $Q_\g$ to $D_\g=\F_\g'(0,\cdot)$ is an approximate right inverse to $\F_\g'(\xi,\cdot)$ for all $\xi\in B(c_{k,\delta}')$ by \eqref{derivislip}.  Hence the kernel $\ker\F_\g'(\xi,\cdot)$ forms a vector bundle over $B(c_{k,\delta}')$ which is canonically oriented by $\oo^\circ_T\otimes(\oo_\RR^\vee)^{\otimes V_s(T')}\otimes\oo_{E_I}$, and the map $\ker D_0\to\ker\F_\g'(\xi,\cdot)$ makes sense for all such $\xi$.  Thus the geometric map on orientations is also given by the simpler map at $\xi=0$
\begin{equation}\label{kernelpregluinggeo}
\ker D_0\xrightarrow{{(1-Q_\g D_\g)}\circ{\calib}\circ{\glue}\circ{\PT}}\ker D_\g
\end{equation}
combined with the canonical identifications
\begin{align}
\oo_{\ker D_0}&=\oo^\circ_{T'}\otimes(\oo_\RR^\vee)^{\otimes V_s(T')}\otimes\oo_{E_I},\\
\oo_{\ker D_\g}&=\oo^\circ_T\otimes(\oo_\RR^\vee)^{\otimes V_s(T')}\otimes\oo_{E_I}.
\end{align}
Now this map \eqref{kernelpregluinggeo} is precisely the sort of kernel pregluing map which defines the ``analytic'' map on orientations.

Strictly speaking, the analytic map $\oo^\circ_{T'}\to\oo^\circ_T$ is defined using a slightly different linearized operator (no $\J$, $E_I$, or point conditions), but this is only a ``finite-dimensional'' difference (note also that $\J$ is canonically oriented since it is a complex vector space).  It is thus straightforward to relate them and see that they give rise to the same analytic map on orientations.
\end{proof}

\bibliographystyle{alpha-a}
\bibliography{contact}
\addcontentsline{toc}{section}{References}

\end{document}